\documentclass[11pt]{amsart}
\usepackage{amsfonts,amssymb,amsmath,euscript}
\usepackage{graphicx}
\usepackage[matrix,arrow,curve]{xy}
\usepackage{wrapfig}

    \newtheorem{thm}{Theorem}                     [section]
    \newtheorem{thm*}{Theorem}
    \newtheorem{prop}[thm]{Proposition}
    \newtheorem{lemma}[thm]{Lemma}
    \newtheorem{cor}[thm]{Corollary}

    \newtheorem{lemma*}{Lemma}    

    \newtheorem{defn}[thm]{Definition}                 
    \newtheorem{obs}{Observation}                 
    \newtheorem{conj}{Conjecture}                 
    \newtheorem{rems}[thm]{Remark}                     
    \newtheorem{ques}{Question}            

    \newtheorem{rems*}{Remark}   

\newcommand{\ndef}{\newcommand*}
\def\rndef{\renewcommand}

\ndef{\myaddress}[1]{\begin{center} \it\small #1 \end{center}}

\ndef{\clA}{{\mathcal A}} \ndef{\rmA}{{\mathrm A}} \ndef{\mbA}{{\mathbb A}} \ndef{\bfA}{{\mathbf A}} \ndef{\euA}{{\EuScript A}} \ndef{\frA}{{\mathfrak A}}
\ndef{\clB}{{\mathcal B}} \ndef{\rmB}{{\mathrm B}} \ndef{\mbB}{{\mathbb B}} \ndef{\bfB}{{\mathbf B}} \ndef{\euB}{{\EuScript B}} \ndef{\frB}{{\mathfrak B}}
\ndef{\clC}{{\mathcal C}} \ndef{\rmC}{{\mathrm C}} \ndef{\mbC}{{\mathbb C}} \ndef{\bfC}{{\mathbf C}} \ndef{\euC}{{\EuScript C}} \ndef{\frC}{{\mathfrak C}}
\ndef{\clD}{{\mathcal D}} \ndef{\rmD}{{\mathrm D}} \ndef{\mbD}{{\mathbb D}} \ndef{\bfD}{{\mathbf D}} \ndef{\euD}{{\EuScript D}} \ndef{\frD}{{\mathfrak D}}
\ndef{\clE}{{\mathcal E}} \ndef{\rmE}{{\mathrm E}} \ndef{\mbE}{{\mathbb E}} \ndef{\bfE}{{\mathbf E}} \ndef{\euE}{{\EuScript E}} \ndef{\frE}{{\mathfrak E}}
\ndef{\clF}{{\mathcal F}} \ndef{\rmF}{{\mathrm F}} \ndef{\mbF}{{\mathbb F}} \ndef{\bfF}{{\mathbf F}} \ndef{\euF}{{\EuScript F}} \ndef{\frF}{{\mathfrak F}}
\ndef{\clG}{{\mathcal G}} \ndef{\rmG}{{\mathrm G}} \ndef{\mbG}{{\mathbb G}} \ndef{\bfG}{{\mathbf G}} \ndef{\euG}{{\EuScript G}} \ndef{\frG}{{\mathfrak G}}
\ndef{\clH}{{\mathcal H}} \ndef{\rmH}{{\mathrm H}} \ndef{\mbH}{{\mathbb H}} \ndef{\bfH}{{\mathbf H}} \ndef{\euH}{{\EuScript H}} \ndef{\frH}{{\mathfrak H}}
\ndef{\clI}{{\mathcal I}} \ndef{\rmI}{{\mathrm I}} \ndef{\mbI}{{\mathbb I}} \ndef{\bfI}{{\mathbf I}} \ndef{\euI}{{\EuScript I}} \ndef{\frI}{{\mathfrak I}}
\ndef{\clJ}{{\mathcal J}} \ndef{\rmJ}{{\mathrm J}} \ndef{\mbJ}{{\mathbb J}} \ndef{\bfJ}{{\mathbf J}} \ndef{\euJ}{{\EuScript J}} \ndef{\frJ}{{\mathfrak J}}
\ndef{\clK}{{\mathcal K}} \ndef{\rmK}{{\mathrm K}} \ndef{\mbK}{{\mathbb K}} \ndef{\bfK}{{\mathbf K}} \ndef{\euK}{{\EuScript K}} \ndef{\frK}{{\mathfrak K}}
\ndef{\clL}{{\mathcal L}} \ndef{\rmL}{{\mathrm L}} \ndef{\mbL}{{\mathbb L}} \ndef{\bfL}{{\mathbf L}} \ndef{\euL}{{\EuScript L}} \ndef{\frL}{{\mathfrak L}}
\ndef{\clM}{{\mathcal M}} \ndef{\rmM}{{\mathrm M}} \ndef{\mbM}{{\mathbb M}} \ndef{\bfM}{{\mathbf M}} \ndef{\euM}{{\EuScript M}} \ndef{\frM}{{\mathfrak M}}
\ndef{\clN}{{\mathcal N}} \ndef{\rmN}{{\mathrm N}} \ndef{\mbN}{{\mathbb N}} \ndef{\bfN}{{\mathbf N}} \ndef{\euN}{{\EuScript N}} \ndef{\frN}{{\mathfrak N}}
\ndef{\clO}{{\mathcal O}} \ndef{\rmO}{{\mathrm O}} \ndef{\mbO}{{\mathbb O}} \ndef{\bfO}{{\mathbf O}} \ndef{\euO}{{\EuScript O}} \ndef{\frO}{{\mathfrak O}}
\ndef{\clP}{{\mathcal P}} \ndef{\rmP}{{\mathrm P}} \ndef{\mbP}{{\mathbb P}} \ndef{\bfP}{{\mathbf P}} \ndef{\euP}{{\EuScript P}} \ndef{\frP}{{\mathfrak P}}
\ndef{\clQ}{{\mathcal Q}} \ndef{\rmQ}{{\mathrm Q}} \ndef{\mbQ}{{\mathbb Q}} \ndef{\bfQ}{{\mathbf Q}} \ndef{\euQ}{{\EuScript Q}} \ndef{\frQ}{{\mathfrak Q}}
\ndef{\clR}{{\mathcal R}} \ndef{\rmR}{{\mathrm R}} \ndef{\mbR}{{\mathbb R}} \ndef{\bfR}{{\mathbf R}} \ndef{\euR}{{\EuScript R}} \ndef{\frR}{{\mathfrak R}}
\ndef{\clS}{{\mathcal S}} \ndef{\rmS}{{\mathrm S}} \ndef{\mbS}{{\mathbb S}} \ndef{\bfS}{{\mathbf S}} \ndef{\euS}{{\EuScript S}} \ndef{\frS}{{\mathfrak S}}
\ndef{\clT}{{\mathcal T}} \ndef{\rmT}{{\mathrm T}} \ndef{\mbT}{{\mathbb T}} \ndef{\bfT}{{\mathbf T}} \ndef{\euT}{{\EuScript T}} \ndef{\frT}{{\mathfrak T}}
\ndef{\clU}{{\mathcal U}} \ndef{\rmU}{{\mathrm U}} \ndef{\mbU}{{\mathbb U}} \ndef{\bfU}{{\mathbf U}} \ndef{\euU}{{\EuScript U}} \ndef{\frU}{{\mathfrak U}}
\ndef{\clV}{{\mathcal V}} \ndef{\rmV}{{\mathrm V}} \ndef{\mbV}{{\mathbb V}} \ndef{\bfV}{{\mathbf V}} \ndef{\euV}{{\EuScript V}} \ndef{\frV}{{\mathfrak V}}
\ndef{\clW}{{\mathcal W}} \ndef{\rmW}{{\mathrm W}} \ndef{\mbW}{{\mathbb W}} \ndef{\bfW}{{\mathbf W}} \ndef{\euW}{{\EuScript W}} \ndef{\frW}{{\mathfrak W}}
\ndef{\clX}{{\mathcal X}} \ndef{\rmX}{{\mathrm X}} \ndef{\mbX}{{\mathbb X}} \ndef{\bfX}{{\mathbf X}} \ndef{\euX}{{\EuScript X}} \ndef{\frX}{{\mathfrak X}}
\ndef{\clY}{{\mathcal Y}} \ndef{\rmY}{{\mathrm Y}} \ndef{\mbY}{{\mathbb Y}} \ndef{\bfY}{{\mathbf Y}} \ndef{\euY}{{\EuScript Y}} \ndef{\frY}{{\mathfrak Y}}
\ndef{\clZ}{{\mathcal Z}} \ndef{\rmZ}{{\mathrm Z}} \ndef{\mbZ}{{\mathbb Z}} \ndef{\bfZ}{{\mathbf Z}} \ndef{\euZ}{{\EuScript Z}} \ndef{\frZ}{{\mathfrak Z}}

\ndef{\tA}{{\widetilde A}} \ndef{\tcA}{{\widetilde\clA}} \ndef{\ttcA}{\widetilde{\tcA}} \ndef{\sfA}{{\textsf A}} \ndef{\ttA}{\widetilde{\tA}} \ndef{\dzA}{{A^\sharp}}
\ndef{\tB}{{\widetilde B}} \ndef{\tcB}{{\widetilde\clB}} \ndef{\ttcB}{\widetilde{\tcB}} \ndef{\sfB}{{\textsf B}} \ndef{\ttB}{\widetilde{\tB}} \ndef{\dzB}{{B^\sharp}}
\ndef{\tC}{{\widetilde C}} \ndef{\tcC}{{\widetilde\clC}} \ndef{\ttcC}{\widetilde{\tcC}} \ndef{\sfC}{{\textsf C}} \ndef{\ttC}{\widetilde{\tC}} \ndef{\dzC}{{C^\sharp}}
\ndef{\tD}{{\widetilde D}} \ndef{\tcD}{{\widetilde\clD}} \ndef{\ttcD}{\widetilde{\tcD}} \ndef{\sfD}{{\textsf D}} \ndef{\ttD}{\widetilde{\tD}} \ndef{\dzD}{{D^\sharp}}
\ndef{\tE}{{\widetilde E}} \ndef{\tcE}{{\widetilde\clE}} \ndef{\ttcE}{\widetilde{\tcE}} \ndef{\sfE}{{\textsf E}} \ndef{\ttE}{\widetilde{\tE}} \ndef{\dzE}{{E^\sharp}}
\ndef{\tF}{{\widetilde F}} \ndef{\tcF}{{\widetilde\clF}} \ndef{\ttcF}{\widetilde{\tcF}} \ndef{\sfF}{{\textsf F}} \ndef{\ttF}{\widetilde{\tF}} \ndef{\dzF}{{F^\sharp}}
\ndef{\tG}{{\widetilde G}} \ndef{\tcG}{{\widetilde\clG}} \ndef{\ttcG}{\widetilde{\tcG}} \ndef{\sfG}{{\textsf G}} \ndef{\ttG}{\widetilde{\tG}} \ndef{\dzG}{{G^\sharp}}
\ndef{\tH}{{\widetilde H}} \ndef{\tcH}{{\widetilde\clH}} \ndef{\ttcH}{\widetilde{\tcH}} \ndef{\sfH}{{\textsf H}} \ndef{\ttH}{\widetilde{\tH}} \ndef{\dzH}{{H^\sharp}}
\ndef{\tI}{{\widetilde I}} \ndef{\tcI}{{\widetilde\clI}} \ndef{\ttcI}{\widetilde{\tcI}} \ndef{\sfI}{{\textsf I}} \ndef{\ttI}{\widetilde{\tI}} \ndef{\dzI}{{I^\sharp}}
\ndef{\tJ}{{\widetilde J}} \ndef{\tcJ}{{\widetilde\clJ}} \ndef{\ttcJ}{\widetilde{\tcJ}} \ndef{\sfJ}{{\textsf J}} \ndef{\ttJ}{\widetilde{\tJ}} \ndef{\dzJ}{{J^\sharp}}
\ndef{\tK}{{\widetilde K}} \ndef{\tcK}{{\widetilde\clK}} \ndef{\ttcK}{\widetilde{\tcK}} \ndef{\sfK}{{\textsf K}} \ndef{\ttK}{\widetilde{\tK}} \ndef{\dzK}{{K^\sharp}}
\ndef{\tL}{{\widetilde L}} \ndef{\tcL}{{\widetilde\clL}} \ndef{\ttcL}{\widetilde{\tcL}} \ndef{\sfL}{{\textsf L}} \ndef{\ttL}{\widetilde{\tL}} \ndef{\dzL}{{L^\sharp}}
\ndef{\tM}{{\widetilde M}} \ndef{\tcM}{{\widetilde\clM}} \ndef{\ttcM}{\widetilde{\tcM}} \ndef{\sfM}{{\textsf M}} \ndef{\ttM}{\widetilde{\tM}} \ndef{\dzM}{{M^\sharp}}
\ndef{\tN}{{\widetilde N}} \ndef{\tcN}{{\widetilde\clN}} \ndef{\ttcN}{\widetilde{\tcN}} \ndef{\sfN}{{\textsf N}} \ndef{\ttN}{\widetilde{\tN}} \ndef{\dzN}{{N^\sharp}}
\ndef{\tO}{{\widetilde O}} \ndef{\tcO}{{\widetilde\clO}} \ndef{\ttcO}{\widetilde{\tcO}} \ndef{\sfO}{{\textsf O}} \ndef{\ttO}{\widetilde{\tO}} \ndef{\dzO}{{O^\sharp}}
\ndef{\tP}{{\widetilde P}} \ndef{\tcP}{{\widetilde\clP}} \ndef{\ttcP}{\widetilde{\tcP}} \ndef{\sfP}{{\textsf P}} \ndef{\ttP}{\widetilde{\tP}} \ndef{\dzP}{{P^\sharp}}
\ndef{\tQ}{{\widetilde Q}} \ndef{\tcQ}{{\widetilde\clQ}} \ndef{\ttcQ}{\widetilde{\tcQ}} \ndef{\sfQ}{{\textsf Q}} \ndef{\ttQ}{\widetilde{\tQ}} \ndef{\dzQ}{{Q^\sharp}}
\ndef{\tR}{{\widetilde R}} \ndef{\tcR}{{\widetilde\clR}} \ndef{\ttcR}{\widetilde{\tcR}} \ndef{\sfR}{{\textsf R}} \ndef{\ttR}{\widetilde{\tR}} \ndef{\dzR}{{R^\sharp}}
\ndef{\tS}{{\widetilde S}} \ndef{\tcS}{{\widetilde\clS}} \ndef{\ttcS}{\widetilde{\tcS}} \ndef{\sfS}{{\textsf S}} \ndef{\ttS}{\widetilde{\tS}} \ndef{\dzS}{{S^\sharp}}
\ndef{\tT}{{\widetilde T}} \ndef{\tcT}{{\widetilde\clT}} \ndef{\ttcT}{\widetilde{\tcT}} \ndef{\sfT}{{\textsf T}} \ndef{\ttT}{\widetilde{\tT}} \ndef{\dzT}{{T^\sharp}}
\ndef{\tU}{{\widetilde U}} \ndef{\tcU}{{\widetilde\clU}} \ndef{\ttcU}{\widetilde{\tcU}} \ndef{\sfU}{{\textsf U}} \ndef{\ttU}{\widetilde{\tU}} \ndef{\dzU}{{U^\sharp}}
\ndef{\tV}{{\widetilde V}} \ndef{\tcV}{{\widetilde\clV}} \ndef{\ttcV}{\widetilde{\tcV}} \ndef{\sfV}{{\textsf V}} \ndef{\ttV}{\widetilde{\tV}} \ndef{\dzV}{{V^\sharp}}
\ndef{\tW}{{\widetilde W}} \ndef{\tcW}{{\widetilde\clW}} \ndef{\ttcW}{\widetilde{\tcW}} \ndef{\sfW}{{\textsf W}} \ndef{\ttW}{\widetilde{\tW}} \ndef{\dzW}{{W^\sharp}}
\ndef{\tX}{{\widetilde X}} \ndef{\tcX}{{\widetilde\clX}} \ndef{\ttcX}{\widetilde{\tcX}} \ndef{\sfX}{{\textsf X}} \ndef{\ttX}{\widetilde{\tX}} \ndef{\dzX}{{X^\sharp}}
\ndef{\tY}{{\widetilde Y}} \ndef{\tcY}{{\widetilde\clY}} \ndef{\ttcY}{\widetilde{\tcY}} \ndef{\sfY}{{\textsf Y}} \ndef{\ttY}{\widetilde{\tY}} \ndef{\dzY}{{Y^\sharp}}
\ndef{\tZ}{{\widetilde Z}} \ndef{\tcZ}{{\widetilde\clZ}} \ndef{\ttcZ}{\widetilde{\tcZ}} \ndef{\sfZ}{{\textsf Z}} \ndef{\ttZ}{\widetilde{\tZ}} \ndef{\dzZ}{{Z^\sharp}}

\ndef{\bfa}{{\mathbf a}}
\ndef{\bfb}{{\mathbf b}}
\ndef{\bfc}{{\mathbf c}}
\ndef{\bfd}{{\mathbf d}}

\ndef{\euu}{{\EuScript u}}

  \ndef{\eps}{\varepsilon}

\let\geq\geqslant
\let\leq\leqslant

\ndef{\lims}[1]{\lim\limits_{#1}}
\ndef{\sums}[1]{\sum\limits_{#1}}
\ndef{\ints}[1]{\int_{#1}}
\ndef{\sups}[1]{\sup\limits_{#1}}
\ndef{\liminfty}[1]{\lims{#1\to\infty}}
\ndef{\suminf}[1]{\sums{#1=1}^\infty}

\ndef{\limo}[1]{\omega\mbox{-}\!\!\!\lims{#1\to\infty}}          
\ndef{\limL}[1]{\rmL\mbox{-}\!\!\!\lims{#1\to\infty}}            
\ndef{\limLOne}[1]{\clL_1\mbox{-}\!\lims{#1}}
\ndef{\tildelimo}[1]{\tilde\omega\mbox{-}\!\!\!\lims{#1\to\infty}}
\ndef{\slim}{\mathrm{s}\mbox{-}\!\!\lim}          
\ndef{\wlim}{\mathrm{w}\mbox{-}\!\lim}          

\ndef{\Aut}{\operatorname{Aut}}      
\ndef{\Ch}{\operatorname{ch}}        
\ndef{\End}{\operatorname{End}}      
\ndef{\Hom}{\operatorname{Hom}}      
\rndef{\ker}{\operatorname{ker}}      
\ndef{\coker}{\operatorname{coker}}      
\ndef{\im}{\operatorname{im}}        
\ndef{\Log}{\operatorname{Log}}      
\ndef{\OP}{\operatorname{OP}}        
\ndef{\Op}{\operatorname{Op}}        
\ndef{\Symb}{\operatorname{Symb}}    
\ndef{\Tr}{\operatorname{Tr}}        
\ndef{\Wres}{\operatorname{Wres}}    
\ndef{\cl}{\operatorname{cl}}        
\ndef{\com}{\operatorname{com}}
\ndef{\const}{\operatorname{const}}  
\ndef{\conv}{\operatorname{conv}}    
\rndef{\det}{\operatorname{det}}     
\ndef{\Var}{\operatorname{Var}}
\ndef{\Cov}{\operatorname{Cov}}

\ndef{\detFK}[1]{\Delta\brs{#1}} 
\ndef{\detFKrel}[2]{\Delta_{#2}\brs{#1}} 

\ndef{\adj}{\operatorname{adj}}    
\ndef{\diag}{\operatorname{diag}}    
\ndef{\dist}{\operatorname{dist}}    
\ndef{\dom}{\operatorname{dom}}      
\ndef{\ec}{\operatorname{ec}}        
\ndef{\id}{\mathrm{Id}}                        
\ndef{\ind}{\operatorname{ind}}      
\ndef{\mydeg}{\operatorname{deg}}    
\ndef{\op}{\operatorname{op}}
\ndef{\rank}{\operatorname{rank}}
\ndef{\res}{\operatorname{res}}      
\ndef{\rng}{\operatorname{ran}}      
\ndef{\sflow}{\operatorname{sf}}     
\ndef{\isf}{\operatorname{isf}}      
\ndef{\sign}{\operatorname{sign}}    
\ndef{\sgn}{\operatorname{sgn}}      
\ndef{\sing}{\operatorname{sing}}    
\ndef{\supp}{\operatorname{supp}}    
\ndef{\tr}{\operatorname{tr}}        
\ndef{\var}{\operatorname{var}}      
\ndef{\vol}{\operatorname{vol}}      
\ndef{\wn}{\operatorname{wn}}        
\ndef{\wres}{\operatorname{wres}}    
\rndef{\Im}{\operatorname{Im}}       
\rndef{\Re}{\operatorname{Re}}       

\ndef{\prng}[1]{\mathrm R_{#1}} 
\ndef{\pker}[1]{\mathrm N_{#1}} 
\ndef{\rprng}[2]{\mathrm R_{#1}^{#2}}           
\ndef{\rpker}[2]{\mathrm N_{#1}^{#2}}           
\ndef{\rsupp}[1]{\supp_r(#1)}
\ndef{\lsupp}[1]{\supp_l(#1)}
\ndef{\rslv}[1]{R_z(#1)}      
\ndef{\HH}{H}                 
\ndef{\tHH}{\tilde \HH}       
\ndef{\VV}{V}                 
\ndef{\Rz}{R_z}               
\ndef{\tRz}{\tR_z}            
\ndef{\psif}[1]{#1^{[1]}} 
\ndef{\WPlus}[1]{W_{#1}(\mbR)} 

\newcommand{\Texp}{\mathrm{T}\!\exp}

\newcommand{\xia}{\xi^{(a)}}
\newcommand{\xis}{\xi^{(s)}}
\newcommand{\mua}{\mu^{(a)}}
\newcommand{\mus}{\mu^{(s)}}

\ndef{\bndl}{\xi}                         
\ndef{\bndlA}{\eta}                       
\ndef{\GlueMap}{\varphi}                  
\ndef{\ChartMap}{h}                       
\ndef{\chern}{\ensuremath{\mathrm{ch}}}
\ndef{\hilb}{\clH}                     
\ndef{\hilba}{\clH^{(a)}}                    
\ndef{\hilbs}{\clH^{(s)}}                    
   \ndef{\hilbasargument}{(\hilb)} 
\ndef{\LpH}[1]{\clL_{#1}\hilbasargument}       
\ndef{\saLpH}[1]{\clL_{sa}^{#1}\hilbasargument}       
\ndef{\clBH}{\clB\hilbasargument}              
\ndef{\ubBH}{\clB_1\hilbasargument}            
\ndef{\clCH}{\clC\hilbasargument}              
\ndef{\clKH}{\clK\hilbasargument}              
\ndef{\clFH}{\clF\hilbasargument}              
\ndef{\clUH}{\clU\hilbasargument}              
\ndef{\clCFH}{{\clC\clF}\hilbasargument}       
\ndef{\saBH}{\clB_{sa}\hilbasargument}         
\ndef{\saCH}{\clC_{sa}\hilbasargument}         
\ndef{\saFH}{\clF_{sa}\hilbasargument}         
\ndef{\saKH}{\clK_{sa}\hilbasargument}         
\ndef{\saCFH}{\clC\clF_{sa}\hilbasargument}    
\ndef{\clUFH}{\clU\clF\hilbasargument}         
\ndef{\Uinj}{\clU_{inj}\hilbasargument}        
\ndef{\UFinj}{\clU\clF_{inj}\hilbasargument}   

\ndef{\spproj}[2]{E^{#1}_{#2}}                      
\ndef{\spprojb}[2]{E^{#2}_{#1}}                     

\ndef{\LpN}[1]{\clL^{#1}(\clN,\tau)}     
\ndef{\saLpN}[1]{\clL^{#1}_{sa}(\clN,\tau)} 
\ndef{\rLpN}[1]{L^{#1}(\clN,\tau)}       
\ndef{\clAND}{(\clA,\clN,D)}             
\ndef{\clBA}{{\clB(\clA)}}
\ndef{\saKN}{{\clK_{sa}(\clN,\tau)}}          
\ndef{\clKN}{{\clK(\clN,\tau)}}          
\ndef{\clKtN}{{\clK(\tilde\clN,\tau)}}   
\ndef{\clFN}{{\clF(\clN,\tau)}}          
\ndef{\saFN}{{\clF_{sa}(\clN,\tau)}}     
\ndef{\clPN}{\clP(\clN)}                 
\ndef{\clQN}{\clQ(\clN,\tau)}            
\ndef{\infPN}{{\clP_\tau^\infty(\clN)}}  
\ndef{\clOF}[2]{\clF_{#1\mbox{-}#2}(\clN,\tau)}         
\ndef{\oind}[2]{{\rm \tau\mbox{-}ind}_{#1\mbox{-}#2}}   
\ndef{\tind}{\tau\mbox{-}\ind}                  
\ndef{\DInd}{\ind_{\clD,\tau}}           
\ndef{\BF}{Breuer-Fredholm}              
\ndef{\skewfred}[2]{$(#1\cdot #2)$ $\tau$\tire Fredholm}   
\ndef{\affl}{\eta}                       
\ndef{\vNa}{von Neumann algebra}         
\ndef{\nsf}{faithful normal semifinite } 
\ndef{\taubrs}[1]{\tau\brackets{#1}}     
\ndef{\sqbrs}[1]{[#1]}        
\ndef{\Sqbrs}[1]{\big[#1\big]}        
\ndef{\SqBrs}[1]{\Big[#1\Big]}        

\ndef{\domd}{\bigcap\limits_{n\ge 0} \dom\;\delta^n}         
\ndef{\DiffOP}{{\rm \clD}}
\ndef{\ADA}{\clA \cup [\clD,\clA]}
\ndef{\DixIdeal}[1]{\LpH{#1,\infty}}               
\ndef{\dixideal}{\ell^{1,\infty}}                  
\ndef{\WDixIdeal}{\LpH{1,\mathrm w}}               
\ndef{\DixIdealPos}[1]{\DixIdeal{#1}_+}            
\ndef{\DixIdealN}[1]{\LpN{#1,\infty}}              
\ndef{\DixIdealNPar}[2]{\clL^{#1,\infty}_{#2}(\clN,\tau)}    
\ndef{\DixIdealNPos}[1]{\LpN{#1,\infty}_+}                   
\ndef{\TrD}{\Tr_\omega}                                      
\ndef{\tauD}{{\tau_\omega}}                                  
\ndef{\ILogN}{\frac 1{\log(1+N)}}
\ndef{\DixNorm}[1]{\norm{#1}_{(1,\infty)}}                   
\ndef{\DixInt}[1]{\ints 0^t \mu_s(#1)\,ds}
\ndef{\DixIntL}[1]{\ints 0^{\lambda_{1/t}(#1)}\mu_s(#1)\,ds}
    \ndef{\SmallIdeal}{{\clL^{1, \mathrm w}}}
    \ndef{\SmallIdealMeas}{{\clL^{1, \mathrm w}_m}}
    \ndef{\DixIntII}[1]{\int_0^t \mu_s(#1)\,ds}
    \ndef{\DixIntf}[1]{\Phi_t(#1)}
    \ndef{\DixIntg}[1]{\Psi_t(#1)}

\ndef{\lpi}{\clL^{1,\pi}(\clN,\tau)}

\ndef{\strl}[1]{\stackrel \longrightarrow {#1}}
\ndef{\IIinfty}{$\mathrm{II}_\infty$\ }

\ndef{\fourier}[1]{\clF(#1)}          
\ndef{\HaarMeasBohrs}{\nu}            
\ndef{\BrownsMeas}{\mu}               
\ndef{\BohrCont}[1]{\tilde{#1}}       
\ndef{\APMean}{{M}}                   
\ndef{\CDSS}{{\clA_B}}                
\ndef{\matr}{{\rm Mat}}               
\ndef{\seque}[1]{\ensuremath{\{#1_n\}_{n=1}^\infty}}    
\ndef{\sequen}[2]{\ensuremath{\{#1_#2\}_{#2=1}^\infty}}    
\ndef{\Seque}[1]{\ensuremath{\left(#1_0,#1_1,#1_2,\dots\right)}}    
\ndef{\Cesaro}{H}                           
\ndef{\CesaroRPlus}{M}                      
\ndef{\Dilation}{D}                         
\ndef{\Shift}{T}                            

\ndef{\norm}[1]{\left\Vert#1\right\Vert}    
\ndef{\TrNorm}[1]{\norm{#1}_1}              
\ndef{\HSNorm}[1]{\norm{#1}_2}              
\ndef{\InftyNorm}[1]{\norm{#1}_\infty}      
\ndef{\normQN}[1]{\norm{#1}_{\clQN}}        
\ndef{\clLpnorm}[2]{\norm{#2}_{\clL^{#1}}}    
\ndef{\clLnorm}[1]{\clLpnorm{1}{#1}}    

\ndef{\ccurve}{\gamma}                      

\ndef{\abs}[1]{\left\lvert#1\right\rvert}   
\ndef{\set}[1]{\left\{#1\right\}}           
\ndef{\brackets}[1]{\left(#1\right)}        
\ndef{\brs}[1]{\brackets{#1}}               
\ndef{\Brs}[1]{\big(#1\big)}                
\ndef{\BRS}[1]{\Big(#1\Big)}                
\ndef{\scal}[2]{\left\langle #1,#2\right\rangle}               
\ndef{\precprec}{\prec\!\!\!\prec}
\ndef{\qeq}{\stackrel?=}
\ndef{\spectrum}[1]{\sigma_{#1}} 
\ndef{\spectruma}[1]{\sigma^{(a)}_{#1}} 
\ndef{\numrange}[1]{\mathrm{W}(#1)}                         
\rndef{\emptyset}{\varnothing}                              
\ndef{\csupp}{c}                           
\ndef{\closure}[1]{\overline{#1}}
\ndef{\linspan}[1]{\mathrm{span}\ {#1}}
\ndef{\bddborel}[1]{B(#1)}                 
\ndef{\charfunc}{\chi}
\ndef{\FrDer}{\euD}                        
\ndef{\LieDer}[1]{\pounds_{#1}\,}          
\ndef{\dds}{\left.\frac d{ds} \right|_{s = 0}}
\ndef{\ortcmp}[1]{#1^{\scriptscriptstyle \perp}}            
\ndef{\Laplace}{\Delta}                    

\ndef{\matrPQ}[3]
{
    \left(
      \begin{array}{cc}
        #1_{11} & #1_{12} \\
        #1_{21} & #1_{22}
      \end{array}
    \right)_{[#2,#3]}
}

\ndef{\margOK}{\marginpar{\bf \small OK}}

\newcounter{margcomcount}
\ndef{\margcom}[1]{\marginpar{\bf \small #1} \addtocounter{margcomcount}{1}
   \index{\indexcom{{\bf COMMENT: #1}}}}

\newcounter{margproof}
\ndef{\margproof}{\marginpar{\bf \small PROOF} \addtocounter{margproof}{1}
  \index{**** \indexcom{{\bf PROOF}}}}

\newcounter{margdetails}
\ndef{\margdetails}{\marginpar{\bf Details} \addtocounter{margdetails}{1}
  \index{**** \indexcom{{\bf DETAILS}}}}

\newcounter{margproofb}
\ndef{\margproofb}[1]{\marginpar{\bf \small Proof(B) #1} \addtocounter{margproofb}{1}
  \index{**** \indexcom{{\bf PROOF(B): #1}}}}

\newcounter{margdetailsb}
\ndef{\margdetailsb}[1]{\marginpar{\bf \small Details(B)} \addtocounter{margdetailsb}{1}
  \index{**** \indexcom{{\bf DETAILS(B): \\ #1}}}}

\newcounter{margdetailsc}
\ndef{\margdetailsc}[1]{\marginpar{\bf \small Details(C)} \addtocounter{margdetailsc}{1}
  \index{**** \indexcom{{\bf DETAILS(C): \\ #1}}}}

\newcounter{margcomcountb}
\ndef{\margcomb}[1]{\marginpar{\bf \small #1} \addtocounter{margcomcountb}{1}
   \index{\indexcom{{\bf COMMENT(B): \\ #1}}}}

\ndef{\mytimes}{\!\times\!}
\ndef{\sss}[1]{\subsubsection{}\label{#1}}

\rndef{\phi}{\varphi} \ndef{\OpenUnitDisk}{D}
\ndef{\RHS}{RHS}                            
\ndef{\LHS}{LHS} 
\ndef{\ttt}{\Leftrightarrow}
\ndef{\then}{\Rightarrow}
\ndef{\tto}{\longrightarrow}
\ndef{\nno}{\nonumber\\}
\ndef{\newn}[1]{\index{#1} {\bfseries #1}}       
\ndef{\la}{\langle}
\ndef{\ra}{\rangle}
\ndef{\dbar}{{\;\bar{\phantom{o}} \!\!\!\! d}}
\ndef{\stl}[1]{\stackrel{\vbox to 0pt{\vss\hbox{$\scriptstyle #1$}}}}
\ndef{\mathcomment}[1]{{\hfill \qquad\qquad\qquad\text{by (#1)}}}        
\ndef{\mathcomm}[1]{{\hfill \qquad\qquad\qquad\qquad\qquad\text{#1}}}        
\ndef{\details}[1]{\smallskip\begin{center} {\bf Here:}
#1\end{center}\medskip} \ndef{\indexcom}[1]{ --- #1}
\ndef{\longsim}{\ \sim \ }              
\ndef{\tire}{-}              
\ndef{\intinfinf}{\int_{-\infty}^\infty}
     \ndef{\npartial}{\slash\!\!\!\partial}
     \ndef{\Heis}{\operatorname{Heis}}
     \ndef{\Solv}{\operatorname{Solv}}
     \ndef{\Spin}{\operatorname{Spin}}
     \ndef{\SO}{\operatorname{SO}}
     \ndef{\Index}{\operatorname{index}}

             \ndef{\p}{\partial}
             \ndef{\dd}{|\clD|}
             \ndef{\n}{\parallel}

\let\LatexCite=\cite  

\let\ifnumref\iffalse 

\ndef{\ifuncited}[4]{\expandafter\ifx\csname used#4\endcsname\relax}

\ndef{\ifcited}[4]{\expandafter\ifx\csname used#4\endcsname\relax\else}

  \ndef{\papertitle}[1]{ \emph{#1}, }
  \ndef{\paperauthor}[2]{#2}  
  \ndef{\pbbi}[9]{%
      \ifcited{#1}{#2}{#3}{#5}%
        \ifnumref%
          \bibitem{#5}\paperauthor{#1}{#6},\papertitle{#7}#8.%
        \else%
          \advance #9 by 1%
          \ifnum#9<1%
            \bibitem[#4]{#5}\paperauthor{#1}{#6}, \papertitle{#7}#8.%
          \else%
            \bibitem[#4$_{\the#9}$]{#5}\paperauthor{#1}{#6},\papertitle{#7}#8.%
          \fi%
        \fi%
      \fi%
  }
  \ndef{\mbbi}[8]{%
     \ifcited{#1}{#2}{#3}{#5}%
        \ifnumref%
          \bibitem{#5}\paperauthor{#1}{#6},\papertitle{#7}#8.%
        \else%
          \bibitem[#4]{#5}\paperauthor{#1}{#6},\papertitle{#7}#8.%
        \fi%
     \fi%
  }

\ndef{\AddCite}[1]{%
   \ifuncited{0}{0}{0}{#1}%
     \expandafter\gdef\csname used#1\endcsname {}%
   \fi%
}

\def\ProcessCite#1,{%
     \ifx\relax#1%
         \let\next=\relax%
     \else%
         \AddCite{#1}%
         \let\next=\ProcessCite%
     \fi%
     \next%
}

\ndef{\AddCites}[1]{\ProcessCite#1,\relax,}

\ndef{\CiteWithoutExtension}[1]{%
   \AddCites{#1}%
   \LatexCite{#1}%
}

\def\CiteWithExtension[#1]#2{%
   \AddCites{#2}%
   \LatexCite[#1]{#2}%
}

\ndef{\CleverCite}{%
    \ifx\NChar[ %
       \let\MyCite=\CiteWithExtension %
    \else %
       \let\MyCite=\CiteWithoutExtension %
    \fi %
    \MyCite%
}

\renewcommand{\cite}{\futurelet\NChar\CleverCite}

      \ndef{\volume}[1]{{\bf #1}}
      \ndef{\VolYearPP}[3]{\ifnum#2=0 (to appear)\else\volume{#1} (#2), #3\fi}
      \ndef{\VolNoYearPP}[4]{\ifnum#3=0 (to appear)\else\volume{#1} #2 (#3), #4\fi}
      \ndef{\libcode}[1]{}

\ndef{\jnActaMath}[3]{Acta Math. \VolYearPP{#1}{#2}{#3}}                       
\ndef{\jnAdvMath}[3]{Adv. in~Math. \VolYearPP{#1}{#2}{#3}}                     
\ndef{\jnAlgAnal}[3]{Algebra i~Analiz \VolYearPP{#1}{#2}{#3}}
\ndef{\jnAmerJMath}[3]{Amer. J. Math. \VolYearPP{#1}{#2}{#3}}                  
\ndef{\jnAmerMathMonth}[3]{Amer. Math. Monthly \VolYearPP{#1}{#2}{#3}}         
\ndef{\jnAnnMath}[4]{Ann. of~Math. \VolNoYearPP{#1}{#2}{#3}{#4}}               
\ndef{\jnAnalMath}[3]{J. Anal. Math. \VolYearPP{#1}{#2}{#3}}                   
\ndef{\jnArchRatMechAnal}[3]{Arch. Rational Mech. Anal. \VolYearPP{#1}{#2}{#3}}                   
\ndef{\jnBullLondMathSoc}[3]{Bull. London Math. Soc. \VolYearPP{#1}{#2}{#3}}   
\ndef{\jnBullAMS}[3]{Bull. Amer. Math. Soc. \VolYearPP{#1}{#2}{#3}}   
\ndef{\jnCanMathBull}[3]{Canad. Math. Bull. \VolYearPP{#1}{#2}{#3}}            
\ndef{\jnCanMath}[3]{Canad. J.~Math. \VolYearPP{#1}{#2}{#3}}             
\ndef{\jnCommMathPhys}[3]{Comm. Math. Phys. \VolYearPP{#1}{#2}{#3}}             
\ndef{\jnCommPDE}[3]{Comm. Partial Differential Equations \VolYearPP{#1}{#2}{#3}}             
\ndef{\jnComptRendue}[3]{C.\,R.~Acad. Sci. Paris S\'er. A-B \VolYearPP{#1}{#2}{#3}}      
\ndef{\jnContMath}[3]{Contemporary Math. \VolYearPP{#1}{#2}{#3}}               %
\ndef{\jnDukeMJ}[3]{Duke Math. J. \VolYearPP{#1}{#2}{#3}}
\ndef{\jnDiffGeom}[3]{J.~Diff. Geom. \VolYearPP{#1}{#2}{#3}}                   
\ndef{\jnErgodicTheory}[3]{Ergodic Theory and Dynamical Systems \VolYearPP{#1}{#2}{#3}} 
\ndef{\jnFuncAnal}[3]{J.~Functional Analysis \VolYearPP{#1}{#2}{#3}}           
\ndef{\jnFunkAnalPril}[4]{Funct. Anal. Appl. \VolNoYearPP{#1}{#2}{#3}{#4}}  
\ndef{\jnGAFA}[3]{GAFA \VolYearPP{#1}{#2}{#3}}                                 
\ndef{\jnIHES}[3]{IHES Publ. Math. (Paris) \VolYearPP{#1}{#2}{#3}}             
\ndef{\jnIEOT}[3]{Integral Equations Operator Theory   \VolYearPP{#1}{#2}{#3}} 
\ndef{\jnIsrMath}[3]{Israel J.~Math. \VolYearPP{#1}{#2}{#3}}                   
\ndef{\jnKTheory}[3]{K-Theory \VolYearPP{#1}{#2}{#3}}                          
\ndef{\jnLetMathPhys}[3]{Lett. Math. Phys. \VolYearPP{#1}{#2}{#3}}             
\ndef{\jnMathAnn}[3]{Math. Ann. \VolYearPP{#1}{#2}{#3}}                        
\ndef{\jnMathAnalAppl}[3]{J.~Math. Anal. and Appl. \VolYearPP{#1}{#2}{#3}}     
\ndef{\jnMathNachr}[3]{Math. Nachr. \VolYearPP{#1}{#2}{#3}}
\ndef{\jnMathPhys}[3]{J. Math. Phys. \VolYearPP{#1}{#2}{#3}}
\ndef{\jnMathSocJap}[3]{J. Math. Soc. Japan \VolYearPP{#1}{#2}{#3}}
\ndef{\jnOperTheory}[3]{J.~Operator Theory \VolYearPP{#1}{#2}{#3}}             
\ndef{\jnPacJMath}[3]{Pacific J.~Math. \VolYearPP{#1}{#2}{#3}}                  
\ndef{\jnPositivity}[3]{Positivity \VolYearPP{#1}{#2}{#3}}
\ndef{\jnProcAmerMS}[3]{Proc. Amer. Math. Soc. \VolYearPP{#1}{#2}{#3}}         
\ndef{\jnProcCambPhilSoc}[3]{Math. Proc. Camb. Phil. Soc. \VolYearPP{#1}{#2}{#3}}
\ndef{\jnReineAngew}[3]{J.~Reine Angew. Math. \VolYearPP{#1}{#2}{#3}}          
\ndef{\jnTokyoMath}[3]{Tokyo J.~Math. \VolYearPP{#1}{#2}{#3}}
\ndef{\jnTopology}[3]{Topology \VolYearPP{#1}{#2}{#3}}
\ndef{\jnTransAmerMathSoc}[3]{Trans. Amer. Math. Soc. \VolYearPP{#1}{#2}{#3}}
\ndef{\jnIzvANSSSR}[3]{Izv. Akad. Nauk SSSR, Ser. Mat. \VolYearPP{#1}{#2}{#3}}
\ndef{\jnIzvVyshUchZav}[3]{Izv. Vyssh. Uch. Zav., Mat. \VolYearPP{#1}{#2}{#3} (Russian)}
\ndef{\jnIzdatLenUniv}[2]{Izdat. Leningrad. Univ., Leningrad, (#1), #2 (Russian)}
\ndef{\jnFieldsInsComm}[3]{Fields Inst. Comm. \VolYearPP{#1}{#2}{#3}}
\ndef{\jnDoklANSSSR}[3]{Dokl. Akad. Nauk SSSR \VolYearPP{#1}{#2}{#3}}
\ndef{\jnMatZametki}[3]{Matem. zametki \VolYearPP{#1}{#2}{#3}}
\ndef{\jnRussMathSurvey}[3]{Russian Math. Surveys \VolYearPP{#1}{#2}{#3}}
\ndef{\jnSibMathJ}[3]{Sib. Math.~J. \VolYearPP{#1}{#2}{#3}}
\ndef{\jnSovMath}[3]{J.~Soviet math. \VolYearPP{#1}{#2}{#3}}
\ndef{\jnTransMoscMathSoc}[3]{Trans. Moscow Math. Soc. \VolYearPP{#1}{#2}{#3}}
\ndef{\jnUMN}[3]{Uspekhi Mat. Nauk \VolYearPP{#1}{#2}{#3}}

\ndef{\bkTransMathMon}[2]{Trans. Math. Monographs, AMS, \volume{#1}, #2}

\ndef{\pbBirkhauser}[1]{Birkh\"auser, Boston, #1}
\ndef{\pbFactorial}[1]{Moscow, Factorial, #1}
\ndef{\pbGauthier}[1]{Gauthier-Villars, Paris, #1}
\ndef{\pbNauka}[1]{Moscow, Nauka, #1 (Russian)}
\ndef{\pbNaukaR}[1]{Москва, Наука, #1}
\ndef{\pbPrinceton}[1]{Princeton University Press, Princeton, New Jersey, #1}
\ndef{\pbPublPerish}[1]{Publish or Perish Inc., Berkeley, #1}
\ndef{\pbSpringer}[1]{Springer-Verlag, #1}

\ndef{\myauthor}[1]{\mbox{#1}}

\ndef{\Agmon}{\myauthor{Sh.\,Agmon}}
\ndef{\Ahiezer}{\myauthor{N.\,I.\,Ahiezer}}
\ndef{\Arazy}{\myauthor{J.\,Arazy}}
\ndef{\Aronszajn}{\myauthor{N.\,Aronszajn}}
\ndef{\Astashkin}{\myauthor{S.\,V.\,Astashkin}}
\ndef{\Atiyah}{\myauthor{M.\,Atiyah}}
\ndef{\Avron}{\myauthor{J.\,E.\,Avron}}
\ndef{\Azamov}{\myauthor{N.\,A.\,Azamov}}
\ndef{\Banach}{\myauthor{S.\,Banach}}
\ndef{\Benameur}{\myauthor{M-T.\,Benameur}}
\ndef{\Bennett}{\myauthor{C.\,Bennett}}
\ndef{\Berezin}{\myauthor{F.\,A.\,Berezin}}
\ndef{\Berline}{\myauthor{N.\,Berline}}
\ndef{\Birman}{\myauthor{M.\,Sh.\,Birman}}
\ndef{\Blackadar}{\myauthor{B.\,Blackadar}}
\ndef{\Bogolyubov}{\myauthor{N.\,N.\,Bogolyubov}}
\ndef{\Bonsall}{\myauthor{F.\,F.\,Bonsall}}
\ndef{\Bony}{\myauthor{J.\,F.\,Bony}}
\ndef{\BoosBavnbek}{\myauthor{B.\,Boo$\beta$-Bavnbek}}
\ndef{\Bott}{\myauthor{R.\,Bott}}
\ndef{\Branges}{\myauthor{L.\,de Branges}}
\ndef{\Bratteli}{\myauthor{O.\,Bratteli}}
\ndef{\Bredon}{\myauthor{G.\,E.\,Bredon}}
\ndef{\Breuer}{\myauthor{M.\,Breuer}}
\ndef{\Brown}{\myauthor{L.\,G.\,Brown}}
\ndef{\Bruneau}{\myauthor{V.\,Bruneau}}
\ndef{\Buslaev}{\myauthor{V.\,S.\,Buslaev}}
\ndef{\Carey}{\myauthor{A.\,L.\,Carey}}
\ndef{\CareyRW}{\myauthor{R.\,W.\,Carey}} 
\ndef{\Cartan}{\myauthor{H.\,Cartan}}
\ndef{\Chilin}{\myauthor{V.\,I.\,Chilin}}
\ndef{\Coburn}{\myauthor{L.\,A.\,Coburn}}
\ndef{\Connes}{\myauthor{A.\,Connes}}
\ndef{\Cornfeld}{\myauthor{I.\,P.\,Cornfeld}}
\ndef{\Daletskii}{\myauthor{Yu.\,L.\,Daletski\u\i}}   
\ndef{\Dixmier}{\myauthor{J.\,Dixmier}}
\ndef{\DoddsPG}{\myauthor{P.\,G.\,Dodds}}
\ndef{\DoddsTK}{\myauthor{T.\,K.\,Dodds}}
\ndef{\Douglas}{\myauthor{R.\,G.\,Douglas}}
\ndef{\Dubrovin}{\myauthor{B.\,A.\,Dubrovin}}
\ndef{\Dugundji}{\myauthor{J.\,Dugundji}}
\ndef{\Duncan}{\myauthor{J.\,Duncan}}
\ndef{\Dunford}{\myauthor{N.\,Dunford}}
\ndef{\Dykema}{\myauthor{K.\,J.\,Dykema}}
\ndef{\Edwards}{\myauthor{R.\,E.\,Edwards}}
\ndef{\Eilenberg}{\myauthor{S.\,Eilenberg}}
\ndef{\Entina}{\myauthor{S.\,B.\,\`Entina}}
\ndef{\Fack}{\myauthor{T.\,Fack}} 
\ndef{\Faddeev}{\myauthor{L.\,D.\,Faddeev}}
\ndef{\Farber}{\myauthor{M.\,Farber}}
\ndef{\Farforovskaya}{\myauthor{Yu.\,B.\,Farforovskaya}}
\ndef{\Federer}{\myauthor{H.\,Federer}}
\ndef{\Fedosov}{\myauthor{B.\,V.\,Fedosov}}
\ndef{\Figiel}{\myauthor{T.\,Figiel}} 
\ndef{\Figueroa}{\myauthor{H.\,Figueroa}}
\ndef{\Fillmore}{\myauthor{P.\,A.\,Fillmore}}
\ndef{\Fomenko}{\myauthor{A.\,T.\,Fomenko}} 
\ndef{\Fomin}{\myauthor{S.\,V.\,Fomin}}
\ndef{\Frohlich}{\myauthor{J.\,Fr\"ohlich}}
\ndef{\Fuglede}{\myauthor{B.\,Fuglede}}
\ndef{\Furutani}{\myauthor{K.\,Furutani}}
\ndef{\Gelfand}{\myauthor{I.\,M.\,Gelfand}}
\ndef{\Gesztesy}{\myauthor{F.\,Gesztesy}}     
\ndef{\Getzler}{\myauthor{E.\,Getzler}} 
\ndef{\Gilkey}{\myauthor{P.\,B.\,Gilkey}}
\ndef{\Gitler}{\myauthor{S.\,Gitler}}
\ndef{\Glazman}{\myauthor{I.\,M.\,Glazman}}
\ndef{\Glimm}{\myauthor{J.\,Glimm}}
\ndef{\Gohberg}{\myauthor{I.\,C.\,Gohberg}}
\ndef{\Goldshtein}{\myauthor{Ya.\,Goldshtein}}
\ndef{\Golze}{\myauthor{F.\,Golze}}
\ndef{\GraciaBondia}{\myauthor{J.\,M.\,Gracia-Bond\'{i}a}}
\ndef{\Greenleaf}{\myauthor{F.\,P.\,Greenleaf}}
\ndef{\Gromov}{\myauthor{M.\,Gromov}}
\ndef{\Gunning}{\myauthor{R.\,C.\,Gunning}}
\ndef{\Haagerup}{\myauthor{U.\,Haagerup}}
\ndef{\Haag}{\myauthor{R.\,Haag}}
\ndef{\Halmos}{\myauthor{P.\,R.\,Halmos}}
\ndef{\Hardy}{\myauthor{G.\,H.\,Hardy}}
\ndef{\Herbst}{\myauthor{I.\,W.\,Herbst}}
\ndef{\Higson}{\myauthor{N.\,Higson}}  
\ndef{\Hoermander}{\myauthor{L.\,Hoermander}} 
\ndef{\Hoffman}{\myauthor{K.\,Hoffman}} 
\ndef{\Ito}{\myauthor{K.\,Ito}}
\ndef{\Ikebe}{\myauthor{T.\,Ikebe}}
\ndef{\Jaffe}{\myauthor{A.\,Jaffe}}
\ndef{\James}{\myauthor{I.\,M.\,James}}
\ndef{\Javrjan}{\myauthor{V.\,A.\,Javrjan}}
\ndef{\Jitomirskaya}{\myauthor{S.\,Jitomirskaya}}
\ndef{\Kadison}{\myauthor{R.\,V.\,Kadison}}
\ndef{\Kalton}{\myauthor{N.\,J.\,Kalton}} 
\ndef{\Kato}{\myauthor{T.\,Kato}} 
\ndef{\Kobayashi}{\myauthor{S.\,Kobayashi}}
\ndef{\Koplienko}{\myauthor{L.\,S.\,Koplienko}}
\ndef{\Korotyaev}{\myauthor{E.\,Korotyaev}}
\ndef{\Kosaki}{\myauthor{H.\,Kosaki}}
\ndef{\Kostrykin}{\myauthor{V.\,Kostrykin}}
\ndef{\Kotani}{\myauthor{S.\,Kotani}}
\ndef{\Krein}{\myauthor{Kre\u\i n}}
\ndef{\KreinMG}{\myauthor{M.\,G.\,Kre\u\i n}}
\ndef{\KreinSG}{\myauthor{S.\,G.\,Kre\u\i n}}
\ndef{\Kuroda}{\myauthor{S.\,T.\,Kuroda}}
\ndef{\Leichtnam}{\myauthor{E.\,Leichtnam}}
\ndef{\Lesch}{\myauthor{M.\,Lesch}}
\ndef{\Lesniewski}{\myauthor{A.\,Lesniewski}}
\ndef{\Levitan}{\myauthor{B.\,M.\,Levitan}}
\ndef{\Lidskii}{\myauthor{V.\,B.\,Lidskii}}
\ndef{\Lifshitz}{\myauthor{I.\,M.\,Lifshitz}}
\ndef{\Lindenstrauss}{\myauthor{J.\,Lindenstrauss}}
\ndef{\Loday}{\myauthor{J.-L.\,Loday}}
\ndef{\Lord}{\myauthor{S.\,Lord}}      
\ndef{\Lorentz}{\myauthor{G.\,Lorentz}}
\ndef{\Magnus}{\myauthor{W.\,Magnus}}
\ndef{\Makarov}{\myauthor{K.\,A.\,Makarov}}
\ndef{\MakarovN}{\myauthor{N.\,Makarov}}
\ndef{\Mathai}{\myauthor{V.\,Mathai}}         
\ndef{\McKean}{\myauthor{H.\,P.\,McKean}}
\ndef{\Mishchenko}{\myauthor{A.\,S.\,Mishchenko}}
\ndef{\Molchanov}{\myauthor{S.\,A.\,Molchanov}}
\ndef{\Moore}{\myauthor{C.\,C.\,Moore}}
\ndef{\Moscovici}{\myauthor{H.\,Moscovici}}  
\ndef{\Motovilov}{\myauthor{A.\,K.\,Motovilov}}
\ndef{\Moyer}{\myauthor{R.\,D.\,Moyer}}
\ndef{\Naboko}{\myauthor{S.\,N.\,Naboko}}
\ndef{\Narasimhan}{\myauthor{R.\,Narasimhan}}
\ndef{\Nomizu}{\myauthor{K.\,Nomizu}}
\ndef{\Novikov}{\myauthor{S.\,P.\,Novikov}}
\ndef{\Osterwalder}{\myauthor{K.\,Osterwalder}}
\ndef{\Patodi}{\myauthor{V.\,Patodi}}
\ndef{\Pagter}{\myauthor{B.\,de~Pagter}}  
\ndef{\Pastur}{\myauthor{L.\,A.\,Pastur}}  
\ndef{\Pavlov}{\myauthor{B.\,S.\,Pavlov}}
\ndef{\Pedersen}{\myauthor{G.\,K.\,Pedersen}}
\ndef{\Peller}{\myauthor{V.\,V.\,Peller}}
\ndef{\Perera}{\myauthor{V.\,S.\,Perera}}
\ndef{\Petunin}{\myauthor{Ju.\,I.\,Petunin}}
\ndef{\Phillips}{\myauthor{J.\,Phillips}}  
\ndef{\Piazza}{\myauthor{P.\,Piazza}}   
\ndef{\Pincus}{\myauthor{J.\,D.\,Pincus}}   
\ndef{\Poincare}{Poincar\'e}
\ndef{\Postnikov}{\myauthor{M.\,M.\,Postnikov}} 
\ndef{\Povzner}{\myauthor{A.\,Ya.\,Povzner}}
\ndef{\Prinzis}{\myauthor{R.\,Prinzis}}
\ndef{\Privalov}{\myauthor{I.\,I.\,Privalov}}
\ndef{\Pushnitski}{\myauthor{A.\,B.\,Pushnitski}} 
\ndef{\Raeburn}{\myauthor{I.\,Raeburn}}
\ndef{\Raikov}{\myauthor{G.\,Raikov}}
\ndef{\Reed}{\myauthor{M.\,Reed}}
\ndef{\Rennie}{\myauthor{A.\,Rennie}}
\ndef{\Rickart}{\myauthor{C.\,E.\,Rickart}}
\ndef{\Riesz}{\myauthor{F.\,Riesz}}
\ndef{\Ringrose}{\myauthor{J.\,Ringrose}}
\ndef{\Rio}{\myauthor{R.\,del Rio}}
\ndef{\Robinson}{\myauthor{D.\,Robinson}}
\ndef{\Rossi}{\myauthor{H.\,Rossi}}
\ndef{\Rudin}{\myauthor{W.\,Rudin}}
\ndef{\Ruelle}{\myauthor{D.\,Ruelle}}
\ndef{\Ruzhansky}{\myauthor{M.\,Ruzhansky}}
\ndef{\Sakai}{\myauthor{Sh.\,Sakai}}
\ndef{\Sargsjan}{\myauthor{I.\,S.\,Sargsjan}}
\ndef{\Sato}{\myauthor{H.\,Sato}}
\ndef{\Schaeffer}{\myauthor{D.\,G.\,Schaeffer}}
\ndef{\Schluchtermann}{\myauthor{G.\,Schluchtermann}}
\ndef{\Schochet}{\myauthor{C.\,Schochet}}
\ndef{\SchroedingerE}{\myauthor{E.\,Schr\"odinger}}
\ndef{\Schroedinger}{\myauthor{Schr\"odinger}}
\ndef{\Schrohe}{\myauthor{E.\,Schrohe}}
\ndef{\Schwartz}{\myauthor{J.\,T.\,Schwartz}}
\ndef{\Sedaev}{\myauthor{A.\,A.\,Sedaev}}
\ndef{\Seiler}{\myauthor{R.\,Seiler}}
\ndef{\Semenov}{\myauthor{E.\,M.\,Semenov}}
\ndef{\Shabat}{\myauthor{B.\,V.\,Shabat}}
\ndef{\Shafarevich}{\myauthor{I.\,R.\,Shafarevich}}
\ndef{\Sharpley}{\myauthor{R.\,Sharpley}}
\ndef{\Shilov}{\myauthor{G.\,E.\,Shilov}}
\ndef{\Shirkov}{\myauthor{D.\,V.\,Shirkov}}
\ndef{\Shubin}{\myauthor{M.\,A.\,Shubin}}
\ndef{\Silverman}{\myauthor{H.\,Silverman}}
\ndef{\Simon}{\myauthor{B.\,Simon}}
\ndef{\Sinai}{\myauthor{Ya.\,G.\,Sinai}}
\ndef{\Singer}{\myauthor{I.\,M.\,Singer}}
\ndef{\Solomyak}{\myauthor{M.\,Z.\,Solomyak}}
\ndef{\Soloviev}{\myauthor{Yu.\,P.\,Soloviev}}
\ndef{\Spivak}{\myauthor{M.\,Spivak}}
\ndef{\Stein}{\myauthor{E.\,M.\,Stein}}
\ndef{\Stenkin}{\myauthor{V.\,V.\,Sten'kin}}
\ndef{\Stratila}{\myauthor{S.\,Stratila}}
\ndef{\Sucheston}{\myauthor{L.\,Sucheston}}
\ndef{\Sukochev}{\myauthor{F.\,A.\,Sukochev}}
\ndef{\Switzer}{\myauthor{R.\,M.\,Switzer}}
\ndef{\SzNagy}{\myauthor{B.\,Sz.-Nagy}}
\ndef{\Takesaki}{\myauthor{M.\,Takesaki}}
\ndef{\Taylor}{\myauthor{M.\,E.\,Taylor}}
\ndef{\Treves}{\myauthor{F.\,Treves}}
\ndef{\Troitsky}{\myauthor{E.\,V.\,Troitsky}}
\ndef{\Tzafriri}{\myauthor{L.\,Tzafriri}}
\ndef{\Varilly}{\myauthor{J.\,C.\,V\'{a}rilly}}
\ndef{\Vergne}{\myauthor{M.\,Vergne}}
\ndef{\Vladimirov}{\myauthor{V.\,S.\,Vladimirov}}
\ndef{\Voiculescu}{\myauthor{D.\,Voiculescu}}
\ndef{\Weiss}{\myauthor{G.\,Weiss}}
\ndef{\Wells}{\myauthor{R.\,O.\,Wells}}
\ndef{\Williams}{\myauthor{J.\,P.\,Williams}}
\ndef{\Winkler}{\myauthor{S.\,Winkler}}
\ndef{\Witten}{\myauthor{E.\,Witten}}
\ndef{\Wodzicki}{\myauthor{M.\,Wodzicki}}
\ndef{\Wojciechowski}{\myauthor{K.\,P.\,Wojciechowski}}
\ndef{\Yafaev}{\myauthor{D.\,R.\,Yafaev}}
\ndef{\Yosida}{\myauthor{K.\,Yosida}}
\ndef{\Zsido}{\myauthor{L.\,Zsido}}

\numberwithin{equation}{section}

\newcommand{\order}{\mathrm{d}\,}
\newcommand{\depth}{\mathrm{\gamma}\,}

\newcommand{\Schrodinger}{Schr\"odinger\ }
\newcommand{\twovector}[2]{\left(\begin{matrix} #1 \\ #2 \end{matrix} \right)}
\newcommand{\ulA}{\underline{A}\,}
\newcommand{\ulB}{\underline{B}\,}
\newcommand{\ulP}{\underline{P}\,}
\newcommand{\ulQ}{\underline{Q}\,}
\newcommand{\clLs}{\clL}
\newcommand{\clLw}{\clL^w}
\newcommand{\ubfA}{\underline{\bfA}\,}
\newcommand{\ubfB}{\underline{\bfB}\,}
\newcommand{\hatlmbu}[1]{\hat u_{\lambda+i0}^{(#1)}(s)}

\newcommand{\mbd}{\mathbf d}

\ndef{\hlambda}{{\mathfrak h_\lambda}}
\rndef{\iff}{\Leftrightarrow}
\ndef{\Rindex}{\euR}
\begin{document}
\title{Spectral flow inside essential spectrum}
\author{Nurulla Azamov}
\address{School of Computer Science, Engineering and Mathematics
   \\ Flinders University
   \\ Bedford Park, 5042, SA Australia.}
\email{nurulla.azamov@flinders.edu.au}
 \keywords{Spectral flow, spectral shift function, singular spectral shift function, absolutely continuous spectral shift function,
   scattering matrix, Birman-\Krein\ formula, Lippmann-Schwinger equation}

 \subjclass[2010]{ 
     Primary 47A55, 47A10, 47A70, 47A40;
     Secondary 35P99.
 }
\begin{center} \tiny \sf\today \ v26.7(arXiv) \end{center}

\begin{abstract} The spectral flow is a classical notion of functional analysis and differential geometry which
was given different interpretations as Fredholm index, Witten index, and Maslov index.
The classical theory treats spectral flow outside the essential spectrum.
Inside essential spectrum, the spectral shift function could be considered as a proper analogue of spectral flow, but unlike the spectral flow,
the spectral shift function is not an integer-valued function.

In this paper it is shown that the notion of spectral flow admits a natural extension for a.e. value of the spectral parameter
inside essential spectrum too and appropriate theory is developed. The definition of spectral flow inside essential spectrum
given in this paper applies to the classical spectral flow and thus gives one more new alternative definition of it.

\medskip
One of the results of this paper is the following

{\bf Theorem.} Let $H_0$ be a self-adjoint operator and let $V$ be a trace class self-adjoint operator
acting on a separable Hilbert space. Let $H_s = H_0+sV, \ s \in \mbC.$ The following four functions are equal for a.e. $\lambda;$
their common value is the spectral flow inside essential spectrum by definition.

1) Density $\xis(\lambda)$ of the singular spectral shift measure $\Delta \mapsto \int_0^1 \Tr\brs{VE_\Delta^{H_s}P^{(s)}(H_s)}\,ds,$
where $E_\Delta^H$ is the spectral measure of a self-adjoint operator~$H$ and $P^{(s)}(H)$
is the orthogonal projection onto the singular subspace of~$H.$

2) The difference $\mus(\lambda) := \mu(\theta, \lambda; H_1,H_0) - \mua(\theta, \lambda; H_1,H_0),$ $\theta \in [0,2\pi),$ where $\mu(\theta, \lambda; H_1,H_0)$ is the Pushnitski $\mu$-invariant
and $\mua(\theta, \lambda; H_1,H_0)$ is the absolutely continuous part of the Pushnitski $\mu$-invariant. This difference does not depend on the angle $\theta.$
The numbers $\mu(\theta, \lambda)$ and $\mua(\theta, \lambda)$ measure the spectral flow of the eigenvalues of the scattering matrix $S(\lambda; H_1,H_0)$ in two distinctive ways.

3) The total resonance index of the pair of self-adjoint operators $H_0,H_1,$ defined by formula $\sum_{r_\lambda \in [0,1]} \ind_{res}(\lambda; H_{r_\lambda},V),$
where one of the many equivalent definitions of the so-called resonance index
$\ind_{res}(\lambda; H_{r_\lambda},V)$ is as follows: let $N_+$ respectively, $N_-$ be the number of those eigenvalues of the operator $(H_s-\lambda -iy)^{-1}V$ (counting multiplicities)
from the upper $\mbC_+$ respectively, lower $\mbC_-$ complex half-plane which converge to the real number $(s - r_\lambda)^{-1}$ as $y\to 0^+;$ then
$$
  \ind_{res}(\lambda; H_{r_\lambda},V) = N_+-N_-.
$$
This definition is independent of $s\in \mbR.$
A real number~$r_\lambda$ is called a resonance point iff $0<N_++N_-<\infty;$ the set of resonance points is a discrete subset of $\mbR$ so the sum above is well-defined.

4) The number
$\sum_{r_\lambda \in [0,1]} \sign\brs{\brs{\ulP_{\lambda+iy}(r_\lambda)}^*V\ulP_{\lambda+iy}(r_\lambda)},$ where $0<y<\!\!<1,$ $\sign$ is the signature of a finite-rank self-adjoint operator,
and $\ulP_{\lambda+iy}(r_\lambda)$ is the Riesz idempotent
$$
  \ulP_{\lambda+iy}(r_\lambda) = \frac 1{2\pi i} \oint_{C} \brs{\sigma - (H_s-\lambda - iy)^{-1}V}^{-1}\,d\sigma,
$$
where $C$ is a contour which encloses all and only those eigenvalues of $(H_s-\lambda -iy)^{-1}V$ which converge to $(s - r_\lambda)^{-1}$ as $y\to 0^+.$
This definition is also independent of $s\in \mbR.$

\medskip Equality of the third and fourth numbers is proved under a much weaker assumption on $H_0$ and $V$
which includes Schr\"odinger operators. Some applications of this result are given, such as $\abs{N_+-N_-} \leq m,$ where $m$ is the dimension
of the vector space of solutions of the Lippmann-Schwinger equation $(1+(r_\lambda-s)F(H_s-\lambda - i0)^{-1}F^*J)u = 0,$ where $V = F^*JF.$
\end{abstract}

\maketitle

\newpage
\tableofcontents

\section{Introduction}
This paper develops the theory of spectral flow inside essential spectrum. In order to put the results of this paper into context,
in this introduction a quick survey is given of relevant parts of the theory of spectral flow, the mathematical theory of scattering,
and related notions from the perspective of this paper.
In fact, the introduction and the main body of the paper are quite independent; the reader may choose to omit reading this introduction (as long as he does not ask
what is the point and origin of the results of this paper), or treat this introduction as an independent survey.
This also explains the relatively large size of this introduction.

\subsection{Spectral flow}
Spectral flow was introduced by \Atiyah, \Patodi\ and \Singer\ in \cite{APS75,APS76}, as the intersection number of eigenvalues
of a continuous path $D_u, \ 0\leq u \leq 1,$ of elliptic self-adjoint pseudo-differential operators on a compact manifold with the line $\lambda = -\eps,$
where $\eps$ is a small positive number. Atiyah, Patodi and Singer remarked in \cite{APS76} that spectral flow could in fact be defined
for any continuous path of self-adjoint Fredholm operators. Essential spectrum of a self-adjoint Fredholm operator does not contain zero,
and so one can formally define spectral flow as the net number of eigenvalues crossing 0 in the positive direction, where it is assumed that if an eigenvalue
crosses 0 in the negative direction then its contribution to spectral flow in negative.
\Singer\ proposed in~1974 that it should be possible to express spectral flow as an integral of a one-form defined in terms of the path of operators.
Such an analytic formula for spectral flow was established by E.\,Getzler in \cite{Ge93Top}:
\begin{equation} \label{F: Getzler}
  \sflow(\mathrm D,g^{-1}\mathrm Dg) = \frac 1 {\sqrt \pi} \int_0^1 \Tr(\dot {\mathrm D}_u e^{-\mathrm D_u^2})\,du,
\end{equation}
where $\mathrm D$ is a self-adjoint operator of an odd $\theta$-summable Fredholm module (see \cite{CoNG} for definition) $(\clA, \hilb, D)$ over a Banach $*$-algebra~$\clA,$
$g$ is a representative of an element $[g]$ of the odd $K$-theory group $K_1(\clA)$ (see e.g. \cite[\S 8]{Bl} or \cite[Chapter 7]{Murphy} for definition),
and $\mathrm D_u = (1-u)\mathrm D + u g^{-1}\mathrm Dg.$
For example \cite{Ge93Top}, in the case $\hilb = L_2(\mbT,d\theta),$ $\clA = C(\mbT),$ $\mathrm D = \frac 1i \frac d{d\theta},$ and $[g]$ is the class of the function $e^{i n\theta},$ one has
$\mathrm D_u = \mathrm D + n u I,$ where~$I$ is the identity operator, so that $\sigma(\mathrm D_u) = \set{k+nu \colon k \in \mbZ}.$ Thus, as $u$ changes from $0$ to $1,$
each real number including zero is crossed by $n$ simple eigenvalues of $\mathrm D_u$ in the positive direction and therefore $\sflow(\mathrm D,g^{-1}\mathrm Dg) = n.$

For a norm continuous path of self-adjoint Fredholm operators $F \colon [a,b] \to \clB(\hilb),$ where $\clB(\hilb)$ is the algebra of bounded operators,
J.\,Phillips \cite{Ph96CMB,Ph97FIC} gave an alternative definition of spectral flow by formula
$$
  \sflow(\set{F_t}) = \sum_{i=1}^n \ec(P_{t_{i-1}},P_{t_i}),
$$
where $P_t = E_{[0,\infty)}^{F_t}$ is the spectral projection of $F_t$ corresponding to the interval $[0,\infty)$ and $\ec(P,Q)$ is the essential co-dimension
of a Fredholm pair of projections $P,Q$ (see \cite{ASS} for definition, see also \cite{AmSin,Kalt97,Kato55}), which is defined as the Fredholm index of the operator
$PQ \colon Q\hilb \to P\hilb.$ It was shown in \cite{Ph96CMB,Ph97FIC} that this definition of spectral flow is well-defined for and independent
of the choice of small enough partitions and that it is homotopically invariant. The spectral flow $\sflow(F_0,F_1)$ for a pair of Fredholm operators $F_0$ and $F_1$
with compact difference is then defined by the above formula for the straight path $(1-t)F_0+tF_1$ connecting $F_0$ and $F_1,$
and for a pair of self-adjoint operators $D_0,D_1$ with compact resolvents and bounded difference the spectral flow is defined by formula
$$
  \sflow(D_0,D_1) = \sflow(\phi(D_0),\phi(D_1)),
$$
where $\phi(x) = x(1+x^2)^{-1/2}.$
The analytic formula for spectral flow~(\ref{F: Getzler}) was generalized by A.\,Carey and \Phillips\ \cite{CP98CJM}, \cite{CP2},
who in particular proved the following formula \cite[Corollary 8.10]{CP2} for spectral flow for two $\theta$-summable operators $D_0$ and $D_1:$
\begin{equation} \label{F: CP2 formula for SF}
  \begin{split}
    \sflow(D_0,D_1) = \frac 1{\sqrt{\pi}} \int_0^1 \Tr\brs{\frac {dD_t}{dt} \,e^{-D_t^2}}\,dt & + \eta_1(D_1) - \eta_1(D_0) \\ & + \frac 12 \Tr\brs{[\ker (D_1)]} - \frac 12 \Tr\brs{[\ker (D_0)]},
  \end{split}
\end{equation}
where $[\ker (D_j)]$ is the projection onto the kernel of $D_j$ and where the real number
$$
  \eta_1(D) = \frac 1{\sqrt{\pi}} \int_1^\infty \Tr\brs{D\,e^{-tD^2}}\,\frac{dt}{\sqrt{t}}
$$
is the so-called $\eta$-invariant of $D_j,$ --- the notion introduced for self-adjoint elliptic operators on compact manifolds by Atiyah, Patodi and Singer in \cite{APS76}.
A formula analogous to~(\ref{F: CP2 formula for SF}) was also established for $p$-summable operators.
It was also shown in \cite{CP2} that the one-form on the affine space of $\theta$-summable self-adjoint operators $\set{D_0+A \colon A \ \text{is a bounded s.-a. operator}}$
given by formula
$$
  \alpha_D(A) = \frac 1{\sqrt{\pi}} \Tr\brs{A\,e^{-D^2}}
$$
is exact.
The nature of integral formulas for spectral flow such as~(\ref{F: Getzler}),~(\ref{F: CP2 formula for SF}) was clarified in \cite{ACS}, where it was proved \cite[(35)]{ACS}
that for any two self-adjoint operators $D_0$ and $D_1$ with compact resolvent such that $D_1-D_0$ is bounded the following formula holds
\begin{equation} \label{F: ACS formula for SF}
  \sflow(\lambda; D_0,D_1) = \xi_{D_1,D_0}(\lambda) + \frac 12 \Tr\brs{[\ker (D_1-\lambda)]} - \frac 12 \Tr\brs{[\ker (D_0 - \lambda)]},
\end{equation}
where $\xi_{D_1,D_0}(\lambda)$ is the so-called spectral shift function. The formula~(\ref{F: ACS formula for SF}) is quite general
in the sense that firstly it allows to easily recover integral formulas of Getzler
(\ref{F: Getzler}) and Carey-Phillips~(\ref{F: CP2 formula for SF}) by averaging over an appropriate function $\phi(\lambda),$ and secondly, unlike other integral formulas
it does not impose on operators $D_0$ and $D_1$ any summability conditions.

Though in \cite{ACS} the operators~$D_r$ were assumed to have compact resolvent, the same technique of proof shows that a connection between the spectral flow and
spectral shift function given by~(\ref{F: ACS formula for SF}) holds for norm-continuous paths~$D_r$ of self-adjoint operators with trace-class difference
if $\lambda$ does not belong to the common essential spectrum of operators~$D_r$ (see also \cite{Pu08AMST}).

\subsection{Spectral shift function}
The works on spectral flow discussed above were written by geometers, and they were interested in spectral flow primarily as a topological invariant and in its connections with other topological invariants,
such as Chern character (see e.g. \cite{KN,Wells} for definition). See also, for instance, \cite{BCPRSW,BFTokyo98,BLP,CPRS1,CPRS2,CPRS3,CM95GAFA}.
A notion closely related to spectral flow appeared in 1952 in the work of \Lifshitz\ \cite{Li52UMN}.
Lifshitz introduced and developed a formalism of the spectral shift function $\xi(\lambda)$ of a pair of self-adjoint operators~$H_0$ and $H_1$ with finite rank difference
$V = H_1-H_0,$ which was defined by equality
\begin{equation} \label{F: def of xi by Lifshitz}
  \xi(\lambda) = \Tr\brs{E_\lambda^{H_1} - E_\lambda^{H_0}}.
\end{equation}
In particular, Lifshitz observed that the spectral shift function formally satisfies the following equality called the trace formula:
\begin{equation} \label{F: Lifshitz formula}
  \Tr\brs{f(H_1) - f(H_0)} = \int_{-\infty}^\infty f'(\lambda) \xi(\lambda)\,d\lambda.
\end{equation}
Lifshitz introduced the spectral shift function in connection with a problem of solid state physics, in which the initial operator~$H_0$
is the Hamiltonian of a pure crystal and~$V$ is the perturbation introduced by a point impurity, and his work had a formal character.
A mathematically rigorous theory of the spectral shift function was created one year later by \KreinMG\
in \cite{Kr53MS}. \Krein\ showed that for any pair of self-adjoint operators~$H_0$ and $H_1$ with trace-class difference $V = H_1 - H_0$ there exists
a unique (up to a set of zero measure, of course) integrable function $\xi(\lambda),$ such that for all functions $f$ from
a class which includes $C_c^2(\mbR),$ the trace formula~(\ref{F: Lifshitz formula}) holds.
\Krein\ also demonstrated in \cite{Kr53MS} by presenting a counter-example that the equality~(\ref{F: def of xi by Lifshitz}) cannot serve as a definition of the spectral shift function, since
the difference $E_\lambda^{H_1} - E_\lambda^{H_0}$ may fail to be trace-class. Further, a description of the largest class of functions $f$ for which
the trace formula~(\ref{F: Lifshitz formula}) holds was given by \Peller\ in terms of Besov spaces in \cite{PelFA85} (see also \cite{Fa75JSM}).
There is a big literature on the spectral shift function, see e.g. \cite{GM00JAnalM,GM03AA,Pu00PDE,Pu97AA,Si98PAMS}.

\Birman\ and \Solomyak\ showed in \cite{BS75SM} that for any self-adjoint operator~$H_0$ and any trace-class self-adjoint operator~$V$
the spectral shift function $\xi_{H_1,H_0}(\lambda)$ satisfies the equality
\begin{equation} \label{F: BS formula for xi}
  \xi(\lambda) = \frac d{d\lambda} \int_0^1 \Tr\brs{V E_\lambda^{H_r}}\,dr \quad \text{a.e.} \ \lambda,
\end{equation}
where \label{Page: Hr} \label{Page: r}
$$
  H_r = H_0+rV, \ \ r \in \mbR,
$$
and where $E_\lambda^{H}$ is the spectral projection of~$H$ corresponding to the interval $(-\infty,\lambda].$
If we are to interpret the spectral shift function $\xi(\lambda)$ as a distribution $\xi(\phi), \phi \in C_c^\infty(\mbR),$ the Birman-Solomyak
formula~(\ref{F: BS formula for xi}) can be rewritten as
\begin{equation} \label{F: BS formula (2)}
  \xi(\phi) = \int_0^1 \Tr\brs{V \phi(H_r)}\,dr \quad \forall \, \phi \in C_c^\infty(\mbR).
\end{equation}
The Birman-Solomyak formula~(\ref{F: BS formula for xi}) rewritten in the form~(\ref{F: BS formula (2)}) makes a clear connection between the integral formulas for spectral flow
(\ref{F: Getzler}),~(\ref{F: CP2 formula for SF}), etc and the spectral shift function: both of them are integrals of one-forms
\begin{equation} \label{F: one-form alpha(f)}
  \alpha^f_H(V) = \Tr(Vf(H))
\end{equation}
on a real affine space~$H_0 + \clA_0$ of self-adjoint operators, where $\clA_0$ is a real vector space of self-adjoint operators.
This connection was observed and used in \cite{ACS} to derive a general integral formula for spectral flow in the case
of self-adjoint operators~$H$ with compact resolvent and $\clA_0 = \clB_{sa}(\clH).$
It was shown in \cite{ACS} that the one-forms~(\ref{F: one-form alpha(f)}) are exact on the affine space $H+\clB_{sa}(\clH)$
for any compactly supported smooth function $f,$
and therefore integrals over all piecewise smooth continuous paths connecting~$H_0$ and~$H_0+V$ coincide and are equal to the right hand side of
(\ref{F: BS formula (2)}). Analogue of this result was proved in \cite{AS2} for the so-called trace-compatible perturbations, which include
self-adjoint operators with compact resolvent and bounded perturbations, as well as arbitrary self-adjoint operators
and trace-class perturbations. An affine space $\clA=H_0+\clA_0$ of self-adjoint operators is called trace-compatible if for any
operator $H \in \clA,$ any perturbation $V \in \clA_0,$ and any compactly supported continuous function $\phi$ the condition
$V \phi(H) \in \clL_1(\hilb)$ holds,
where $\clL_1(\hilb)$ is the class of operators with finite trace.
This definition was motivated by the distribution version~(\ref{F: BS formula (2)}) of the Birman-Solomyak formula
(\ref{F: BS formula for xi}), since trace-compatibility is the least requirement which one needs to impose
on operators~$H_0+rV$ to give sense to the integral in~(\ref{F: BS formula (2)}).

One of the important developments in the theory of the spectral shift function occurred when \Buslaev\ and \Faddeev\
observed in \cite{BF60DAN} a connection between the spectral shift function and the phase shift of the scattering matrix.
This connection for trace-class perturbations of self-adjoint operators was established by \Birman\ and \KreinMG\ in \cite{BK62DAN};
namely, for self-adjoint operators~$H_0$ and $H_1$ with trace-class difference $V = H_1-H_0$ they proved the formula
\begin{equation} \label{F: Birman-Krein formula}
  e^{-2\pi i \xi(\lambda)} = \det S(\lambda; H_1,H_0),
\end{equation}
where $S(\lambda; H_1,H_0)$ is the scattering matrix for the pair $(H_1,H_0)$ (see e.g. \cite{Ya}), definition of which follows in the next subsection, $\det$ is
the Fredholm determinant (see e.g. \cite[Chapter 4]{GK}, \cite[Chapter 3]{SimTrId2} or \cite[\S XIII.7]{RS4})
and $\xi(\lambda)$ is the spectral shift function of the pair $(H_1,H_0).$

\subsection{Scattering theory}
The scattering operator $\bfS(H_1,H_0)$ for a pair
of self-adjoint operators is defined by formula (see e.g. \cite{BW,RS3,Ya})
\begin{equation} \label{F: bfS=W*(+)W(-)}
  \bfS(H_1,H_0) = W_+^*(H_1,H_0)W_-(H_1,H_0),
\end{equation}
where the M\"oller wave operators $W_\pm(H_1,H_0)$ are defined, if they exist, as strong operator limits
\begin{equation} \label{F: W(pm) class-l def-n}
  W_\pm(H_1,H_0) = \lim_{t \to \pm \infty} e^{i t H_1} e^{-i t H_0}P^{(a)}(H_0),
\end{equation}
where $P^{(a)}(H_0)$ is the orthogonal projection onto the absolutely continuous subspace of $H_0$
(for definition, see e.g. \cite[Theorem VII.4]{RS1} and the definition preceding this theorem).
The classical Kato-Rosenblum theorem (\cite{KaPJA57,RoPJM57}, see also \cite[Theorem XI.8]{RS3}, \cite[Theorem 6.2.3]{Ya}) asserts that if the difference $H_1-H_0$ is trace-class, then the wave operators $W_\pm(H_1,H_0)$
exist and are therefore complete (by symmetry of the condition $H_1-H_0 \in \clL_1(\hilb)$),
which implies that the scattering operator~(\ref{F: bfS=W*(+)W(-)}) exists as well. Completeness of wave operators means that both operators $W_+(H_1,H_0)$ and
$W_-(H_1,H_0)$ are partial isometries whose initial space is the absolutely continuous subspace $\hilb^{(a)}(H_0)$ with respect to~$H_0$ and the final space is
the absolutely continuous subspace $\hilb^{(a)}(H_1)$ with respect to $H_1.$
One of the many versions of the Spectral Theorem asserts that, given a self-adjoint operator~$H_0,$
the absolutely continuous subspace $\hilb^{(a)}(H_0)$ of~$H_0$ admits representation as a direct integral of Hilbert spaces
\begin{equation} \label{F: euF: hilb(a) to direct int-l}
  \euF \colon \hilb^{(a)}(H_0) \to \int_{\hat \sigma_{H_0}}^\oplus \hlambda\,\rho(d\lambda),
\end{equation}
such that for any $f \in \hilb^{(a)}(H_0) \cap \dom(H_0)$ the equality
$$
  \euF(H_0 f)(\lambda) = \lambda \euF(f)(\lambda)
$$
holds for a.e. $\lambda \in \hat \sigma_{H_0},$
where $\hat \sigma_{H_0}$ is a core of the absolutely continuous spectrum of~$H_0,$
$\set{\hlambda, \,\lambda \in \hat \sigma_{H_0}}$ is a measurable family of Hilbert spaces, $\rho$ is an absolutely continuous Borel measure with Borel support $\hat \sigma_{H_0}$
and $\euF$ is a unitary operator; for definition of direct integral of Hilbert spaces see e.g. \cite{BW,BSbook}.
By Kato-Rosenblum theorem, the scattering operator $\bfS(H_1,H_0)$ is a partial isometry with initial and final space $\hilb^{(a)}(H_0),$
further, the scattering operator $\bfS(H_1,H_0)$ commutes with~$H_0;$ these properties of the scattering operator imply (see e.g. \cite{BE,BYa92AA,BYa92AA2,Ya})
that in the spectral representation~(\ref{F: euF: hilb(a) to direct int-l}) of the absolutely continuous part of the Hilbert space
the scattering operator~(\ref{F: bfS=W*(+)W(-)}) is represented by a direct integral
\begin{equation} \label{F: bfS = int S(l)dl}
  \bfS(H_1,H_0) = \int_{\hat \sigma_{H_0}}^\oplus S(\lambda; H_1,H_0)\,\rho(d\lambda),
\end{equation}
where $\set{S(\lambda; H_1,H_0),\,\lambda \in \hat \sigma_{H_0}}$ is a measurable family of unitary operators on fiber Hilbert spaces~$\hlambda.$
The spectral parameter~$\lambda$ has physical meaning of energy~$E;$ the fiber Hilbert space~$\hlambda$ is often called an energy shell.
Physicists call the unitary operator $S(\lambda; H_1,H_0)$ the on-shell scattering operator, while the scattering operator $\bfS(H_1,H_0)$ itself
is called the off-shell scattering operator (see e.g. \cite[\S 3-b]{TayST}, see also \cite[Theorem XI.42]{RS3} and a discussion followed after this theorem).
In physics there is a famous stationary formula
mainly due to B.\,Lippmann and J.\,Schwinger \cite{LippSch50} and Gell'Mann-Goldberger \cite{GellGold53}
for the on-shell scattering operator (see e.g. \cite{TayST}, \cite[Theorem XI.42]{RS3})
\begin{equation} \label{F: stationary formula from Taylor}
  \begin{split}
    \la \mathbf p'| \mathsf S | \mathbf p \ra & = \delta_3(\mathbf p' - \mathbf p) - 2\pi i \delta(E_{p'}-E_p) \la \mathbf p'| V(1-G^0(E_p+i0)V)^{-1}| \mathbf p \ra
    \\ & = \delta_3(\mathbf p' - \mathbf p) - 2\pi i \delta(E_{p'}-E_p) \la \mathbf p'| (V+VG(E_p+i0)V)| \mathbf p \ra,
  \end{split}
\end{equation}
which follows from combination of \cite[(3.7), (8.11) and (8.22)]{TayST}. This is a version of the stationary formula for one spinless particle, being scattered
by a potential $V;$ there are stationary formulas for particles with a spin and for multi-particle systems as well, see e.g. \cite{TayST}.

In rigorous mathematical notation the stationary formula~(\ref{F: stationary formula from Taylor}) for a self-adjoint operator~$H_0$ and its trace-class perturbation $H_1=H_0+V$ should have been written as
\begin{equation} \label{F: unrigorous stationary formula}
  S(\lambda; H_1,H_0) = 1 - 2 \pi i \,\euF_\lambda\, V(1-R_{\lambda+i0}(H_0)V)^{-1} \euF^*_\lambda,
\end{equation}
where $\euF_\lambda \colon \hilb^{(a)}(H_0) \to \hlambda$ is a fiber of the unitary operator~(\ref{F: euF: hilb(a) to direct int-l}).
But, unfortunately, the expression on the right hand side of~(\ref{F: unrigorous stationary formula})
does not make sense for two reasons: firstly, the limit of the resolvent $R_{\lambda+i0}(H_0) := (H_0-\lambda-i0)^{-1}$
does not in general exist even in the weakest of all reasonable topologies (for a discussion of this question see e.\,g.~\cite[\S 6.1]{Ya}),
and secondly, the operator $\euF_\lambda$ is not well-defined for a particular value of $\lambda.$

A mathematically rigorous version of the stationary formula for the scattering matrix~(\ref{F: stationary formula from Taylor})
was established by \Faddeev\ \cite{Fa64} (see also \cite{LF58}) in the setting of Friedrichs-Faddeev model \cite{Fr38,Fr48,FrBook,Ya}.
In Friedrichs-Faddeev model the initial self-adjoint operator $H_0$ is an operator of multiplication by
the independent variable $x$ in the Hilbert space $L_2[a,b; \mathfrak h],$ $-\infty \leq a < b \leq \infty,$
of square-integrable $\mathfrak h$-valued functions, where $\mathfrak h$ is a fixed Hilbert space,
and the perturbation operator~$V$ is an integral operator
$$Vf(x) = \int_a^b v(x,y)f(y)\,dy,$$ with sufficiently regular kernel $v \colon [a,b]^2 \to \mathfrak h.$
A detailed exposition of stationary scattering theory for Friedrichs-Faddeev model can be found in \cite[Chapter 4]{Ya}.

Another important setting is short range potential scattering theory, see e.g. \cite{Povz53,Povz55,Ikebe60,Kato69,Agm,KuJMSJ73I,KuJMSJ73II}; expositions of this theory and literature can be found in
\cite{Agm,Kur}, see also \cite{Ya2001}. In potential scattering theory the initial operator $H_0$ is
the Laplace operator
\begin{equation} \label{F: H0=-Delta}
  H_0 f = - \Delta f
\end{equation}
on the Hilbert space $L_2(\mbR^n),$ where the domain of $H_0$ is the Sobolev space $\mathsf H_2(\mbR^n)$ (see e.g. \cite[IX.6]{RS2} for definition);
a short range perturbation~$V$ is an operator of multiplication by a measurable function $q \colon \mbR^n \to \mbR,$ which satisfies an estimate
$\abs{q(x)} \leq C (1+\abs{x}^2)^{-\rho/2},$ where $\rho > 1$ (in \cite{Agm} short range potentials are defined by a weaker condition
of integral type). The perturbed operator~$H$ is the \Schrodinger operator
\begin{equation} \label{F: Hu=H0(u)+qu}
  Hu(x) = -\Delta u(x) + q(x)u(x).
\end{equation}
In this case the spectral structure of the initial operator $H_0$ is completely transparent since $H_0$ can be diagonalized by the Fourier transform $\euF,$ that is,
\begin{equation} \label{F: H0=F* xi2 F}
  H_0 = \euF^* M_{\abs{\xi}^2}\euF,
\end{equation}
where $M_{\abs{\xi}^2}$ is the operator of multiplication by $\abs{\xi}^2.$ So, in this case $\hilb^{(a)}(H_0) = \hilb$ and in the decomposition
(\ref{F: euF: hilb(a) to direct int-l}) one can take a core of the absolutely continuous spectrum $\hat \sigma_{H_0}$ to be $(0,\infty),$
the measure $\rho(d\lambda)$ to be Lebesgue measure $d\lambda$ and the fiber Hilbert space~$\hlambda$ to be $L_2(\Sigma_{\sqrt \lambda}),$
where $\Sigma_{\sqrt \lambda} = \set{\xi \in \mbR^n_\xi \colon \abs{\xi} = \sqrt{\lambda}}$ is the sphere with surface measure inherited from $\mbR^n_\xi.$
The scattering operator~(\ref{F: bfS=W*(+)W(-)}) for the pair of operators $(H,H_0)$ given by~(\ref{F: Hu=H0(u)+qu}) and~(\ref{F: H0=-Delta})
exists and it admits the decomposition~(\ref{F: bfS = int S(l)dl}). Further, for all $\lambda > 0$
with possible exception of a discrete subset $e_+(H)$ \label{Page e+(H)} of positive values of~$\lambda$ the stationary formula for the scattering matrix holds in the following form
\begin{equation} \label{F: stationary formula for Schrodinger op-r}
  \begin{split}
    S(\lambda) & = 1 - 2\pi i c(\lambda) \gamma_0(\lambda) \euF V(1 + R_{\lambda+i0}(H_0)V)^{-1}\euF^* \gamma_0^\diamondsuit(\lambda)
    \\ & = 1 - 2\pi i c(\lambda) \gamma_0(\lambda) \euF (V - V R_{\lambda+i0}(H)V)\euF^* \gamma_0^\diamondsuit(\lambda),
  \end{split}
\end{equation}
explanation of which follows\footnote{A sign mismatch in formulas~(\ref{F: stationary formula from Taylor}) and
(\ref{F: stationary formula for Schrodinger op-r}) comes from definitions of the resolvent $R_z(H) = (H-z)^{-1}$ and of the Green operator $G(z) = (z-H)^{-1},$ as it is defined in \cite[\S 8-a]{TayST}}
(for details see \cite{Agm,Kur}). Firstly, here $c(\lambda)$ is a constant which occurs as a result of change from Cartesian coordinates
to polar coordinates in the momentum space~$\mbR^n_\xi.$ For any $s \in \mbR$ let $L_{2,s}(\mbR^n)$
be the weighted Hilbert space of measurable functions $u \colon \mbR^n \to \mbC$ for which
$$
  \norm{u}_{0,s} := \int_{\mbR^n} \abs{u(x)}(1+\abs{x}^2)^{s/2}\,dx < \infty,
$$
and let
$$
  \mathsf H_{m,s}(\mbR^n) = \set{u \colon D^\alpha u \in L_{2,s}(\mbR^n), 0 \leq \abs{\alpha} \leq m}
$$
be the weighted Sobolev space with norm $$\norm{u}_{m,s} = \brs{\sum_{\abs{\alpha}\leq m}\norm{D^\alpha u}_{0,s}^2}^{1/2}.$$
A rigorous treatment of the stationary formula in potential scattering theory is based on the following theorems, proofs of which can be found in \cite{Agm,Kur}.
In general, a form of the Limiting Absorption Principle is of the utmost importance for stationary scattering theory.
\begin{thm} \label{T: Agmon potentials}
If $q(x)$ is a short range potential, then there exists $\eps'>0$ such that
for any $s \in \mbR$ and for all $\eps \in (0,\eps')$
the operator of multiplication by $q(x)$ is a compact operator from the Hilbert space $\mathsf H_{2,s}(\mbR^n)$ to the Hilbert
space $L_{2,1+s+\eps}(\mbR^n).$
\end{thm}
\begin{thm} \label{T: LAP for -Delta} (The Limiting Absorption Principle for $-\Delta,$ see \cite[Theorem 4.1]{Agm}, \cite[\S 4.4]{Kur})
Let $H_0 = -\Delta$ with domain $\mathsf H_2(\mbR^n).$
For any $s>1/2$ and for any $\lambda > 0$ the resolvents $R_{\lambda\pm iy}(H_0)$ as operators from $L_{2,s}(\mbR^n)$
to $\mathsf H_{2,-s}(\mbR^n)$ converge in the uniform operator topology as $y \to 0,$
so the bounded operators $R_{\lambda\pm i0}(H_0) \in \clB(L_{2,s}(\mbR^n), \mathsf H_{2,-s}(\mbR^n))$ exist.
\end{thm}
\begin{thm} \label{T: Agmon 3.1} (see e.g. \cite[Theorem 3.1]{Agm}, \cite[Theorem XIII.33]{RS4})
Let $H = -\Delta+V$ be a \Schrodinger operator with domain $\mathsf H_2(\mbR^n),$
where~$V$ is a short range potential.
The set $e_+(H)$ of positive eigenvalues of~$H$ is a discrete subset of $(0,\infty),$ all eigenvalues from $e_+(H)$ have finite multiplicity
and the only possible limit points of $e_+(H)$ are 0 and $\infty.$
\end{thm}
\begin{thm} \label{T: LAP for -Delta+V} (The Limiting Absorption Principle for $-\Delta+V,$ see \cite[Theorem 4.2]{Agm}, \cite[\S 5.3]{Kur})
Let $H = -\Delta+V$ be a \Schrodinger operator with short range potential $V.$
For any $s>1/2$ and for any $\lambda > 0$ not in $e_+(H)$ the resolvents $R_{\lambda\pm iy}(H)$ as operators from $L_{2,s}(\mbR^n)$ to $\mathsf H_{2,-s}(\mbR^n)$ converge
in the uniform operator topology as $y \to 0,$
so the bounded operators $R_{\lambda\pm i0}(H) \in \clB(L_{2,s}(\mbR^n), \mathsf H_{2,-s}(\mbR^n))$ exist.
\end{thm}

\noindent
Further, for any $s \in \mbR$ the Fourier transform $\euF$ is a unitary operator from $L_{2,s}(\mbR^n)$ onto $\mathsf H_{{s}}(\mbR^n).$
For any $s>1/2$ the term $\gamma_0(\lambda)$ in~(\ref{F: stationary formula for Schrodinger op-r}) is a well-defined bounded operator from
the Hilbert space $\mathsf H_s(\mbR^n)$ to $L_2(\Sigma_{\sqrt \lambda})$ (the Trace Theorem, see e.g. \cite[\S 2]{Agm}, \cite[Theorem 4.2.1]{Kur});
namely, the operator $\gamma_0(\lambda)$ is a continuous extension of the restriction operator $$C_c^\infty(\mbR^n_\xi) \ni f \mapsto f\big|_{\Sigma_{\sqrt \lambda}} \in L_2(\Sigma_{\sqrt \lambda}).$$
Finally, the bounded operator $\gamma_0^\diamondsuit(\lambda) \colon L_2(\Sigma_{\sqrt \lambda}) \to \mathsf H_{-s}$ can be defined for any $s > 1/2$ by formula
\begin{equation} \label{F: gamma0 diamond}
  \scal{\gamma_0^\diamondsuit(\lambda) f}{g}_{-s,s} = \scal{f}{\gamma_0(\lambda) g}_{L_2(\Sigma_{\sqrt \lambda})},
\end{equation}
where $f \in L_2(\Sigma_{\sqrt \lambda})$ and $g \in \mathsf H_{s}(\mbR^n)$ and $\scal{\cdot}{\cdot}_{-s,s}$ is the natural pairing of Hilbert spaces $\mathsf H_{-s}(\mbR^n)$ and $\mathsf H_{s}(\mbR^n),$
defined by formula
$
  \scal{f}{g}_{-s,s} = \int_{\mbR^n} \overline{\hat f(\xi)} g(\xi) \,d\xi.
$
So, the stationary formula~(\ref{F: stationary formula for Schrodinger op-r}) acquires a precise meaning if factors in the right hand side of this formula
are understood as acting between appropriately chosen Hilbert spaces as follows:
\begin{equation*}
  \begin{split}
    L_2(\Sigma_{\sqrt \lambda}) \stackrel {\gamma_0}\longleftarrow \mathsf H_{s-\eps+\eps'} & \stackrel {\euF}\longleftarrow L_{2,s-\eps+\eps'} \stackrel {V}\longleftarrow  \mathsf H_{2,-1+s-\eps}
    \stackrel {R_{\lambda+i0}(H)} {\longleftarrow\!\!\!-\!\!-\!\!-}
    \\ & \stackrel {R_{\lambda+i0}(H)} {\longleftarrow\!\!\!-\!\!-\!\!-}   L_{2,1-s+\eps} \stackrel {V}\longleftarrow  L_{2,-s}
    \stackrel {\euF^*}\longleftarrow  \mathsf H_{-s} \stackrel {\gamma_0^\diamondsuit}\longleftarrow L_2(\Sigma_{\sqrt \lambda}),
  \end{split}
\end{equation*}
as long as the numbers $s,$ $1-s+\eps$ and $s-\eps+\eps'$ are chosen so that they are all $ >1/2;$ it is obviously possible to choose such $s,\eps,\eps'.$
The set of eigenvalues $e_+(H)$ of~$H$ is related to the set of points~$\lambda$ for which the operator $1+R_{\lambda+i0}(H_0)V$ is not invertible
(see e.g. proof of \cite[Theorem 4.2]{Agm}), and the operator $H E^H_{(0,\infty)\setminus e_+(H)}$ is absolutely continuous \cite[Theorem 6.1]{Agm}.

A mathematically rigorous version of the stationary formula~(\ref{F: unrigorous stationary formula}) for arbitrary self-adjoint trace-class perturbations of arbitrary self-adjoint operators
was proved in \cite{BE} (see also \cite{Ya}). To give~(\ref{F: unrigorous stationary formula})
a rigorous meaning, one needs to introduce an artificial factorization of the perturbation operator~$V.$ Assuming that~$V$ is trace-class,
it is possible to write~$V$ in the form $G^*JG,$ where~$G$ is a Hilbert-Schmidt operator acting from the Hilbert space $\hilb$ to possibly another Hilbert space~$\clK$
and where $J$ is a bounded operator on this auxiliary Hilbert space $\clK.$ Using the factorization $V = G^*JG,$ the formal formula~(\ref{F: unrigorous stationary formula}) can be rewritten as
\begin{equation*}
  S(\lambda; H_1,H_0) = 1 - 2 \pi i \,(\euF_\lambda G^*)\, J(1-G R_{\lambda+i0}(H_0)G^*J)^{-1} G\,\euF^*_\lambda, \ \text{a.e.} \ \lambda \in \mbR,
\end{equation*}
or, introducing notation
\begin{equation} \label{F: Z(l;G)}
  Z_0(\lambda; G) = \euF_\lambda G^*
\end{equation}
and
$$
  T_{\lambda+i0}(H_0) = GR _{\lambda+i0}(H_0)G^*,
$$
as
\begin{equation} \label{F: rigorous stationary formula}
  S(\lambda; H_1,H_0) = 1 - 2 \pi i \,Z_0(\lambda; G) J(1-T_{\lambda+i0}(H_0)J)^{-1}Z_0^*(\lambda; G), \ \text{a.e.} \ \lambda \in \mbR.
\end{equation}
In this formula the two hindrances mentioned above are overcome: the abstract limiting absorption principle
(a theorem proved in \cite{BE,deBranges}, see Theorem~\ref{T: Ya thm 6.1.9} below) asserts that the limit
$T_{\lambda+i0}(H_0)$ exists in Hilbert-Schmidt norm for a.e.~$\lambda,$ and the product $Z_0(\lambda; G) = \euF_\lambda G^*$
also makes sense for a.e.~$\lambda$ as an operator from~$\clK$ to~$\hlambda$ and moreover this product is Hilbert-Schmidt.
Nonetheless, it should be noted that while $S(\lambda; H_1,H_0)$ is defined by the right hand side of~(\ref{F: rigorous stationary formula})
for almost every value of~$\lambda,$ still for no particular choices of~$\lambda \in \mbR$ is the operator $S(\lambda; H_1,H_0)$ well-defined.
The source of this uncertainty is in the factor $Z_0(\lambda; G)$ definition~(\ref{F: Z(l;G)}) of which involves the unitary operator $\euF$
from the spectral representation~(\ref{F: euF: hilb(a) to direct int-l}). This uncertainty is not possible to eradicate, since
in the spectral representation~(\ref{F: euF: hilb(a) to direct int-l}) the choice of a core $\hat \sigma$ of absolutely continuous spectrum is arbitrary, partially due to possible presence
of pure point and singular continuous spectrum, and since the measure $\rho$ can be replaced by any other measure of the same spectral type. This circumstance was not considered as a hindrance
in abstract scattering theory in which one works as a rule with two operators, --- initial~$H_0$ and perturbed $H_1.$
However, in~\cite{Az} in an attempt to find an operator version of the Birman-\Krein\ formula~(\ref{F: Birman-Krein formula}) the following formula was derived
\begin{equation} \label{F: Texp formula for S(l)}
  S(\lambda; H_1,H_0) = \Texp \brs{-2\pi i \int_0^1 w_+(\lambda; H_0,H_r)Z_r(\lambda; G)JZ_r^*(\lambda; G)w_+(\lambda; H_r,H_0)\,dr},
\end{equation}
where subindex $r$ in $Z_r$ indicates that in~(\ref{F: Z(l;G)}) the unitary operator $\euF$ is from the spectral representation of $H_r = H_0+rV,$
and where the so-called wave matrix (see e.g. \cite{Ya})
$$
  w_\pm(\lambda; H_1,H_0) \colon \hlambda(H_0) \to \hlambda(H_1)
$$
is taken from the direct integral representation of the wave operator $W_\pm(H_1,H_0):$
\begin{equation} \label{F: W(pm)= int w(pm)}
  W_\pm(H_1,H_0) = \int^\oplus_{\hat \sigma_{H_0}} w_\pm(\lambda; H_1,H_0)\,\rho(d\lambda),
\end{equation}
analogous to the spectral representation~(\ref{F: bfS = int S(l)dl}) of the scattering operator $\bfS(H_1,H_0).$
(For a rigorous definition and basic properties of the chronological exponential $\Texp\brs{\int_a^b A(s)\,ds}$ of a path of trace-class operators $A(s)$
continuous in trace-class norm which were used in the proof of~(\ref{F: Texp formula for S(l)}) see \cite[Appendix A]{Az3v6}; for formal definition of $\Texp$ see e.g. \cite[Chapter 4]{BSh}).
Proof of the formula~(\ref{F: Texp formula for S(l)}) relies on validity of the stationary formula~(\ref{F: rigorous stationary formula})
for a continuous family $\set{H_r \colon r \in [0,1]}$ of operators, and, more importantly, it
requires the operators $w_+(\lambda; H_r,H_0)$ and $Z_r(\lambda; G)$
to be well-defined for a \emph{continuous} set $[0,1]$ of values of $r.$ For this reason, proof of~(\ref{F: Texp formula for S(l)})
works only under stringent conditions on the operators~$H_0$ and~$V$ which ensure existence of operators $w_+(\lambda; H_r,H_0)$ and $Z_r(\lambda; G).$
As it was discussed above, these stringent conditions which were postulated in \cite{Az} hold for a class of short-range \Schrodinger operators.
Further, it was observed in \cite{Az} that provided the operator $S(\lambda; H_1,H_0) - 1$ is trace class
the equality~(\ref{F: Texp formula for S(l)}) implies the following modified Birman-\Krein\ formula
\begin{equation} \label{F: Birman-Krein-Azamov formula}
  e^{-2\pi i \xia(\lambda)} = \det S(\lambda; H_1,H_0),\ \text{a.e.} \ \lambda \in \mbR,
\end{equation}
where the function $\xia(\lambda) = \xia_{H_1,H_0}(\lambda),$ called in \cite{Az} absolutely continuous spectral shift function,
can be defined as the density of the absolutely continuous measure $\xia(\phi), \, \phi \in C_c(\mbR),$ given by formula
\begin{equation} \label{F: xia def}
  \xia(\phi) = \int_0^1 \Tr\brs{V \phi(H_r^{(a)})}\,dr, \quad \phi \in C_c(\mbR).
\end{equation}
Here the self-adjoint operator~$H_r^{(a)}$ is the absolutely continuous part of~$H_r.$
Analogously, one can define the singular spectral shift function $\xis(\lambda),$ which
can be defined as the density of the absolutely continuous measure $\xis(\phi), \, \phi \in C_c(\mbR),$ defined by formula
\begin{equation} \label{F: xis def}
  \xis(\phi) = \int_0^1 \Tr\brs{V \phi(H_r^{(s)})}\,dr, \quad \phi \in C_c(\mbR),
\end{equation}
where the self-adjoint operator~$H_r^{(s)}$ is the singular part of~$H_r.$
One can note that definitions of the functions $\xia$ and $\xis$ are modifications of Birman-Solomyak formula~(\ref{F: BS formula (2)})
for the spectral shift function $\xi,$ and these functions are related by equality
\begin{equation} \label{F: xi=xia+xis}
  \xi = \xia + \xis,
\end{equation}
which is an immediate consequence of~(\ref{F: BS formula (2)}),~(\ref{F: xia def}) and~(\ref{F: xis def}).
In particular, absolute continuity of the measure $\xis$ follows from the equality~(\ref{F: xi=xia+xis}).
Now, the Birman-\Krein\ formula~(\ref{F: Birman-Krein formula}) combined with~(\ref{F: Birman-Krein-Azamov formula}) implies the equality
$e^{-2\pi i \xis(\lambda)} = 1$ for a.e.~$\lambda,$ that is,
\begin{equation} \label{F: xis is in Z}
  \xis(\lambda) \in \mbZ \ \ \text{for a.e.} \ \lambda \in \mbR.
\end{equation}
By Weyl's Theorem on stability of essential spectrum of a self-adjoint operator under relatively compact perturbations (see e.g. \cite[\S IV.5.6]{Kato}, \cite[\S XIII.4]{RS4}),
the essential spectra of all operators~$H_r = H_0+rV$ are identical. Hence, it follows from definition~(\ref{F: xia def}) that the absolutely continuous
spectral shift function $\xia$ vanishes outside the common essential spectrum of operators~$H_r.$ Therefore, outside the essential spectrum the singular spectral shift function $\xis$
coincides with spectral shift function; equivalently, it coincides with the spectral flow. But unlike the spectral flow the singular spectral shift function
is still defined inside the essential spectrum too as an a.e. integer-valued function.
On the basis of this observation, it was suggested in \cite{Az} (see also \cite{Az3v6}) that the singular spectral shift function should be regarded
as a natural extension of spectral flow into essential spectrum. This definition of spectral flow inside essential spectrum has a significant drawback in the sense
that definition~(\ref{F: xis def}) is hard to work with, since it requires diagonalization of a continuous family of self-adjoint operators.
In \cite{Az7} a new equivalent definition of spectral flow inside essential spectrum called total \emph{resonance index} was given.
The total resonance index coincides with singular spectral shift function $\xis(\lambda)$ for a.e.~$\lambda,$ but unlike
the singular spectral shift function $\xis(\lambda)$ the resonance index is a quite tangible
and easy to work with notion. Resonance index is defined as a difference of two non-negative integers and it makes sense outside essential spectrum too, thus providing
a new definition of spectral flow. In this paper we also show that resonance index is equal to the signature of a finite-rank self-adjoint operator naturally associated
with the data $(\lambda, H, V).$

These considerations however are based on the formula~(\ref{F: Texp formula for S(l)}).
A rigorous justification and a proof of this formula, given in \cite{Az3v6} for trace-class perturbations, required development of a new approach to stationary scattering theory.
It turns out that~(\ref{F: Texp formula for S(l)}) holds under much weaker conditions; the proof
is based on adjustment of the new approach to stationary scattering theory
given in~ \cite{Az3v6}. This approach is discussed in the next subsection.

\subsection{Constructive approach to stationary scattering theory}
In one of the basic settings of abstract mathematical scattering theory one studies arbitrary initial self-adjoint operator~$H_0$ and a relatively trace-class perturbation $H_1= H_0+V$
of~$H_0.$ In this setting not only proof of the formula~(\ref{F: Texp formula for S(l)}) given in \cite{Az} does not work, but the formula itself does not make sense,
since for any fixed value of the coupling constant $r$ the ingredients of this formula such as $w_+(\lambda; H_0,H_r)$ and $Z_r(\lambda; G)$ are defined only for a.e. $\lambda.$
Indeed, the right hand side of~(\ref{F: Texp formula for S(l)}), which involves a continuous family of such operators, may be defined only for a set of values of~$\lambda$
which can potentially be as small as the empty set; more importantly whatever this set is one has no control over it.
This circumstance is apparently a serious hindrance on the way of any attempt to give sense and to prove the formula~(\ref{F: Texp formula for S(l)}).
In fact, a proof of~(\ref{F: Texp formula for S(l)}) for arbitrary self-adjoint trace-class perturbations of arbitrary self-adjoint operators
required to give new definitions of basic notions and to give new proofs of basic theorems of abstract scattering theory.
There are several reasons for this; firstly, definition of the operator $Z_r(\lambda; G)$ involves the operator $\euF_\lambda$ from the spectral representation
(\ref{F: euF: hilb(a) to direct int-l}) for the operator~$H_r,$ and for this reason the set of values of the spectral
parameter~$\lambda$ for which $Z_r(\lambda; G)$ is defined cannot be pinpointed: it is an arbitrary core of spectrum of~$H_r.$
Secondly, in the classical approach to abstract scattering theory \cite{BE,Ya}, the scattering matrix $S(\lambda; H_1,H_0)$ cannot be defined for a fixed single value of
$\lambda.$ This situation is analogous to the fact that while the notion of a measurable function makes perfect sense,
value of a measurable function at a given point does not. Thirdly, if one traces out a proof given in e.g. \cite{BE,Ya} of a formula, involving the scattering matrix $S(\lambda; H_1,H_0),$
such as the stationary formula~(\ref{F: rigorous stationary formula}), then one finds that during numerous steps of the proof one throws out from an initial
core of absolutely continuous spectrum $\hat \sigma_{H_0}$ several finite and/or countable families of null sets. It is necessary to stress here that firstly an initial core of
absolutely continuous spectrum is chosen arbitrarily and it is not a constructive object, secondly, the null sets being thrown away from a core depend on arbitrarily chosen objects,
with no clear connections to the main objects of study, namely, the operators~$H_0$ and~$V.$

An approach to scattering theory which partly addresses this issue was given in the paper of Kato and Kuroda \cite{KK71}.
In this paper the authors construct wave matrices $w_\pm(\lambda; H_1,H_0)$ for a set of full Lebesgue measure which depends on a fixed vector space $\euX$
in the Hilbert space. However, in \cite{KK71} only a fixed pair of self-adjoint operators $(H_1,H_0)$ is studied and it remains unclear how the theory presented in \cite{KK71}
could be applied to prove~(\ref{F: Texp formula for S(l)}) and~(\ref{F: xis is in Z}).
On the other hand, numerous monographs and surveys on mathematical scattering theory, e.g. \cite{BW,RS3,Ya,BYa92AA2}, which appeared after publication of \cite{KK71}, do not contain a discussion of this problem.

An approach to scattering theory for trace-class perturbations of arbitrary self-adjoint operators was developed in \cite{Az3v6} with primary aim to give sense
and to prove formula~(\ref{F: Texp formula for S(l)}) for the scattering matrix $S(\lambda; H_1,H_0).$ Unlike the conventional approach
of \cite{BE,Ya}, in the approach to scattering theory given in \cite{Az3v6} one first defines the wave matrices $w_\pm(\lambda; H_1,H_0)$
and the scattering matrix $S(\lambda; H_1,H_0)$ for all values of the spectral parameter~$\lambda$ from an explicit set of full Lebesgue measure $\Lambda,$ which is defined
beforehand, while the wave operators $W_\pm(H_1,H_0)$ and the scattering operator $\bfS(H_1,H_0)$ thus become derivative objects which are \emph{defined}
by formulas~(\ref{F: W(pm)= int w(pm)}) and~(\ref{F: bfS = int S(l)dl}). Further, in the course of constructing the theory, not a single number from the full set $\Lambda$ is removed,
and all objects of the scattering theory are explicitly constructed, in contrast to conventional scattering theory.
The main steps of this theory are as follows.
Proofs of the following theorems are given in \cite{Az3v6} in case of trace-class $V$ and will appear in \cite{AzDa} in general case, see also
\cite{Az10} for the general case.

I. The main data for constructing a scattering theory are a self-adjoint operator $H_0$ on a Hilbert space $\hilb$
and a self-adjoint perturbation operator~$V.$ The pair $H_0,V$ is assumed to be compatible in a certain sense specified below.
In addition to these data, one needs an additional structure. This additional structure is a rigging operator.
A rigging operator~$F$ is a closed operator with trivial kernel and co-kernel which acts
from the main Hilbert space $\hilb$ to some auxiliary Hilbert space $\clK,$ such that the operator~$V$ admits
a well-defined decomposition $V = F^*JF$ with a bounded self-adjoint operator $J$ on $\clK.$ All objects of scattering theory discussed below depend only on the data $H_0,V$ and $F.$

The pair $(H_0,F)$ must be such that the operator
$$
  T_{z}(H_0) := F R_z(H_0)F^* = F(H_0-z)^{-1}F^*,
$$
called the sandwiched resolvent, is well-defined and compact for non-real $z.$

II. The next step is to define the set of values of the spectral parameter~$\lambda$ for which the wave matrices $w_\pm(\lambda; H_1,H_0)$ are to be defined.
The set $\Lambda(H_0,F)$ is defined as the set of all real numbers~$\lambda$ such that the limits
$$
  \lim_{y \to 0} T_{\lambda\pm iy}(H_0)
$$
exist in the uniform norm.

To ensure existence of the spectral shift functions~(\ref{F: Birman-Krein formula}) and~(\ref{F: xia def}) one has to impose
an additional condition that the operator $\Im T_{z}(H_0)$ is trace-class and that
$$
  \lim_{y \to 0^+} \Im T_{\lambda+iy}(H_0)
$$
exists in trace-class norm, but for the scattering theory this is not necessary and this can be done at a later stage.
It turns out however that, unlike the situation with functions $\xi$ and $\xia,$ to be able to define $\xis$ one does not need a trace-class condition.

The set $\Lambda(H_0,F)$ is assumed to have full Lebesgue measure.
In certain important cases this assumption holds.
The corresponding theorems are called the limiting absorption principle.
Two of the main cases for which the limiting absorption principle holds are
\begin{enumerate}
 \item an arbitrary self-adjoint operator $H_0$ and a Hilbert-Schmidt rigging operator $F$ (see e.g. \cite[Theorems 6.1.5 and 6.1.9]{Ya}) and
 \item a Schr\"odinger operator $H_0 = -\Delta+V_0$ and a rigging operator $F = \sqrt{\abs{V}},$ where $V_0$ and $V$
   are short range potentials (Theorems~\ref{T: LAP for -Delta} and~\ref{T: LAP for -Delta+V}).
\end{enumerate}

The role of the set $\Lambda(H_0,F)$ in constructive approach to stationary scattering theory
is about the same as the role of the set $(0,\infty) \setminus e_+(H)$ from Theorem~\ref{T: Agmon 3.1} in potential scattering theory.
But while the structure of the set $e_+(H)$ is quite simple (see Theorem~\ref{T: Agmon 3.1}), the set $\mbR \setminus \Lambda(H_0,F)$
is more or less an arbitrary set of Lebesgue measure zero;
for instance, the singular operator $H E^H_{\mbR \setminus \Lambda(H_0,F)}$ may contain, --- in the worst scenario, everywhere dense pure point and singular continuous spectra.


III. Since the wave operators $w_\pm(\lambda; H_1,H_0)$ act between the fiber Hilbert spaces $\hlambda(H_0)$ and $\hlambda(H_1),$
the next logical step is construction of fiber Hilbert spaces of the spectral representation~(\ref{F: euF: hilb(a) to direct int-l})
and the direct integral on the right hand side of~(\ref{F: euF: hilb(a) to direct int-l}). The fiber Hilbert space $\hlambda(H_0)$ is defined
as a (closed) subspace of~$\clK$ by equality
\begin{equation} \label{F: hlambda def of}
  \hlambda(H_0) = \closure{\im \sqrt{\Im T_{\lambda+i0}(H_0)}},
\end{equation}
that is, the fiber Hilbert space $\hlambda(H_0)$ is the closure of the image of the compact non-negative operator $\sqrt{\Im T_{\lambda+i0}(H_0)}.$
The family of Hilbert spaces
$$
  \set{\hlambda(H_0) \colon \lambda \in \Lambda(H_0,F)}
$$
is measurable, where as a measurability base one can take orthogonal projections
of vectors from an orthonormal basis of~$\clK$ onto $\hlambda(H_0) \subset \clK.$ Hence, one can construct a direct integral of Hilbert spaces $\euH(H_0)$ by formula
\begin{equation} \label{F: euH(H0)=int L(H0,F) hlambda}
  \euH(H_0) = \int^\oplus_{\Lambda(H_0,F)} \hlambda(H_0)\,d\lambda.
\end{equation}
The complement of the set $\Lambda(H_0,F)$ is a support of the singular spectrum of $H_0$
in the sense that the operator $H_0 E_{\Lambda(H_0,F)}^{H_0}$ is absolutely continuous. 
In other words, the singular spectrum of $H_0$ including all eigenvalues of $H_0$
is left out from $\Lambda(H_0,F).$ Dimensions of the fiber Hilbert spaces $\hlambda(H_0)$  can be finite including zero.
A core of the absolutely continuous spectrum of $H_0$ can be defined by formula
\begin{equation} \label{F: core of L(H0,F)}
  \hat \sigma_{H_0} = \set{\lambda \in \Lambda(H_0,F) \colon \dim \hlambda(H_0) > 0}.
\end{equation}
In particular, a measure $\rho$ from the spectral representation~(\ref{F: euF: hilb(a) to direct int-l}) has the same
spectral type as the restriction of Lebesgue measure $d\lambda$ to the set $\hat \sigma_{H_0}.$
Therefore, if one wishes, in the direct integral~(\ref{F: euH(H0)=int L(H0,F) hlambda}) the set $\Lambda(H_0,F)$ can be replaced by the core
(\ref{F: core of L(H0,F)}), but it is more convenient to work with the set $\Lambda(H_0,F).$

IV. The next step is construction of the unitary isomorphism $\euF$ from the spectral representation~(\ref{F: euF: hilb(a) to direct int-l})
and its fiber $\euF_\lambda.$ To distinguish non-constructive object $\euF$ from its constructive counter-part to be defined, the latter is denoted by $\euE.$
By definition, for any vector $\phi$ from the dense linear manifold
$$
  F^*\clK =: \hilb_+ \subset \hilb
$$
the value of $\euE_\lambda(H_0)$ at $\phi$ is defined by formula
\begin{equation} \label{F: euE(F*psi)=sqrt Im T}
  \euE_\lambda(H_0)\phi = \pi^{-1/2}\sqrt{\Im T_{\lambda+i0}(H_0)} \psi \in \hlambda(H_0),
\end{equation}
where $\psi$ is the unique vector from~$\clK$ such that $\phi = F^*\psi.$
Justification of these definitions is given by the following theorem.
\begin{thm} \label{T: euE} 
Let~$H_0$ be a self-adjoint operator
on a Hilbert space $\hilb$ with a rigging operator $F\colon \hilb \to \clK.$
The linear operator~$\euE=\euE(H_0)$ which acts from the dense subspace $\hilb_+ = F^*\clK$ of $\hilb$ to the direct integral
Hilbert space~(\ref{F: euH(H0)=int L(H0,F) hlambda}) and which is defined by the equality\label{Page: euE(lambda)}
$$
  \euE(F^*\psi)(\lambda) = \euE_\lambda(H_0)(F^*\psi) = \pi^{-1/2}\sqrt{\Im T_{\lambda+i0}(H_0)}\psi
$$
is a bounded operator, whose continuous prolonging to $\hilb$ is a surjective isometric operator
with initial subspace $\hilb^{(a)}(H_0).$
In particular, the operator~$\euE$ is a natural isomorphism of the Hilbert spaces $\hilb^{(a)}(H_0)$ and~(\ref{F: euH(H0)=int L(H0,F) hlambda})
provided there is a fixed rigging~$F$ in $\hilb$ compatible with $H_0.$
Moreover, restriction of the operator~$H_0$ to its absolutely continuous subspace $\hilb^{(a)}(H_0)$
in the representation of $\hilb^{(a)}(H_0)$ by the direct integral~(\ref{F: euH(H0)=int L(H0,F) hlambda}) acts as follows: for any $f \in \hilb^{(a)}(H_0)$
and for a.e. $\lambda \in \Lambda(H_0,F)$
\begin{equation} \label{F: euE(H0f)(l)=l euE(f)(l)}
  \euE(H_0 f)(\lambda) = \lambda \euE(f)(\lambda).
\end{equation}
In other words, the operator \,$\euE$\, and the direct integral \,$\euH(H_0)$ \,diagonalize the absolutely continuous part of the self-adjoint operator~$H_0.$
\end{thm}
\noindent If a vector $f$ belongs to the image of $F^*,$ then the equality~(\ref{F: euE(H0f)(l)=l euE(f)(l)}) holds for all $\lambda \in \Lambda(H_0,F).$
Theorem~\ref{T: euE} is in fact the spectral theorem for the absolutely continuous part of a self-adjoint operator.
Importance of Theorem~\ref{T: euE} comes from the fact that it gives an explicit diagonalisation of the absolutely continuous part of an arbitrary self-adjoint operator.
This is a difficult problem; for instance, in the case of potential scattering, while the free Hamiltonian $H_0 = -\Delta$ is easily diagonalized by the Fourier transform (see~(\ref{F: H0=F* xi2 F})),
diagonalization of the \Schrodinger operator $H = -\Delta+V$ requires (or in essence is equivalent to) calculation of the wave matrices (see e.g. \cite[(83)]{RS3}, \cite[\S 10-a, (10.2)]{TayST})
(which is a difficult problem), so that, in fact, often wave operators are defined via eigenfunction expansion of the perturbed operator.
Compared to this situation, in Theorem~\ref{T: euE} the self-adjoint operator $H_0$ is \emph{arbitrary}. This is a key circumstance, since once explicit eigenfunction
expansions of an operator $H_0$ and of its perturbation $H = H_0+V$ are found, one may try to define the wave matrix by a formula analogous to \cite[(83)]{RS3} or \cite[\S 10-a, (10.2)]{TayST}.

The operator $\euE_\lambda(H_0)$ which acts from $\hilb$ to $\hlambda(H_0)$ makes perfect sense for all values of~$\lambda$ from the full set $\Lambda(H_0,F).$
In this regard, it is different from $\euF_\lambda$ of~(\ref{F: euF: hilb(a) to direct int-l}). The operator $\euE_\lambda(H_0)$ will be called an \emph{evaluation operator}.
Theorem~\ref{T: euE} implies that the operator~(\ref{F: Z(l;G)}) can be unambiguously defined for all~$\lambda$ from the full set $\Lambda(H_0,F)$ by formula
$$
  Z_0(\lambda; F) = \euE_\lambda(H_0) F^*.
$$
But, in actual fact, this formula makes the operator $Z_0(\lambda; F)$ redundant, since the operator $\euE_\lambda(H_0)$ in the right hand
side of this equality is unambiguously defined for an explicit set of full Lebesgue measure $\Lambda(H_0,F).$

V. Once the fiber Hilbert spaces $\hlambda(H_0)$ have been constructed, one can define wave matrices
\begin{equation} \label{F: w(pm): hl(0)to hl(1)}
  w_\pm(\lambda; H_1,H_0) \colon \hlambda(H_0) \to \hlambda(H_1),
\end{equation}
for all real numbers~$\lambda$ from the intersection of sets $\Lambda(H_0,F)$ and $\Lambda(H_1,F).$ Initially, the operator
$w_\pm(\lambda; H_1,H_0)$ is defined as a form on a dense subspace $\euE_\lambda(H_1)F^*\clK \times \euE_\lambda(H_0)F^*\clK$ of the direct
product $\hlambda(H_1) \times \hlambda(H_0)$ by formula \cite[Definition 5.2.1]{Az3v6}: for any $F^*f, F^*g \in F^*\clK$
\begin{equation} \label{F: w(pm) form def-n}
  \scal{\euE_\lambda(H_1)F^*f}{w_\pm(\lambda; H_1,H_0)\euE_\lambda(H_0)F^*g} = \scal{f}{\SqBrs{1 - T_{\lambda\mp i0}(H_1)J} \frac 1 \pi \Im T_{\lambda+ i0}(H_0)g}.
\end{equation}
The idea to define the wave matrices by a formula similar to~(\ref{F: w(pm) form def-n}) was taken from \cite[Definition 2.7.2]{Ya}.
\begin{thm} \label{T: properties of w(pm)} 
(1) For any $\lambda \in \Lambda(H_0,F) \cap \Lambda(H_1,F)$ the formula~(\ref{F: w(pm) form def-n}) correctly defines a bounded operator~(\ref{F: w(pm): hl(0)to hl(1)}). Moreover, this bounded operator is unitary.
(2) For any three values (not necessarily distinct) $r_1, r_2, r_3$ of the coupling constant $r$ such that $\lambda \in \Lambda(H_{r_1},F) \cap \Lambda(H_{r_2},F) \cap \Lambda(H_{r_3},F)$
the following multiplicative property holds:
$$
  w_\pm(\lambda; H_{r_3},H_{r_1}) = w_\pm(\lambda; H_{r_3},H_{r_2})w_\pm(\lambda; H_{r_2},H_{r_1}).
$$
In particular, for any $\lambda \in \Lambda(H_0,F)$ \ $w_\pm(\lambda; H_0,H_0) = 1$ and
for any $\lambda \in \Lambda(H_0,F) \cap \Lambda(H_1,F)$ \ $w_\pm ^*(\lambda; H_1,H_0) = w_\pm(\lambda; H_0,H_1).$
\end{thm}

VI. Once the wave matrices $w_\pm(\lambda; H_1,H_0)$ are defined and their basic properties are proved, one can \emph{define} the wave operators
\begin{equation} \label{F: def of W(pm)}
  W_\pm(H_1,H_0) \colon \euH(H_0) \to \euH(H_1)
\end{equation}
by a formula, similar to~(\ref{F: W(pm)= int w(pm)}):
\begin{equation} \label{F: W(pm)=int w(pm)(2)}
  W_\pm(H_1,H_0) = \int^\oplus_{\Lambda(H_0,F) \cap \Lambda(H_1,F)} w_\pm(\lambda; H_1,H_0)\,d\lambda.
\end{equation}
Here instead of absolutely continuous subspaces $\hilb^{(a)}(H_0)$ and $\hilb^{(a)}(H_1)$ between which wave operators act
one can use the Hilbert spaces $\euH(H_0)$ and $\euH(H_1),$ since by Theorem~\ref{T: euE} the Hilbert spaces
$\hilb^{(a)}(H)$ and $\euH(H)$ are naturally isomorphic via the unitary operator $\euE(H).$
The following theorem demonstrates that definition~(\ref{F: W(pm)=int w(pm)(2)}) of the wave operator coincides with the classical definition of the wave operator.
\begin{thm} 
  Wave operators defined by formulas~(\ref{F: W(pm)=int w(pm)(2)}) and~(\ref{F: w(pm) form def-n}) are equal to the right hand side of~(\ref{F: W(pm) class-l def-n}).
\end{thm}
\noindent Further, Theorems~\ref{T: euE} and~\ref{T: properties of w(pm)} immediately imply well-known properties for wave operators \cite[Theorems 5.4.1, 5.4.2, Corollary 5.4.3]{Az3v6}:
\begin{enumerate}
  \item The wave operators~(\ref{F: def of W(pm)}) are unitary (as operators from $\euH(H_0)$ to $\euH(H_1)$).
  \item (Multiplicative property) For any three real numbers $r_1,r_2,r_3,$ not necessarily distinct,
    $$W_\pm(H_{r_3},H_{r_1}) = W_\pm(H_{r_3},H_{r_2})W_\pm(H_{r_2},H_{r_1}).$$
  \item $W_\pm^*(H_1,H_0) = W_\pm(H_0,H_1).$
  \item $W_\pm(H_0,H_0)$ is the identity operator on $\euH(H_0).$
  \item $H_1 W_\pm(H_1,H_0) = W_\pm(H_1,H_0)H_0$ (intertwining property).
  \item For any bounded measurable function $h$ on $\mbR$ $$h(H_1) W_\pm(H_1,H_0) = W_\pm(H_1,H_0)h(H_0).$$
  \item The absolutely continuous parts of $H_0$ and $H_1$ are unitarily equivalent (Kato-Rosenblum Theorem).
\end{enumerate}

VII. The scattering matrix $S(\lambda; H_1,H_0)$ is defined as an operator $\hlambda(H_0) \to \hlambda(H_0)$
for all values of the spectral parameter~$\lambda$ from the intersection $\Lambda(H_0,F) \cap \Lambda(H_1,F)$ by formula
\cite[Definition 7.1.1]{Az3v6}
\begin{equation} \label{F: S(l)=w+*(l)w-(l)}
  S(\lambda; H_1,H_0) = w_+^*(\lambda; H_1,H_0)w_-(\lambda; H_1,H_0).
\end{equation}
Note that in conventional approach this formula is a theorem (see e.g. \cite{Ya}), which is proved for a.e.~$\lambda$ from an unspecified set of full measure.
Many of the well-known properties of the scattering matrix $S(\lambda; H_1,H_0)$ such as unitarity follow immediately from this definition and Theorem~\ref{T: properties of w(pm)}
\cite[Theorem 7.1.2]{Az3v6}. The scattering operator $\bfS(H_1,H_0)$ is defined by formula
\begin{equation} \label{F: bfS=int S(l)(2)}
  \bfS(H_1,H_0) = \int^\oplus_{\Lambda(H_0,F) \cap \Lambda(H_1,F)} S(\lambda; H_1,H_0)\,d\lambda.
\end{equation}
Equalities~(\ref{F: W(pm)=int w(pm)(2)}) and~(\ref{F: bfS=int S(l)(2)}) imply the classical definition~(\ref{F: bfS=W*(+)W(-)}) of the scattering operator $\bfS(H_1,H_0).$

VIII. Now we return to the formula~(\ref{F: Texp formula for S(l)}).
Before proceeding to a proof of~(\ref{F: Texp formula for S(l)}), one needs to give meaning to the right hand side of
(\ref{F: Texp formula for S(l)}). This raises the following question: if $H_r=H_0+rV$ and if $\lambda \in \Lambda(H_0,F),$
then for which values of $r$ one also has
\begin{equation} \label{F:  lambda in Lambda(Hr,F)}
  \lambda \in \Lambda(H_r,F)?
\end{equation}
This question is important, since the wave matrices $w_\pm(\lambda; H_r,H_0)$
and the scattering matrix $S(\lambda; H_r,H_0)$ are defined for those values of the coupling constant $r$ for which the inclusion~(\ref{F:  lambda in Lambda(Hr,F)}) holds.
The following well-known theorem answers this question; for a proof see e.g. \cite[Theorem 4.1.11]{Az3v6}.
\begin{thm} \label{T: 4.1.11 Az3}
Let $H_0$ be a self-adjoint operator on a Hilbert space $\hilb$ with a rigging operator $F \colon \hilb \to \clK,$
let $V = F^*JF,$ where $J$ is a bounded operator on~$\clK$ and let $H_r = H_0+rV.$
If a real number~$\lambda$ belongs to the set $\Lambda(H_0,F)$ (so in particular the operator
$T_{\lambda+ i0}(H_0)$ exists and is compact), then for any $r \in \mbR$ the number~$\lambda$ belongs to the set
$\Lambda(H_r,F)$ if and only if one (and hence all) of the following four operators is invertible:
$$
  1 + r J T_{\lambda\pm i0}(H_0), \ 1 + r T_{\lambda\pm i0}(H_0)J.
$$
In particular, the set of values of the coupling constant $r,$ for which $\lambda \notin \Lambda(H_r,F),$
is a discrete subset of the real line.
\end{thm}
The set $\set{r \in \mbR \colon \lambda \notin \Lambda(H_r,F)}$ is of importance; elements of this set will be called \emph{resonance points} of the triple ($\lambda; H_0,V),$
the set itself will be called \emph{resonance set} and will be denoted by\label{Page: resonance set}
$R(\lambda; H_0,V)$ (this set depends on~$F$ too, but this dependence is not indicated in the notation). One of the reasons~$\lambda$ may fail
to belong to the set $\Lambda(H_r,F)$ is that~$\lambda$ may be an eigenvalue of~$H_r$ \cite[Proposition 4.1.10]{Az3v6}.

Theorem~\ref{T: 4.1.11 Az3} states that the perturbed operator $H_r = H_0+rV$ possesses a \emph{coupling constant regularity property}.
Coupling constant regularity was observed already by \Aronszajn\ \cite{Ar57} (see also \cite{AD56}) in the course of study of boundary value perturbations of singular Sturm-Liouville equations.
Later coupling constant regularity for general rank-one perturbations was used by \Simon\ and T.\,Wolff \cite{SW}, Simon \cite[Chapters 12,13]{SimTrId2}
and others (e.g. \cite{RMS,RJMS,Go94}; see \cite{SimTrId2} for more references) in a study of singular continuous spectrum and Anderson localization for random Hamiltonians.

A corollary of Theorem~\ref{T: 4.1.11 Az3} is that the operators $\euE_\lambda(H_r),$ $w_\pm(\lambda; H_r,H_0)$ and $S(\lambda; H_r,H_0)$
are defined for all values of the coupling constant $r$ \emph{except the discrete resonance set} $R(\lambda; H_0,V).$

Now we are in position to formulate the stationary formula for the scattering matrix.
\begin{thm} \label{T: Theorem 7.2.2 of Az3} 
Let $\lambda \in \Lambda(H_0,F).$ For all $r \notin R(\lambda; H_0,V)$ the scattering matrix
$S(\lambda; H_r,H_0),$ which is defined by equality~(\ref{F: S(l)=w+*(l)w-(l)}) as an operator on the fiber Hilbert space
(\ref{F: hlambda def of}), satisfies the equality
\begin{equation} \label{F: Theorem 7.2.2 of Az3}
  S(\lambda; H_r,H_0) = 1 - 2i \sqrt{\Im T_{\lambda+i0}(H_0)} \,\,rJ (1 + r T_{\lambda+i0}(H_0)J)^{-1} \sqrt{\Im T_{\lambda+i0}(H_0)}.
\end{equation}
\end{thm}
The right hand side of the equality~(\ref{F: Theorem 7.2.2 of Az3}),
known as modified scattering matrix, is defined on the whole auxiliary Hilbert space~$\clK$ and it is not difficult to check by a direct calculation that it is a unitary operator on the whole Hilbert space.
The equality~(\ref{F: Theorem 7.2.2 of Az3}) shows that
the right hand side of~(\ref{F: Theorem 7.2.2 of Az3}) can be interpreted as a proper scattering matrix, given that the fiber Hilbert space is defined by equality~(\ref{F: hlambda def of}).
Recalling definition of the evaluation operator~(\ref{F: euE(F*psi)=sqrt Im T}), the equality~(\ref{F: Theorem 7.2.2 of Az3}) can be rewritten in more familiar terms as follows
\begin{equation} \label{F: Theorem 7.2.2 of Az3 (v2)}
 \begin{split}
   S(\lambda; H_r,H_0) = 1 - 2\pi i \euE_\lambda(H_0) F^* rJ (1 + r T_{\lambda+i0}(H_0)J)^{-1} F \euE_\lambda^*(H_0)
 \end{split}
\end{equation}
\begin{rems} \rm The expression $\euE_\lambda^*(H_0)$ on its own does not make sense since
$\euE_\lambda(H_0)$ as an operator $\hilb \to \hlambda(H_0)$ with domain $F^*\clK$ as defined by~(\ref{F: euE(F*psi)=sqrt Im T}) in general
is not closable, but the product $F \euE_\lambda^*(H_0)$ is a well-defined compact operator
from the Hilbert space $\hlambda(H_0)$ to the Hilbert space $\clK$ for every $\lambda \in \Lambda(H_0,F);$ for details see \cite[\S\S 2.6, 2.6.1, 2.15, 5.1]{Az3v6}.
\end{rems}
\noindent The formula~(\ref{F: Theorem 7.2.2 of Az3 (v2)}) coincides with~(\ref{F: rigorous stationary formula}),
but, unlike the formula~(\ref{F: rigorous stationary formula}), in the formula~(\ref{F: Theorem 7.2.2 of Az3 (v2)}) the full set $\Lambda(H_0,F) \cap \Lambda(H_r,F)$
of values of the spectral parameter~$\lambda$ (energy) for which it makes sense is explicitly given.
Finally, the formula~(\ref{F: Theorem 7.2.2 of Az3 (v2)}) can be written as (see \cite[(7.6)]{Az3v6})
$$
  S(\lambda; H_r,H_0) = 1 - 2\pi i \euE_\lambda(H_0) \,rV (1 + r R_{\lambda+i0}(H_0)V)^{-1} \euE_\lambda^\diamondsuit(H_0),
$$
provided the operators $\euE_\lambda(H_0),$~$V$ and $R_{\lambda+i0}$ are interpreted as acting between appropriate pairs of Hilbert spaces $\hilb_-, \hilb_+$ and~$\hlambda$
(see \cite[\S\S 5.1, 2.15]{Az3v6} for details):
$$
  \hlambda \stackrel{\euE_\lambda(H_0)}\longleftarrow \hilb_+ \stackrel{V}\longleftarrow \hilb_-
  \stackrel{R_{\lambda+i0}(H_0)}\longleftarrow \hilb_+ \stackrel{V}\longleftarrow \hilb_- \stackrel{\euE_\lambda^\diamondsuit(H_0)}\longleftarrow \hlambda.
$$
Here $\euE_\lambda^\diamondsuit(H_0)$ is a modified adjoint (see \cite[\S 2.6.1]{Az3v6}), definition of which is an abstract version
of~(\ref{F: gamma0 diamond}); for definition of Hilbert spaces $\hilb_\pm$ see also p.\,\pageref{Page: hilb+} of this paper.

Theorem~\ref{T: Theorem 7.2.2 of Az3} allows us to overcome a hindrance on the way to a proof of formula~(\ref{F: Texp formula for S(l)}).
\begin{thm} \label{T: 7.3.4 Az3v6} \cite[Theorem 7.3.4]{Az3v6} For all values of the spectral parameter~$\lambda$ from the set $\Lambda(H_0,F) \cap \Lambda(H_1,F)$ of full Lebesgue measure
\begin{equation} \label{F: Texp formula for S(l) (2)}
  S(\lambda; H_1,H_0) = \Texp \brs{-2\pi i \int_0^1 w_+(\lambda; H_0,H_r)\euE_\lambda(H_r)F^*JF \euE_\lambda^*(H_r)w_+(\lambda; H_r,H_0)\,dr}.
\end{equation}
\end{thm}
This theorem allows to prove the following theorem.
\begin{thm} \label{T: C 8.2.5 Az3} \cite[Corollary 8.2.5]{Az3v6} Let $H_0$ be a self-adjoint operator, let~$V$ be a trace-class self-adjoint operator and let $H_r = H_0+rV.$ For a.e. $\lambda \in \mbR$
$$
  \det S(\lambda; H_1,H_0) = e^{- 2 \pi i \xia(\lambda)},
$$
where $\xia(\lambda)$ is the absolutely continuous spectral shift function defined as the density of the absolutely continuous measure~(\ref{F: xia def}).
\end{thm}
This formula combined with Birman-\Krein\ formula~(\ref{F: Birman-Krein formula}), implies that the singular spectral shift function $\xis(\lambda)$ of the pair $(H_0,H_1)$
defined as density of measure~(\ref{F: xis def}) is a.e. integer-valued (see~(\ref{F: xis is in Z})). In fact, in \cite{Az3v6} another proof of the theorem~(\ref{F: xis is in Z})
was given, so that the Birman-\Krein\ formula becomes its corollary. This proof is relevant to the content of this paper; for this reason its main idea is outlined in the next paragraph.

Let $U(r), r \in [a,b],$ be a path of unitary operators, such that $U(a) = 1,$ $U(r) - 1$ is trace-class for all $r \in [a,b]$
and the function $r \mapsto U(r) - 1$ is continuous in trace-class norm. These conditions on the operator $U(r)$ imply that spectrum of $U(r)$
consists of isolated eigenvalues on the unit circle with~$1$ as only one point in the essential spectrum of $U(r).$
As~$r$ decreases from~$b$ to~$a,$ eigenvalues of the unitary operator $U(r)$ converge continuously to~$1.$ So, given a point $e^{i\theta}$ on the unit circle,
one may calculate spectral flow through the point $e^{i\theta},$ which, following \cite{Pu01FA}, is called $\mu$-invariant of the path $U(r).$

The scattering matrix $S(\lambda; H_1,H_0)$ for any given value of~$\lambda$ from the full set $\Lambda(H_0,F) \cap \Lambda(H_1,F)$
is a unitary matrix of class $1 + \clL_1(\hlambda)$ (that is, $S(\lambda; H_1,H_0)-1$ is a trace class operator on $\hlambda(H_0)$).
There exist two natural paths which continuously connect the scattering matrix $S(\lambda; H_1,H_0)$ with the identity operator on $\hlambda(H_0).$
In the first path one changes the imaginary part of the spectral parameter $y = \Im z$ in the stationary formula~(\ref{F: Theorem 7.2.2 of Az3 (v2)}) or~(\ref{F: Theorem 7.2.2 of Az3})
for the scattering matrix from $+\infty$ to $0:$
\begin{equation} \label{F: tilde S(l+iy)}
  [0,+\infty] \ni y \mapsto S(\lambda+iy; H_1,H_0) = 1 - 2\pi i \euE_{\lambda+iy}(H_0) F^* J (1 + T_{\lambda+iy}(H_0)J)^{-1} F \euE_{\lambda+iy}^*(H_0),
\end{equation}
where in accordance with (\ref{F: euE(F*psi)=sqrt Im T}) $$\euE_{\lambda+iy}(H_0) F^* = \pi^{-1/2} \sqrt{\Im T_{\lambda+iy}(H_0)}.$$
One can show that this path is continuous in trace-class topology.
In order to get a second way of connecting $S(\lambda; H_1,H_0)$ with the identity operator the following theorem (which was initially observed in \cite{Az2}) is used.
\begin{prop} \label{P: 7.2.5 Az3} \cite[Proposition 7.2.5]{Az3v6}
The scattering matrix $S(\lambda; H_r,H_0)$ as a meromorphic function of the coupling constant $r$ admits
analytic continuation to the real axis.
\end{prop}
\begin{rems} \rm
In \cite{Az3v6} this proposition in fact precedes Theorem~\ref{T: 7.3.4 Az3v6} and is used in its proof.
Indeed, though the integrand of the chronological exponential in~(\ref{F: Texp formula for S(l) (2)}) is defined for all $r$ except the discrete resonance set
$R(\lambda; H_0,V),$ to define the chronological exponential itself one needs the integrand to be continuous in trace-class norm.
\end{rems}
Proposition~\ref{P: 7.2.5 Az3} provides the second way of connecting continuously the scattering matrix $S(\lambda; H_1,H_0)$ with the identity operator:
via the continuous mapping
\begin{equation} \label{F: r to S(l;Hr,H0)}
  [0,1] \ni r \mapsto S(\lambda; H_r,H_0) \in 1 + \clL_1(\hlambda(H_0)).
\end{equation}
\noindent The $\mu$-invariant of the path~(\ref{F: tilde S(l+iy)}) was introduced in \cite{Pu01FA} where it was denoted by $\mu(\theta,\lambda; H_1,H_0).$
The $\mu$-invariant of the path~(\ref{F: r to S(l;Hr,H0)}) was introduced in \cite{Az2,Az3v6}  where it was denoted by $\mua(\theta,\lambda; H_1,H_0).$
Relation of these $\mu$-invariants to the spectral shift functions $\xi,\xia$ and $\xis$ is given by following theorems.
\begin{thm} \label{T: Push} \cite{Pu01FA}
For a.e. $\lambda \in \mbR$
$$
  \xi_{H_1,H_0}(\lambda) = - \frac 1{2\pi} \int_0^{2\pi} \mu(\theta,\lambda; H_1,H_0)\,d\theta
$$
\end{thm}
\begin{thm} \label{T: 9.2.2 Az3} \cite[Theorem 9.2.2]{Az3v6}
For a.e. $\lambda \in \mbR$
$$
  \xia_{H_1,H_0}(\lambda) = - \frac 1{2\pi} \int_0^{2\pi} \mua(\theta,\lambda; H_1,H_0)\,d\theta
$$
\end{thm}
\begin{thm} \label{T: 9.7.3 Az3} \cite[Theorem 9.7.3]{Az3v6} The difference
$$
  \mus(\theta,\lambda; H_1,H_0) := \mu(\theta,\lambda; H_1,H_0) - \mua(\theta,\lambda; H_1,H_0)
$$
does not depend on the angle $\theta$ and for a.e. $\lambda \in \mbR$ it is equal to minus the density $-\xis(\lambda; H_1,H_0)$
of the singular spectral shift measure $\xis(\phi)$ as defined by~(\ref{F: xis def}).
In particular, the function $\xis(\lambda; H_1,H_0)$ is almost everywhere integer-valued.
\end{thm}
\noindent Theorem~\ref{T: Push} of A.\,Pushnitski was given a new proof in \cite{Az3v6} (see \cite[Theorem 9.6.1]{Az3v6}).
Theorems~\ref{T: 9.7.3 Az3} and~\ref{T: C 8.2.5 Az3} give a new proof of the Birman-\Krein\ formula~(\ref{F: Birman-Krein formula}).

The last assertion of Theorem~\ref{T: 9.7.3 Az3} gives a reason to call the function $\xis(\lambda)$ \emph{spectral flow inside essential spectrum},
since $\xis(\lambda)$ coincides with the spectral flow outside of the essential spectrum and it is a.e. integer-valued inside the essential spectrum as well.

The following diagram demonstrates the relationship between $\mu$- and $\mua$-invariants.
In this diagram for a fixed real value of the spectral parameter~$\lambda$ we consider the scattering matrix $S(\lambda+iy; H_r,H_0)$ as a function of
$(r,y),$ where $r$ is the coupling constant and $y$ is the imaginary part of the spectral parameter.

\hskip -0.5cm
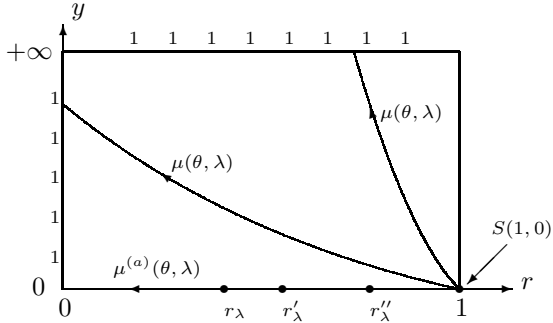
\begin{picture}(125,150)
\put(40,40){\vector(1,0){170}}
\put(40,40){\vector(0,1){100}}
\put(40,130){\line(1,0){150}}
\put(65,133){\tiny 1 \quad 1 \quad 1 \quad 1 \quad 1 \quad 1 \quad 1 \quad 1 }
\put(190,40){\line(0,1){90}}
\put(35,50){\tiny 1}
\put(35,65){\tiny 1}
\put(35,80){\tiny 1}
\put(35,95){\tiny 1}
\put(35,110){\tiny 1}

\put(185,30){{\small~$1$}}
\put(35,30){{\small~$0$}}
\put(25,38){{\small~$0$}}
\put(40,144){{\small~$y$}}
\put(15,127){{\small~$+\infty$}}
\put(210,40){{\small~$r$}}
\put(190,40){\circle*{3}}
\put(101,40){\circle*{3}}
\put(101,30){\tiny $r_\lambda$}
\put(38,10){\small The three points $r_\lambda, r'_\lambda,r''_\lambda$ represent resonance points from $[0,1]$.}
\put(123,40){\circle*{3}}
\put(123,30){\tiny $r'_\lambda$}
\put(156,40){\circle*{3}}
\put(156,30){\tiny $r''_\lambda$}
\put(208,57){\vector(-1,-1){15}}
\put(202,61){\tiny $S(1,0)$}
\thicklines
\put(40,130){\line(1,0){150}}
\put(40,40){\line(0,1){90}}
\put(245,130){\small $S(r,y) := S(\lambda+iy; H_r,H_0)$ is continuous in the}
\put(245,118){\small rectangle except the so-called resonance points.}
\put(245,106){\small On the left $r=0$ and upper $y=+\infty$ rims of this}
\put(245,94){\small rectangle $S(\lambda+iy; H_r,H_0) = 1.$ \ $\mu(\theta,\lambda)$ is the}
\put(245,82){\small spectral flow of eigenvalues of $S(r,y)$ through $e^{i\theta}$}
\put(245,70){\small corresponding to any path which connects $(1,0)$}
\put(245,58){\small with the left or the upper rim as long as it}
\put(245,46){\small avoids the resonance points.}
\put(245,34){\small $\mua(\theta,\lambda)$ is the spectral flow of eigenvalues }
\put(245,22){\small of $S(\lambda; H_r,H_0)$ as $r$ goes from~$1$ to $0.$}

\thinlines
\qbezier(190,40)(170,60)(150,130)
\put(159,100){\vector(-1,3){3}}
\put(160,105){{\tiny $\mu(\theta,\lambda)$}}

\qbezier(190,40)(100,60)(40,110)
\put(80,82){\vector(-3,2){3}}
\put(81,86){{\tiny $\mu(\theta,\lambda)$}}

\put(68,40){\vector(-1,0){3}}
\put(59,46){{\tiny $\mua(\theta,\lambda)$}}

\end{picture}

\noindent
The operators $S(\lambda; H_r,H_0)$ and $S(\lambda+i0; H_r,H_0)$ are identical outside the resonance points.
The group $\clU_1$ of unitary operators of the form ``1 + trace class'' has a non-trivial homotopical structure
and the difference between the operators $S(\lambda; H_r,H_0)$ and $S(\lambda+i0; H_r,H_0)$ is revealed in the way one connects them with the
base point~$1$ of the group $\clU_1.$

\medskip
The functions $\xi(\lambda),$ $\xia(\lambda)$ and $\xis(\lambda)$ are integrable, and so in general one cannot talk about value of these functions
at a given point $\lambda.$ But Theorems~\ref{T: Push},~\ref{T: 9.2.2 Az3} and~\ref{T: 9.7.3 Az3} allow to define values of these functions explicitly
on the set of full measure $\Lambda(H_0,F) \cap \Lambda(H_1,F),$ since the right hand sides of equalities in these theorems are well-defined for all
$\lambda$ from the set $\Lambda(H_0,F) \cap \Lambda(H_1,F).$ This is an important point, since if the perturbed operator $H_1$ is replaced by $H_r = H_0+rV$
with arbitrary real number $r,$ then for every fixed value of~$\lambda$ from $\Lambda(H_0,F)$
the expressions $\xi(\lambda; H_r,H_0),$ $\xia(\lambda; H_r,H_0)$ and $\xis(\lambda; H_r,H_0)$ can be considered as functions of the coupling constant $r.$
Behaviour of these functions of $r$ is explained by the following theorem.
\begin{thm} \label{T: Az 9.7.6} \cite[Proposition 8.2.3, Theorem 9.7.6, Corollary 9.7.7]{Az3v6} For every~$\lambda$ from the set $\Lambda(H_0,F)$
of full Lebesgue measure the following assertions hold:
\begin{enumerate}
  \item The function $r \mapsto \xia(\lambda; H_r,H_0)$ is a function analytic in a neighbourhood of $\mbR.$
  \item The function $r \mapsto \xis(\lambda; H_r,H_0)$ is a locally constant integer-valued function with a discrete set of discontinuity points
which coincides with the resonance set $R(\lambda; H_0,V)$ (see Theorem~\ref{T: 4.1.11 Az3} and the paragraph after it for definition of the resonance set).
  \item As a consequence, the function $r \mapsto \xi(\lambda; H_r,H_0)$ is a piecewise continuous locally analytic function and discontinuity points of this function
are resonance points of the triple $(\lambda; H_0,V).$
\end{enumerate}
\end{thm}

\subsection{Resonance index}
Theorem~\ref{T: Az 9.7.6} implies, in particular, that if for $\lambda \in \Lambda(H_0,F)$ there are no resonance points in an interval $[a,b],$
then $\xia(\lambda; H_b,H_a) = \xi(\lambda; H_b,H_a).$ It also suggests that the (integer) jump of the singular spectral shift function $\xis(\lambda; H_0,H_1)$
at a resonance point~$r_\lambda \in [0,1]$ should depend only on the triple $(\lambda; H_{r_\lambda},V).$ Indeed, to a triple $(\lambda; H_{r_\lambda},V)$
one can assign an integer number, which in this paper is called \emph{resonance index} and is denoted by
$$
  \ind_{res}(\lambda; H_{r_\lambda},V).
$$
This number is defined as follows.
Firstly, it can be observed that by Theorem~\ref{T: 4.1.11 Az3} a real number~$r_\lambda$ is a resonance point of the triple $(\lambda; H_0,V)$
if and only if the real number $\sigma_\lambda=-r_\lambda^{-1}$ is an eigenvalue of the operator $T_{\lambda+i0}(H_0)J.$ Further, the number~$r_\lambda$ is a singular point (a pole) of the meromorphic
factor $(1+rT_{\lambda+i0}(H_0)J)^{-1}$ which is part of the stationary formula~(\ref{F: Theorem 7.2.2 of Az3 (v2)}) for the scattering matrix $S(\lambda; H_r,H_0).$
Still, according to Proposition~\ref{P: 7.2.5 Az3}, the scattering matrix $S(\lambda; H_r,H_0)$ does not have a singularity at $r = r_\lambda.$
This happens due to the fact that this singularity belongs to the singular subspace of $H_0,$ which is eliminated by factors $\euE_\lambda(H_0)F^*$
and $F\euE_\lambda^*(H_0)$ of the stationary formula. In order to reveal this hidden singularity, one has to shift the spectral parameter $\lambda+i0$
slightly off the real axis. Since $\sigma_\lambda$ is an isolated eigenvalue of the compact operator $T_{\lambda+i0}(H_0)J,$ it is stable
but it may split into several eigenvalues
\begin{equation} \label{F: shifted sigma(l+iy)}
  \sigma_{\lambda+iy}^1, \ldots, \sigma_{\lambda+iy}^N,
\end{equation}
where~$N$ is the multiplicity of $\sigma_\lambda,$
which are therefore eigenvalues of the compact operator $T_{\lambda+iy}(H_0)J$ from the group of $\sigma_\lambda.$ It is well-known and is not difficult to show
that none of the shifted eigenvalues~(\ref{F: shifted sigma(l+iy)}) is a real number. Therefore, the following definition makes sense:
resonance index $\ind_{res}(\lambda; H_{r_\lambda},V)$ of the triple $(\lambda; H_{r_\lambda},V)$ is the difference
\begin{equation} \label{F: N(+)-N(-)}
  N_+ - N_-,
\end{equation}
where $N_+$ (respectively, $N_-$) is the number of shifted eigenvalues of the group of $\sigma_\lambda$ in the upper (respectively, lower) complex half-plane.
Definition the resonance index is correct in the sense that it does depend on the choice of the ``initial'' operator $H_0,$ as the following lemma with a simple proof asserts.
\begin{lemma} Let $\lambda \in \Lambda(H_0,F).$ Let a real number $s$ be such that~$\lambda$ also belongs to the full set $\Lambda(H_s,F).$
Further, let~$r_\lambda$ be a resonance point of the triple $(\lambda; H_0,V)$ (that is, $\lambda \notin \Lambda(H_{r_\lambda},F)$).
Then the real number $\sigma_\lambda(s) = (s-r_\lambda)^{-1}$ is an eigenvalue of the operator $T_{\lambda+i0}(H_s)J$
of the same algebraic multiplicity~$N$ as that of the eigenvalue $\sigma_\lambda(0) = (0-r_\lambda)^{-1}$ of the operator $T_{\lambda+i0}(H_0)J$
and if~$\lambda$ is shifted off the real axis
to $\lambda+iy$ with small and positive $y,$ then the number of split eigenvalues from the group of $(s-r_\lambda)^{-1}$ in the upper complex half-plane is equal to $N_+.$
\end{lemma}
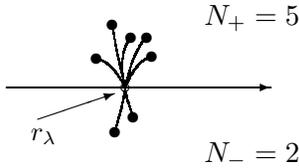
\begin{picture}(125,90)
\put(10,40){\vector(1,0){100}}
\put(85,64){{$N_+=5$}}
\put(85,12){{$N_-=2$}}
\put(55,40){\circle{3}}    
\put(44,51){\circle*{4}}   \qbezier(55,40)(50,50)(44,51) 
\put(50,64){\circle*{4}}   \qbezier(55,40)(52,52)(50,64) 
\put(57,59){\circle*{4}}   \qbezier(55,40)(52,52)(57,59) 
\put(63,59){\circle*{4}}   \qbezier(55,40)(58,54)(63,59) 
\put(65,52){\circle*{4}}   \qbezier(55,40)(59,50)(65,52) 
\put(58,29){\circle*{4}}   \qbezier(55,40)(55,37)(58,29) 
\put(51,23){\circle*{4}}   \qbezier(55,40)(55,37)(51,23) 

\put(22,28){\vector(3,1){29}}
\put(16,20){{\small~$r_\lambda$}}
\end{picture}

\noindent Introduction of this notion is justified by the following theorem, see \cite[Theorem 3.8]{Az7}. Since the paper \cite{Az7} is not published,
outline of the proof of this theorem is given in section \ref{S: res.index as s.ssf}.
\begin{thm} \label{T: xis=ind res} Let $H_0$ be a self-adjoint operator on a Hilbert space $\hilb$
with a Hilbert-Schmidt rigging operator $F \colon \hilb \to \clK.$ Let~$V$ be a trace-class self-adjoint operator
which admits decomposition $V = F^*JF$ with bounded operator $J$ on $\clK$ and let $a<b.$ Then for every real number~$\lambda$ from the set
$\Lambda(H_a,F) \cap \Lambda(H_b,F)$ of full Lebesgue measure the following equality holds:
\begin{equation} \label{F: xis=ind res}
  \xis(\lambda; H_b,H_a) = \sum_{r_\lambda} \ind_{res}(\lambda; H_{r_\lambda},V),
\end{equation}
where the sum is taken over all resonance points~$r_\lambda$ of the triple $(\lambda; H_0,V)$ from the interval~$[a,b].$
\end{thm}
\noindent
In other words, as the value of the coupling constant $r$ changes from $a$ to $b,$ the locally constant function $[a,b] \ni r \mapsto \xis(\lambda; H_r,H_a)$
jumps at every encountered resonance point~$r_\lambda \in [a,b]$ by the integer $\ind_{res}(\lambda; H_{r_\lambda},V).$
Theorem~\ref{T: xis=ind res} gives a computable and tangible representation for values of the function $\xis(\cdot; H_b,H_a),$
which is initially defined as the density of the singular spectral shift measure~(\ref{F: xis def}), and as such seems to be difficult to handle
(indeed, the formula~(\ref{F: xis def}) requires in particular calculation of singular parts of a continuous family of self-adjoint operators).
In particular, this theorem allows to prove the following \cite[Theorem 4.3]{Az7} 
\begin{thm} There exist a self-adjoint operator $H_0$ and a rank-one self-adjoint operator~$V$ such that the pair $(H_0,V)$ is irreducible
and the restriction of the singular spectral shift function $\xis(\cdot; H_0+V,H_0)$ of this pair to the absolutely continuous spectrum $\sigma_{a.c.}(H_0)$ of $H_0$
is a non-zero element of $L_1(\sigma_{a.c.}(H_0),dx).$
\end{thm}
The construction of such a pair may not be interesting, but at least this theorem shows that the decomposition~(\ref{F: xi=xia+xis}) is non-trivial.

The expression on the right hand side of~(\ref{F: xis=ind res}) will be called total resonance index for the pair $H_0,H_1 = H_0+V.$
For values of the spectral parameter~$\lambda$ which lie outside of essential spectrum of $H_0$ the singular spectral shift function
coincides with spectral flow, and therefore it follows from~(\ref{F: xis=ind res}) that the total resonance index provides a new definition of spectral flow.
Moreover, the notion of resonance index which was discovered in the course of study of the singular spectral shift function, makes sense even in finite dimensions.
Resonance index represents a new approach to calculation of spectral flow, which in essence is ``flow of eigenvalues''. Indeed, in order to find out how many eigenvalues
of a path of self-adjoint operators $\set{H_0+rV \colon 0 \leq r \leq 1}$ crossed in the positive direction a fixed point~$\lambda$ outside of spectrum of the initial $H_0$ and final $H_1$ operators,
one can either try to keep track of each eigenvalue and count how many times and in which direction it crossed~$\lambda,$ or instead of that one can try to detect moments
of ``time'' (coupling constant)~$r_\lambda$ for which an event ``$\lambda$ is an eigenvalue of $H_{r_\lambda}$'' occurs and then to decide where the eigenvalue has come from and where it is going. The first approach
requires continuous enumeration of eigenvalues  (which for general continuous paths is not a trivial problem even in finite dimensions, see \cite[\S II.5.2]{Kato}),
but inside of essential spectrum this approach does not work since eigenvalues embedded into essential spectrum are extremely unstable (for some striking examples see e.g. \cite[\S 12.5]{SimTrId2}).
In the second approach a detector of eigenvalues needs to be told how to decide in which direction a detected eigenvalue is moving.
The answer to this question is to tell the counter: calculate the resonance index of the triple $(\lambda; H_{r_\lambda},V),$
that is, choose any real $s$ such that~$\lambda$ is not an eigenvalue of $H_s$ and find those eigenvalues of the operator
\begin{equation} \label{F: 1498}
  R_{\lambda+iy}(H_s)V
\end{equation}
with a very small $y>0,$ which are close to $(s-r_\lambda)^{-1}.$ Then the difference $N_+-N_-$ of the eigenvalues in~$\mbC_+$ and~$\mbC_-$
will show the net number of eigenvalues crossing~$\lambda$ in the positive direction at the moment of ``time'' $r = r_\lambda.$
Remarkably, this algorithm works equally well for eigenvalues embedded into essential spectrum,
so even if an eigenvalue appears suddenly from the continuous spectrum and then dissolves in it immediately afterwards
one is still able to determine which direction it appeared from and in which direction it dissolved. The difference is that
the condition ``$\lambda$ is an eigenvalue of~$H_r$'' should be replaced by the condition $\lambda \in \Lambda(H_r,F),$
or, equivalently, $r \in R(\lambda; H_0,V).$ As a consequence, to define spectral flow inside
essential spectrum one has to consider singular points instead of eigenvalues, as a non-trivial spectral flow
inside essential spectrum may be a result of moving singular continuous spectrum.


Finally we discuss the origin of terminology ``resonance points'', ``resonance index'' etc, used in this paper.
This paragraph of introduction has a formal character as it frequently refers to physical concepts and phenomena; its partial aim is to explain/justify usage of the word ``resonance'',
though this formal and remote connection with quantum scattering may be found interesting.
The justification of this terminology can be even more necessary since the word ``resonance'' has several meanings and this word is used in this paper since it is associated with
a quantum scattering phenomenon and as such it has little to do with, for instance, pushing a child on a swing.
A resonance in quantum scattering is associated with a sharp variation of the scattering cross-section as a function of energy, see e.g. \cite[\S XVIII.6]{Bohm}.
The value of energy $\lambda_0$ of a projectile at which this sharp variation occurs is called resonance energy.
Physicists associate resonances with other phenomena (see e.g. \cite[\S XII.6]{RS4}, \cite[Chapter 13]{TayST} or \cite[\S XVIII.6]{Bohm}, more specifically, see e.g. the last sentence on p.\,431
of that section and (6.1)):
\begin{enumerate}
  \item poles of the scattering matrix as a function of energy which are close to the real axis,
  \item a rapid increase of a scattering phase $\theta_j(\lambda)$ ($=2\delta_l(E)$ in physical notation) by $2\pi$ as the energy~$\lambda$
    of a projectile crosses a resonance value $\lambda_0$
  \item existence of a quasi-stationary (or meta-stable) state with energy $\lambda_0,$
  \item and finally a time delay for the interval of time between the moments of entering and leaving the interaction region around the target
    by the projectile compared to the same time-interval for non-interacting projectile.
\end{enumerate}
These phenomena are non-trivially related to each other and to the fact that at resonance energy the projectile
can be captured by the target into a nearly bound meta-stable state ``target-projectile'' (see e.g. introduction to \cite[Chapter 13]{TayST}).
These phenomena except the time delay will have mathematical analogues in our setting if one fixes the value of energy~$\lambda$ and considers as a variable the value of the coupling constant $r:$
\begin{enumerate}
   \item resonance point~$r_\lambda$ is a pole of the factor $(1+r T_{\lambda+i0}J)^{-1}$ from the stationary formula~(\ref{F: Theorem 7.2.2 of Az3}) for the scattering matrix,
   \item Theorem~\ref{T: 9.7.3 Az3} and the formula~(\ref{F: xis=ind res}) are expressions of the fact that as the coupling constant $r$ crosses a resonant value~$r_\lambda$ at least one of the scattering phases
jumps by an integer multiple of $2\pi,$
  \item by Theorem~\ref{T: 4.1.11 Az3} a value $r$ of the coupling constant at a given energy~$\lambda$ is resonant if and only if the equation
\begin{equation} \label{F: (1+rJT)psi=0}
  (1+r JT_{\lambda+i0}(H_0))\psi = 0
\end{equation}
has a non-trivial solution $\psi,$ which can be interpreted as a quasi-stationary state.
\end{enumerate}
Further, unlike the physical resonances, in this paper an idealized situation
is considered in the sense that (1) the pole of the scattering matrix is not near the real axis, but is exactly on it, (2) the scattering phase does not \emph{change rapidly} by $2\pi$
at resonance point, but \emph{jumps} by an integer multiple of $2\pi,$ (3) finally, while a physical quasi-stationary state is nevertheless a scattering state
in the sense that sooner or later the projectile leaves the target and can be observed, the quasi-stationary state represented by a solution of the equation
(\ref{F: (1+rJT)psi=0}) is not a scattering state, in the sense that it does not belong to the fiber Hilbert space $\hlambda(H_r).$ 
The latter may be attributed to the possibility that in this idealized situation, --- a pole exactly on the real axis, the projectile
gets captured by the target and never leaves it; see e.g. Pearson's example in \cite[\S XI.4, p.\,70]{RS3}, which shows that this scenario is mathematically possible.
This is also in accordance with a physical fact that time delay is proportional to the inverse width $1/\Gamma$
of the resonance bump (= imaginary part of the resonance pole), which (the width $\Gamma$) is zero
(see e.g. \cite[(13.10)]{TayST}, \cite[\S XVIII.6, p.\,432]{Bohm}). 

\subsection{Main results}
This subsection gives a list of main results of this paper.

Let $\clA = \set{H_r = H_0+rV \colon r \in \mbR}$ be an affine line of self-adjoint operators~$H_r$
on a separable complex Hilbert space $\hilb$ and let $\clK$ be another Hilbert space.
These operators are assumed to satisfy the following conditions:
\begin{enumerate}
  \item All self-adjoint operators $H_r, \ r \in \mbR$ have a common dense domain $\euD.$
  This implies that domain of $V$ contains $\euD.$
  \item The operator $V$ admits a factorization $V = F^*JF,$ where $F \colon \hilb \to \clK$ is a closed operator with trivial kernel and co-kernel,
  and $J \colon \clK \to \clK$ is a bounded operator. It is assumed that the factorization is such that the domain of $F$ contains $\euD.$
  \item Let $R_z(H_r) = (H_r-z)^{-1}$ be the resolvent of~$H_r,$ $z \in \mbC\setminus \mbR.$ Since by (1) for any $r\in \mbR$ the range of $R_z(H_r)$ contains $\euD,$ by the first two assumptions
  the operator $FR_z(H_r)F^*$ is well-defined on the dense domain of~$F^*.$ It is assumed that the operator $FR_z(H_r)F^*$ is bounded and moreover is compact.
  This operator will be denoted by $T_z(H_r).$
  \item It is assumed that the set $\Lambda(\clA,F)$ of points~$\lambda$ such that for some $r\in \mbR$ the norm limit
  $$
    T_{\lambda+i0}(H_r) := \lim_{y \to 0^+} T_{\lambda+iy}(H_r)
  $$
  exists and therefore is compact, has full Lebesgue measure. This is the main assumption, called the Limiting Absorption Principle.
  Numbers from the full set $\Lambda(\clA,F)$ will be called essentially regular for the affine space~$\clA.$
  Given an essentially regular number~$\lambda,$ a point $r$ for which the limit $T_{\lambda+i0}(H_r)$
  exists will be called \emph{regular at} $\lambda,$ otherwise it will be called \emph{resonant at} $\lambda.$
\end{enumerate}

The set $R(\lambda; H_0,V)$ of resonant at~$\lambda$ numbers is a discrete subset of $\mbR;$ dependence on $F$ is not indicated in the notation
$R(\lambda; H_0,V).$

By Weyl's theorem, all operators $H_r \in \clA$ have common essential spectrum which is denoted~$\sigma_{ess}.$
Let $\Pi$ be the set which is defined as a disjoint union of $\mbC \setminus \sigma_{ess}$ and of two copies
$\Lambda(\clA,F)+i0$ and $\Lambda(\clA,F)-i0$ of $\Lambda(\clA,F)$
$$
  \Pi = \brs{\mbC \setminus \sigma_{ess}} \cup \brs{\Lambda(\clA,F)+i0} \cup \brs{\Lambda(\clA,F)-i0}.
$$
If $\lambda \notin \sigma_{ess},$ then $\lambda = \lambda+i0 = \lambda-i0,$ but otherwise $\lambda+i0 \neq \lambda-i0.$
Thus, the operator $T_z(H_r)$ as a function of $z$ is defined on the set $\Pi$ with the exception of those~$z=\lambda\pm i0$ for which $r$ is a resonance point.

For $z \in \Pi$ let
$$
  A_z(s) = T_z(H_s)J, \qquad B_z(s) = J T_z(H_s)
$$
Given a number $z \in \Pi,$ a number $r_z \in \mbC$ is called a \emph{resonance point} corresponding to $z,$ if $r_z$ is a pole of the meromorphic function $s \mapsto A_z(s).$
We define the vector spaces
$$
  \Upsilon_z^k(r_z) = \set{u \in \clK \colon (1+(r_z-s)T_z(H_s)J)^k u = 0}
  \quad \text{and} \quad
  \Upsilon_z(r_z) = \bigcup _{k =1,2,\ldots} \Upsilon_z^k(r_z),
$$
$$
  \Psi_z^k(r_z) = \set{\psi \in \clK \colon (1+(r_z-s)JT_z(H_s))^k \psi = 0}
  \quad \text{and} \quad
  \Psi_z(r_z) = \bigcup _{k =1,2,\ldots} \Psi_z^k(r_z),
$$
and idempotents
\begin{equation*} 
  P_z(r_z) = \frac 1{2\pi i} \oint _{C(\sigma_z(s))} \brs{\sigma - A_z(s)}^{-1}\,d\sigma, \quad
  Q_z(r_z) = \frac 1{2\pi i} \oint _{C(\sigma_z(s))} \brs{\sigma - B_z(s)}^{-1}\,d\sigma,
\end{equation*}
where $C(\sigma_z(s))$ is a small circle enclosing the eigenvalue $\sigma_z(s) = (s-r_z)^{-1}$ of $A_z(s),$
such that there are no other eigenvalues inside or on the circle.
These vector spaces and idempotents do not depend on the choice of $s \in \mbR,$ as long as, in case~$z$ belongs to the boundary $\partial \Pi,$
the operator $A_z(s)$ exists (Propositions~\ref{P: res eq-n is correct} and~\ref{P: Pz is well-defined}).
Many properties of the vector space $\Upsilon_z^k(r_z)$ and the idempotent $P_z(r_z)$ are similar
to those of the vector space $\Psi_z^k(r_z)$ and the idempotent $Q_z(r_z);$
for this reason only properties of the former are given.
The idempotent $P_z(r_z)$ has the following properties~(\ref{F: Pz(rz)=res Az(s)}):
\begin{equation*}
  P_z(r_z) = \frac 1{2\pi i} \oint_{C(r_z)} A_z(s)\,ds;
\end{equation*}
for any two different resonance points $r_z^1$ and $r_z^2$~(\ref{F: Pz(1)Pz(2)=0})
\begin{equation*}
  P_z(r_z^1) P_z(r_z^2) = 0.
\end{equation*}
With every resonance point~$r_z$ the following three non-negative integers are associated which are respectively called \emph{geometric multiplicity}, \emph{algebraic multiplicity} and \emph{order} of $r_z:$
$$
  m = \dim \Upsilon_z^1(r_z),\qquad N = \dim \Upsilon_z(r_z), \qquad d = \min \set{k \in \mbN \colon \Upsilon_z^k(r_z) = \Upsilon_z(r_z)}.
$$
A number $r_z$ is resonant for $z$ if and only if the number $\bar r_z$ is resonant for $\bar z,$ in which case the numbers $m,N,$ and $d$
are the same for $r_z$ and $\bar r_z.$

The nilpotent operators $\bfA_z(r_z)$ and $\bfB_z(r_z)$ are defined by formulas
$$
  \bfA_z(r_z) = \frac {1}{2\pi i}\oint_{C(r_z)} (s-r_z) A_z(s)\,ds \quad \text{and} \quad \bfB_z(r_z) = \frac {1}{2\pi i}\oint_{C(r_z)} (s-r_z) B_z(s)\,ds,
$$
where $C(r_z)$ is a small contour which encloses the resonance point $r_z$ and no other resonance points.

Section \ref{S: Resonance points} also contains an exposition of other properties of the idempotents $P_z(r_z)$ and $Q_z(r_z)$ and the nilpotent operators
$\bfA_z(r_z)$ and $\bfB_z(r_z)$ which are used repeatedly throughout this paper, such as
$$
  (P_z(r_z))^* = Q_{\bar z}(\bar r_z), \ \ (\bfA_z(r_z))^* = \bfB_{\bar z}(\bar r_z).
$$
$$
  JP_z(r_z) = Q_z(r_z)J, \ \ J\bfA_z(r_z) = \bfB_z(r_z)J.
$$
Further, for a fixed $z \in \Pi,$ the function $A_z(s)$ is a meromorphic function of $s$ whose poles are exactly the resonance points corresponding to $z.$
The Laurent expansion of $A_z(s)$ at a pole $r_z$ is
\begin{equation*} 
  A_z(s) = \tilde A_{z,r_z}(s) + \frac 1{s-r_z} P_z(r_z) + \frac 1{(s-r_z)^2} \bfA_z(r_z) + \ldots +
  \frac 1{(s-r_z)^d} \bfA_z^{d-1}(r_z),
\end{equation*}
where $\tilde A_{z,r_z}(s) $ is the holomorphic part.

\noindent


In section~\ref{S: Geom meaning of Ups(1)} we study the relationship between eigenvectors of $H_{r_\lambda}$ corresponding to an eigenvalue~$\lambda$
and resonance vectors of order 1.

\begin{thm} \label{I:T: res eq-n and eigenvectors} (Theorem \ref{T: res eq-n and eigenvectors})
Let~$\lambda$ be an essentially regular point, let $H_0 \in \clA$ be regular at~$\lambda$ operator,
let $V \in \clA_0(F),$ let~$r_\lambda$ be a real resonance point of the triple $(\lambda; H_0,V)$
and let $r$ be a regular point of the triple $(\lambda; H_0,V).$
If~$\lambda$ is an eigenvalue of the operator~$H_{r_\lambda} = H_0+r_\lambda V$
with eigenvector~$\chi \in \euD = \dom(H_0),$ then the vector $u = F\chi$ is a resonance vector of order~1,
that is,
\begin{equation*} 
  \brs{1+(r_\lambda-r)T_{\lambda+i0}(H_{r})J}u = 0.
\end{equation*}
\end{thm}

\begin{cor} \label{I:C: mult of lambda <= dim Ups(1)} (Corollary \ref{C: mult of lambda <= dim Ups(1)})
If~$\lambda$ is an essentially regular point, then the
geometric multiplicity of~$\lambda$ as an eigenvalue of the self-adjoint operator $H_{r_\lambda} = H_0+r_\lambda V$
does not exceed dimension of the vector space $\Upsilon_{\lambda+i0}^1(r_\lambda),$
that is,
$$
  \dim \clV_\lambda \leq \dim \Upsilon^1_{\lambda+i0}(r_\lambda),
$$
where $\clV_\lambda$ is the eigenspace of $H_{r_\lambda}$ corresponding to the eigenvalue~$\lambda.$
\end{cor}

\begin{thm} \label{I:T: infty mult-ty then not essentially regular} (Theorem \ref{T: infty mult-ty then not essentially regular})
  If~$\lambda$ is an eigenvalue of infinite multiplicity for at least one self-adjoint operator~$H$ from the affine space $\clA = H_0+\clA_0(F),$
  then~$\lambda$ is not an essentially regular point of the pair $(\clA,F),$ that is, $\lambda \notin \Lambda(\clA,F).$
\end{thm}

Now we return to the discussion of spectral flow inside essential spectrum. Since inside essential spectrum a non-trivial spectral flow can be generated in absence of
any eigenvalues, the notion of multiplicity of eigenvalue needs to be properly generalized. To this end, there is the following
\begin{thm} \label{I: T: if lambda notin ess sp, ...} (Theorem~\ref{T: if lambda notin ess sp, ...}) \ Let~$\lambda$ be a real number which does not belong to the essential spectrum and let~$r_\lambda$ be
a resonance point of the triple $(\lambda; H_0,V)$
(that is,~$\lambda$ is an eigenvalue of $H_{r_\lambda}$). Let $s$ be any non-resonant point of the triple $(\lambda; H_0,V).$
The rigging operator~$F$ is a linear isomorphism of the vector space $\clV_\lambda$ of eigenvectors of $H_{r_\lambda}$ corresponding to the eigenvalue~$\lambda$
and the vector space $\Upsilon_{\lambda+i0}^1(r_\lambda)$ of eigenvectors of the operator $T_{\lambda+i0}(H_s)J$ corresponding to the eigenvalue $(s-r_\lambda)^{-1}.$
\end{thm}

\noindent
Theorems \ref{I:T: res eq-n and eigenvectors} and \ref{I: T: if lambda notin ess sp, ...} give a rationale to call the integer number $\dim \Upsilon_{\lambda+i0}^1(r_\lambda)$
\emph{multiplicity of the singular spectrum} of the self-adjoint operator
$H_{r_\lambda}$ at~$\lambda.$ That this is a reasonable definition is further confirmed by the U-turn Theorem~\ref{I:T: U-turn for res index}.

\begin{thm} \label{I:T: mult of s.c. spectrum: drastic version} (Theorem \ref{T: mult of s.c. spectrum: drastic version})
If~$H_{r_\lambda}$ is resonant at an essentially regular point~$\lambda,$
then the vector space
$$
  \Upsilon^1_{\lambda+i0}(r_\lambda) = \Upsilon_\lambda^1(H_{r_\lambda},V)
$$ does not depend on a regularizing operator~$V.$
\end{thm}

In section~\ref{S: res index} we introduce a class~$\clR$ of finite-rank operators which do not have
non-zero real eigenvalues. A so-called $\euR$-index for operators~$A$
of class~$\clR$ is defined as the difference $N_+-N_-:$
$$
  \Rindex(A) = N_+-N_-,
$$
where $N_+$ and $N_-$ are the numbers of eigenvalues of~$A$ in the upper $\mbC_+$ and lower $\mbC_-$ half-planes respectively.
Some elementary properties of $\euR$-index and a new proof of \Krein's theorem \cite{Kr53MS}
$$
  \euR(R_z(H)V) = \sign(V),
$$
where~$H$ is a self-adjoint operator and~$V$ is a finite rank self-adjoint operator, are given.

Further, in this section \emph{resonance index} of a triple $(\lambda, H_{r_\lambda},V)$ is introduced, which can be defined by formula
$$
  \ind_{res}(\lambda, H_{r_\lambda},V) = \euR(A_{\lambda+iy}(s)P_{\lambda+iy}(r_\lambda)) \ \ \text{for all small enough } y.
$$
Given a finite set $\Gamma = \set{r_z^1, \ldots, r_z^M}$ of resonance points corresponding to $z \in \Pi,$
we denote by $P_z(\Gamma)$ and $Q_z(\Gamma)$ the idempotents
$$
  P_z(\Gamma) = \sum_{r_z \in \Gamma} P_z(r_z) \quad \text{and} \quad Q_z(\Gamma) = \sum_{r_z \in \Gamma} Q_z(r_z)
$$
respectively. By $\bar \Gamma$ we denote the set $\set{\bar r_z^1, \ldots, \bar r_z^M}.$

The following theorem is one of the main technical results of this paper.
\begin{thm} \label{I:T: res matrix is positive for set of up-points} (Theorem \ref{T: res matrix is positive for set of up-points})
If $\Gamma = \set{r_z^1, \ldots, r_z^M}$ is
a finite set of resonance up-points corresponding to a non-real number~$z,$ then the operator
$$
  \Im z \, Q_{\bar z}(\bar \Gamma) J P_z(\Gamma)
$$
is non-negative and its rank is equal to the rank of $P_z(\Gamma).$
\end{thm}

\begin{thm} \label{I:T: sign of res matrix for set of res points} (Theorem \ref{T: sign of res matrix for set of res points})
If $\Gamma = \set{r_z^1, \ldots, r_z^M}$ is
a finite set of resonance points corresponding to a non-real number~$z,$ then the signature of the finite-rank self-adjoint operator
$
  Q_{\bar z}(\bar \Gamma) J P_z(\Gamma)
$
is equal to the $R$-index of the operator
$
  \Im z \, A_z(s)P_z(\Gamma).
$
\end{thm}
Theorems \ref{I:T: res matrix is positive for set of up-points} and \ref{I:T: sign of res matrix for set of res points}
are non-trivial even in the finite-dimensional case $\dim \hilb < \infty,$
that is, for matrices. 


In section \ref{S: vectors of type I} we prove the following
\begin{prop} \label{I:P: euE (Vf)=0, k>1} Let~$\lambda$ be an essentially regular point,
let $\set{H_0+rV \colon r \in \mbR}$ be a line regular at~$\lambda,$
let~$r_\lambda$ be a real resonance point of the path $\set{H_0+rV \colon r \in \mbR}$ at~$\lambda$ and
let~$k$ be a positive integer. If $u_{\lambda\pm i0}(r_\lambda) \in \Upsilon_{\lambda\pm i0}(r_\lambda)$ is a resonance vector of order
$k\geq 1$ at $\lambda\pm i0,$ then 
for all non-resonant values of~$s$ the following equality holds:
\begin{equation} \label{I:F: (J psi,Im TJ psi)=c(-2)/s2+...}
  \scal{J u_{\lambda\pm i0}(r_\lambda)}{\Im T_{\lambda \pm i0}(H_s)J u_{\lambda\pm i0}(r_\lambda)}
    = \frac {c_{\pm 2}}{(s-r_\lambda)^2} + \frac {c_{\pm 3}}{(s-r_\lambda)^3} + \ldots + \frac {c_{\pm k}}{(s-r_\lambda)^k},
\end{equation}
where, in case $k\geq 2,$ for $j = 2,\ldots,k$
\begin{equation*} \label{I:F: c(pm j)}
 \begin{split}
   c_{\pm j} & = \Im \scal{u_{\lambda\pm i0}(r_\lambda)}{J\bfA_{\lambda \pm i0}^{j-1}(r_\lambda) u_{\lambda\pm i0}(r_\lambda)}
         \\ & = - \Im \scal{u_{\lambda\pm i0}(r_\lambda)}{J\bfA_{\lambda \mp i0}^{j-1}(r_\lambda) u_{\lambda\pm i0}(r_\lambda)}.
 \end{split}
\end{equation*}
In particular, if $u_{\lambda\pm i0}(r_\lambda) \in \Upsilon_{\lambda\pm i0}(r_\lambda)$ is a resonance vector of order 1, then
\begin{equation*} \label{I:F: (J psi,Im TJ psi)=0}
  \scal{J u_{\lambda\pm i0}(r_\lambda)}{\Im T_{\lambda \pm i0}(H_s)J u_{\lambda\pm i0}(r_\lambda)} = 0.
\end{equation*}
\end{prop}
Further, in section \ref{S: vectors of type I} we introduce and study the so-called vectors of type I.
These are vectors which satisfy any of the following equivalent conditions.

\begin{thm} \label{I:T: type I vectors} Let~$r_\lambda$ be a real resonance point of the line $\gamma = \set{H_r \colon r \in \mbR},$ corresponding to a real number $\lambda \in \Lambda(\gamma,F).$
Let $u \in \clK.$ The following assertions are equivalent:
\begin{enumerate}
  \item $u \in \Upsilon_{\lambda+i0}(r_\lambda)$ and for all non-resonant real numbers $s$
  $$
    \sqrt{\Im T_{\lambda+i0}(H_s)}Ju = 0.
  $$
  \item $u \in \Upsilon_{\lambda-i0}(r_\lambda)$ and for all non-resonant real numbers $s$
  $$
    \sqrt{\Im T_{\lambda+i0}(H_s)}Ju = 0.
  $$
  \item $u \in \Upsilon_{\lambda+i0}(r_\lambda)$ and for all non-resonant real numbers $s$
  $$
    A_{\lambda+i0}(s)u = A_{\lambda-i0}(s)u.
  $$
  \item $u \in \Upsilon_{\lambda-i0}(r_\lambda)$ and for all non-resonant real numbers $s$
  $$
    A_{\lambda+i0}(s)u = A_{\lambda-i0}(s)u.
  $$
  \item $u \in \Upsilon_{\lambda+i0}(r_\lambda)$ or $u \in \Upsilon_{\lambda-i0}(r_\lambda)$ and for all $j=0,1,2,\ldots,d-1,$ where $d$ is the order of~$r_\lambda,$
  $$
    \bfA^j _{\lambda+i0}(r_\lambda) u = \bfA^j _{\lambda-i0}(r_\lambda) u.
  $$
  \item $u \in \Upsilon_{\lambda+i0}(r_\lambda)$ and there exists a non-resonant real number $r$ such that for all $j=0,1,2,\ldots$
  $$
    (A_{\lambda+i0}(r)-A_{\lambda-i0}(r)) \bfA^j _{\lambda+i0}(r_\lambda) u = 0.
  $$
  \item $u \in \Upsilon_{\lambda-i0}(r_\lambda)$ and there exists a non-resonant real number $r$ such that for all $j=0,1,2,\ldots$
  $$
    (A_{\lambda+i0}(r)-A_{\lambda-i0}(r)) \bfA^j _{\lambda-i0}(r_\lambda) u = 0.
  $$
  \item $u \in \Upsilon_{\lambda+i0}(r_\lambda)$ and all the coefficients $c_{+j}$ from the equality (\ref{I:F: (J psi,Im TJ psi)=c(-2)/s2+...})
  are equal to zero.
  \item $u \in \Upsilon_{\lambda-i0}(r_\lambda)$ and all the coefficients $c_{-j}$ from the equality (\ref{I:F: (J psi,Im TJ psi)=c(-2)/s2+...})
  are equal to zero.
\end{enumerate}
The set $\Upsilon_{\lambda+ i0}^{\mathrm I}(r_\lambda)$ of vectors which satisfy any of these equivalent conditions is a vector subspace
of the vector space $\Upsilon_{\lambda+ i0}(r_\lambda) \cap \Upsilon_{\lambda - i0}(r_\lambda)$ and the vector space
$\Upsilon_{\lambda+ i0}^{\mathrm I}(r_\lambda)$ is invariant with respect to both $\bfA^j _{\lambda+i0}(r_\lambda)$ and $\bfA^j _{\lambda-i0}(r_\lambda).$
\end{thm}

For the nilpotent operator $\bfA_z(r_z)$ on the vector space $\Upsilon_z(r_z)$ there exists a Jordan basis
$(u_\nu^{(j)}),$ $\nu=1,\ldots,m,$ $j=1,\ldots,d_\nu,$ where we assume that $d_1 \geq d_2 \geq \ldots \geq d_m;$
that is, a basis of $\Upsilon_z(r_z)$ such that $\bfA_z(r_z)u_\nu^{(j)} = u_\nu^{(j-1)}$ assuming that $u_\nu^{(0)}=0.$
Every Jordan basis $(u_\nu^{(j)})$ induces a decomposition of the vector space $\Upsilon_z(r_z)$ into a direct sum
$$
  \Upsilon_z(r_z) = \Upsilon_z^{[1]}(r_z) \dotplus \ldots \dotplus \Upsilon_z^{[m]}(r_z),
$$
where $\Upsilon_z^{[\nu]}(r_z)$ is the linear span of vectors $u_\nu^{(1)},\ldots,u_\nu^{(d_\nu)}$
and where $\dotplus$ denotes direct sum of linear spaces.
We call this decomposition a \emph{Jordan decomposition} of~$\Upsilon_z(r_z).$
\label{Page: Jordan decomp-n}

Proposition \ref{I:P: euE (Vf)=0, k>1} and Theorem \ref{I:T: type I vectors} are used to prove the following theorem
which in its turn is essentially used in section \ref{S: U-turn theorem}.
\begin{thm} \label{I:T: on vectors with property L} If a resonance vector $u^{(k)} \in \Upsilon_{\lambda\pm i0}(r_\lambda)$ has order $k$ then the vectors
$$
  u^{(1)}, \ldots, u^{(\lceil k/2\rceil)}
$$
are of type~I, where $\lceil k/2\rceil$ is the smallest integer not less than $k/2$ and $u^{(j)} = \bfA^{k-j}_{\lambda \pm i0}(r_\lambda)u^{(k)}.$
\end{thm}
For example, assume that the geometric multiplicity $m=12$ and order $d=6;$
if a Jordan basis $(u_\nu^{(j)})$ of $\Upsilon_{\lambda+i0}(r_\lambda)$ is represented by the left of the following two Young diagrams \label{Page: Young diagram}

\noindent\begin{picture}(192,96)
\put(0,0){\line(0,1){96}}\put(16,0){\line(0,1){96}} \put(32,0){\line(0,1){96}} \put(48,0){\line(0,1){96}}
\put(64,0){\line(0,1){80}} \put(80,0){\line(0,1){80}} \put(96,0){\line(0,1){48}} \put(112,0){\line(0,1){48}}
\put(128,0){\line(0,1){32}} \put(144,0){\line(0,1){32}} \put(160,0){\line(0,1){16}} \put(176,0){\line(0,1){16}}
\put(192,0){\line(0,1){16}}
\put(0,0){\line(1,0){192}}\put(0,16){\line(1,0){192}}\put(0,32){\line(1,0){144}}\put(0,48){\line(1,0){112}}
\put(0,64){\line(1,0){80}}\put(0,80){\line(1,0){80}}\put(0,96){\line(1,0){48}}
\put(3,5){\tiny $u_{1}^{(1)}$}\put(3,21){\tiny $u_{1}^{(2)}$}\put(3,37){\tiny $u_{1}^{(3)}$}\put(3,53){\tiny $u_{1}^{(4)}$}\put(3,69){\tiny $u_{1}^{(5)}$}\put(3,85){\tiny $u_{1}^{(6)}$}
\put(19,5){\tiny $u_{2}^{(1)}$}\put(19,21){\tiny $u_{2}^{(2)}$}\put(19,37){\tiny $u_{2}^{(3)}$}\put(19,53){\tiny $u_{2}^{(4)}$}\put(19,69){\tiny $u_{2}^{(5)}$}\put(19,85){\tiny $u_{2}^{(6)}$}
\put(35,5){\tiny $u_{3}^{(1)}$}\put(35,21){\tiny $u_{3}^{(2)}$}\put(35,37){\tiny $u_{3}^{(3)}$}\put(35,53){\tiny $u_{3}^{(4)}$}\put(35,69){\tiny $u_{3}^{(5)}$}\put(35,85){\tiny $u_{3}^{(6)}$}
\put(51,5){\tiny $u_{4}^{(1)}$}\put(51,21){\tiny $u_{4}^{(2)}$}\put(51,37){\tiny $u_{4}^{(3)}$}\put(51,53){\tiny $u_{4}^{(4)}$}\put(51,69){\tiny $u_{4}^{(5)}$}
\put(67,5){\tiny $u_{5}^{(1)}$}\put(67,21){\tiny $u_{5}^{(2)}$}\put(67,37){\tiny $u_{5}^{(3)}$}\put(67,53){\tiny $u_{5}^{(4)}$}\put(67,69){\tiny $u_{5}^{(5)}$}
\put(83,5){\tiny $u_{6}^{(1)}$}\put(83,21){\tiny $u_{6}^{(2)}$}\put(83,37){\tiny $u_{6}^{(3)}$}
\put(99,5){\tiny $u_{7}^{(1)}$}\put(99,21){\tiny $u_{7}^{(2)}$}\put(99,37){\tiny $u_{7}^{(3)}$}
\put(115,5){\tiny $u_{8}^{(1)}$}\put(115,21){\tiny $u_{8}^{(2)}$}
\put(131,5){\tiny $u_{9}^{(1)}$}\put(131,21){\tiny $u_{9}^{(2)}$}
\put(147,5){\tiny $u_{10}^{(1)}$}
\put(163,5){\tiny $u_{11}^{(1)}$}
\put(179,5){\tiny $u_{12}^{(1)}$}
\end{picture}
\qquad
\begin{picture}(192,96)
\put(0,0){\line(0,1){96}}\put(16,0){\line(0,1){96}} \put(32,0){\line(0,1){96}} \put(48,0){\line(0,1){96}}
\put(64,0){\line(0,1){80}} \put(80,0){\line(0,1){80}} \put(96,0){\line(0,1){48}} \put(112,0){\line(0,1){48}}
\put(128,0){\line(0,1){32}} \put(144,0){\line(0,1){32}} \put(160,0){\line(0,1){16}} \put(176,0){\line(0,1){16}}
\put(192,0){\line(0,1){16}}
\put(0,0){\line(1,0){192}}\put(0,16){\line(1,0){192}}\put(0,32){\line(1,0){144}}\put(0,48){\line(1,0){112}}
\put(0,64){\line(1,0){80}}\put(0,80){\line(1,0){80}}\put(0,96){\line(1,0){48}}
\put(3,5){\tiny $u_{1}^{(1)}$}\put(3,21){\tiny $u_{1}^{(2)}$}\put(3,37){\tiny $u_{1}^{(3)}$}
\put(19,5){\tiny $u_{2}^{(1)}$}\put(19,21){\tiny $u_{2}^{(2)}$}\put(19,37){\tiny $u_{2}^{(3)}$}
\put(35,5){\tiny $u_{3}^{(1)}$}\put(35,21){\tiny $u_{3}^{(2)}$}\put(35,37){\tiny $u_{3}^{(3)}$}
\put(51,5){\tiny $u_{4}^{(1)}$}\put(51,21){\tiny $u_{4}^{(2)}$}\put(51,37){\tiny $u_{4}^{(3)}$}
\put(67,5){\tiny $u_{5}^{(1)}$}\put(67,21){\tiny $u_{5}^{(2)}$}\put(67,37){\tiny $u_{5}^{(3)}$}
\put(83,5){\tiny $u_{6}^{(1)}$}\put(83,21){\tiny $u_{6}^{(2)}$}
\put(99,5){\tiny $u_{7}^{(1)}$}\put(99,21){\tiny $u_{7}^{(2)}$}
\put(115,5){\tiny $u_{8}^{(1)}$}
\put(131,5){\tiny $u_{9}^{(1)}$}
\put(147,5){\tiny $u_{10}^{(1)}$}
\put(163,5){\tiny $u_{11}^{(1)}$}
\put(179,5){\tiny $u_{12}^{(1)}$}
\end{picture}

\noindent then according to Theorem \ref{I:T: on vectors with property L} all vectors shown on the right Young diagram are of type I.

In section \ref{S: res.index and sign res matrix} we prove that the resonance index is equal to the signature of the resonance matrix.

\begin{thm} \label{I:T: property M} (Theorem \ref{T: property M})
The idempotents $P_{\lambda\pm i0}(r_\lambda)$ are linear isomorphisms of the vector spaces $\Upsilon_{\lambda\mp i0}(r_\lambda)$
and $\Upsilon_{\lambda\pm i0}(r_\lambda).$
\end{thm}

Theorem \ref{I:T: property M} is used in the proof of the following theorem which is one of the main results of this paper.

\begin{thm} \label{I:T: res.ind=sign res.matrix} (Theorem \ref{T: res.ind=sign res.matrix})
For any real resonance point~$r_\lambda$
the signatures of the resonance matrices $\sign(Q_{\lambda\mp i0}(r_\lambda)JP_{\lambda\pm i0}(r_\lambda))$
of~$r_\lambda$ are the same and are equal to the resonance index
of the triple $(\lambda,H_{r_\lambda},V);$ that is,
$$
  \sign(Q_{\lambda\mp i0}(r_\lambda)JP_{\lambda\pm i0}(r_\lambda)) = \ind_{res}(\lambda; H_{r_\lambda},V).
$$
\end{thm}


In Section \ref{S: U-turn theorem} we prove Theorem \ref{I:T: U-turn for res index} which is one of the main results of this paper.

\begin{thm} \label{I:T: U-turn for res matrix} (Theorem \ref{T: U-turn for res matrix})
If~$r_\lambda$ is a real resonance point
corresponding to~$z = \lambda \pm i0,$ then the absolute value of the signature of the resonance
matrices $Q_{\lambda\mp i0}(r_\lambda)JP_{\lambda\pm i0}(r_\lambda)$ is less or equal to the dimension of the vector space
$\Upsilon_{\lambda+i0}^1(r_\lambda):$
$$
  \abs{\sign Q_{\lambda\mp i0}(r_\lambda)JP_{\lambda\pm i0}(r_\lambda)} \leq \dim \Upsilon_{\lambda+i0}^1(r_\lambda).
$$
\end{thm}

Theorems \ref{I:T: res.ind=sign res.matrix} and \ref{I:T: U-turn for res matrix} have the following corollary. 

\begin{thm} \label{I:T: U-turn for res index} (Theorem \ref{T: U-turn for res index}) For all real resonance points~$r_\lambda$
$$
  \abs{\ind_{res}(\lambda; H_{r_\lambda}, V)} \leq \dim \Upsilon_{\lambda+i0}^1(r_\lambda).
$$
\end{thm}
\noindent
Theorem \ref{I:T: U-turn for res index} has the following meaning: the increment of the spectral flow inside essential spectrum
which occurs at a resonance point~$r_\lambda$ cannot be larger than the multiplicity of the singular spectrum
of $H_{r_\lambda}$ at~$\lambda.$

The numbers $N_\pm$ from the definition of the resonance index give more information about the behaviour of points of the singular spectrum than the difference $N_+-N_-.$
Appealing to a shop-keeper's doorbell, a customer may open the door and leave without entering the shop. In this
case the doorbell rings but the number of customers in the shop remains the same (that is, increment of spectral flow is zero).
In other words, a ring of the doorbell condition $r \in R(\lambda; H_0,V)$ does not necessarily mean that an ``eigenvalue'' crossed~$\lambda,$
e.g., if $\lambda$ is outside the essential spectrum, an eigenvalue can make a U-turn at~$\lambda.$
Theorems~\ref{T: xis=ind res} and~\ref{I:T: U-turn for res index} imply that
if there is an eigenvalue $\lambda_j(r)$ of a path~$H_r$ making a U-turn at~$\lambda$ when $r=r_\lambda,$
then $N_+>0$ and $N_->0$ so that contributions of that eigenvalue to $N_+$ and $N_-$ cancel each other.
In particular, if the eigenvalue $\lambda_j(r_\lambda) = \lambda$ of $H_{r_\lambda}$ making a U-turn is non-degenerate, then
$N_+=N_-,$ so that $\ind_{res}(\lambda; H_{r_\lambda},V)$ is zero.
On page \pageref{Page: figures} of this paper eight diagrams are given which correspond to eight qualitatively different eigenvalue behaviors in case $N_+=5$
and $N_-= 2.$


\bigskip
The main result of section \ref{S: property U} is Theorem \ref{T: Property U}.
Proof of this theorem relies on certain algebraic relations between operators $P_{\lambda \pm i0}(r_\lambda)$
and $\bfA_{\lambda \pm i0}(r_\lambda)$ which are proved in this section.

\label{Page: point with property C}
A real resonance point~$r_\lambda$ will be said to have \emph{property~$C$} if the vector spaces $\Upsilon_{\lambda\pm i0}(r_\lambda)$
admit Jordan decompositions
\begin{equation*} 
  \Upsilon_{\lambda+i0}(r_\lambda) = \Upsilon^{[1]}_{\lambda+i0}(r_\lambda) \dotplus \Upsilon^{[2]}_{\lambda+i0}(r_\lambda) \dotplus \ldots \dotplus \Upsilon^{[m]}_{\lambda+i0}(r_\lambda)
\end{equation*}
and
\begin{equation*} 
  \Upsilon_{\lambda-i0}(r_\lambda) = \Upsilon^{[1]}_{\lambda-i0}(r_\lambda) \dotplus \Upsilon^{[2]}_{\lambda-i0}(r_\lambda) \dotplus \ldots \dotplus \Upsilon^{[m]}_{\lambda-i0}(r_\lambda)
\end{equation*}
such that for all $j=1,2,\ldots,m$ the following equalities hold:
\begin{equation*} 
  P_{\lambda + i0}(r_\lambda) \Upsilon^{[\nu]}_{\lambda - i0}(r_\lambda) = \Upsilon^{[\nu]}_{\lambda + i0}(r_\lambda)
  \quad \text{and} \quad P_{\lambda - i0}(r_\lambda) \Upsilon^{[\nu]}_{\lambda + i0}(r_\lambda) = \Upsilon^{[\nu]}_{\lambda - i0}(r_\lambda).
\end{equation*}

\begin{thm} \label{I:T: Property U} (Theorems \ref{T: Property U} and \ref{T: Q(+) is a lin. isom-m})
For any $z = \lambda \pm i0 \in \partial \Pi,$ for any real resonance point~$r_\lambda \in \mbR$ with property~$C,$
corresponding to~$z$ and for any $j=1,2,3,\ldots$ (1) restriction of the idempotent
operator $P_{\lambda\pm i0}(r_\lambda)$ to $\Upsilon^j_{\lambda\mp i0}(r_\lambda)$
is a linear isomorphism of the vector spaces $\Upsilon^j_{\lambda\mp i0}(r_\lambda)$ and $\Upsilon^j_{\lambda\pm i0}(r_\lambda),$
and
(2) the idempotent $Q_{\lambda\pm i0}(r_\lambda)$ is a linear isomorphism
of the vector spaces~$\Psi^j_{\lambda\mp i0}(r_\lambda)$ and~$\Psi^j_{\lambda\pm i0}(r_\lambda)$
for all $j=1,2,\ldots.$

In other words, for points~$r_\lambda$ with property~$C,$ for all $j=1,2,3,\ldots$ we have commutative diagrams of linear isomorphisms:
\begin{equation*} 
  \xymatrix{
    \Psi_{\lambda+i0}^j(r_\lambda)  && \Upsilon_{\lambda+i0}^j(r_\lambda) \ar[ll]_J  \\
                    &&                               \\
    \Psi_{\lambda-i0}^j(r_\lambda) \ar[uu]^{Q_{\lambda+i0}(r_\lambda)} && \Upsilon_{\lambda-i0}^j(r_\lambda) \ar[ll]^J \ar[uu]_{P_{\lambda+i0}(r_\lambda)}
  }
\qquad\qquad
  \xymatrix{
    \Psi_{\lambda+i0}^j(r_\lambda) \ar[dd]_{Q_{\lambda-i0}(r_\lambda)} && \Upsilon_{\lambda+i0}^j(r_\lambda) \ar[ll]_J \ar[dd]^{P_{\lambda-i0}(r_\lambda)} \\
                    &&                               \\
    \Psi_{\lambda-i0}^j(r_\lambda)  && \Upsilon_{\lambda-i0}^j(r_\lambda) \ar[ll]^J
  }
\end{equation*}
\end{thm}
Real resonance points for which the conclusion of this theorem holds are called points with \emph{property~$U$}.
Thus, property $C$ implies property $U.$ Plainly, every point of geometric multiplicity 1 has property~$C$ and therefore it has property~$U$ too.
We conjecture that every real resonance point has properties~$C$ and~$U.$

In section \ref{S: questions of indend from F} we consider some questions of independence from the choice of the rigging operator~$F.$

\begin{thm} \label{I:T: res.ind independent of F} (Theorem \ref{T: res.ind independent of F})
The resonance index $\ind_{res}(\lambda; H,V)$ does not depend on the choice of the rigging operator~$F$
as long as~$\lambda$ is essentially regular for the pair $(\clA,F),$ where $\clA = \set{H+rV \colon r \in \mbR}$ and $V$ is
a regularizing direction for an operator $H$ which is resonant at~$\lambda.$
\end{thm}

In section~\ref{S: res point of type I} we study a class of the
so-called real resonance points of type~I. By definition, a real
resonance point~$r_\lambda$ is of type~I if and only if for some
regular value of~$s$
\begin{equation*}
  \sqrt{\Im T_{\lambda+i0}(H_s)}\, JP_{\lambda+i0}(r_\lambda) = 0.
\end{equation*}

\begin{thm} \label{I:T: type I thm} (Theorem \ref{T: type I thm})
Let~$\lambda$ be an essentially regular point for the pair $(\clA,F).$
Let~$H_0 \in \clA$ be an operator regular at~$\lambda$ and let $V \in \clA_0(F).$ Let~$r_\lambda \in \mbR$ be a resonance point of the path $\set{H_0+rV \colon r \in \mbR}.$
The following assertions are all equivalent to~$r_\lambda$ of being of type I.
\begin{enumerate}
\item[(i$_\pm$)] For any regular point~$r$ \ \ $\sqrt{\Im T_{\lambda+i0}(H_r)}J P_{\lambda\pm i0}(r_\lambda) = 0.$
\item[(i$^*_\pm$)] There exists a regular point~$r$ such that $\sqrt{\Im T_{\lambda+i0}(H_r)}J P_{\lambda\pm i0}(r_\lambda) = 0.$
\item[(ii$_\pm$)] For any regular point~$r$ \ \ $\sqrt{\Im T_{\lambda+i0}(H_r)}Q_{\lambda\pm i0}(r_\lambda) = 0.$
\item[(ii$^*_\pm$)] There exists a regular point~$r$ such that \ \ $\sqrt{\Im T_{\lambda+i0}(H_r)} Q_{\lambda\pm i0}(r_\lambda) = 0.$
\item[(iii$_\pm$)] The meromorphic function
$$
  w_\pm(s) := \sqrt{\Im T_{\lambda+i0}(H_0)}[1+sJT_{\lambda\pm i0}(H_0)]^{-1}
$$
is holomorphic at $s = r_\lambda.$
\item[(iii$'_\pm$)] The meromorphic function
$$
  w_\pm(s)J = \sqrt{\Im T_{\lambda+i0}(H_0)} J [1+sT_{\lambda\pm i0}(H_0)J]^{-1}
$$
is holomorphic at $s = r_\lambda.$
\item[(iv$_\pm$)] The meromorphic function
$$
  w^\dagger_\pm(s) = [1+sT_{\lambda\mp i0}(H_0)J]^{-1}\sqrt{\Im T_{\lambda+i0}(H_0)}
$$
is holomorphic at $s = r_\lambda.$
\item[(v$_\pm$)] The residue of the function $w_\pm(s)$ at $s = r_\lambda$ is zero.
\item[(vi$_\pm$)] For all $\pm$-resonance vectors the real numbers $c_{-j}$ from Proposition~\ref{P: euE (Vf)=0, k>1}
are all zero.
\item[(vii)] The function \
$
  s \mapsto \Im T_{\lambda+i0}(H_s)
$ \
is holomorphic at $s = r_\lambda.$
\item[(viii)] The function \
$
  s \mapsto J\Im T_{\lambda+i0}(H_s)J
$ \
is holomorphic at $s = r_\lambda.$
\end{enumerate}
Moreover, assertions obtained from {\rm (i$_\pm$)--(ii$_\pm$)} and {\rm (i$^*_\pm$)--(ii$^*_\pm$)} by removing the square root are also equivalent to these ones.
\end{thm}

The following theorem shows that the property of being of type I is a generic property of real resonance points.
\begin{thm} \label{I:T: suff conditions for type I} (Theorems \ref{T: outside ess sp then type I}, \ref{T: order 1 then type I}, \ref{T: V>0, then type I})
Let~$\lambda$ be an essentially regular point, let~$H_0 \in \clA$ and let $V \in \clA_0(F)$ be a regularizing direction at~$\lambda.$
Let~$r_\lambda$ be a real resonance point of the triple $(\lambda, H_0,V).$
If at least one of the following three conditions hold,
\medskip
\begin{enumerate}
  \item~$\lambda$ does not belong to the (necessarily common) essential spectrum of operators from~$\clA,$
  \item order of~$r_\lambda$ is equal to 1,
  \item the operator $V$ is non-negative or non-positive,
\end{enumerate}
then~$r_\lambda$ is a point of type~I.
\end{thm}

\noindent For every real resonance point~$r_\lambda$ of type I the idempotents $P_{\lambda+i0}(r_\lambda)$ and $P_{\lambda-i0}(r_\lambda)$
coincide.
We say that a real resonance point~$r_\lambda$ has \emph{property~$S$} if
kernels of the idempotents $P_{\lambda+i0}(r_\lambda)$ and $P_{\lambda-i0}(r_\lambda)$ coincide.

\begin{prop} \label{I:P: property S} (Proposition \ref{P: property S})
Let~$\lambda$ be an essentially regular point and let~$r_\lambda$ be a real resonance point of a triple $(\lambda, H_0,V).$
The following assertions are equivalent:
\begin{enumerate}
  \item[(i)] \ $r_\lambda$ has property $S.$
  \item[(ii)] \ $P_{\lambda+i0}(r_\lambda) P_{\lambda-i0}(r_\lambda) = P_{\lambda+i0}(r_\lambda)$ and $P_{\lambda-i0}(r_\lambda) P_{\lambda+i0}(r_\lambda) = P_{\lambda-i0}(r_\lambda).$
  \item[(iii)] \ $\im Q_{\lambda+i0}(r_\lambda) = \im Q_{\lambda-i0}(r_\lambda).$
  \item[(iv)] \ $Q_{\lambda+i0}(r_\lambda) Q_{\lambda-i0}(r_\lambda) = Q_{\lambda-i0}(r_\lambda)$ and $Q_{\lambda-i0}(r_\lambda) Q_{\lambda+i0}(r_\lambda) = Q_{\lambda+i0}(r_\lambda).$
  \item[(v)] \ $Q_{\lambda-i0}(r_\lambda)JP_{\lambda+i0}(r_\lambda) = JP_{\lambda+i0}(r_\lambda).$
  \item[(vi)] \ $Q_{\lambda+i0}(r_\lambda)JP_{\lambda-i0}(r_\lambda) = JP_{\lambda-i0}(r_\lambda).$
  \item[(vii)] \ $Q_{\lambda-i0}(r_\lambda)JP_{\lambda+i0}(r_\lambda) = Q_{\lambda-i0}(r_\lambda)J.$
  \item[(viii)] \ $Q_{\lambda+i0}(r_\lambda)JP_{\lambda-i0}(r_\lambda) = Q_{\lambda+i0}(r_\lambda)J.$
  \item[(ix)] \ $Q_{\lambda-i0}(r_\lambda)JP_{\lambda+i0}(r_\lambda) = Q_{\lambda+i0}(r_\lambda) J P_{\lambda-i0}(r_\lambda).$
\end{enumerate}
\end{prop}

\begin{prop} \label{I:P: points without property S exist} (Proposition \ref{P: points without property S exist})
Every resonance point of type~I has property~$S.$ There are resonance points which do not have property $S,$
and there are points with property~$S$ which are not of type~I.
\end{prop}

\noindent
\label{Page: point with property P}
Let's say that a real resonance point~$r_\lambda$ has \emph{property~$P$} if $P_{\lambda+i0}(r_\lambda) = P_{\lambda-i0}(r_\lambda).$
Relations between real resonance points with different properties are given in the following diagram, where arrows stand for implications.
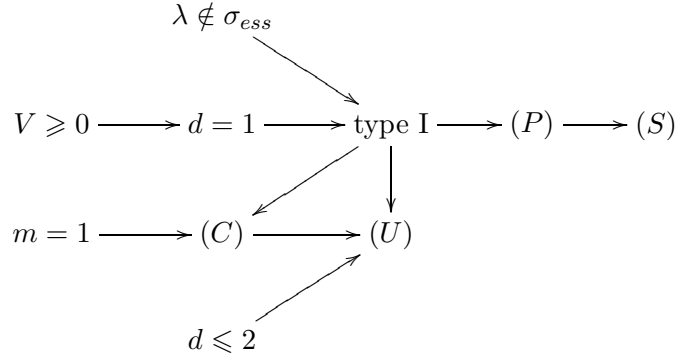
\begin{equation*}
  \xymatrix{
                           & \lambda \notin \sigma_{ess} \ar[rd]   &                                    &                    \\
    V \geq 0 \ar[r]        & d=1  \ar[r]                           &  \text{type~I} \ar[r]\ar[d]\ar[ld] &  (P) \ar[r] & (S)  \\
     m=1 \ar[r]            & (C) \ar[r]                            &  (U)                               &                    \\
                           & d\leq 2 \ar[ru]                       & & \\ 
  }
\end{equation*}

In section~\ref{S: Pert embedded eig} we study behaviour of a non-degenerate eigenvalue embedded into essential spectrum
under a regularizing perturbation~$V.$ 
In this subsection we in particular construct real resonance points which do not have property $S,$
and real resonance points with property $S,$ which don't have property~$P.$

Assume that $\lambda$ is an eigenvalue of a self-adjoint operator $H_{r_\lambda}$ with eigenvector~$\chi.$
Then the Hilbert space $\hilb$ on which $H_{r_\lambda}$ acts can be represented as
$$
  \hilb = \hat \hilb \oplus \mbC,
$$
such that the operator $H_{r_\lambda}$ takes the form
\begin{equation*}
  H_{r_\lambda} = \left(\begin{matrix}
    \hat H_{r_\lambda} & 0 \\
    0 & \lambda
  \end{matrix}\right)
\end{equation*}
where $\hat H_{r_\lambda}$ is the restriction of $H_{r_\lambda}$ to $\hat \hilb.$
Let
\begin{equation*}
  V = \left(\begin{matrix}
    \hat V & \hat v \\
    \scal{\hat v}{\cdot} & \alpha
  \end{matrix}\right)
\end{equation*}
be the representation of the operator $V.$ We assume that the rigging operator $F \colon \hilb \to \clK$ has representation
\begin{equation*}
  F = \left(\begin{matrix}
    \hat F & 0 \\
    0 & 1
  \end{matrix}\right).
\end{equation*}
In this case $V = F^*JF,$ where $J$ has representation
\begin{equation*}
  J = \left(\begin{matrix}
    \hat J & \hat \psi \\
    \la \hat \psi, \cdot \ra & \alpha
  \end{matrix}\right)
\end{equation*}
such that $\hat V = \hat F ^* \hat J \hat F$ and $\hat v = \hat F^* \hat \psi.$
The vector $\hat \psi$ is connected with the eigenvector $\chi$ by the equality $\hat \psi = JF\chi - \alpha F\chi.$
Finally, we assume that $\lambda$ is a regular point of the pair $(\hat H_{r_\lambda},\hat F):$
$$
  \lambda \in \Lambda(\hat H_{r_\lambda},\hat F).
$$
Let
$$
  \hat u_z(s) = \hat F R_z(\hat H_s)\hat F^* \hat \psi,
$$
where $T_z(\hat H_s)=\hat F R_z(\hat H_s)\hat F^*$ and let $\hat A_z(s) = T_z(\hat H_s)\hat J.$
The operator $A_{\lambda+i0}(r_\lambda)$ does not exist since $\lambda \notin \Lambda(H_{r_\lambda},F),$
but the sliced operator $\hat A_{\lambda+i0}(r_\lambda)$ and the vector $\hat u_{\lambda+i0}(r_\lambda)$ exist due
to the condition~$\lambda \in \Lambda(\hat H_{r_\lambda},\hat F).$
%

The following lemma and the theorem describe properties of the resonance point~$r_\lambda.$

\begin{lemma} 
Order of the resonance point~$r_\lambda$
is not less than 2 if and only if $\alpha = 0.$
If this is the case, then the vector space~$\Upsilon^2_{\lambda+i0}(r_\lambda)$ is two-dimensional and is generated by vectors
$$
  F\chi = \twovector{0}{1} \quad \text{and} \quad \twovector{\hat u_{\lambda+i0}(r_\lambda)}{0},
$$
which have orders~$1$ and~$2$ respectively.
\end{lemma}

\begin{thm} \label{I:T: order k case}  (Theorem \ref{T: order k case})
Let~$d$ be an integer not less than two.
The order of the real resonance point~$r_\lambda$ is equal to~$d$ if and only if the vectors
\begin{equation*} 
  \hat u_{\lambda+i0}(r_\lambda), \ \hat A_{\lambda+i0}(r_\lambda)\hat u_{\lambda+i0}(r_\lambda), \ldots, \hat A_{\lambda+i0}^{d-3}(r_\lambda)\hat u_{\lambda+i0}(r_\lambda)
\end{equation*}
are orthogonal to the vector $\hat \psi$ but the vector $\hat A_{\lambda+i0}^{d-2}(r_\lambda)\hat u_{\lambda+i0}(r_\lambda)$ is not.
If this is the case, then for all $j=1,2,\ldots,d$ the vector space~$\Upsilon^j_{\lambda+i0}(r_\lambda)$ is $j$-dimensional and is generated by vectors
\begin{equation*} 
  \twovector{0}{1}, \ \ \twovector{\hat u_{\lambda+i0}(r_\lambda)}{0}, \ \ \twovector{\hat A_{\lambda+i0}(r_\lambda)\hat u_{\lambda+i0}(r_\lambda)}{0}, \ \ \ldots, \ \
   \twovector{\hat A^{j-2}_{\lambda+i0}(r_\lambda)\hat u_{\lambda+i0}(r_\lambda)}{0},
\end{equation*}
which have orders $1,2,\ldots,j$ respectively.
\end{thm}

The following diagram shows interdependence of sections 2--14.
A dashed arrow means that the dependence is of notational and terminological character.
In particular, the section \ref{S: Pert embedded eig} is almost independent of the other sections, but motivation for this section comes from previous ones.
The core of this paper are sections \ref{S: res.index as s.ssf}, \ref{S: sign of res matrix}, \ref{S: res.index and sign res matrix} and \ref{S: U-turn theorem}.
Having said this, ideologically all sections are interconnected in the sense that they represent different aspects of the same subject
given in the title of this paper.
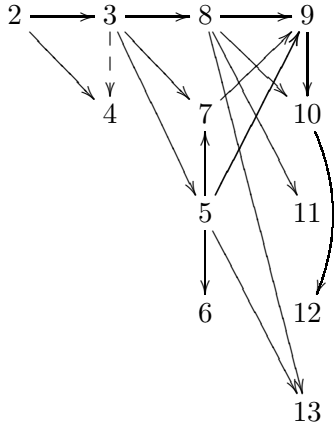
\begin{equation*}
  \xymatrix{
    2 \ar[r]\ar[rd] & 3 \ar[r]\ar@{-->}[d]\ar[rd]\ar[rdd] & 8 \ar[r]\ar[rd]\ar[rdd]\ar[ddddr]       & 9 \ar[d] \\
                    & 4                                   & 7 \ar[ur]                               & 10 \ar@/^.8pc/[dd] \\
                    &                                     & 5 \ar[u]\ar[uur]\ar[d]\ar[uur]\ar[ddr]  & 11 \\
                    &                                     & 6                                       & 12 \\
                    &                                     &                                         & 13 \\
  }
\end{equation*}
In section~\ref{S: open problems} some open problems are stated.
Finally, for reader's convenience there is a detailed index.

\subsection{Future work}
\subsubsection{Integrity of singular spectral shift function for relatively trace-class perturbations.}
So far the property (\ref{F: xis is in Z}) of the singular spectral shift function has been proved for trace-class perturbations.
There is a paper in preparation \cite{AzDa} in which this result will be proved for relatively trace-class perturbations.
A special case of this result is the following
\begin{thm} \label{T: xis for Schrodinger}
Let $H_0= -\frac {d^2}{dx^2}+V_0(x)$ be a one-dimensional Schr\"odinger operator, where $V_0(x)$ is a bounded measurable real-valued function and let $V$
be an operator of multiplication by a real-valued measurable function $V(x)$ such that $\abs{V(x)} \leq \const (1+\abs{x})^{-1-\eps}$
for some $\eps > 0$ and let $H_r = H_0+rV.$ Let
$$
  \xis(\phi) = \int_0^1 \Tr(V \phi(H_r^{(s)}))\,dr, \quad \phi \in C_c(\mbR),
$$
where $H_r^{(s)}$ is the singular part of~$H_r.$ Then $\xis$ is an absolutely continuous measure whose density $\xis(\lambda)$ (denoted by the same symbol) is integer-valued
for a.e. $\lambda.$
\end{thm}
The bulk of the proof of this theorem is a modification for relatively trace-class perturbations of the approach to scattering theory given in \cite{Az3v6}
and discussed in this introduction. This modification was given in \cite{Az10}
with an aim to prove Theorem \ref{T: xis for Schrodinger}. For reasons mentioned in this introduction classical approaches to scattering theory
do not allow to prove this theorem.

There is a work in progress with the aim to prove an analogue of Theorem \ref{T: xis for Schrodinger} for $n$-dimensional Schr\"odinger operators.

\subsubsection{Resonance index and singular $\mu$-invariant.}
For trace-class perturbations the singular spectral shift function admits three other equivalent descriptions, as singular $\mu$-invariant, total resonance index
and total signature of resonance matrix. These three definitions do not require the perturbation to be trace class or to be relatively trace class, ---
all we need to assume is that the perturbation is relatively compact and that the limiting absorption principle holds.
In this paper it is shown that resonance index and signature of resonance matrix are equal under these two conditions. In \cite{Az11} it will be proved that the singular $\mu$-invariant
is equal to the total resonance index given the same conditions.

\subsection{Acknowledgements}
  I thank Thomas Daniels for a scrupulous and critical reading of this paper
  which greatly reduced the number of inaccuracies and typos.
  I also thank Prof. Peter Dodds and Prof. Jerzy Filar for their moral support.

\section{Preliminaries}
\label{S: Preliminaries}
\subsection{Operators on a Hilbert space}
Details, concerning the material of this section, can be found in \cite{GK,Kato,RS1,SimTrId2}.
A partial aim of these preliminaries is to fix notation and terminology.

Throughout this paper, $\mbR$ is the field of real numbers and $\mbC$ is the field of complex numbers.
The calligraphic letters $\hilb$ and~$\clK$ will denote complex separable Hilbert spaces --- finite or infinite dimensional.
The scalar product $\scal{\cdot}{\cdot}$ is assumed to be linear with respect to the second argument and anti-linear with respect to the first.
If it is necessary to distinguish the Hilbert spaces $\hilb$ and $\clK,$ the former will be called the main Hilbert space,
and the latter will be called auxiliary Hilbert space; having said this, it should be noted that majority of operators considered in this paper
act on the auxiliary Hilbert space~$\clK$ rather than the main one~$\hilb.$
The letter~$H$ with possible indices will denote a self-adjoint operator on~$\hilb.$ The letter~$F$ will always denote a fixed densely defined closed operator
from $\hilb$ to~$\clK$ which has trivial kernel and co-kernel. The letter $\Lambda$ with arguments will always denote a measurable subset of $\mbR$
of full Lebesgue measure. Throughout this paper the word ``operator'' means a linear operator.

The letter~$V$ will be used to denote a self-adjoint operator on $\hilb$ with some conditions imposed on it.
We shall consider perturbation $H_r=H_0+rV$ of a self-adjoint operator~$H_0$ by a real multiple of~$V;$
the multiple itself, called a coupling constant, will be denoted by the letters $s$ and $r$ (with possible subindexes).

A subset~$A$ of a metric space $X$ is \emph{discrete} if intersection of~$A$ with any compact subset of $X$ is finite.

If $L_1$ and $L_2$ are two closed subspaces of a Hilbert space such that $L_1\cap L_2 = \set{0},$ then by
$L_1 \dotplus L_2$ we denote the direct sum of $L_1$ and $L_2.$ If in addition to this the subspaces $L_1$ and $L_2$
are orthogonal then the direct sum of $L_1$ and $L_2$ we denote by $L_1 \oplus L_2$ instead of $L_1 \dotplus L_2.$

By $\ker(A)$ the \emph{kernel} of an operator~$A$ is denoted and $\im(A)$ will denote the \emph{range} or the \emph{image} of $A.$ \label{Page: im(A)}
The \emph{resolvent set} $\rho_T$ of a densely-defined closed operator~$T$ on a Hilbert space $\hilb$ consists of all complex numbers~$z$ such that
the operator $T - z$ is a bijection of the domain $\dom(T)$ onto $\hilb;$ for such $z$ the bounded inverse\label{Page: Rz(s)}
$$
  R_z(T) = (T-z)^{-1},
$$
called \emph{resolvent} of $T,$ exists.
The spectrum $\sigma_T$ or $\sigma(T)$ of a densely-defined closable operator~$T$ on a Hilbert space is the complement of the resolvent set.
For two bounded operators $S$ and~$T$ one has (see e.g. \cite[Proposition 2.2.3]{BR1})
\begin{equation} \label{F: spec(AB)=spec(BA)}
  \sigma_{ST} \cup \set{0} = \sigma_{TS} \cup \set{0}.
\end{equation}
Let~$T$ be a closed operator on a Hilbert space~$\clK$ and let $z \in \mbC.$
Non-zero vectors~$u$ from~$\clK$ such that $(T-z)^k u = 0$ for some $k =1,2,\ldots$ are called \emph{root vectors} of~$T$ corresponding to an eigenvalue $z.$
A point~$z$ of the spectrum of~$T$ is called an \emph{isolated eigenvalue} of finite algebraic multiplicity if~$z$ is an isolated point of $\sigma(T)$ and if
the \emph{algebraic multiplicity} $\mu_T(z)$ of~$z$ defined by
$$
  \mu_T(z) := \dim \set{u \in \clK \colon \exists k \in \mbZ_+ \ (T-z)^k u = 0}
$$
is finite. The set of all isolated eigenvalues of finite algebraic multiplicity of an operator~$T$ is denoted by $\sigma_{d}(T).$
If~$T$ is compact, the function $\mu_T$ of~$z$ is called \emph{spectral measure} of $T.$
\label{Page: mu(A)}
If~$S$ and~$T$ are bounded operators, such that~$ST$ and~$TS$ are both compact, then the following stronger version of (\ref{F: spec(AB)=spec(BA)}) holds:
\begin{equation} \label{F: mu(AB)=mu(BA)}
  \mu_{ST}|_{\mbC\setminus \set{0}} = \mu_{TS}\big|_{\mbC\setminus \set{0}}.
\end{equation}
Further, for any compact operator~$T$
\begin{equation} \label{F: mu(A*)=bar mu(A)}
  \mu_{T^*} = \bar \mu_{T},
\end{equation}
where $\bar \mu_{T}(z) = \mu_T(\bar z).$

A closed operator~$T$ is said to be \emph{Fredholm} if the range of~$T$ is a closed subspace of finite co-dimension and
the kernel of~$T$ is finite-dimensional (see \cite[IV.5.1]{Kato}). A~bounded operator~$T$ is Fredholm if and only if there exists a bounded operator $S$
such that the operators $ST-1$ and $TS-1$ are compact, such an operator $S$ is called parametrix of $T.$ In other words, bounded Fredholm operators are invertible up to compact operators.
By definition, \emph{essential spectrum} $\sigma_{ess}(T)$ of a closed operator~$T$ consists of all complex numbers~$z$ such that the operator
$T-z$ is not Fredholm; in this regard note that in \cite{Kato} the essential spectrum is defined as the set of
all complex numbers~$z$ such that the operator $T-z$ is not semi-Fredholm, see \cite[\S IV.5]{Kato}.
There are also other definitions of the essential spectrum, but all of them coincide for self-adjoint operators.
Since in this paper we shall be concerned with the essential spectrum of self-adjoint operators and of their relatively compact perturbations,
this definition suffices.
Essential spectrum of a self-adjoint operator~$H$ admits another characterization: the essential spectrum of~$H$
is the spectrum of~$H$ from which all isolated eigenvalues of finite multiplicity are removed.
In general,
$$
  \sigma_{ess}(T) \subset \sigma(T) \setminus \sigma_d(T),
$$
but this inclusion may be strict \cite{Kato}.

Let~$H$ and $V$ be two self-adjoint operators on a Hilbert space~$\hilb.$ The operator $V$ is said to be \emph{relatively compact} with respect to $H,$
if $R_z(H)V$ is a bounded operator on $\dom(V) \subset \hilb$ for some $z \in \rho_H$ such that its continuous prolonging to $\hilb$ is a compact operator.
In this case the operator $R_z(H)V$ is bounded with compact prolonging for any $z \in \rho_H.$
Weyl's Theorem asserts that essential spectrum of a self-adjoint operator is stable under relatively 
compact perturbations (see e.g. \cite[\S IV.5.6]{Kato}, \cite[\S XIII.4]{RS4}).

The spectrum of a closed operator $T$ on a Hilbert space is upper semi-continuous: for any neighbourhood~$O$ of the spectrum of~$T$
there exists~$\delta>0$ such that for all bounded~$S$ with $\norm{S} < \delta$ the spectrum of~$S+T$ is a subset of~$O.$

In general, the spectrum is not continuous in the sense that for a bounded operator~$T$ there may exist $z \in \sigma(T)$
and a neighbourhood $O$ of~$z$ such that for any $\delta>0$ there exists a bounded operator $S$ with $\norm{S}< \delta$
such that $\sigma(T+S) \cap O = \emptyset.$

For brevity, the identity operator on a Hilbert space is denoted by~$1;$ in particular, the scalar operator of multiplication by a number $c$ will be denoted by $c$
instead of $cI.$ An \emph{idempotent operator} is a bounded operator $P$ such that $P^2=P.$
If~$A$ and $B$ are two bounded operators such that $z \notin \sigma_{AB} \cup \set{0},$ then
\begin{equation} \label{F: (s-AB)(-1)A=...}
  (z - AB)^{-1} A = A(z - BA)^{-1}.
\end{equation}
The condition $z \notin \sigma_{AB} \cup \set{0}$ implies that $z \notin \sigma_{BA},$ so that
the right hand side of the above equality makes sense. Hence, the equality itself follows from obvious equality $A(z - BA) = (z - AB)A.$

The \emph{real} $\Re A$ and \emph{imaginary} $\Im A$ parts of a bounded operator~$A$ on a Hilbert space are defined by formulas
\label{Page: Re(A)}
$$
  \Re A = \frac{A+A^*}2 \quad \text{and} \quad \Im A = \frac{A-A^*}{2i}.
$$
\emph{Rank} of an operator~$A$ is the dimension of the image of $A.$
The \emph{signature} $\sign(A)$ of a finite-rank self-adjoint operator~$A$ is an integer defined as follows:
\label{Page: sign(A)}
\begin{equation} \label{F: sign(A)}
  \sign(A) = \rank A_+ - \rank A_-,
\end{equation}
where~$A_+$ (respectively,~$A_-$) is the positive (respectively, negative) part of~$A.$
In this regard note that, given a self-adjoint operator~$A,$ the word ``signature'' is also used for the operator $f(A),$ where $f(x)$ is the sign-function,
but in this paper this notion will not be used and therefore there is no danger of confusion.

\begin{lemma} If~$A$ is an operator of rank $N < \infty,$ then there exists $\eps>0$ such that for any operator $B$ of norm less than $\eps$
the inequality $\rank(A+B) \leq N$ implies the equality \mbox{$\rank(A+B) = N.$}
\end{lemma}
\noindent In other words, small enough perturbations of finite rank operators which do not increase the rank preserve the rank.
This lemma is a direct consequence of the upper-semicontinuity of spectrum.

\begin{lemma} \label{L: finite-rank lemma} Let $M$ be a finite-rank self-adjoint operator on a Hilbert space $\clK.$ If $\clL$ is a vector subspace of $\clK$ such that
for any non-zero $f \in \clL$ the scalar product $\scal{f}{Mf}$ is positive, then
$$
  \dim \clL \leq \rank M_+.
$$
\end{lemma}
\begin{proof} Let $\clM_+$ be the vector space spanned by eigenvectors of $M$ corresponding to positive eigenvalues and assume contrary to the claim that
$\dim \clL > \dim \clM_+.$ Then the intersection $\clM_+^\perp \cap \clL$ is a vector subspace of dimension at least 1. If $f$ is a non-zero vector from
$\clM_+^\perp \cap \clL,$ then $\scal{f}{Mf} > 0$ since $f \in \clL$ and $\scal{f}{Mf} \leq 0$ since $f \in \clM_+^\perp.$
\end{proof}

\begin{lemma} \label{L: sign(M)=sign(FMF)} Let $M$ be a self-adjoint finite rank operator on a Hilbert space $\clK$ and let $F\colon \hilb \to \clK$ be a closed operator with zero kernel and co-kernel.
If $\im(M) \subset \dom(F^*),$ then
the product $F^*MF$ is a well defined finite-rank self-adjoint operator for which $\rank(M) = \rank(F^*MF)$ and $\sign(M) = \sign(F^*MF).$
\end{lemma}
\begin{proof} Let~$\clM_+$ (respectively, $\clM_-$) be the vector spaces spanned by eigenvectors of $M$ corresponding to positive (respectively, negative) eigenvalues of $M$
and let $\clM =\clM_+\oplus \clM_-.$
Since $\im(M) \subset \dom(F^*),$ the product~$F^*MF$ is well defined.
Since~$F$ and~$F^*$ have zero kernel, the ranks of operators~$F^*MF$ and~$M$ are the same.
Since the range of~$F$ contains the vector spaces~$\clM_\pm,$ the vector spaces $\clL_\pm = F^{-1}\clM_\pm$ are well-defined and $\dim \clM_\pm = \dim \clL_\pm.$
For any non-zero vector $f = F^{-1} g\in \clL_+,$ where $g \in \clM_+,$
we have
$$
  \scal{f}{F^*MFf} =\scal{Ff}{MFf} = \scal{g}{Mg} > 0.
$$
It follows from this and Lemma~\ref{L: finite-rank lemma} that $\dim \clL_+ = \dim \clM_+$ is not larger than the rank of the positive part of $F^*MF.$
Similarly, one shows that $\dim \clM_-$ is not larger than the rank of the negative part of $F^*MF.$ Combining this with equality $\rank(M) = \rank(F^*MF)$
implies that $\sign(M) = \sign(F^*MF).$
\end{proof}

If~$T$ is a compact operator on a Hilbert space, then the sequence of $s$-numbers $s_1(T), s_2(T), \ldots$ of~$T$
is the sequence of eigenvalues of the compact operator $\abs{T} := \sqrt{T^*T}$ listed in non-increasing order
such that each eigenvalue is repeated in the list according to its multiplicity.
Let $p \in [1,\infty].$
The notation $\clL_p(\hilb)$ denotes the class of all compact operators~$T$ acting on $\hilb$ such that the $p$-norm $\norm{T}_p$ of $T,$ defined by equality
$$
  \norm{T}_p^p := \sum_{n = 1}^\infty s_n(T)^p, \ \text{if } p<\infty; \ \ \norm{T}_\infty := s_1(T) = \norm{T}, \ \text{if } p = \infty,
$$
is finite. The linear space $\clL_p(\hilb)$ with thus defined norm is an invariant operator ideal \cite{GK},
called $p$-th Schatten ideal.
This means, in particular, that if $T \in \clL_p(\hilb)$ and if $A,B$ are bounded operators, then $ATB \in \clL_p(\hilb)$
and $\norm{ATB}_p \leq \norm{A}\norm{T}_p\norm{B}.$ The ideal $\clL_\infty$ consists of all compact operators on~$\hilb.$
Operators from the first Schatten ideal $\clL_1(\hilb)$ are called \emph{trace class operators},
operators from the second Schatten ideal $\clL_2(\hilb)$ are called \emph{Hilbert-Schmidt operators}.
\emph{Trace-class norm} $\norm{T}_1$ of a trace-class operator~$T$ is equal to $\Tr\abs{T}$
and \emph{Hilbert-Schmidt norm} $\norm{T}_2$ of a Hilbert-Schmidt operator~$T$ is equal to $\sqrt{\Tr(\abs{T}^2)}.$
The \emph{trace} $\Tr \colon \clL_1(\hilb) \to \mbC$
is a linear continuous functional, defined for trace class operators by $\Tr(T) = \sum_{n=1}^\infty \scal{\phi_n}{T\phi_n},$ where $\set{\phi_n}_{n=1}^\infty$
is an orthonormal basis of~$\hilb.$ If $A,B$ are bounded operators such that $AB$ and $BA$ are trace-class, then
$$
  \Tr(AB) = \Tr(BA).
$$
Further, for any trace-class operator~$T$ the equality $\Tr(T^*) = \overline{\Tr(T)}$ holds.
For any trace class operator~$T$ the sequence of eigenvalues $\set{\lambda_j(T)}_{j=1}^\infty$
of the operator~$T$ is summable; the Lidskii theorem asserts that
\begin{equation} \label{F: Lidskii Thm}
  \Tr(T) = \sum_{j=1}^\infty \lambda_j(T).
\end{equation}

\begin{lemma} \label{L: An to A iff An to A in p-norm}
Let $p\geq 1.$ If $A, A_1, A_2, A_3,\ldots$ is a sequence of finite-rank operators on a Hilbert space such that the sequence of ranks of~$A_n$ is bounded,
then~$A_n$ converges to~$A$ as $n\to\infty$ in the uniform norm if and only if~$A_n$ converges to~$A$ in $p$-norm as $n\to\infty.$
\end{lemma}
\begin{proof} If~$N$ is the largest of ranks of operators $A, A_1, A_2, A_3,\ldots,$ then $\norm{A} \leq \norm{A}_p \leq N \norm{A}.$
\end{proof}

\subsection{Analytic operator-valued functions}
For definition and detailed study of vector-valued holomorphic functions see e.g. \cite{HPh,Kato,RS1}.
Let $T(\kappa)$ be a~single-valued holomorphic function with values in bounded operators;
assume that $T(\kappa)$ is defined in some domain~$G$ of the complex plane except a discrete set of singular points.
In a deleted neighbourhood $0 < \abs{\kappa - \kappa_0}< \delta$ of a singular point~$\kappa_0 \in G$ the function $T(\kappa)$ admits Laurent expansion at~$\kappa_0$
\begin{equation} \label{F: T(kappa)=}
  T(\kappa) = \tilde T(\kappa) + \sum_{j=1}^\infty (\kappa-\kappa_0)^{-j} T_{j},
\end{equation}
where $\tilde T(\kappa)$ is a function of~$\kappa$ holomorphic in the neighbourhood of~$\kappa_0$ (including~$\kappa_0$) and $T_1, T_2,\ldots$ are some bounded operators.
A~function~$T$ defined on~$G$ is said to be \emph{meromorphic} in $G$ if it is holomorphic everywhere on~$G$ except possibly a discrete subset of singular points,
such that at each singular point~$\kappa_0$ the sum in its Laurent expansion (\ref{F: T(kappa)=}) is finite.

\begin{thm} \label{T: Analytic Fredholm alternative} (Analytic Fredholm alternative) \
Let~$G$ be an open connected subset of $\mbC.$ Let $T \colon G \to \clL_\infty(\hilb)$
be a holomorphic family of compact operators in $G.$ If the family of operators $1+T(\kappa)$
is invertible at some point~$\kappa_1 \in G,$ then it is invertible at all points of~$G$ except the discrete set
$$
  \euN := \set{\kappa \in G \colon -1 \in \sigma(T(\kappa))}.
$$
Further, the operator-function~$F(\kappa) := (1+T(\kappa))^{-1}$ is meromorphic in $G$ and the set of its poles is $\euN.$ Moreover, in the Laurent expansion of
$F(\kappa)$ in a neighbourhood of any point~$\kappa_0 \in \euN$ the coefficients of negative powers of~$\kappa-\kappa_0$ are finite-rank operators.
\end{thm}
For proof of this theorem see e.g. \cite[Theorem VI.14]{RS1}, \cite[Theorem 1.8.2]{Ya}.

\subsection{Divided differences}
If $f(s)$ is a function of one variable, then the divided difference of~$f$ of first order is the function
$$
  f^{[1]}(s_1,s_2) = \frac{f(s_2)-f(s_1)}{s_2-s_1}.
$$
Divided difference of order~$k$ of a function $f(s)$ is a function $f^{[k]}(s_1,\ldots,s_{k+1})$ of $k+1$ variables $s_1,\ldots,s_{k+1}$
which is defined inductively by equality
$$
  f^{[k]}(s_1,\ldots,s_{k+1}) = \frac{f^{[k-1]}(s_2,\ldots,s_{k+1})-f^{[k-1]}(s_1,\ldots,s_{k})}{s_{k+1}-s_1}.
$$
We shall use two facts about divided differences.
\begin{lemma} \label{L: div-d diff-nce I}
Divided difference of order $k-1$ of a function $f$ is equal to
$$
  f^{[k-1]}(s_1,\ldots,s_k) = \sum_{j=1}^k f(s_j) \prod_{i=1, i \neq j}^k \frac 1{s_j-s_i}.
$$
\end{lemma}
\begin{lemma} \label{L: div-d diff-nce II}
Divided difference of order $k-1$ of a function $f$ is zero if and only if $f$ is a polynomial of degree $\leq k-2.$
\end{lemma}
\noindent Proofs of these lemmas can be found in e.g. \cite{Bakh}.

\subsection{Rigged Hilbert spaces}
A rigging of a Hilbert space $\hilb$ is a triple $(X,\hilb,X^*)$
which in addition to the Hilbert space $\hilb$ itself consists of
a normed or more generally locally convex space $X$ and its
adjoint $X^*$ such that $X$ is continuously embedded into $\hilb$
and $\hilb$ is continuously embedded into~$X^*$ and these
embeddings have dense ranges, see e.g. \cite{BeShu}. The rigging
normed space $X$ is often introduced as the range of a certain
operator acting on~$\hilb.$ In this case it is possible to
consider the operator itself as the rigging. In this paper we
follow this view-point. Further, the normed space~$X$ can itself
be a Hilbert or pre-Hilbert space. In this case elements
$f,g,\ldots$ of $X$ can be considered as elements of both $X^*$
via the Riesz-Fisher theorem and of $\hilb$ via the natural
embedding $X \hookrightarrow \hilb,$ and in this case it is
assumed that the equality $\scal{f}{g}_{\hilb} =
\scal{f}{g}_{(X,X^*)}$ holds, where $\scal{f}{g}_{(X,X^*)}$ is the
value of the linear functional $g\in X^*$ on the vector $f \in X.$
The number $\scal{f}{g}_{(X,X^*)}$ is often denoted by
$\scal{f}{g}_{1,-1}.$


A~\emph{rigging}~$F$ \label{Page: F} on a Hilbert space $\hilb$
is a closed operator from $\hilb$ to another Hilbert space~$\clK$
with trivial kernel and co-kernel.
Endowing a Hilbert space $\hilb$ with a rigging operator~$F$ generates a triple of Hilbert spaces
\label{Page: hilb+}
\begin{equation} \label{F: hilb+}
  \hilb_+, \ \hilb, \ \hilb_-,
\end{equation}
where the Hilbert space $\hilb_+$ is the completion of the vector space $\im \abs{F}$
endowed with the scalar product
$$
  \scal{f}{g}_{\hilb_+} = \scal{\abs{F}^{-1}f}{\abs{F}^{-1}g}_{\hilb}
$$
and the Hilbert space $\hilb_-$ is the completion of the vector space $\dom \abs{F}$ endowed with the scalar product
$$
  \scal{f}{g}_{\hilb_-} = \scal{\abs{F}f}{\abs{F}g}_{\hilb}.
$$
Similarly, the operator $F^*$ considered as a rigging in $\clK,$ generates a triple of Hilbert spaces
$$
  \clK_+, \ \clK, \ \clK_-.
$$
The mapping $\abs{F}$ prolongs to an isomorphism of $\hilb_\alpha$ and $\hilb_{\alpha-1},$ $\alpha = 0,1.$
Similarly, the mapping $\abs{F^*}$ prolongs to an isomorphism of $\clK_\alpha$ and $\clK_{\alpha-1},$ $\alpha = 0,1.$
The (prolonging of) rigging operator~$F$ itself can be considered as an isomorphism
$$
  F \colon \hilb \simeq \clK_+
$$
or as an isomorphism
$$
  F \colon \hilb_- \simeq \clK.
$$
Similarly, the operator $F^{-1}$ can be treated as an isomorphism $\clK_+ \simeq \hilb$
or as an isomorphism $\clK \simeq \hilb_{-}.$

\subsection{Limiting absorption principle}
Let $\hilb$ and $\clK$ be two complex separable Hilbert spaces and let
\begin{equation} \label{F: F}
  F\colon \hilb \to \clK
\end{equation}
be a fixed rigging operator in~$\hilb.$
Let $\clA_0=\clA_0(F)$ \label{Page: clA0} be a real normed space of self-adjoint operators~$V$ \label{Page: V}
of the form
\begin{equation} \label{F: V}
  V=F^*JF,
\end{equation}
where $J$ is an element of a real subspace of the algebra of bounded self-adjoint operators on $\clK.$
\label{Page: J} The norm of the space $\clA_0(F)$ is defined by $\norm{V}_F = \norm{J}.$ Let~$H$ be a self-adjoint operator on~$\hilb.$
\label{Page: clA} The affine space of self-adjoint operators of the form $H+V,$ where $V \in \clA_0(F),$
will be denoted by~$\clA = \clA(H,F),$ that is,
\begin{equation} \label{F: clA}
  \clA = H + \clA_0(F).
\end{equation}

Here we have firstly introduced the rigging operator~$F$ and then using it we have introduced the affine space~$\clA.$
In fact, the operator~$F$ has a service nature while the affine space~$\clA$ comes directly from formulation of a problem.
Hence, in practice, given an affine space $\clA$ one has to find a rigging~$F$ which makes the pair $(\clA,F)$ compatible in the sense that all the conditions
imposed on this pair are satisfied.

We frequently use notation \label{Page: Tz(s)}
\begin{equation} \label{F: Tz(H)}
  T_z(H) := FR_z(H)F^*.
\end{equation}
The operator $T_z(H)$ is often called a sandwiched resolvent.

We assume that all operators $V$ from the real vector space $\clA_0(F)$ are relatively compact perturbations of some operator $H$ from the real affine space~$\clA,$
that is, we assume that
$\dom(H) \subset \dom(V)$
and that the operator $R_z(H)V$ is bounded and its continuous prolonging is compact:
\begin{equation} \label{F: Rz(H)V is compact}
  \text{the operator} \ R_z(H)V \ \text{is compact}.
\end{equation}

Since all perturbation operators $V = H_1-H_0,$ where $H_1,H_0$ is any pair of operators from $\clA,$
are supposed to be relatively compact with respect to $H_0,$ this implies that $H_0$ and $H_1=H_0+V$ have the same domain.
That is, domains of all operators $H$
from the affine space $\clA$ coincide; we denote this common domain by $\euD:$
\begin{equation} \label{F: euD=dom(H)}
  \text{for any } H \in \clA \ \ \dom(H) = \euD.
\end{equation}
Further, since all perturbations $V \in \clA_0(F)$ of operators $H$ from $\clA$ are relatively compact, Weyl's theorem implies that
all operators $H$ from $\clA$ have a common essential spectrum:
$$
  \forall H_0,H_1 \in \clA \ \ \sigma_{ess}(H_0) = \sigma_{ess}(H_1).
$$
This common essential spectrum we denote by $\sigma_{ess}.$ The subset $\sigma_{ess}$ of $\mbR$ depends only on~$\clA.$
This allows us to talk about the essential spectrum of the affine space~$\clA.$

The operator~$F$ is not assumed to be bounded; therefore, one needs to clarify the meaning
of operators (\ref{F: V}) and (\ref{F: Tz(H)}). Domain of any perturbation operator $V$ contains $\euD:$
\begin{equation} \label{F: euD subset dom(V)}
  \euD \subset \dom(V).
\end{equation}
Additionally we assume that
\begin{equation} \label{F: euD subset dom(F)}
  \euD \subset \dom(F).
\end{equation}
By (\ref{F: euD subset dom(V)}), for any $H_0,H_1 \in \clA$ domain of any perturbation operator $V=H_1-H_0$ contains~$\euD;$ therefore,
any operator~$J$ from (\ref{F: V}) satisfies
\begin{equation} \label{F: JF euD subset dom(F*)}
  JF \euD \subset \dom(F^*).
\end{equation}
Since by (\ref{F: euD=dom(H)}) for any $H \in \clA$ the range of the resolvent $R_z(H)$ is equal to~$\euD,$ and on this subspace the operator $F$ is defined by
the assumption (\ref{F: euD subset dom(F)}),
the sandwiched resolvent (\ref{F: Tz(H)}) is defined at least on the dense domain of $F^*.$ It will always be assumed that the operator~(\ref{F: Tz(H)})
is bounded on~$\dom(F^*)$ and that its continuous prolonging to $\clK$ is compact:
\begin{equation} \label{F: Tz(H) is compact}
  T_z(H) \ \text{is compact.}
\end{equation}
This also implies that for any bounded subset $\Delta$ of $\mbR$
\begin{equation} \label{F: FE(Delta) is compact}
  FE_\Delta^H \ \text{is compact.}
\end{equation}
Indeed, by (\ref{F: Tz(H) is compact}), the operator $\Im T_z(H) = (F\sqrt{\Im R_z(H)})(F\sqrt{\Im R_z(H)})^*$ is compact, and hence so is the operator
$F\sqrt{\Im R_z(H)}.$ This implies (\ref{F: FE(Delta) is compact}).
Using this one can show that for a bounded rigging operator~$F$ the condition (\ref{F: Tz(H) is compact}) implies (\ref{F: Rz(H)V is compact}).

\begin{lemma} If $FR_z(H)F^*$ is compact for some $z \in \rho(H),$ then
$FR_w(H)F^*$ is compact for any other $w \in \rho(H).$ Further,
the function $\mbC\setminus \mbR \ni z \mapsto T_z(H)$ is a holomorphic function.
\end{lemma}
\begin{proof} Without loss of generality we can assume that $y = \Im z > 0.$
If $FR_z(H)F^*$ is compact then so is $FR_{\bar z}(H)F^* = (FR_z(H)F^*)^*$
and therefore the operator $\brs{F\sqrt{\Im R_z(H)}}\brs{F\sqrt{\Im R_z(H)}}^* = F(\Im R_z(H))F^* = (FR_z(H)F^* - FR_{\bar z}(H)F^*)/(2i)$ is also compact.
It follows that $F\sqrt{\Im R_z(H)}$ is compact. Since the function $\mbR \ni x \mapsto R_z(x)/\sqrt{\Im R_z(x)},$ where $R_z(x) = (x-z)^{-1},$ is bounded
(by $y^{-1/2}$ as can be easily checked), the operator $R_z(H)/\sqrt{\Im R_z(H)}$ is also bounded. It follows that the operator $F R_z(H)$ is compact.
Since the function $\mbR \ni x \mapsto R_w(x)/R_z(x)$ is bounded, it follows that $F R_w(H)$ is compact.
Hence, $F R_z(H)R_w(H)F^*$ is compact. Since $F R_z(H)R_w(H)F^* = (z-w) F (R_z(H)-R_w(H))F^*$ and since $F R_z(H)F^*$ is also compact,
it follows that $F R_w(H)F^*$ is compact too.
The second assertion follows from equivalence of weak and strong analyticity.
\end{proof}

Given an operator~$H$ from the affine space~$\clA,$ the notation \label{Page: Lambda(H,F)}
\begin{equation} \label{F: Lambda(H,F)}
  \Lambda(H,F)
\end{equation}
will be used to denote the set of all real numbers~$\lambda$ for which the limit
\begin{equation} \label{F: T(l+i0) exists}
  T_{\lambda+i0}(H) := \lim_{y \to 0^+} T_{\lambda+iy}(H) \ \ \text{exists in the uniform topology}.
\end{equation}
Since $\brs{T_z(H)}^* = T_{\bar z}(H)$ and since the operation of taking adjoint is continuous in the uniform topology,
it follows that the norm limit $T_{\lambda+i0}(H)$ exists if and only if the norm limit $T_{\lambda-i0}(H)$ exists.
Thus, if $\lambda \in \Lambda(H,F)$ then also
\begin{equation} \label{F: Im T(l+i0) exists}
  \Im T_{\lambda+i0}(H) := \lim_{y \to 0^+} \Im T_{\lambda+iy}(H) \ \ \text{exists in the norm topology.}
\end{equation}
In fact, one often requires a stronger form of convergence for the imaginary part $\Im T_{\lambda+i0}(H)$
and this additional condition is imposed when needed.

{\bf L.\,A.\,P. Assumption.} Throughout this paper we assume that the pair $(\hilb,F)$ and the affine space $\clA$ satisfy the limiting absorption principle:
{\it For any self-adjoint operator $H \in \clA$ on the Hilbert space~$\hilb$ with rigging~$F$ the set~(\ref{F: Lambda(H,F)}) has full Lebesgue measure.
}

Though the set $\Lambda(H,F)$ has full Lebesgue measure in certain cases of interest, for the development of the theory of spectral flow
inside essential spectrum it is not quite necessary. As long as the set $\Lambda(H,F)$ contains at least one point, one may study spectral flow through that point.

As it was mentioned in the introduction, L.\,A.\,P. Assumption holds for Schr\"odinger operators with short range potentials.
Another setting in which L.\,A.\,P. Assumption holds is given by the following theorem.

\begin{thm} \label{T: Ya thm 6.1.9} { \rm~\cite{BE,deBranges}\cite[Theorem 6.1.9]{Ya}}
If~$H_0$ is a self-adjoint operator acting on a Hilbert space~$\hilb$
and if~$F$ is a Hilbert-Schmidt operator from $\hilb$ to another Hilbert space $\clK,$ then for a.e. $\lambda \in \mbR$ the operator-valued function
$F R_{\lambda + iy}(H_0)F^*$ has a limit in Hilbert-Schmidt norm as $y \to 0.$
\end{thm}
The limiting absorption principle plays an important role in the stationary approach to scattering theory
(see \cite{BE,deBranges,KK71,Ya}). Proof of the limiting absorption principle in all cases of interest is a difficult problem.
But for this paper it is a postulate and of utmost importance.

\subsection{Resonant at~$\lambda$ and regular at~$\lambda$ operators}
Given a self-adjoint operator~$H$ and a perturbation $V=F^*JF$ one is usually interested in points~$\lambda$
for which the limiting absorption principle~(\ref{F: T(l+i0) exists}) holds. In contrast to this, in this work
we are mainly interested in points~$\lambda$ for which the limiting absorption principle fails.
However, there can be points~$\lambda$ for which~(\ref{F: T(l+i0) exists}) fails for any operator $H \in \clA.$
This is an indication of the fact that~$\lambda$ is a very singular value of the spectral parameter.
We exclude such points from our study; for the present work we are interested in those points~$\lambda$ for which
(\ref{F: T(l+i0) exists}) fails for some but not for all operators from~$\clA.$ We introduce appropriate notation and terminology.

Let \label{Page: Lambda(clA,F)}
\begin{equation} \label{F: Lambda(clA,F)}
  \Lambda(\clA,F) := 
  \bigcup_{H \in \clA} \Lambda(H,F) \subset \mbR.
\end{equation}
Since the sets $\Lambda(H,F),$ for $H \in \clA,$ have full Lebesgue measure, the set $\Lambda(\clA,F)$
also has full Lebesgue measure. A~real number from the set $\Lambda(\clA,F)$ will be called an \emph{essentially regular point} \cite[\S 4.2]{Az3v6}.
\label{Page: ess reg point}
Points which are not essentially regular exist; for example, it will be shown that an eigenvalue of infinite
multiplicity cannot be essentially regular (Theorem~\ref{T: infty mult-ty then not essentially regular}).
But a real number may fail to be essentially regular even if~$\lambda$ is not an eigenvalue. This may happen inside the essential spectrum only,
since outside the essential spectrum all points are essentially regular.
This indicates to the nature of non essentially regular points as those of infinite singularity.

The notation 
$$
  \Pi_+ = \Pi_+(\clA, F), \quad \text{respectively,} \  \ \Pi_-=\Pi_-(\clA,F),
$$
will be used to denote the union of the open upper complex half-plane, respectively, of the open lower complex half-plane,
and the set $\Lambda(\clA, F).$ The letter $\Pi$ will denote the disjoint union of the sets $\Pi_+$ and $\Pi_-.$ \label{Page: Pi}
Thus, the boundary $\partial \Pi$ of $\Pi$ is the disjoint union of two copies $\partial \Pi_+ = \Lambda(\clA,F)$ and $\partial \Pi_- = \Lambda(\clA,F)$
of the same set. The conjugation $z \mapsto \bar z$ swaps $\Pi_+$ and $\Pi_-.$ Elements of the boundary $\partial \Pi_\pm$ are written as $\lambda\pm i0,$
where $\lambda \in \mbR.$ Elements of $\Pi$ will usually be denoted by $z,$ the real part of~$z$ is denoted as a rule by~$\lambda$
and the complex part of~$z$ is denoted by~$y.$ Thus, $y$ is an element of the set $(-\infty,0-] \cup [0+,\infty).$ The real number~$\lambda$
will be fixed throughout most of this paper.

Let~$\lambda$ be an essentially regular point of the pair $(\clA,F)$ and let~$H$ be an operator from~$\clA.$
We say that the operator\label{Page: op-r res at lambda}
\begin{equation} \label{F: op-r resonant at lambda}
  H \ \text{is \emph{resonant at}} \ \lambda \ or \ \text{\emph{not regular at}} \ \lambda, \ \text{if and only if} \ \lambda \notin \Lambda(H,F).
\end{equation}
Thus,~$H$ is resonant at~$\lambda$ if and only if the limit~(\ref{F: T(l+i0) exists}) does not exist.
Otherwise it will be said that
$$
  H \ \text{is \emph{regular at}} \ \lambda \ or \ \text{\emph{non-resonant at}} \ \lambda, \ \text{if and only if} \ \lambda \in \Lambda(H,F).
$$
The set of all resonant at~$\lambda$ operators from the affine space~$\clA$ will be denoted by \label{Page: R(l,A,F)}
\begin{equation} \label{F: R(lambda)}
  R(\lambda; \clA,F).
\end{equation}
The set $R(\lambda; \clA,F)$ will be called the \emph{resonance set} at~$\lambda.$
The following theorem is well-known; what may be new is the way we interpret it.
\begin{thm} \label{T: Az 4.1.11} 
Let~$\lambda$ be an essentially regular point of the pair $(\clA,F),$ let~$H_0 \in \clA$
be an operator regular at~$\lambda$ and let $V = F^*JF \in \clA_0(F).$
The following five assertions are equivalent:
  \begin{enumerate}
    \item [(i)] The operator~$H_0+V$ is resonant at~$\lambda.$
    \item[(ii$_\pm$)] The operator $1 + J T_{\lambda\pm i0}(H_0)$ is not invertible.
    \item[(iii$_\pm$)] The operator $1 + T_{\lambda\pm i0}(H_0)J$ is not invertible.
  \end{enumerate}
\end{thm}
\begin{proof} The equivalence of (i) and (ii$_+$) can easily be derived from the equality
$$
  T_{\lambda+iy}(H_0+V) = \SqBrs{1+T_{\lambda+iy}(H_0)J}^{-1}T_{\lambda+iy}(H_0),
$$
which in its turn follows from the second resolvent identity (see (\ref{F: T=(...)(-1)T}) below).
Equivalence of~(i) to other items is proved similarly.
\end{proof}
\noindent This theorem has the following simple but important corollary.
\begin{thm} \label{T: Az Th. 4.2.5} For every essentially regular point $\lambda \in \mbR,$
the resonance set $R(\lambda; \clA, F)$ is a closed nowhere dense subset of~$\clA.$ Moreover,
the intersection of any real-analytic path in $\clA$ with the resonance set $R(\lambda; \clA,F)$ is either
a discrete set or coincides with the path itself.
\end{thm}
Proof of this theorem follows verbatim that of \cite[Theorem 4.2.5]{Az3v6}.

\noindent
The left figure below shows how a more or less typical two-dimensional section of the resonance set $R(\lambda; \clA,F),$
which has two resonance lines and two resonance points, may look.

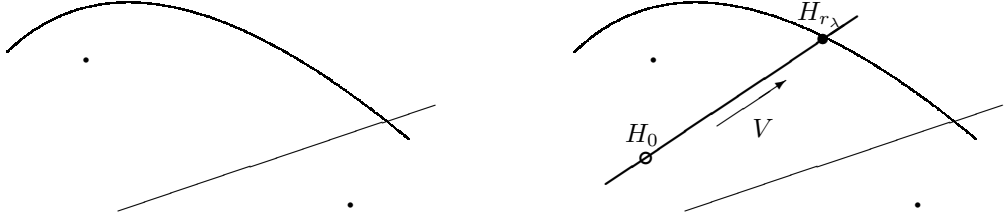
\begin{picture}(200,90)
\put(138,12){\circle*{2}}
\put(38,67){\circle*{2}}
\thinlines
\put(50,10){\line(3,1){120}}
\qbezier(8,70)(60,120)(160,37)
\end{picture}
\quad
\begin{picture}(200,90)
\put(138,12){\circle*{2}}
\put(38,67){\circle*{2}}
\thinlines
\put(50,10){\line(3,1){120}}
\qbezier(8,70)(60,120)(160,37)

\thicklines
\put(35,30){\circle{4}}
\put(24,36){{\small~$H_0$}}
\put(20,20){\line(3,2){95}}
\thinlines
\put(62,42){\vector(3,2){26}}
\put(72,38){{\small~$V$}}

\put(102,75){\circle*{4}}
\put(89,82){{\small~$H_{r_\lambda}$}}
\end{picture}

\noindent
Let $\gamma = \set{H_r = H_0+rV \colon r \in \mbR}$ be a straight line or a path of operators in the affine space~$\clA.$
If $\lambda \in \mbR$ is an essentially regular point of the pair $(\clA,F),$ then
according to Theorem~\ref{T: Az Th. 4.2.5} there are two possible scenarios:
all points of $\gamma$ except a discrete subset are regular at~$\lambda$ or
all points of $\gamma$ are resonant at~$\lambda.$
In the first case we say that $\gamma$ is \emph{regular at}~$\lambda.$
\label{Page: ess regular line}
A real number~$r$ will be said to be a \emph{resonance point} \label{Page: res point} of the line~$\gamma = \set{H_s \colon s \in \mbR}$ at~$\lambda,$
if~$H_r$ is resonant at~$\lambda.$ 
A regular line~$\gamma$ may only have a discrete set of resonance points. We shall mainly be concerned with only one of them which will be denoted by~$r_\lambda.$
\label{Page: r(lambda)}
The right figure above shows a regular at~$\lambda$ operator~$H_0 \in \clA$ and a~direction $V \in \clA_0(F);$
the line~$\gamma$ intersects the resonance set $R(\lambda; \clA, F)$ at point~$H_{r_\lambda}.$
If an operator $H \in \clA$ is resonant at~$\lambda$ then a perturbation $V \in \clA_0(F)$ will be called a
\emph{regularizing direction} for~$H$ at~$\lambda$ if the straight line $\gamma$ which passes through~$H$ in direction of~$V$ is regular at~$\lambda.$
In the picture the operator~$H_{r_\lambda}$ is resonant at~$\lambda$ and~$V$ is a regularizing direction for~$H_{r_\lambda}$ at $\lambda;$
in fact, in the case of the figure every direction, which is parallel to the two-dimensional section of the affine space $\clA$ shown in the figure, is regularizing
for~$H_{r_\lambda}$ at~$\lambda.$

\begin{prop} \label{P: Az 4.1.10}
If $\lambda \in \Lambda(\clA,F)$ is an eigenvalue of an operator $H \in \clA,$ then~$H$ is resonant at~$\lambda.$
\end{prop}
\noindent
Proof of this proposition is the same as that of \cite[Proposition 4.1.10]{Az3v6}.
This proposition shows one source of resonance points, but a point~$r$ can be resonant even if~$\lambda$ is not an eigenvalue of~$H_r.$

Proposition \ref{P: Az 4.1.10} implies the following
\begin{cor} \label{C: Az 4.1.10} If $\lambda\in \Lambda(\clA,F)$ does not belong to the essential spectrum $\sigma_{ess}$ of~$\clA,$
then an operator~$H$ from~ $\clA$ is resonant at $\lambda$ if and only if $\lambda$ is an eigenvalue of~$H.$
\end{cor}
\begin{proof} The (if) implication follows from Proposition \ref{P: Az 4.1.10}.
We prove (only if) part.
Assume the contrary: $\lambda$ is not an eigenvalue of~$H.$ Since also $\lambda \notin \sigma_{ess}(H),$ it follows that $\lambda$ belongs to the resolvent set of~$H.$
In this case the norm limit $R_{\lambda+i0}(H)$ of the resolvent exists even without sandwiching by $F$ and $F^*,$ and therefore $H$ is regular at~$\lambda.$
\end{proof}

\subsection{Operators $A_z(s)$ and $B_z(s)$}
Let $z \in \Pi.$ We shall frequently use notation \label{Page: Az(s)}
\begin{equation} \label{F: Az(s)}
  A_z(s) = T_z(H_{s})J.
\end{equation}
Sandwiched version of the second resolvent identity
\begin{equation} \label{F: T=(...)(-1)T}
  T_z(H_r) - T_z(H_s) = (s-r)T_z(H_r)JT_z(H_s)
\end{equation}
implies the equality
\begin{equation} \label{F: II resolvent identity}
  A_z(r) - A_z(s) = (s-r)A_z(r)A_z(s),
\end{equation}
From this equality one can infer that the operator $1+(s-r)A_z(r)$
must be invertible. Hence, the operator~$A_z(s)$ satisfies the equality
\begin{equation} \label{F: A(s)=(1+(s-r)A(r))(-1)A(r)}
  A_z(s) = (1+(s-r)A_z(r))^{-1}A_z(r),
\end{equation}
which also implies that
\begin{equation} \label{F: A(s) and A(r) commute}
  A_z(s) A_z(r) = A_z(r) A_z(s).
\end{equation}
Since the operator $A_z(r)$ is compact, by the analytic Fredholm alternative (Theorem~\ref{T: Analytic Fredholm alternative}),
the equality~(\ref{F: A(s)=(1+(s-r)A(r))(-1)A(r)}) gives meromorphic continuation of the function~$A_z(s)$ of~$s$
to the whole complex plane $\mbC.$ The equality~(\ref{F: A(s) and A(r) commute}) holds also for this meromorphic continuation.
Moreover, Theorem~\ref{T: Analytic Fredholm alternative} and~(\ref{F: A(s)=(1+(s-r)A(r))(-1)A(r)}) imply the following
\begin{lemma} \label{L: A(z,s) is meromorphic}
The function $(z,s) \mapsto A_z(s)$ is a meromorphic function of two complex variables~$z$ and~$s$
in the domain $\Pi^\circ \times \mbC$ of~$\mbC^2.$
\end{lemma}

The equality~(\ref{F: II resolvent identity}) implies that
\begin{equation} \label{F: Dn Az(s)=...}
  \frac {d^n}{ds^n} A_z(s) = (-1)^n n! A_z^{n+1}(s).
\end{equation}
We also use notation \label{Page: Bz(s)}
\begin{equation} \label{F: Bz(s)}
  B_z(s) = JT_z(H_{s}).
\end{equation}
One can check that the following analogue of the identity~(\ref{F: A(s)=(1+(s-r)A(r))(-1)A(r)}) holds:
\begin{equation} \label{F: B(s)=(1+(s-r)B(r))(-1)B(r)}
  B_z(s) = (1+(s-r)B_z(r))^{-1}B_z(r),
\end{equation}
which implies the equality
\begin{equation} \label{F: B(s) and B(r) commute}
  B_z(s) B_z(r) = B_z(r) B_z(s).
\end{equation}
Using the equalities~(\ref{F: A(s)=(1+(s-r)A(r))(-1)A(r)}) and~(\ref{F: B(s)=(1+(s-r)B(r))(-1)B(r)}) one can check that
\begin{equation} \label{F: A(z)(s)*=B(bz)(bs)}
  \brs{A_z(s)}^* = B_{\bar z}(\bar s).
\end{equation}

We shall also use the following well-known equality (see e.g. \cite[p.\,144]{KK71} \cite[(99)]{RS3} \cite[(4.8)]{Az3v6})
\begin{equation} \label{F: Az3v6 (4.8)}
 \begin{split}
  \Im T_{z}(H_s) & = (1+(s-r)T_{\bar z}(H_r)J)^{-1} \Im T_{z}(H_r) (1+(s-r)JT_{z}(H_r))^{-1}
   \\ & = (1+(s-r)A_{\bar z}(r))^{-1} \Im T_{z}(H_r) (1+(s-r)B_{z}(r))^{-1}.
 \end{split}
\end{equation}
This equality holds for all real numbers $s$ and $r,$ if~$z$ does not belong to $\partial \Pi;$ otherwise, if $z = \lambda\pm i0$ and if the
line $\set{H_r = H_0+rV \colon r \in \mbR}$ is regular at~$\lambda,$ then this equality holds for all real numbers $s$ and~$r$
as long as they do not belong to the resonance set $R(\lambda; H_0,V).$ In particular, the right hand side of this
equality provides meromorphic continuation of the left hand side as function of $s$ to the whole complex plane.

\begin{lemma} \label{L: T(l+iy) to T(l+i0)} As $y \to 0$ the holomorphic function $T_{\lambda+iy}(H_s)$ of $s$ converges to
$T_{\lambda+i0}(H_s)$ uniformly on any compact subset of $\Pi$ which does not contain resonance points corresponding to $\lambda+i0.$
\end{lemma}

In what follows, the spectra of the operators~$A_z(s)$ and $B_z(s)$ will play an important role.
Since~$A_z(s)$ and $B_z(s)$ are compact operators, their spectra consists of isolated eigenvalues of finite multiplicity
and zero. By~(\ref{F: mu(AB)=mu(BA)}), the spectral measures of~$A_z(s)$ and $B_z(s)$ coincide,
and therefore it suffices to consider the spectrum of~$A_z(s).$
Eigenvalues of~$A_z(s)$ will be denoted by $\sigma_z = \sigma_z(s).$
As it will be seen later (Proposition~\ref{P: res eq-n is correct}), eigenvalues (with their multiplicities) of~$A_z(s)$ for different $s$
are connected by a simple relation:
$\sigma_z(s) = (s - r_z)^{-1},$ where~$r_z$ is a complex number independent of $s.$

Occasionally we also consider operators \label{Page: ulAz(s)}
\begin{equation} \label{F: ulA and ulB}
  \ulA_z(s) = R_z(H_s)V \quad \text{and} \quad \ulB_z(s) = VR_z(H_s).
\end{equation}
Spectral properties of these operators are identical to those of $A_z(s)$ and $B_z(s).$
By the limiting absorption principle the operators $A_z(s)$ and $B_z(s)$ have well-defined limits $A_{\lambda\pm i0}(s))$ and $B_{\lambda\pm i0}(s)$
as $z=\lambda+iy$ approaches $\lambda\pm i0$ unlike the operators $\ulA_z(s)$
and $\ulB_z(s);$ since eventually the limit $z=\lambda+iy \to \lambda\pm i0$ will be taken,
this is the main reason to work with the former pair of operators rather than the latter. But as long as $z$ stays outside the real axis
or outside the common essential spectrum of operators $H_s,$ practically all other properties of these two pairs of operators are almost identical
and as a consequence they will be stated only for $A_z(s)$ and $B_z(s).$ Nearly all objects, such as to be introduced later $P_z(r_z),\bfA_z(r_z),$ etc,
which are naturally associated with operators
$A_z(s)$ and $B_z(s),$ have their analogues for $\ulA_z(s)$ and $\ulB_z(s);$ these analogues will be distinguished by underlining, e.g.
$\ulP_z(r_z),\ubfA_z(r_z),$ etc.

The following lemma is well-known.
\begin{lemma} \label{L: Az Lemma 4.1.4}
If $z$ is a non-real number, then compact operators~$A_z(s)$ and~$B_z(s)$ do not have real eigenvalues except possibly zero.
Moreover, if the operator $V$ is non-negative (respectively, non-positive), then all eigenvalues of operators $R_z(H_s)V,$ $VR_z(H_s),$~$A_z(s)$
and~$B_z(s)$ belong to that open complex half-plane $\Pi_\pm$ which~$z$ belongs to
(respectively, does not belong to).
\end{lemma}
\noindent
However, if~$z$ belongs to $\partial \Pi,$ then the operators~$A_z(s)$ and $B_z(s)$ may have non-zero real eigenvalues.
In fact, it is these real eigenvalues of~$A_z(s)$ which are of the most interest for the present and with a bit of
exaggeration it can be said that this paper is mainly devoted to investigation of these real eigenvalues.

%

\section{Analytic properties of $A_z(s)$}
\label{S: Resonance points}
\subsection{Vector spaces~$\Upsilon_z(r_z),$ \,$\Psi_z(r_z)$}
Throughout this paper we assume that $H_0$ is a self-adjoint operator from the affine space~(\ref{F: clA})
and that $V$ is a self-adjoint operator from the real vector space $\clA_0(F)$ with factorization~(\ref{F: V}).
Let~$\lambda$ be a fixed real number. We assume that the line
$$
  \gamma := \set{H_0+rV \colon r \in \mbR}
$$
is regular at $\lambda;$ by definition this means that there exists a non-resonant value of the coupling constant~$r,$
that is, for some value of~$r$ the inclusion $H_r \in R(\lambda; \clA,F)$ fails (equivalently, the inclusion $\lambda \in \Lambda(H_r,F)$ holds).
In this case the set $R(\lambda; H_0,V)$ of resonance points~$r_\lambda$ is a discrete subset of $\mbR,$ by Theorem~\ref{T: Az 4.1.11}.

Let~$z$ be a point of $\Pi,$ let~$r_z$ be a complex number and let~$k$ be a positive integer.
Let $s$ be any number for which the operator $A_z(s)$ is defined. The equation
\begin{equation} \label{F: res eq-n}
  \brs{1+(r_z-s) A_z(s)}^ku = 0
\end{equation}
will be called the \emph{resonance equation of order~$k$} for the pair $(z,r_z).$ The resonance equation of order~1 is nothing else but the Lippmann-Schwinger equation.

\begin{defn} \label{D: res point rz}
A complex number~$r_z$ will be said to be a \emph{resonance point}\label{Page: rz} corresponding to $z \in \Pi,$
if the resonance equation~(\ref{F: res eq-n}) of order $k=1$ has a non-zero solution.
\end{defn}
\noindent
In other words,~$r_z$ is a resonance point if and only if the number \label{Page: sigma z(s)}
\begin{equation} \label{F: sigma z(s)=(s-rz)(-1)}
  \sigma_z(s) := (s-r_z)^{-1}
\end{equation}
is a non-zero eigenvalue of the compact operator~$A_z(s).$ Real resonance points~$r_\lambda$ were defined earlier
in the paragraph following Theorem~\ref{T: Az Th. 4.2.5} and these definitions are consistent with each other.
It will be shown below (Proposition~\ref{P: res eq-n is correct}) that definition of the resonance point does not depend on $s.$
Hence,~$r_z$ depends only on $z,$~$H_0,$~$V$ and, in case $z \in \partial \Pi,$ also on $F.$
If~$z$ lies outside of the boundary $\partial \Pi,$ then this definition does not depend on the rigging operator~$F,$
since in this case both operators $A_z(s) = F R_z(H_s) F^*J$ and $R_z(H_s) F^*J F = R_z(H_s) V$ make sense and they have the same non-zero eigenvalues
by~(\ref{F: spec(AB)=spec(BA)}).

According to the correspondence~(\ref{F: sigma z(s)=(s-rz)(-1)}) between resonance points $r_z$ and eigenvalues $\sigma_z(s)$ of a compact operator $A_z(s),$
the set of resonance points corresponding to a given $z \in \Pi$ is a discrete subset of $\mbC.$
Also, the formula~(\ref{F: A(s)=(1+(s-r)A(r))(-1)A(r)}) shows that resonance points corresponding to~$z$ are exactly the poles of the meromorphic function $A_z(s).$
For this reason, resonance points may sometimes be called poles.

Solutions of the resonance equation~(\ref{F: res eq-n}) of order~$k$ will usually be denoted by $u,$ $u_z$ or
$u_z(r_z)$ \label{Page: uu}
and will be called \emph{resonance vectors of order} $\leq k.$ \label{Page: res vector}
Order~$k$ of a resonance vector~$u$ is the smallest positive integer such that~$u$ is a solution of
the resonance equation~(\ref{F: res eq-n}) of order~$k.$ Order of a resonance vector will be denoted by $\order(u).$ \label{Page: order of vector}
If necessary we write $\order_z(u)$ instead of $\order(u);$ also, instead of $\order_{\lambda\pm i0}(u)$ we often write $\order_{\pm}(u).$

The finite-dimensional vector space of all resonance vectors of order~$\leq k$ will be denoted by~$\Upsilon^k_z(r_z).$ \label{Page: Upsilon(j)}
To be precise one should indicate dependence of this vector on operators~$H_0,V$ by writing, say, $\Upsilon^k_z(r_z; H_0,V),$
but since throughout this paper the operators~$H_0$ and~$V$ are fixed, the simpler notation will be used.
The same remark applies to many other objects to be introduced later.
A vector $u = u_z(r_z)$ will be said to be a resonance vector of order $k,$ if $u$ is a resonance vector of order $\leq k$
but not a vector of order $\leq k-1.$
It was proved in \cite{Az7} that the set of solutions of the equation~(\ref{F: res eq-n}) does not depend on $s.$
We give here the proof for completeness and for readers' convenience.
\begin{prop} \label{P: res eq-n is correct}
  Let $z \in \Pi$ and let~$r_z$ be a resonance point corresponding to~$z.$
  The vector space~$\Upsilon_z^k(r_z)$ of solutions of the equation~(\ref{F: res eq-n})
  does not depend on $s \in \mbR.$
\end{prop}
\begin{proof}
We prove this theorem using induction on $k.$ Let $u$ be a solution of~(\ref{F: res eq-n}) with $k=1$
for the value of $s = r,$ so that  $A_z(r) u = \brs{r - r_z}^{-1} u.$
It follows from this and~(\ref{F: A(s)=(1+(s-r)A(r))(-1)A(r)}) that
\begin{equation*}
  \begin{split}
    A_z(s) u & = (1+(s-r)A_z(r))^{-1} A_z(r) u
    \\ & = \SqBrs{1+(s-r) \cdot \frac 1{r - r_z}} ^{-1} \frac 1{r - r_z} u = \frac 1{s - r_z}u.
  \end{split}
\end{equation*}
Hence, if $u$ is a solution of~(\ref{F: res eq-n}) with $k=1$ for one value of $s,$ then $u$ is a solution of~(\ref{F: res eq-n}) with $k=1$
for any other regular value of $s$ too.
Now assume that the assertion is true for $k=n$ and let $u$ be a solution of~(\ref{F: res eq-n}) with $k=n+1$ for the value of $s = r.$
Then
$$
  (1+(r_z-r)A_z(r))(1+(r_z-r)A_z(r))^nu = 0.
$$
It follows from this and induction base, applied to the vector $(1+(r_z-r)A_z(r))^nu,$ that
$$
  (1+(r_z-s)A_z(s))(1+(r_z-r)A_z(r))^nu_{z} = 0.
$$
Since, by~(\ref{F: A(s) and A(r) commute}), the operators~$A_z(s)$ and $A_z(r)$ commute, it follows that
$$
  (1+(r_z-r)A_z(r))^n (1+(r_z-s)A_z(s)) u = 0.
$$
By the induction assumption, applied to the vector $(1+(r_z-s)A_z(s))u,$ this implies that
$$
  (1+(r_z-s)A_z(s))^n (1+(r_z-s)A_z(s)) u = 0.
$$
\end{proof}
The sequence
$$
  \Upsilon^1_z(r_z) \subset \Upsilon^2_z(r_z) \subset \ldots \subset \Upsilon^k_z(r_z) \subset \ldots \subset \clK,
$$
stabilizes. The union of the vector spaces~$\Upsilon_z^1(r_z), \Upsilon_z^2(r_z), \ldots $ will be denoted by~$\Upsilon_z(r_z).$ \label{Page: Upsilon}

\label{Page: point of order d}
A resonance point~$r_z$ will be said to have \emph{order~$d$},
if there are resonance vectors of order~$d,$ but there are no resonance vectors of order $d+1.$
In other words, the order~$d$ of a resonance point~$r_z$ is the integer
\begin{equation} \label{F: d}
  d = \min \set{k \in \mbN \colon \Upsilon_z^k(r_z) = \Upsilon_z^{k+1}(r_z)} = \min \set{k \in \mbN \colon \Upsilon_z^k(r_z) = \Upsilon_z(r_z)}.
\end{equation}
Apart from order~$d,$ with every resonance point~$r_z$ another two positive integers are naturally associated: geometric multiplicity~$m$
defined by equality \label{Page: m}
\begin{equation} \label{F: m}
  m = \dim \Upsilon_z^1(r_z)
\end{equation}
and algebraic multiplicity~$N$
defined by equality \label{Page: N}
\begin{equation} \label{F: N}
  N = \dim \Upsilon_z(r_z).
\end{equation}
Obviously, $d+m-1\leq N.$
Throughout this paper the letters~$d,$~$m$ and~$N$ will be used only with these meanings, unless it is specifically stated otherwise.

The equation
\begin{equation} \label{F: cores eq-n}
  \brs{1+(r_z-s) B_z(s)}^k\psi = 0
\end{equation}
will be called \emph{co-resonance equation of order $k.$} Solutions of the co-resonance equation of order~$k$
will be denoted by $\psi$ or $\psi_z$ or $\psi_z(r_z)$ \label{Page: psi} and will be called \emph{co-resonance vectors} \label{Page: cores vector}
of order $\leq k.$ The finite-dimensional vector space of all
co-resonance vectors of order~$\leq k$ will be denoted by~$\Psi_z^k(r_z).$ \label{Page: Psi(j)}
A co-resonance vector $\psi$ has order~$k$ if it has order $\leq k$ but not $\leq k-1.$
The sequence
$$
  \Psi^1_z(r_z) \subset \Psi^2_z(r_z) \subset \ldots \subset \Psi_z^k(r_z) \subset \ldots \subset \clK,
$$
stabilizes; its union will be denoted by~$\Psi_z(r_z).$ \label{Page: Psi}
Similarly to Proposition~\ref{P: res eq-n is correct},
one can prove the following
\begin{prop} \label{P: cores eq-n is correct}
  Let $z \in \Pi$ and let~$r_z$ be a resonance point corresponding to~$z.$
  The vector space~$\Psi_z^k(r_z)$ of solutions of the equation~(\ref{F: cores eq-n})
  does not depend on $s \in \mbR.$
\end{prop}
\noindent
This proposition also follows from Proposition~\ref{P: res eq-n is correct} and Lemma~\ref{L: j-dimensions coincide}.

\begin{lemma} \label{L: j-dimensions coincide}
Let $z \in \Pi$ and let~$r_z$ be a resonance point corresponding to~$z.$
Dimensions of four vector spaces~$\Upsilon^j_{z}(r_z),$~$\Upsilon^j_{\bar z}(\bar r_z),$~$\Psi^j_z(r_z)$
and~$\Psi^j_{\bar z}(\bar r_z)$ coincide for all $j=1,2,\ldots$ Moreover,
for all $j=1,2,\ldots$ and all non-resonant real numbers $s$ the mappings $J \colon \Upsilon^j_z(r_z) \to \Psi^j_z(r_z)$ and
$T_z(H_s) \colon \Psi^j_z(r_z) \to \Upsilon^j_z(r_z)$ are linear isomorphisms.

In particular, dimensions of four vector spaces~$\Upsilon_{z}(r_z),$~$\Upsilon_{\bar z}(\bar r_z),$~$\Psi_z(r_z)$
and~$\Psi_{\bar z}(\bar r_z)$ coincide and $J$ is a linear isomorphism of the vector spaces~$\Upsilon_z(r_z)$ and~$\Psi_z(r_z).$
\end{lemma}
\begin{proof}
Let $j$ be a positive integer and $s$ a real number. The resonance equation~(\ref{F: res eq-n}) implies that if $u \in \Upsilon^j_z(r_z),$ then
$Ju \in \Psi^j_z(r_z).$ Also, if $Ju = 0,$ where~$u$ is a solution of the resonance equation~(\ref{F: res eq-n}), then it follows from
this equation, after expanding brackets, that $u = 0.$ Hence, $J$ is an injective linear operator from $\Upsilon^j_z(r_z)$ into $\Psi^j_z(r_z).$

Similarly, the co-resonance equation~(\ref{F: cores eq-n}) implies that if $\psi \in \Psi^j_z(r_z),$ then
$T_z(H_s)\psi \in \Upsilon^j_z(r_z);$ further,
if $T_z(H_s)\psi = 0,$ where~$\psi$ is a solution of the co-resonance equation~(\ref{F: cores eq-n}), then it follows from
this equation that $\psi = 0.$ Hence, $T_z(H_s)$ is an injective linear operator from $\Psi^j_z(r_z)$ into $\Upsilon^j_z(r_z).$

Thus, the vector spaces $\Upsilon^j_z(r_z)$ and $\Psi^j_z(r_z)$ are linearly isomorphic and the mappings $J \colon \Upsilon^j_z(r_z) \to \Psi^j_z(r_z)$ and
$T_z(H_s) \colon \Psi^j_z(r_z) \to \Upsilon^j_z(r_z)$ are linear isomorphisms.

Further, let $S = [1+(s-r_z)A_z(s)]^j;$ then
\begin{equation*} 
   \dim \Upsilon_z^j(r_z) = \dim \ker S = \dim \ker S^* = \dim \Psi_{\bar z}^j(\bar r_z),
\end{equation*}
where the first and the third equalities directly follow from definitions of the vector spaces $\Upsilon_z^j(r_z)$
and $\Psi_{\bar z}^j(\bar r_z)$ and the second equality follows from the fact that the Fredholm index of $S$ is zero,
since the operator $S-1$ is compact.
It follows that dimensions of the four vector spaces ~$\Upsilon^j_{z}(r_z),$~$\Upsilon^j_{\bar z}(\bar r_z),$~$\Psi^j_z(r_z)$
and~$\Psi^j_{\bar z}(\bar r_z)$ are the same.
\end{proof}

\begin{cor} \label{C: r(bar z)=bar r(z)} If~$r_z$ is a resonance point of algebraic multiplicity~$N,$ order~$d$ and geometric multiplicity $m,$
corresponding to $z,$ then $\bar r_z$ is a resonance point of algebraic multiplicity~$N,$ order~$d$ and geometric multiplicity $m,$
corresponding to $\bar z.$
\end{cor}

\begin{cor} \label{C: Upsilon is invariant}
The vector spaces~$\Upsilon_z(r_z)$ and~$\Upsilon_z^{k}(r_z),$ $k=1,2,\ldots,$ are invariant under the operator~$A_z(s) = T_z(H_s)J$
for any non-resonant $s \in \mbR.$
\end{cor}

\subsection{Idempotents $P_z(r_z)$ and $Q_z(r_z)$}

For a given element~$z$ of $\Pi$ with a corresponding resonance point~$r_z \in \mbC$
an idempotent operator $P_z(r_z),$ which acts on the Hilbert space~$\clK$ and
has the range~$\Upsilon_z(r_z),$ will be defined by equality \label{Page: Pz(rz)}
\begin{equation} \label{F: Pz(rz)=oint (sigma-Az)(-1)d sigma}
  P_z(r_z) = \frac 1{2\pi i} \oint _{C(\sigma_z(s))} \brs{\sigma - A_z(s)}^{-1}\,d\sigma
\end{equation}
where $C(\sigma_z(s))$ is a small circle enclosing the eigenvalue~(\ref{F: sigma z(s)=(s-rz)(-1)})
of the operator~$A_z(s),$ so that there
are no other eigenvalues of this operator on or inside the circle.
The contour integral in~(\ref{F: Pz(rz)=oint (sigma-Az)(-1)d
sigma}) and in all the following contour integrals are taken in
the uniform operator topology.

Apart of the operator $P_z(r_z)$ we shall sometimes need its modification \label{Page: ulPz(rz)}
\begin{equation} \label{F: ulPz(rz)=oint (sigma-Az)(-1)d sigma}
  \ulP_z(r_z) = \frac 1{2\pi i} \oint _{C(\sigma_z(s))} \brs{\sigma - \ulA_z(s)}^{-1}\,d\sigma,
\end{equation}
where $\ulA_z(s) = R_z(H_s)V.$ As long as the variable $z$ is non-real,
properties of $P_z(r_z)$ and $\ulP_z(r_z)$ are quite similar; for this reason they are given only for the operator~$P_z(r_z).$
An essential difference between $P_z(r_z)$ and $\ulP_z(r_z)$ is that the former operator $P_z(r_z)$ has the limit $P_{\lambda\pm i0}(r_z)$
as $z$ approaches its real part~$\lambda$ from above or below, while the latter operator may not have such a limit.
In fact, this is the main reason for considering $P_z(r_z)$ instead of $\ulP_z(r_z).$ The same remark applies to other ``underlined'' versions
of operators to be introduced later.

The following assertion was proved
in \cite{Az7}; its proof is given below for completeness.
\begin{prop} \label{P: Pz is well-defined} The idempotent operator $P_z(r_z),$ defined by equality~(\ref{F: Pz(rz)=oint (sigma-Az)(-1)d sigma}), does not depend on~$s.$
\end{prop}
\begin{proof} Let $P_1$ and $P_2$ be two idempotents $P_z(r_z)$ defined for two different values $s_1$ and $s_2$ of~$s.$
Since by Proposition~\ref{P: res eq-n is correct} these idempotents have the same range $\Upsilon_z(r_z),$ we have
$P_1P_2 = P_2$ and $P_2P_1 = P_1.$ Since, by~(\ref{F: A(s) and A(r) commute}), operators $A_{z}(s_1)$ and
$A_{z}(s_2)$ commute, it follows from~(\ref{F: Pz(rz)=oint (sigma-Az)(-1)d sigma}) that $P_1$ and $P_2$ also commute.
It follows that $P_1 = P_2P_1 = P_1P_2 = P_2.$
\end{proof}
\noindent
Another proof of this proposition follows from Proposition~\ref{P: Pz(rz)=res Az(s)}.

We also define an idempotent operator $Q_z(r_z),$ which acts on the Hilbert space~$\clK$ and has
the range~$\Psi_z(r_z),$ by equality \label{Page: Qz(rz)}
\begin{equation} \label{F: Qz(rz)=oint (sigma-Bz)(-1)d sigma}
  Q_z(r_z) = \frac 1{2\pi i} \oint _{C(\sigma_z(s))} \brs{\sigma - B_{z} (s)}^{-1}\,d\sigma,
\end{equation}
where the contour $C(\sigma_{z}(s))$ is the same as in~(\ref{F: Pz(rz)=oint (sigma-Az)(-1)d sigma}).
The ``underlined'' version of $Q_z(r_z)$ is defined by formula \label{Page: ulQz(rz)}
\begin{equation} \label{F: ulQz(rz)=oint (sigma-Bz)(-1)d sigma}
  \ulQ_z(r_z) = \frac 1{2\pi i} \oint _{C(\sigma_z(s))} \brs{\sigma - \ulB_{z} (s)}^{-1}\,d\sigma.
\end{equation}

Proof of the following proposition is similar to that of Proposition~\ref{P: Pz is well-defined}.
\begin{prop} The idempotent operator $Q_z(r_z),$ defined by equality~(\ref{F: Qz(rz)=oint (sigma-Bz)(-1)d sigma}), does not depend on~$s.$
\end{prop}
The following equality follows from definitions~(\ref{F: Pz(rz)=oint (sigma-Az)(-1)d sigma})
and~(\ref{F: Qz(rz)=oint (sigma-Bz)(-1)d sigma}) of idempotents $P_z(r_z)$ and $Q_z(r_z),$
norm continuity of taking adjoint $T \mapsto T^*,$ and~(\ref{F: A(z)(s)*=B(bz)(bs)}).
\begin{equation} \label{F: Pz*=Q(bar z)}
  \brs{P_z(r_z)}^* = Q_{\bar z}(\bar r_z).
\end{equation}

\begin{prop} \label{P: Pz(rz)=res Az(s)} Let $z \in \Pi$ and let $r_z$ be a resonance point corresponding to~$z.$
The idempotent $P_z(r_z)$ is equal to the residue of the function~$A_z(s)$ of $s$ corresponding to the pole~$r_z:$
\begin{equation} \label{F: Pz(rz)=res Az(s)}
  P_z(r_z) = \frac 1{2\pi i} \oint_{C(r_z)} A_z(s)\,ds,
\end{equation}
where $C(r_z)$ is a small circle enclosing~$r_z$ in counter-clockwise direction.
\end{prop}
\begin{proof} Let~$r$ be a complex number which lies outside of the circle $C(r_z).$
The equality~(\ref{F: A(s)=(1+(s-r)A(r))(-1)A(r)}) implies
\begin{equation*}
  \begin{split}
     \oint_{C(r_z)} A_z(s)\,ds & = \oint_{C(r_z)} \brs{1+(s-r)A_z(r)}^{-1}A_z(r)\,ds
     \\ & = \oint_{C(r_z)} \frac 1{s-r}\brs{1-\brs{1+(s-r)A_z(r)}^{-1}}\,ds.
  \end{split}
\end{equation*}
Since~$r$ lies outside of the circle $C(r_z),$ the integral of $\frac 1{s-r}$ vanishes. Hence,
\begin{equation*}
     \oint_{C(r_z)} A_z(s)\,ds = \oint_{C(r_z)} \frac {1}{r-s}\brs{1+(s-r)A_z(r)}^{-1}\,ds.
\end{equation*}
We make the change of variables $\sigma = \frac {1}{r-s}.$ When~$s$ goes around~$r_z$ in counter-clockwise direction,
so does the variable~$\sigma$ around $\sigma_z(r):=\frac {1}{r-r_z}.$ Hence, from the last equality we obtain
\begin{equation*}
  \begin{split}
     \oint_{C(r_z)} A_z(s)\,ds & = \oint_{C(\sigma_z(r))} \sigma\brs{1-\sigma^{-1}A_z(r)}^{-1} \sigma^{-2}\,d\sigma
     \\ & = \oint_{C(\sigma_z(r))} \brs{\sigma-A_z(r)}^{-1} \,d\sigma
     \\ & = 2\pi i P_z(r_z),
  \end{split}
\end{equation*}
where $C(\sigma_z(r))$ is the image of the contour $C(r_z)$ under the mapping $s \mapsto \sigma = \frac {1}{r-s}.$
\end{proof}

One similarly proves the next proposition.
\begin{prop} \label{P: Q(+)=res B(+)(s)}
Let $z \in \Pi$ and let $r_z$ be a resonance point corresponding to~$z.$
The idempotent $Q_z(r_z)$ is equal to the residue
of the function~$B_z(s)$ of $s$ corresponding to the pole~$r_z:$
\begin{equation} \label{F: Qz(rz)=res Bz(s)}
  Q_z(r_z) = \frac 1{2\pi i} \oint_{C(r_z)} B_z(s)\,ds,
\end{equation}
where $C(r_z)$ is a small circle enclosing~$r_z$ in counter-clockwise direction.
\end{prop}

The following proposition directly follows from definition~(\ref{F: Pz(rz)=oint (sigma-Az)(-1)d sigma}),
but nonetheless we give another proof of it.
\begin{prop}
If for a given $z \in \Pi$ the operator~$A_z(s)$ has two different poles~$r_z^1$ and~$r_z^2,$ then the corresponding
idempotents $P_z(r_z^1)$ and $P_z(r_z^2)$ satisfy the equality
\begin{equation} \label{F: Pz(1)Pz(2)=0}
  P_z(r_z^1) P_z(r_z^2) = 0.
\end{equation}
\end{prop}
\begin{proof}
Proposition~\ref{P: Pz(rz)=res Az(s)} and~(\ref{F: II resolvent identity}) imply that
\begin{equation}
  \begin{split}
    P_z(r_z^1) P_z(r_z^2) & = \frac 1{(2\pi i)^2} \oint_{C_t(r_z^1)} \oint_{C_s(r_z^2)} A_z(t)A_z(s)\,dt\,ds
    \\ & = \frac 1{(2\pi i)^2} \oint_{C_t(r_z^1)} \oint_{C_s(r_z^2)} \frac{A_z(t)-A_z(s)}{s-t}\,dt\,ds,
  \end{split}
\end{equation}
where the contours $C_t(r_z^1)$ and $C_s(r_z^2),$ enclosing (only) the points~$r_z^1$ and~$r_z^2$ respectively,
can be chosen so that they do not intersect and (therefore) do not enclose one another. Under this choice of the contours,
the function $\frac{A_z(t)}{s-t}$ of $s$ is holomorphic on and inside the contour $C_s(r_z^2),$ and therefore its integral vanishes.
For an analogous reason, the integral of $\frac{A_z(s)}{s-t}$ vanishes too.
\end{proof}
Similarly, one shows that
$$
  Q_z(r_z^1) Q_z(r_z^2) = 0.
$$
Now we note some relations between operators $P_z(r_z),$ $Q_z(r_z),$ $T_z(H_s),$ and $J$ which will be used later.
Let $z \in \Pi$ and let $r_z$ be a resonance point corresponding to $z.$
The equality~(\ref{F: (s-AB)(-1)A=...}) combined with definitions~(\ref{F: Pz(rz)=oint (sigma-Az)(-1)d sigma})
and~(\ref{F: Qz(rz)=oint (sigma-Bz)(-1)d sigma}) of idempotents $P_z(r_z)$ and $Q_z(r_z)$ imply equalities
\begin{equation}  \label{F: JP=QJ}
  JP_z(r_z) = Q_{z}(r_z)J,
\end{equation}
\begin{equation}  \label{F: PT=TQ}
  \ \ P_z(r_z)T_z(H_s) = T_z(H_s)Q_{z}(r_z).
\end{equation}
The following equalities follow from Lemma~\ref{L: j-dimensions coincide} and~(\ref{F: JP=QJ}):
\begin{equation} \label{F: JP=QJP}
  JP_z(r_z) = Q_z(r_z)JP_z(r_z) = Q_z(r_z)J.
\end{equation}
The equality
\begin{equation}  \label{F: AP=PA}
  \ \ A_z(s)P_z(r_z) = P_z(r_z)A_z(s)
\end{equation}
is a direct consequence of~(\ref{F: Pz(rz)=oint (sigma-Az)(-1)d sigma}).

\subsection{Nilpotent operators $\bfA_z(r_z)$ and $\bfB_z(r_z)$}
\label{SS: bfA and bfB}
Let~$z \in \Pi$ and let~$r_z$ be a resonance point corresponding to~$z.$
We introduce a compact operator $\bfA_z(r_z)$ on the auxiliary Hilbert space~$\clK$ by equality \label{Page: bfA(rz)}
\begin{equation} \label{F: def of bfA}
  \bfA_z(r_z) = \frac {1}{2\pi i}\oint_{C(r_z)} (s-r_z) A_z(s)\,ds,
\end{equation}
where $C(r_z)$ is a small circle which contains only one resonance point~$r_z$
and which is counter-clockwise oriented around~$r_z.$
Quite often dependence of the operator $\bfA_z(r_z)$ on~$r_z$ will not be indicated, especially in proofs.
Also, instead of $\bfA_z(r_z)^j$ we shall write $\bfA_z^j(r_z).$

Similarly, one introduces an operator \label{Page: bfB(rz)}
\begin{equation} \label{F: def of bfB}
  \bfB_z(r_z) = \frac {1}{2\pi i}\oint_{C(r_z)} (s-r_z) B_z(s)\,ds.
\end{equation}

Apart of $\bfA_z(r_z)$ and $\bfB_z(r_z)$ we may sometimes need their ``underlined'' versions \label{Page: ubfA(rz)}
\begin{equation} \label{F: def of ubfA and ubfB}
  \ubfA_z(r_z) = \frac {1}{2\pi i}\oint_{C(r_z)} (s-r_z) \ulA_z(s)\,ds \ \ \text{and} \ \ \ubfB_z(r_z) = \frac {1}{2\pi i}\oint_{C(r_z)} (s-r_z) \ulB_z(s)\,ds.
\end{equation}
But since many properties of $\ubfA_z(r_z)$ and $\bfA_z(r_z),$ etc, are similar, they are given only for $\bfA_z(r_z),$ etc.

\begin{prop} \label{P: bfA to j}
Let $z \in \Pi$ and let~$r_z \in \mbC$ be a resonance point corresponding to~$z.$
For any positive integer $j$
\begin{equation} \label{F: bfA to j}
  \bfA_z^j(r_z) = \frac {1}{2\pi i}\oint_{C(r_z)} (s-r_z)^{j} A_z(s)\,ds.
\end{equation}
\end{prop}
\begin{proof} Let $\bfA_z^{(j)}$ be the right hand side of the last equality.
The claim will be proved if it is shown that for any two non-negative integers~$m$ and~$k$ the equality
$\bfA_z^{(m)}\bfA_z^{(k)} = \bfA_z^{(m+k)}$ holds.
We have
\begin{equation*}
  \begin{split}
    \bfA_z^{(m)}\bfA_z^{(k)} & = \frac {1}{(2\pi i)^2}\oint_{C_s(r_z)} (s-r_z)^{m} A_z(s)
       \brs{\oint_{C_t(r_z)} (t-r_z)^{k} A_z(t) \,dt} \,ds
       \\ & = \frac {1}{(2\pi i)^2}\oint_{C_s(r_z)} \oint_{C_t(r_z)} (s-r_z)^{m}
       (t-r_z)^{k} \frac{A_z(s) - A_z(t)}{t-s} \,dt \,ds.
  \end{split}
\end{equation*}
In this expression it can be assumed that the contour of integration $C_s(r_z)$ lies strictly inside the contour
$C_t(r_z).$ Under this choice of contours the second summand of the integrand which contains~$A_z(t)$
is holomorphic inside $C_s(r_z)$ with respect to~$s$ and therefore its integral over $C_s(r_z)$ vanishes:
\begin{equation*}
  \begin{split}
     \bfA_z^{(m)}\bfA_z^{(k)} & = \frac {1}{(2\pi i)^2}\oint_{C_s(r_z)} \oint_{C_t(r_z)} (s-r_z)^{m}
       (t-r_z)^{k} \frac{A_z(s)}{t-s} \,dt \,ds
       \\ & = \frac {1}{(2\pi i)^2}\oint_{C_s(r_z)} (s-r_z)^{m} A_z(s) \brs{\oint_{C_t(r_z)}
       \frac{(t-r_z)^{k}}{t-s}\,dt} \,ds
       \\ & = \frac {1}{2\pi i}\oint_{C_s(r_z)} (s-r_z)^{m} A_z(s) \cdot
       (s-r_z)^{k}\,ds
       \\ & = \frac {1}{2\pi i}\oint_{C_s(r_z)} (s-r_z)^{m+k} A_z(s)\,ds
       \\ & = \bfA_z^{(m+k)},
  \end{split}
\end{equation*}
where in the third equality the Cauchy integral formula is used.
\end{proof}
Proposition~\ref{P: Pz(rz)=res Az(s)} and~(\ref{F: bfA to j}) allow us to write a bit informally
\begin{equation} \label{F: P=bfA to 0}
  P_z(r_z) = \bfA_z^0(r_z).
\end{equation}
With this convention, (\ref{F: bfA to j}) holds for $j=0$ too, according to (\ref{F: Pz(rz)=res Az(s)}).

Relation~(\ref{F: A(s) and A(r) commute}), combined with~(\ref{F: Pz(rz)=res Az(s)}) and~(\ref{F: def of bfA}), implies that
\begin{equation} \label{F: PA=AP=A}
  \bfA_z(r_z) P_z(r_z) = P_z(r_z) \bfA_z(r_z) = \bfA_z(r_z).
\end{equation}
This equality also follows from general theory of operator-valued holomorphic functions \cite{Kato}.

If $r_z^1$ and $r_z^2$ are two different resonance points corresponding to $z,$ then
\begin{equation} \label{F: bfAz(1)bfAz(2)=0}
  \bfA_z(r_z^1) \bfA_z(r_z^2) = 0.
\end{equation}
Indeed, $\bfA_z(r_z^1) \bfA_z(r_z^2) = \bfA_z(r_z^1) P_z(r_z^1) P_z(r_z^2)\bfA_z(r_z^2) = 0,$
where the first equality follows from~(\ref{F: PA=AP=A}) and the second equality follows from~(\ref{F: Pz(1)Pz(2)=0}).

The equalities
\begin{equation} \label{F: Qz(1)Qz(2)=0}
  Q_z(r_z^1) Q_z(r_z^2) = 0
\end{equation}
and
\begin{equation} \label{F: bfBz(1)bfBz(2)=0}
  \bfB_z(r_z^1) \bfB_z(r_z^2) = 0
\end{equation}
can be proved by the same argument; they also follow from~(\ref{F: Pz(1)Pz(2)=0}) and~(\ref{F: bfAz(1)bfAz(2)=0}), using~(\ref{F: Pz*=Q(bar z)})
and~(\ref{F: A*(z)=B(bar z)}).
It follows from definitions~(\ref{F: def of bfA}) and~(\ref{F: def of bfB}) that for any $z \in \Pi$
and any resonance point~$r_z$ corresponding to~$z$
\begin{equation} \label{F: A*(z)=B(bar z)}
  \bfA^*_z(r_z) = \bfB_{\bar z}(\bar r_z)
\end{equation}
and, since $JA_z(s) = B_z(s)J,$ that
\begin{equation} \label{F: JA=BJ}
  J\bfA_z(r_z) = \bfB_z(r_z)J.
\end{equation}
Similarly to~(\ref{F: bfA to j}),~(\ref{F: P=bfA to 0}) and~(\ref{F: PA=AP=A})
we have
\begin{equation} \label{F: bfB to j}
  \bfB_z^j(r_z) = \frac {1}{2\pi i}\oint_{C(r_z)} (s-r_z)^{j} B_z(s)\,ds,
\end{equation}
\begin{equation} \label{F: Q=bfB to 0}
  Q_z(r_z) = \bfB_z^0(r_z),
\end{equation}
\begin{equation} \label{F: QB=BQ=B}
  \bfB_z Q_z(r_z) = Q_z(r_z) \bfB_z(r_z) = \bfB_z(r_z).
\end{equation}

Recall that a resonance point~$r_z$ is a pole of the meromorphic function~$A_z(s)$ of $s.$
Proposition~\ref{P: bfA to j} implies that Laurent series of the function~$A_z(s)$
in a neighbourhood of the pole~$r_z$ is given by, for some positive integer~$d,$ \label{Page: tilde A}
\begin{equation} \label{F: Laurent for A+(s)}
  A_z(s) = \tilde A_{z,r_z}(s) + \frac 1{s-r_z} P_z(r_z) + \frac 1{(s-r_z)^2} \bfA_z(r_z) + \ldots +
  \frac 1{(s-r_z)^d} \bfA_z^{d-1}(r_z),
\end{equation}
where $\tilde A_{z,r_z}(s)$ is the holomorphic part of the Laurent series.
It will be shown later that the integer $d$ is equal to the order of the resonance point $r_z.$
This Laurent series is an analogue of~(\ref{F: T(kappa)=}); the difference is that~(\ref{F: Laurent for A+(s)}) is a Laurent series
of a function of the coupling constant, while~(\ref{F: T(kappa)=}) is a Laurent series of a function of the spectral parameter (energy).
The finiteness of the Laurent series follows from the fact that~$(s-r_z)^{-1}$
is an isolated eigenvalue of finite multiplicity of a compact operator~$A_z(s).$
It follows from~(\ref{F: Laurent for A+(s)}) that if~$r_z^1$ and~$r_z^2$ are two resonance points, then
\begin{equation} \label{F: Laurent for A+(rz1,rz2)(s)}
  \begin{split}
  A_z(s) = \tilde A_{z,r_z^1,r_z^2}(s) & + \frac 1{s-r_z^1} P_z(r_z^1) + \frac 1{(s-r_z^1)^2} \bfA_z(r_z^1) + \ldots + \frac 1{(s-r_z^1)^{d_1}} \bfA_z^{d_1-1}(r_z^1)
  \\ & + \frac 1{s-r_z^2} P_z(r_z^2) + \frac 1{(s-r_z^2)^2} \bfA_z(r_z^2) + \ldots +
  \frac 1{(s-r_z^2)^{d_2}} \bfA_z^{d_2-1}(r_z^2),
  \end{split}
\end{equation}
where $d_\nu$ is the order of~$r_z^\nu$ and where the meromorphic function $\tilde A_{z,r_z^1,r_z^2}(s)$ is holomorphic at~$r_z^1$ and~$r_z^2.$
Similarly, the expansion~(\ref{F: Laurent for A+(rz1,rz2)(s)}) can be written for any finite set of resonance points~$r_z^1,r_z^2,\ldots$
If the perturbation operator~$V$ has finite rank, then the set of resonance points~$r_z$ is finite and the Laurent expansion,
similar to~(\ref{F: Laurent for A+(rz1,rz2)(s)}) but written for the set of all resonance points, gives Mittag-Leffler representation of the meromorphic function~$A_z(s).$
Whether this is true for infinite-rank~$V$ is unknown to me.
The equalities~(\ref{F: Laurent for A+(s)}) and~(\ref{F: PA=AP=A}) imply that
$$
  \tilde A_{z,r_z}(s) P_z(r_z) = P_z(r_z)\tilde A_{z,r_z}(s).
$$
In fact, it will be shown later that this product is equal to zero.

\begin{lemma} \label{L: nilpotent term of (sigma-A)(-1)}
Let $z \in \Pi$ and let $r_z$ be a resonance point corresponding to~$z.$
For any non-negative~$k$ and any non-resonance $r$
\begin{equation} \label{F: Laurent terms for A+(s)}
  \begin{split}
  \oint_{C(\sigma_z(r))} & (\sigma-\sigma_z(r))^k \brs{\sigma-A_z(r)}^{-1} \,d\sigma \\
      &  = \frac 1{(r-r_z)^k} \oint_{C(r_z)} \brs{\frac{s-r_z}{r-r_z} + \brs{\frac{s-r_z}{r-r_z}}^2 + \ldots}^k A_z(s)\,ds.
  \end{split}
\end{equation}
where $C(\sigma_z(r))$ is an anti-clockwise oriented contour which encloses the pole $\sigma_z(r) = (r-r_z)^{-1},$
and where $C(r_z)$ is an anti-clockwise oriented small enough contour which encloses only the pole~$r_z$ and such that the above series converges for all $s \in C(r_z).$
\end{lemma}
\noindent
Proof of this lemma is a calculation similar to the one from the proof of Proposition~\ref{P: Pz(rz)=res Az(s)}, but it is given here for the sake of completeness.
\begin{proof} The contour $C(r_z)$ can be chosen as a small enough circle with centre at~$r_z$ such that the number~$r$ lies outside of it.
In this case the geometric series in the right hand side of (\ref{F: Laurent terms for A+(s)}) converges.
This allows to rewrite the right hand side as follows:
\begin{equation*}
  \begin{split}
    (E) :=  \frac 1{(r-r_z)^k} \oint_{C(r_z)} & \brs{\frac{s-r_z}{r-r_z} + \brs{\frac{s-r_z}{r-r_z}}^2 + \ldots}^k A_z(s)\,ds
     \\  & = \frac 1{(r-r_z)^k} \oint_{C(r_z)} \brs{\frac{s-r_z}{r-r_z} \cdot \brs{1 - \frac{s-r_z}{r-r_z}}^{-1}}^k A_z(s)\,ds
     \\  & = \frac 1{(r-r_z)^k} \oint_{C(r_z)} \brs{\frac{s-r_z}{r-s}}^k A_z(s)\,ds.
  \end{split}
\end{equation*}
Now, following proof of Proposition~\ref{P: Pz(rz)=res Az(s)}, we obtain
\begin{equation*}
    (E) = \frac 1{(r-r_z)^k} \oint_{C(r_z)} \brs{\frac{s-r_z}{r-s}}^k \frac 1{s-r}\brs{1-\brs{1+(s-r)A_z(r)}^{-1}}\,ds.
\end{equation*}
Since~$r$ lies outside of the contour $C(r_z),$ it follows that
\begin{equation*}
    (E) = \frac 1{(r-r_z)^k} \oint_{C(r_z)} \brs{\frac{s-r_z}{r-s}}^k \frac 1{r-s}\brs{1+(s-r)A_z(r)}^{-1}\,ds.
\end{equation*}
Let $\sigma = \frac {1}{r-s}.$ When the variable~$s$ goes around~$r_z$ in counter-clockwise direction,
so does the variable~$\sigma$ around $\sigma_z(r)=\frac {1}{r-r_z}.$
Noting that
$$
  \frac 1{(r-r_z)^k} \brs{\frac{s-r_z}{r-s}}^k = (\sigma - \sigma_z(r))^k
$$
and
$$
  \frac 1{r-s}\brs{1+(s-r)A_z(r)}^{-1}\,ds = (\sigma - A_z(r))^{-1}\,d\sigma
$$
completes the proof.
\end{proof}

\begin{prop} \label{P: (sigma-Az(r))to(-1)=...}
Let $z \in \Pi$ and let~$r_z \in \mbC$ be a resonance point corresponding to~$z.$
The terms with negative powers in the Laurent expansion of the function $\brs{\sigma-A_z(r)}^{-1}$ of~$\sigma$
at $\sigma = \sigma_z(r) = (r-r_z)^{-1}$ are linear combinations of powers of $\bfA_z(r_z).$
In particular, taking $k=1$ in~(\ref{F: Laurent terms for A+(s)}) gives the coefficient of $(\sigma - \sigma_z)^{-2}:$
\begin{equation*} 
  \begin{split}
   \frac 1{2 \pi i}\oint_{C(\sigma_z(r))} & (\sigma-\sigma_z(r)) \brs{\sigma-A_z(r)}^{-1} \,d\sigma
       = \sigma^2_z(r) \bfA_z(r_z) + \sigma^3_z(r) \bfA^2_z(r_z) + \ldots.
  \end{split}
\end{equation*}
Taking $k=d-1$ in~(\ref{F: Laurent terms for A+(s)}), where~$d$ is the order of the resonance point~$r_z,$ gives
the coefficient of $(\sigma - \sigma_z)^{-d}:$
\begin{equation} \label{F: d-th Laurent coef for A+(s)}
  \begin{split}
  \oint_{C(\sigma_z(r))} & (\sigma-\sigma_z(r))^{d-1} \brs{\sigma-A_z(r)}^{-1} \,d\sigma
       = \sigma_z^{2d-2}(r) \bfA^{d-1}_z(r_z).
  \end{split}
\end{equation}
For other values of~$k$ the coefficient of $(\sigma - \sigma_z)^{-k-1}$ in~(\ref{F: Laurent terms for A+(s)})
has the form
\begin{equation} \label{F: k-th Laurent coef for A+(s)}
  \sigma_z^{2k}(r) \bfA^k_z(r_z) + \ldots,
\end{equation}
where dots $\ldots$ denote terms containing powers $\bfA^j_z(r_z)$ with $j>k.$
\end{prop}
\begin{proof} This immediately follows from~(\ref{F: Laurent for A+(s)}) and~(\ref{F: Laurent terms for A+(s)}).
\end{proof}

One can prove an assertion, similar to Proposition~\ref{P: (sigma-Az(r))to(-1)=...}, for the operator $B_z(s).$
\begin{prop} \label{P: nilpotent terms of (sigma-B)(-1)} The terms with negative powers in the Laurent expansion
of the function $\brs{\sigma-B_z(r)}^{-1}$ of~$\sigma$ at $\sigma = \sigma_z(r)$
are linear combinations of powers of $\bfB_z(r_z).$ 
\end{prop}

Similarly to~(\ref{F: Laurent for A+(s)}) we have
\begin{equation} \label{F: Laurent for B+(s)}
  B_z(s) = \tilde B_{z,r_z}(s) + \frac 1{s-r_z} Q_z(r_z) + \frac 1{(s-r_z)^2} \bfB_z(r_z) + \ldots +
  \frac 1{(s-r_z)^d} \bfB_z^{d-1}(r_z),
\end{equation}
where $\tilde B_{z,r_z}(s)$ is the holomorphic part of the Laurent series.
Relations~(\ref{F: Laurent for A+(s)}),~(\ref{F: Laurent for B+(s)}),~(\ref{F: JP=QJ}), and~(\ref{F: JA=BJ}) imply that holomorphic parts
$\tilde A_{z,r_z}(s)$ and $\tilde B_{z,r_z}(s)$ satisfy the relation
$$
  J\tilde A_{z,r_z}(s) = \tilde B_{z,r_z}(s)J.
$$

\subsection{Resonance vectors of order~$k$}
\label{SS: res vectors of order k}
Using a polarization type argument, the equality~(\ref{F: II resolvent identity}) allows
to rewrite the left hand side of the resonance equation~(\ref{F: res eq-n}) of order~$k$ as an expression, linearly dependent on~$A_z(s_j).$
This is done in the following proposition. As will be seen, this rewriting of the resonance equation will prove useful.
\begin{prop} \label{P: prod of A's}
If $z \in \Pi,$ if~$r_z$ is a resonance point corresponding to~$z$ and if distinct numbers $s_1, \ldots, s_k$ are non-resonant, then
\begin{equation}  \label{F: prod of A's}
  \prod_{j=1}^k \SqBrs{1+(r_z-s_j) A_z(s_j)} = \sum_{j=1}^k (s_j-r_z)^{k-1}\brs{1+(r_z-s_j) A_z(s_j)} \prod_{i=1, i\neq j}^k (s_j-s_i)^{-1}.
\end{equation}
\end{prop}
\begin{proof} For $k=1$ this equality is trivial. In case of $k=2,$ the second resolvent identity~(\ref{F: II resolvent identity}) implies
\begin{equation} \label{F: the k=2 case}
  \begin{split}
     [1+(r_z&-s)A_z(s)][1+(r_z-r)A_z(r)]
     \\ & = 1 +(r_z-s)A_z(s) +(r_z- r) A_z(r) + \frac{(r_z-s)(r_z-r)}{s-r} (A_z(r) - A_z(s))
     \\ & = \frac {s-r_z}{s-r} \brs{1+(r_z-s) A_z(s)}  + \frac {r-r_z}{r-s}\brs{1+(r_z-r) A_z(r)}
  \end{split}
\end{equation}
and this gives~(\ref{F: prod of A's}) for $k=2.$
Assuming that~(\ref{F: prod of A's}) holds for $k-1$ instead of $k,$ we have
\begin{equation*}
  \begin{split}
     (E) & := \prod_{j=1}^k \SqBrs{1+(r_z-s_j) A_z(s_j)} = (1+(r_z-s_k) A_z(s_k))\prod_{j=1}^{k-1} \SqBrs{1+(r_z-s_j) A_z(s_j)}
     \\ & = \brs{1+(r_z-s_k) A_z(s_k)}\sum_{j=1}^{k-1} (s_j-r_z)^{k-2}\brs{1+(r_z-s_j) A_z(s_j)} \prod_{i=1, i\neq j}^{k-1} (s_j-s_i)^{-1}.
  \end{split}
\end{equation*}
Applying~(\ref{F: the k=2 case}) to the product $\brs{1+(r_z-s_k) A_z(s_k)}\brs{1+(r_z-s_j) A_z(s_j)}$ gives
\begin{equation*}
  \begin{split}
     \\ (E) & = \sum_{j=1}^{k-1} (s_j-r_z)^{k-2}  \SqBrs{\frac {s_k-r_z}{s_k-s_j}\brs{1+(r_z-s_k) A_z(s_k)} + \frac {s_j-r_z}{s_j-s_k} \brs{1+(r_z-s_j) A_z(s_j)}}
     \\ &   \hskip 10.5cm  \times \prod_{i=1, i\neq j}^{k-1} (s_j-s_i)^{-1}
     \\ & = \sum_{j=1}^{k-1} (s_j-r_z)^{k-1}\brs{1+(r_z-s_j) A_z(s_j)} \prod_{i=1, i\neq j}^{k} (s_j-s_i)^{-1}
     \\ & \qquad \qquad \qquad \qquad - (s_k-r_z)\brs{1 +(r_z-s_k) A_z{(s_k)}} \sum_{j=1}^{k-1} (s_j-r_z)^{k-2}\prod_{i=1, i\neq j}^{k} (s_j-s_i)^{-1}.
  \end{split}
\end{equation*}
Thus the proof will be complete if it is shown that
\begin{equation} \label{F: obvious equality}
  \begin{split}
    \sum_{j=1}^{k} (s_j-r_z)^{k-2}\prod_{i=1, i\neq j}^{k} (s_j-s_i)^{-1} = 0.
  \end{split}
\end{equation}
By Lemma~\ref{L: div-d diff-nce I}, the left hand side of this equality is the divided difference of order $k-1$ of the function $f(s) = (s-r_z)^{k-2}.$
Hence, the equality~(\ref{F: obvious equality}) follows from Lemma~\ref{L: div-d diff-nce II}.
\end{proof}

%
%
Proposition~\ref{P: prod of A's} and Proposition~\ref{P: res eq-n is correct} imply the following assertion.
\begin{thm} \label{T: sum prod psi = 0}
The resonance equation~(\ref{F: res eq-n})
of order~$k$ is equivalent to any of the following two equations:
\begin{equation} \label{F: prod of A(s)'s = 0}
  \prod_{j=1}^k \brs{1+(r_z-s_j) A_z(s_j)} u = 0
\end{equation}
or
\begin{equation} \label{F: boring formula}
  \sum_{j=1}^k (s_j-r_z)^{k-1}\brs{u + (r_z-s_j) A_z(s_j)u} \prod_{i=1, i\neq j}^k (s_j-s_i)^{-1} = 0,
\end{equation}
where $s_1, \ldots,s_k$ is any set of~$k$ non-resonance points.
\end{thm}
\begin{proof} Commutativity property~(\ref{F: A(s) and A(r) commute}) of~$A_z(s)$ and
Proposition~\ref{P: res eq-n is correct} imply that the resonance equation~(\ref{F: res eq-n}) is equivalent to~(\ref{F: prod of A(s)'s = 0}).
Proposition~\ref{P: prod of A's} implies that~(\ref{F: prod of A(s)'s = 0}) is equivalent to~(\ref{F: boring formula}).
\end{proof}
\begin{thm} \label{T: Laurent for A(s)psi}
If $u^{(k)}$ is a resonance vector of order $k,$ then
\begin{equation} \label{F: A psik=...}
  A_z(s)u^{(k)} = \sum_{j=0}^{k-1} \frac {u^{(k-j)}}{(s-r_z)^{j+1}}
\end{equation}
where $u^{(k-j)}$ is a resonance vector of order $k-j.$
Moreover,
\begin{equation} \label{F: psi(j)=oint...Apsi(k)}
  u^{(k-j)} = \bfA^j_z (r_z)u^{(k)},
\end{equation}
and thus, the operator $\bfA^j_z(r_z)$ lowers order of a resonance vector
$u \in \Upsilon_z(r_z)$ by $j,$ where $j=1,2,\ldots.$
\end{thm}
\noindent In particular, the operator $\bfA_z(r_z)$ is nilpotent: $\bfA_z^d(r_z)=0,$
where~$d$ is the order of the point~$r_z,$ and the geometric multiplicity~$m$ of the resonance point~$r_z$
is equal to $m = \dim \ker \bfA_z(r_z).$
\begin{proof}
We prove this assertion by induction on $k.$
For $k=1$ the equality~(\ref{F: A psik=...}) is equivalent to the resonance equation~(\ref{F: res eq-n}) of order $k=1.$
Assume that the assertion holds for $k=n-1$ and let $u = u^{(n)}$ be a vector of order $n.$ Since $u$ satisfies the resonance equation of order $n,$
it follows from Theorem~\ref{T: sum prod psi = 0} that $u$ satisfies the equality~(\ref{F: boring formula}). Hence, taking in~(\ref{F: boring formula})
(with $k=n$) \ $s = s_n$ we obtain that the vector~$A_z(s)u$ has the form
\begin{equation} \label{F: A psin=...}
  A_z(s)u = \sum_{j=0}^{n-1} \frac {u^{(n-j)}}{(s-r_z)^{j+1}},
\end{equation}
where $u^{(n-j)}, \ j=0,\ldots,n-1,$ are some vectors; we have to show that the vector $u^{(n-j)}$ has order $n-j$ for all $j=0,\ldots,n-1.$
Applying to both sides of the equality~(\ref{F: A psin=...}) the operator $1+(r_z-r)A_z(r)$ and using commutativity of operators $A_z(s)$ and $A_z(r)$~(\ref{F: A(s) and A(r) commute}) we obtain
\begin{equation} \label{F: 3947}
  \begin{split}
    A_z(s)\SqBrs{1+(r_z-r)A_z(r)}u & = \SqBrs{1+(r_z-r)A_z(r)}A_z(s)u
    \\ & = \sum_{j=0}^{n-1} \frac {\phi^{(n-j-1)}}{(s-r_z)^{j+1}},
  \end{split}
\end{equation}
where
\begin{equation} \label{F: 8837}
  \phi^{(n-j-1)} = \SqBrs{1+(r_z-r)A_z(r)}u^{(n-j)}, \ \ j=0,1,\ldots,n-1
\end{equation}
and where $\phi^{(0)} = 0.$ Since $u$ is a resonance vector of order $n,$
the vector $\SqBrs{1+(r_z-r)A_z(r)}u$ is a resonance vector of order $n-1.$ Hence, by induction assumption, it follows
from~(\ref{F: 3947}) that the vector $\phi^{(n-j-1)}$ has order $n-j-1.$
Since the operator $1+(r_z-r)A_z(r)$ decreases order of a resonance vector by 1, this and~(\ref{F: 8837}) imply that $u^{(n-j)}$
is a vector of order $n-j.$ Proof of the first part of the theorem is complete.

The equality~(\ref{F: psi(j)=oint...Apsi(k)}) follows from~(\ref{F: bfA to j}) and~(\ref{F: A psik=...}).
\end{proof}
\noindent
The equality~(\ref{F: A psik=...}) can be rewritten as
\begin{equation} \label{F: Az(s)Pz(rz)=sum...}
  A_z(s)P_z(r_z) = \sum_{j=0}^{d-1} (s-r_z)^{-j-1}\bfA_z^j(r_z),
\end{equation}
where~$d$ is the order of the resonance point~$r_z.$

\begin{cor} The holomorphic part $\tilde A_{z,r_z}(s)$ of the meromorphic function~$A_z(s)$
in a neighbourhood of~$r_z$ satisfies the equality
\begin{equation} \label{F: tilde AP=0}
  \tilde A_{z,r_z}(s) P_z(r_z) = P_z(r_z) \tilde A_{z,r_z}(s) = 0.
\end{equation}
\end{cor}
\noindent This follows from~(\ref{F: Laurent for A+(s)}),~(\ref{F: PA=AP=A})
and~(\ref{F: Az(s)Pz(rz)=sum...}).


\begin{prop} \label{P: [1+sAz(r)](-1)Pz(rz)} 
If~$r_z$ is a resonance point of order~$d$ and if $r$ and $s+r$ are regular points such that $\abs{s} < \abs{r-r_z},$ then
\begin{equation} \label{F: weird f-n}
  \begin{split}
     [1+sA_z(r)]^{-1}P_z(r_z) & = \sum_{j=0}^{d-1} (r-r_z)^{-j}R_j\brs{\frac{s}{r_z-r}} \bfA_z^j(r_z),
  \end{split}
\end{equation}
where $R_j(w),$ $j=0,1,2,\ldots$ are some holomorphic functions given by power series centered at $w = 0$ with radius of convergence equal to~$1.$
\end{prop}
\begin{proof} The numbers $r$ and $s+r$ are to be regular points for the equality (\ref{F: weird f-n}) to hold since otherwise the operator $A_z(r)$ does not exist or the operator
$1+sA_z(r)$ is not invertible.

It follows from~(\ref{F: Dn Az(s)=...}) and~(\ref{F: Az(s)Pz(rz)=sum...}) that
\begin{equation*}
  \begin{split}
    A_z^{n+1}(r) P_z(r_z) & = \frac{(-1)^n}{n!} \frac {d^n}{dr^n} \sum_{j=0}^{d-1} \frac {1}{(r-r_z)^{j+1}} \bfA_z^j(r_z)
      \\ & = \frac{1}{n!} \sum_{j=0}^{d-1} \frac {(j+1)(j+2)\ldots(j+n)}{(r-r_z)^{j+n+1}} \bfA_z^j(r_z)
      \\ & = \sum_{j=0}^{d-1} \frac {C^{n}_{n+j}}{(r-r_z)^{j+n+1}} \bfA_z^j(r_z),
  \end{split}
\end{equation*}
where $C^{n}_{n+j}$ is the binomial coefficient.
Using this, we have, for small enough $s,$
\begin{equation*}
  \begin{split}
     [1+sA_z(r)]^{-1}P_z(r_z) & = \sum_{n=0}^\infty (-s)^n A_z^n(r)P_z(r_z)
     \\ & = \sum_{n=0}^\infty (-s)^n \sum_{j=0}^{d-1} \frac {C^{n-1}_{n+j-1}}{(r-r_z)^{j+n}} \bfA_z^j(r_z)
     \\ & = \sum_{j=0}^{d-1} (r-r_z)^{-j} \sum_{n=0}^\infty \brs{\frac{-s}{r-r_z}}^n C^{n-1}_{n+j-1}\bfA_z^j(r_z).
  \end{split}
\end{equation*}
The functions $R_j(w) = \sum_{n=0}^\infty C^{n-1}_{n+j-1}w^n,$ $j=1,2,\ldots,$ are holomorphic functions with radius of convergence equal to~$1.$
It follows that~(\ref{F: weird f-n}) holds for all small enough $s,$ and therefore by analytic continuation it holds for all $s$
such that the operator $1+sA_z(r)$ is invertible and $\abs{s} < \abs{r-r_z}.$
The last equality also shows that if $\abs{s} < \abs{r-r_z},$ then the function $[1+sA_z(r)]^{-1}P_z(r_z)$ admits analytic continuation to non-regular points $s$
which belong to the disk $\abs{s} < \abs{r-r_z}.$
\end{proof}

Recall that underlined versions $\ulP_z(r_z)$ and $\ubfA_z(r_z)$ of operators $P_z(r_z)$ and $\bfA_z(r_z)$
are defined by formulas (\ref{F: ulPz(rz)=oint (sigma-Az)(-1)d sigma}) and (\ref{F: def of ubfA and ubfB}).
In the following proposition we use the underlined operators, since for ``non-underlined'' operators it does not make sense.
\begin{prop} For any resonance point $r_z$ corresponding to a non-real number~$z,$
\begin{equation} \label{F: (H-z)P=-VA}
  (H_{r_z} - z)\ulP_z(r_z) = - V\ubfA_z(r_z).
\end{equation}
\end{prop}
\begin{proof} From~(\ref{F: A(s)=(1+(s-r)A(r))(-1)A(r)}) we have
$$
  (1+(s-r)\ulA_z(r)) \ulA_z(s) = \ulA_z(r).
$$
Substituting here instead of $\ulA_z(s)$ its Laurent expansion~(\ref{F: Laurent for A+(s)}),
we find a Laurent expansion of the left hand side as a function of~$s.$ Since the right hand side is constant, all coefficients except one in this Laurent expansion
are zero. In particular, calculating the coefficient of $(s-r_z)^{-1}$ we find that
$$
  (1+(r_z-r)\ulA_z(r))\ulP_z(r_z) = - \ulA_z(r)\ubfA_z(r_z).
$$
Multiplying both sides of this equality by $H_r-z$ gives~(\ref{F: (H-z)P=-VA}).
\end{proof}
The equality~(\ref{F: (H-z)P=-VA}) is plainly equivalent to the following proposition.
\begin{cor} \label{C: nasty proposition}
Let~$z$ be a non-real number and let~$r_z$ be a resonance point corresponding to~$z.$
If $u_z^{(k)} = F\chi_z^{(k)}$ is a vector of order $k,$ then
\begin{equation*} 
  \begin{split}
    (H_{r_z}-z)\chi_z^{(1)} & = 0 ,\\
    (H_{r_z}-z)\chi_z^{(2)} & = - V\chi_z^{(1)}, \\
    \ldots \\
    (H_{r_z}-z)\chi_z^{(k)} & = - V\chi_z^{(k-1)}, \\
  \end{split}
\end{equation*}
where the vectors $u_z^{(j)} = F\chi_z^{(j)}$ satisfy~(\ref{F: psi(j)=oint...Apsi(k)}).
\end{cor}

\subsection{Holomorphic part of $A_z(s)$}
In this subsection we study the holomorphic part $\tilde A_{z,r_z}(s)$ of the Laurent expansion~(\ref{F: Laurent for A+(s)}) of the function $A_z(s)$
at a resonance point $s = r_z.$ 
\begin{prop} If $z \in \Pi$ and if~$r_z$ is a resonance point corresponding to $z,$ then for any non-resonant value of $s$
we have
\begin{equation} \label{F: tilde A(r)=...}
  \tilde A_{z,r_z}(r) = \tilde A_{z,r_z}(s)(1+(s-r)\tilde A_{z,r_z}(r))
\end{equation}
as equality between two holomorphic functions of $r.$ 
\end{prop}
\begin{proof} Using~(\ref{F: A(s)=(1+(s-r)A(r))(-1)A(r)}) and the Laurent expansion~(\ref{F: Laurent for A+(s)}) of $A_z(s)$ we have
\begin{equation*}
  \begin{split}
    A_z(r) & = A_z(s) (1+(s-r)A_z(r))
    \\ & = A_z(s) \brs{1+(s-r)\tilde A_{z,r_z}(r) + (s-r) \sum_{j=1}^d (r-r_z)^{-j}\bfA_z^{j-1}}.
  \end{split}
\end{equation*}
Here we consider both sides of this equality as meromorphic functions of $r,$ so $s$ is a fixed number.
One can see that the holomorphic part of $(s-r) \sum_{j=1}^d (r-r_z)^{-j}\bfA_z^{j-1}$ at $r = r_z$ is $-P_z(r_z).$
Hence, comparing holomorphic at $r=r_z$ parts of the last equality, we find that
\begin{equation} \label{F: tilde A=A(1+(s-r)tilde A-P)}
  \tilde A_{z,r_z}(r) = A_z(s) \brs{1+(s-r)\tilde A_{z,r_z}(r) - P_z(r_z)}.
\end{equation}
Equalities~(\ref{F: tilde AP=0}) and~(\ref{F: PA=AP=A}) combined with~(\ref{F: Laurent for A+(s)})
imply that $A_z(s)(1-P_z(r_z)) = \tilde A_{z,r_z}(s)$ and that $A_z(s) \tilde A_{z,r_z}(r) = \tilde A_{z,r_z}(s) \tilde A_{z,r_z}(r).$
Combining these equalities with~(\ref{F: tilde A=A(1+(s-r)tilde A-P)}) gives~(\ref{F: tilde A(r)=...}).
\end{proof}
Another way to prove this proposition is to observe that, since $P_z(r_z)$ and $A_z(s)$ commute, the kernel of $P_z(r_z)$ reduces $A_z(s)$ and
by~(\ref{F: Laurent for A+(s)}) the reduction is
$\tilde A_{z,r_z}(s).$ Hence, the claim follows from~(\ref{F: A(s)=(1+(s-r)A(r))(-1)A(r)}) and~(\ref{F: tilde AP=0}).
From this observation it also follows that the kernel and range of the operator $\tilde A_{z,r_z}(r)$ do not depend on $r.$
Using the equality~(\ref{F: tilde A(r)=...}) and a standard Fredholm alternative argument one can show that the operator
$1+(s-r)\tilde A_{z,r_z}(r)$ is invertible, so that
\begin{equation} \label{F: tiA(s)=(1+(s-r)tiA(r))(-1)tiA(r)}
  \tilde A_{z,r_z}(s) = (1+(s-r)\tilde A_{z,r_z}(r))^{-1} \tilde A_{z,r_z}(r).
\end{equation}
Similar equalities also hold for functions $\tilde A_{z,r_z^1,r_z^2},$ etc.

Since~$r_z$ is a pole of $A_z(s)$ the expression $A_z(r_z)$ does not make sense, but the value $\tilde A_{z,r_z}(r_z)$ of the holomorphic part $\tilde A_{z,r_z}(s)$
at $s=r_z$ is defined. In particular, we have
\begin{equation} \label{F: tiA(s)=(1+(s-rz)tiA(rz))(-1)tiA(rz)}
  \tilde A_{z,r_z}(s) = (1+(s-r_z)\tilde A_{z,r_z}(r_z))^{-1} \tilde A_{z,r_z}(r_z).
\end{equation}
The equality~(\ref{F: tiA(s)=(1+(s-rz)tiA(rz))(-1)tiA(rz)}) allows to find Taylor series of $\tilde A_{z,r_z}(s)$ at $s = r_z:$
$$
  \tilde A_{z,r_z}(s) = \tilde A_{z,r_z}(r_z) - \tilde A_{z,r_z}^2(r_z)(s-r_z) + \tilde A_{z,r_z}^3(r_z)(s-r_z)^2 - \ldots
$$
It is possible that $\tilde A_{z,r_z}(r_z) = 0$ but this is very unlikely, since this would imply
that $\tilde A_{z,r_z}(s) = 0$ for all $s$ and therefore according to~(\ref{F: Laurent for A+(s)})
that~$r_z$ is the only resonance point corresponding to~$z.$

Similar properties hold also for the holomorphic part $\tilde B_z(s)$ of the function $B_z(s).$
One can also see that
$$
  \brs{\tilde A_{z,r_z}(s)}^* = \tilde B_{\bar z, \bar r_z}(\bar s) \quad \text{and} \quad  J\tilde A_{z,r_z}(s)=\tilde B_{z,r_z}(s)J.
$$

\section{Geometric meaning of $\Upsilon^1_{\lambda+i0}(r_\lambda)$}
\label{S: Geom meaning of Ups(1)}

This and subsequent sections are independent of each other.

In scattering theory one may distinguish three types of vectors: scattering states, bound states
and trapped states (see e.g. \cite{TayST,RS3}). Bound states describe localized particles, scattering states describe particles which are free
at $t \to \pm \infty,$ and finally trapped states describe particles which are free at $t \to \pm \infty$ but localized at $t \to \mp \infty,$
and as such trapped states describe processes of capture and decay. Bound states are eigenvectors of the full Hamiltonian $H=H_0+V,$
so they are attributed to the point spectrum; the vector space of scattering states of a fixed energy~$\lambda$ can be seen as a fiber Hilbert space $\hlambda$
(on-shell Hilbert space) and thus they can be attributed to absolutely continuous spectrum; finally, trapped vectors should be attributed to
singular continuous spectrum. While the states $\psi$ of all three types are eigenvectors of the full Hamiltonian in the sense that they satisfy the eigenvector
equation $H \psi = \lambda \psi,$ only bound states belong to the Hilbert space. Scattering and trapped states are usually called generalized eigenvectors.
For the Schr\"odinger operator $-\Delta +V,$ the scattering and trapped states are given by functions which do not belong to $L_2(\mbR^\nu).$
In an abstract setting one may consider a rigged Hilbert space~(\ref{F: hilb+}) to describe generalized eigenvectors. That is, proper eigenvectors are elements of
$\hilb$ while generalized eigenvectors are elements of $\hilb_-.$ Since the rigging operator $F$ provides natural isomorphisms of Hilbert spaces
$\hilb$ and $\clK_+$ on one hand, and of Hilbert spaces $\hilb_-$ and $\clK$ on the other hand, one may also treat proper eigenvectors as elements of the Hilbert space~$\clK_+$
and generalized eigenvectors as elements of the Hilbert space~$\clK.$

Let~$H_0$ be a self-adjoint operator from the affine space~(\ref{F: clA}) which is regular at an essentially regular point $\lambda$ and let $V \in \clA_0(F)$ be a perturbation.
At a discrete set of real resonance points~$r_\lambda$ of the triple $(\lambda, H_0,V)$ the operator~$H_0+r_\lambda V$ ceases to be regular at~$\lambda.$
A natural question is why this can happen. By Proposition~\ref{P: Az 4.1.10}, one reason is that~$\lambda$ can be an eigenvalue of
$H_0+r_\lambda V.$ For $\lambda$ outside of the essential spectrum this is the only reason. But if $\lambda$ belongs to the essential spectrum then
the operator~$H_0+r_\lambda V$ may still fail to be regular at~$\lambda$ even if~$\lambda$ is not an eigenvalue.
Intuitively, if~$H_r$ is regular at $\lambda$ then all generalized eigenvectors are scattering states which form the fiber Hilbert space $\hlambda.$
Therefore it is natural to expect that if $\lambda$ is not an eigenvalue of~$H_r$ but nevertheless~$H_r$ is not regular at $\lambda,$
then the operator~$H_0+r_\lambda V$ should have trapped eigenvectors,
that is, generalized eigenvectors which are neither proper eigenvectors nor the elements of the Hilbert space $\hlambda$ of scattering states.
Results of this section formally confirm this assertion.
Namely, it is shown that the vector space
$$
  \Upsilon^1_{\lambda+i0}(r_\lambda)
$$
of solutions of the equation
$$
  u + (r_\lambda-r)T_{\lambda+i0}(H_r)Ju = 0
$$
can be considered as a proper replacement of the vector space of proper eigenvectors
in the sense that the latter space is naturally linearly isomorphic to a subspace of $\Upsilon^1_{\lambda+i0}(r_\lambda).$
The linear isomorphism is natural in the sense that it is given by the rigging operator~$F.$
Thus, dimension of the vector space $\Upsilon^1_{\lambda+i0}(r_\lambda)$ consists of two summands, the dimension of the vector space
of proper eigenvectors and the dimension of the vector factor-space of trapped vectors defined up to an eigenvector.

The eigenvalue equation for the perturbed operator~$H_r=H_0+rV$
$$
  (H_0+rV) \chi = \lambda \chi
$$
can be rewritten formally as the homogeneous Lippmann-Schwinger equation (\cite{LippSch50}, see also e.g. \cite[(81)]{RS3}, \cite{TayST})
\begin{equation} \label{F: Lippmann-Schwinger eq-n}
  \chi + r(H_0-\lambda)^{-1}V\chi = 0.
\end{equation}
If~$\lambda$ lies outside the essential spectrum, then the Lippmann-Schwinger equation makes perfect sense and is equivalent to the eigenvalue equation,
but if~$\lambda$ belongs to the essential spectrum, then the Lippmann-Schwinger equation should be rewritten to make sense. One way of doing this is
to factorize the perturbation~$V$ as $G^*JG,$ where~$G$ is an operator acting from the ``main'' Hilbert space
to an auxiliary Hilbert space~$\clK,$ and to rewrite the Lippmann-Schwinger equation as an equation for a vector $u = G\chi$ in~$\clK$ as follows
(see e.g. \cite[Lemma~4.7.8]{Ya})
\begin{equation} \label{F: Lippmann-Schwinger eq-n(2)}
  u + rG(H_0-\lambda-i0)^{-1}G^* J u = 0.
\end{equation}
This can be done as long as the limiting absorption principle holds, that is, as long as the limit operator $G(H_0-\lambda-i0)^{-1}G^*$
acting on the Hilbert space $\clK$ exists. The vector~$\chi$ may afterwards be recovered by $\chi = G^{-1}u,$
but this vector may not belong to the Hilbert space~$\hilb.$ The number~$\lambda$
for which~(\ref{F: Lippmann-Schwinger eq-n(2)}) has a non-zero solution
is an eigenvalue of $H_0+V$ if and only if the vector $\chi = G^{-1}u$
exists and belongs to $\hilb,$ that is, iff $u$ belongs to the range of the operator~$G.$
But even if $u$ does not belong to the range of~$G,$ the number~$\lambda$ is still to be considered as a singular point of the spectrum of~$H_r$
due to the presence of ``trapped states''.

As a final remark we note that though a factorization $G^*JG$ of the perturbation~$V$
looks to be an unnatural nuisance, which is however necessary for technical reasons,
in the current setting there is a fixed rigging operator~$F$ and the perturbation~$V$ admits a factorization $F^*JF$ by the very definition. 
\begin{center} * \ \ * \ \ *
\end{center}
In this section we shall use two well-known properties of a self-adjoint operator~$H$:
for any real number~$\lambda$
\begin{equation} \label{F: fact 1}
  \frac{(H-\lambda)^2}{(H-\lambda)^2+y^2} \to 1 \quad \text{strongly as} \ y \to 0
\end{equation}
and if a real number~$\lambda$ is not an eigenvalue of~$H$ then
\begin{equation} \label{F: fact 2}
  \frac{y(H-\lambda)}{(H-\lambda)^2+y^2} \to 0 \quad \text{weakly as } \ y \to 0.
\end{equation}

\begin{thm} \label{T: res eq-n and eigenvectors}
Let~$\lambda$ be an essentially regular point, let $H_0 \in \clA$ be regular at~$\lambda$ operator,
let $V \in \clA_0(F),$ let~$r_\lambda$ be a real resonance point of the triple $(\lambda; H_0,V)$
and let $r$ be a regular point of the triple $(\lambda; H_0,V).$
If~$\lambda$ is an eigenvalue of the operator~$H_{r_\lambda} = H_0+r_\lambda V$
with eigenvector~$\chi \in \euD = \dom(H_{r_\lambda}),$ then the vector $u = F\chi$ is a resonance vector of order~1,
that is,
\begin{equation} \label{F: res eq 1}
  \brs{1+(r_\lambda-r)T_{\lambda+i0}(H_{r})J}u = 0.
\end{equation}
\end{thm}
\begin{proof}
Firstly we note that by (\ref{F: euD subset dom(F)}) the vector $F\chi$ is well-defined, since the domain of $F$ contains the common domain of operators $H \in \clA.$
The eigenvalue equation~$H_{r_\lambda}\chi = \lambda \chi$ implies the equality
\begin{equation} \label{F: (Hr-l)chi=rV chi (0)}
  (H_{r}-\lambda)\chi = (r-r_\lambda)V\chi.
\end{equation}
Here both sides are well-defined since~$H_r$ and $H_{r_\lambda}$ have common domain $\euD$ by (\ref{F: euD=dom(H)})
and by (\ref{F: euD subset dom(V)}) the domain of $V$ contains $\euD.$
Hence, for any~$z$ with $\Im z \neq 0$ we have
$$
  F R_z(H_{r})(H_{r}-\lambda)\chi = (r-r_\lambda)F R_z(H_{r}) V\chi.
$$
Since $V \chi = F^*JF\chi$ and $\lambda \in \Lambda(H_{r},F),$
by the Limiting Absorbtion Principle Assumption (see~(\ref{F: Lambda(H,F)}) and~(\ref{F: T(l+i0) exists}))
the limit of the right hand side of the last equality exists in the
uniform operator topology as $z = \lambda\pm iy \to \lambda\pm i0$
and therefore so does the limit of the left hand side:
$$
  F R_{\lambda\pm i0}(H_{r})(H_{r}-\lambda)\chi = (r-r_\lambda)F R_{\lambda\pm i0}(H_{r}) V\chi.
$$
Adding these equalities gives
\begin{equation*} \label{F: Re=Re}
  F \Re R_{\lambda + i0}(H_{r})(H_{r}-\lambda)\chi = (r-r_\lambda)F \Re R_{\lambda +i0}(H_{r}) V\chi.
\end{equation*}
Since, by~(\ref{F: fact 1}), $\Re R_{\lambda + iy}(H_{r})(H_{r}-\lambda) \to 1$ in
the strong operator topology as $y\to 0,$ it follows from the last equality that
\begin{equation} \label{F: F chi=rF Re R Hr V chi}
  F \chi = (r-r_\lambda)F \Re R_{\lambda +i0}(H_{r}) V\chi.
\end{equation}
Since~$r$ is a regular point of the path $\set{H_s \colon s \in \mbR},$ by Proposition~\ref{P: Az 4.1.10},
$\lambda$ is not an eigenvalue of $H_{r}.$ It follows from this and~(\ref{F: fact 2})
that $\Im R_{\lambda + iy}(H_{r})(H_{r}-\lambda) \to 0$ in the weak
operator topology as $y \to 0.$ Since~$F E_\Delta(H_r)$ is compact by (\ref{F: FE(Delta) is compact}), it follows
that
$
  F \Im R_{\lambda + i0}(H_{r})(H_{r}-\lambda)\chi = 0.
$
Combining this with~(\ref{F: (Hr-l)chi=rV chi (0)}) gives equality
$$
  0 = (r-r_\lambda)F \Im R_{\lambda + i0}(H_{r})V\chi.
$$
Multiplying this equality by $i$ and adding it to~(\ref{F: F chi=rF Re R Hr V chi}), one gets the equality
$$
  F \chi = (r-r_\lambda)F R_{\lambda +i0}(H_{r}) V\chi.
$$
Since $V = F^*JF,$ this can be rewritten as
$$
  (1 + (r_\lambda-r)F R_{\lambda +i0}(H_{r}) F^*J)F\chi = 0.
$$
This equality is identical to~(\ref{F: res eq 1}) with $u = F\chi.$
Hence, $u=F\chi$ is a resonance vector of order~1. 
\end{proof}
A resonance vector~$u$ will be called \emph{regular}, if $u \in \clK_+.$ \label{Page: regular vector}
Since the rigging operator~$F$ has trivial kernel, Theorem~\ref{T: res eq-n and eigenvectors} implies that to linearly independent
eigenvectors $\chi_1,\ldots,\chi_N$ of~$H_0$ there correspond linearly independent regular resonance vectors
$u_1 = F\chi_1, \ldots, u_N = F\chi_N \in \Upsilon^1_{\lambda+i0}(r_\lambda).$
Hence,
\begin{cor} \label{C: mult of lambda <= dim Ups(1)}
If~$\lambda$ is an essentially regular point, then the
geometric multiplicity of~$\lambda$ as an eigenvalue of the self-adjoint operator $H_{r_\lambda} = H_0+r_\lambda V$
does not exceed dimension of the vector space $\Upsilon_{\lambda+i0}^1(r_\lambda),$
that is,
$$
  \dim \clV_\lambda \leq \dim \Upsilon^1_{\lambda+i0}(r_\lambda),
$$
where $\clV_\lambda$ is the eigenspace of $H_{r_\lambda}$ corresponding to the eigenvalue~$\lambda.$
\end{cor}
\noindent
Corollary~\ref{C: mult of lambda <= dim Ups(1)} allows to present an example of a point~$\lambda$ which is not essentially regular.
\begin{thm} \label{T: infty mult-ty then not essentially regular}
  If~$\lambda$ is an eigenvalue of infinite multiplicity for at least one self-adjoint operator~$H$ from the affine space $\clA = H_0+\clA_0(F),$
  then~$\lambda$ is not an essentially regular point of the pair $(\clA,F),$ that is, $\lambda \notin \Lambda(\clA,F).$
\end{thm}
\begin{proof} Without loss of generality it can be assumed that the operator~$H$ is equal to~$H_0.$
Assume the contrary to the statement of the theorem: for some perturbation $V \in \clA_0(F)$
and some necessarily non-zero $r \in \mbR$ the number~$\lambda$ belongs to the set
$\Lambda(H_r,F),$ where~$H_r = H_0+rV.$
Since~$\lambda$ is an eigenvalue of infinite multiplicity of~$H_0$ and $\clV_\lambda$ is the corresponding infinite-dimensional subspace
of eigenvectors, by Theorem~\ref{T: res eq-n and eigenvectors} for the non-resonant point~$r$ the linear subspace $F(\clV_\lambda)$ consists of eigenvectors
of a compact operator $A_{\lambda+i0}(r) = T_{\lambda+i0}(H_r)J$ corresponding to the eigenvalue $1/r.$ Since~$F$ has trivial kernel, the subspace $F(\clV_\lambda)$
is also infinite-dimensional. This contradicts the compactness of the operator $T_{\lambda+i0}(H_r)J.$
\end{proof}
So far in this section no conditions were imposed on~$\lambda$ except the condition of essential regularity. If, however,~$\lambda$
lies outside the essential spectrum, then one can prove more refined version of Theorem~\ref{T: res eq-n and eigenvectors}.
\begin{lemma} \label{L: Upsilon(1) subset F(hilb)} Let~$\lambda$ be an essentially regular point, let~$H_{r_\lambda}$ be resonant at~$\lambda$
and let~$V$ be a regularizing direction.
If~$\lambda$ is an isolated eigenvalue of~$H_{r_\lambda},$ then all resonance vectors of first order are $\clK_+$-vectors,
that is, all vectors $u \in \Upsilon^1_{\lambda+i0}(r_\lambda)$ are of the form $u = F\chi$ for some vector $\chi \in \hilb.$
\end{lemma}
\begin{proof} Assume that~$u$ is a resonance vector of order 1:
\begin{equation} \label{F: star eq-n}
  (1 + (r_\lambda - r)T_{\lambda+i0}(H_r)J) u = 0.
\end{equation}
Since~$V$ is a regularizing direction, by Corollary \ref{C: Az 4.1.10}
for some~$r$ the number~$\lambda$ belongs to the resolvent set of~$H_r.$
It follows that $R_{\lambda+i0}(H_r) = (H_r-\lambda)^{-1}$ exists as a bounded operator in~$\hilb.$ Hence, the vector
$$
  \chi := (r - r_\lambda)R_{\lambda+i0}(H_r)F^*J u \in \hilb_-
$$
is a well-defined element of the Hilbert space $\hilb,$ where the vector $F^*Ju$ is well-defined by (\ref{F: JF euD subset dom(F*)}).
It follows from this and (\ref{F: star eq-n}) that $u = F\chi.$
Hence, the vector $u = F\chi$ belongs to $\clK_+ \supset F\hilb.$
\end{proof}

\begin{thm} \label{T: if lambda notin ess sp, ...} If~$\lambda$ does not belong to the essential spectrum $\sigma_{ess},$ then
the rigging operator~$F$ is a linear isomorphism of the eigenspace $\clV_\lambda$ of~$H_{r_\lambda}$ and
the vector space~$\Upsilon_\lambda^1(r_\lambda).$ In particular,
$$
  \dim \clV_\lambda = \dim \Upsilon_\lambda^1(r_\lambda).
$$
\end{thm}
\begin{proof}
Since~$F$ has trivial kernel, it follows from Theorem~\ref{T: res eq-n and eigenvectors} that~$F$
is an injective linear mapping from $\clV_\lambda$ to~$\Upsilon_\lambda^1(r_\lambda).$
Hence, it has to be shown that~$F$ maps $\clV_\lambda$ onto~$\Upsilon_\lambda^1(r_\lambda).$
Let $u \in \Upsilon_\lambda^1(r_\lambda),$ so that~$u$ satisfies~(\ref{F: res eq 1}).
By Lemma~\ref{L: Upsilon(1) subset F(hilb)}, there exists $\chi  \in \hilb$ such that $u = F\chi.$
Since~$V$ is a regularizing direction and since~$\lambda$ is an isolated eigenvalue,
the resolvent $(H_r-\lambda)^{-1}$ exists (as a bounded operator) for any non-resonant $r.$
Hence, the equation~(\ref{F: res eq 1}), which the vector~$u$ satisfies by definition,
can be written as
$$
  u +(r_\lambda-r) F R_{\lambda+i0}(H_{r})F^*Ju = 0,
$$
where $R_{\lambda+i0}(H_{r}) = (H_r-\lambda)^{-1}$ is a bounded operator.
Replacing $u$ by $F\chi$ gives the equality
$$
  F\chi +(r_\lambda-r) F R_{\lambda+i0}(H_{r})F^*JF\chi = 0.
$$
Since~$F$ has trivial kernel, it follows that
$$
  \chi +(r_\lambda-r) R_{\lambda+i0}(H_{r})F^*JF\chi = 0.
$$
Applying the operator~$H_r-\lambda$ to both sides of this equality gives
$$
  (H_r-\lambda)\chi +(r_\lambda-r) V\chi = 0.
$$
Thus,~$H_{r_\lambda}\chi = \lambda \chi,$ that is, $u = F\chi$ is an image of an eigenvector $\chi \in \clV_\lambda.$
\end{proof}
The statement of Theorem~\ref{T: if lambda notin ess sp, ...} is not final in the sense that the condition $\lambda \notin \sigma_{ess}$
in fact might be redundant. In this regard, see Conjecture~\ref{Conj: res eq-n and eigenvectors} from section \ref{S: open problems}.

\subsection{Multiplicity of singular spectrum}
\label{SS: mult of s.c. spectrum}
%

Theorem~\ref{T: if lambda notin ess sp, ...} implies that if~$\lambda$ does not belong
to the essential spectrum $\sigma_{ess},$ then the vector space~$\Upsilon_\lambda^1(r_\lambda) = \Upsilon_\lambda^1(H_{r_\lambda},V)$ does not depend on $V.$
This raises a natural question: is this statement true in general? It turns out that the answer to this question is positive.
This is the content of the following theorem.
This is a simple but interesting fact, since it allows to introduce multiplicity of singular spectrum at an essentially regular point~$\lambda$
as the dimension of the vector space $\Upsilon_\lambda^1(H_{r_\lambda},V).$

\begin{thm} \label{T: mult of s.c. spectrum: drastic version}
If~$H_{r_\lambda}$ is resonant at an essentially regular point~$\lambda,$
then the vector space
$$
  \Upsilon^1_{\lambda+i0}(r_\lambda) = \Upsilon_\lambda^1(H_{r_\lambda},V)
$$ does not depend on a regularizing operator~$V.$
\end{thm}
\begin{proof} To simplify formulas, without loss of generality we assume that~$r_\lambda = 0.$

Let $V=F^*JF$ and $V'=F^*J'F$ be two regularizing operators. We have to show that if a vector $u \in \clK$
satisfies the equation
\begin{equation} \label{F: first eq-n}
  [1 - FR_{\lambda+i0}(H_0+V)F^*J]u = 0,
\end{equation}
then $u$ also satisfies the equation
\begin{equation} \label{F: second eq-n}
  [1 - FR_{\lambda+i0}(H_0+V')F^*J']u = 0.
\end{equation}
We have for $y>0$
\begin{equation*}
  \begin{split}
    & FR_{\lambda+iy}(H_0+V)F^*Ju - FR_{\lambda+iy}(H_0+V')F^*J'u \\
    & \mbox{ }\quad = F\SqBrs{R_{\lambda+iy}(H_0+V) - R_{\lambda+iy}(H_0+V')}F^*Ju - FR_{\lambda+iy}(H_0+V')F^*[J'-J]u \\
    & \mbox{ }\quad = F\SqBrs{R_{\lambda+iy}(H_0+V')(V'-V)R_{\lambda+iy}(H_0+V)}F^*Ju - FR_{\lambda+iy}(H_0+V')F^*[J'-J]u \\
    & \mbox{ }\quad = FR_{\lambda+iy}(H_0+V')F^*(J'-J)\SqBrs{F R_{\lambda+iy}(H_0+V)F^*Ju - u}.
  \end{split}
\end{equation*}
Since $u$ satisfies~(\ref{F: first eq-n}), the expression in the last pair of square brackets vanishes as $y \to 0^+.$
Since $FR_{\lambda+iy}(H_0+V')F^*$ converges in norm as $y \to 0^+$ it follows that
$$
  FR_{\lambda+i0}(H_0+V)F^*Ju - FR_{\lambda+i0}(H_0+V')F^*J'u = 0.
$$
Adding this equality to~(\ref{F: first eq-n}) we obtain~(\ref{F: second eq-n}).
\end{proof}
Theorem~\ref{T: mult of s.c. spectrum: drastic version} allows us to consider the vector space~$\Upsilon^1_{\lambda+i0}(r_\lambda)$
as analogue of the vector space of eigenvectors when a point~$\lambda$ of singular spectrum belongs to the essential spectrum.

Later in section \ref{S: U-turn theorem} we will show that $\dim \Upsilon_{\lambda+i0}^1(r_\lambda)$ does not depend on the choice of the rigging operator~$F$ too.


%

\section{Resonance index}
\label{S: res index}
\subsection{$R$-index}
\label{SS: class clR}
\begin{defn} \label{D: R-index} \rm Let $\clK$ be a Hilbert space. The class $\clR = \clR(\clK)$\label{Page: clR} of operators consists of
all finite-rank operators $A\colon \clK \to \clK,$ which satisfy the following two conditions:
\begin{enumerate}
  \item The spectrum of~$A$ does not contain real numbers except zero: $\spectrum{A} \cap \mbR = \set{0}.$
  \item For any $f \in \clK$ the equality $A^2 f = 0$ implies $Af = 0.$
\end{enumerate}
By definition, the \emph{$R$-index}\label{Page: Rindex} of an operator~$A$ from the class~$\clR$ is the integer
$
  \Rindex(A) = N_+-N_-,
$
where~$N_+$ and~$N_-$ are the numbers of eigenvalues of~$A$ counting multiplicities
in~$\mbC_+$ and $\mbC_-$ respectively.
\end{defn}
The second condition in the definition of the class $\clR$ means that zero is an eigenvalue of order~1 for any operator~$A$ from~$\clR.$

If~$N$ is a positive integer, then $\clR_N$\label{Page: clR(N)} will denote the subset of~$\clR$ which consists of operators
of rank~$N.$ The union $\bigcup _{n \leq N}\clR_n$ will be denoted by $\clR_{\leq N}.$

\noindent A list of some elementary properties of the $R$-index is given in the following lemma.
\begin{lemma} \label{L: Rindex(AB)=Rindex(BA)} Let~$A$ and $B$ be two bounded operators
and let~$N$ be a positive integer.
\begin{enumerate}
  \item[(i)] If $AB$ and $BA \in \clR,$ then
  $
    \Rindex(AB) = \Rindex(BA).
  $
  \item[(ii)] If $A$ belongs to the class $\clR$ and if $S$ is a bounded invertible operator, then the operator~$S^{-1}AS$ also belongs to the class $\clR$ and
  $
    \Rindex(S^{-1}AS) = \Rindex(A).
  $
  \item[(iii)] If $A \in \clR,$ then also $A^* \in \clR$ and
  $
    \Rindex(A^*) = -\Rindex(A).
  $
  \item[(iv)] If $A \in \clR_N,$ then there exists a neighbourhood of~$A$ in $\clR_{\leq N},$ which is a subset of $\clR_N$ and such that
  $\Rindex(B) = \Rindex(A)$ for all $B$ from the neighbourhood. That is, the $R$-index is a locally constant function on $\clR_{N}.$

  \item[(v)] If $A \in \clR_N$ and if~$k$ is a non-negative integer, then there exists a neighbourhood of~$A$ in $\clR_{\leq N+k},$ such that
  $\abs{\Rindex(B) - \Rindex(A)} \leq k$ for all $B$ from the neighbourhood.

  \item[(vi)] If~$A$ and $B$ belong to~$\clR$ and if $AB = BA = 0,$ then $A+B$ also belongs to~$\clR$ and
  $
    \Rindex(A+B) = \Rindex(A) + \Rindex(B).
  $
\end{enumerate}
\end{lemma}
\begin{proof} (i) \ This equality follows from~(\ref{F: mu(AB)=mu(BA)}), which asserts that spectral measures of operators $AB$ and $BA$ coincide outside of zero.

(ii) It is easy to check that if $A \in \clR$ and $S$ is a bounded invertible operator, then~$S^{-1}AS \in \clR.$
Hence, the equality $\Rindex(S^{-1}AS) = \Rindex(A)$ follows from the item (i) applied to operators~$AS$ and~$S^{-1}.$

(iii) If~$A$ satisfies the first condition of the definition of the class~$\clR,$ then so does $A^*$ by~(\ref{F: mu(A*)=bar mu(A)}).
Since $A$ and $A^*$ are finite-rank we may assume that $\clK$ is finite-dimensional. In this case the second condition for $A^*$ also follows from~(\ref{F: mu(A*)=bar mu(A)}).
The equality of item (iii) follows from the equality~(\ref{F: mu(A*)=bar mu(A)}).

(iv) Small enough perturbations of~$A$ cannot decrease the rank of $A.$ Hence, a small enough neighbourhood $O$ of~$A$
in $\clR_{\leq N}$ is a subset of $\clR_{N}.$ The half-plane $\mbC_+$ or $\mbC_-$ to which an eigenvalue belongs is stable under small enough perturbations.
For any operator from the neighbourhood $O$ no other non-zero eigenvalues can emerge from zero, since this would increase the rank of $A.$
Thus, any operator~$B$ from a small enough neighbourhood has the same $R$-index as that of $A.$

(v) Small enough perturbations of~$A$ do not change the half-plane~$\mbC_\pm$ which the non-zero eigenvalues
of~$A$ belong to. Hence, if $B$ belongs to a small enough neighbourhood $O$ of~$A$ in $\clR_{\leq N+k}$ then,
since $\rank B \leq N+k,$ no more than~$k$ non-zero eigenvalues can emerge from zero as~$A$ is perturbed to $B.$
Therefore, the $R$-indices of~$A$ and $B$ may differ by no more than $k.$

(vi) Let $v$ be a root vector of order~$k$ corresponding to a non-zero eigenvalue~$\sigma$ of~$A,$
that is, $(A-\sigma)^k v = 0$ and $(A-\sigma)^{k-1}v \neq 0.$ The equality $BA = 0$ implies that
$0 = B(A-\sigma)^k v = \sigma^k B v,$ or $Bv = 0.$ Therefore, since~$A$ and $B$ commute,
$(B+A-\sigma)^k v = (A-\sigma)^k v = 0$ and $(B+A-\sigma)^{k-1} v = (A-\sigma)^{k-1} v \neq 0.$
It follows that a non-zero number $\sigma$ is an eigenvalue for~$A$ if and only if
it is also an eigenvalue of the same algebraic multiplicity for $A+B.$
The same assertion holds for $B$ instead of $A.$
Hence, the spectral measure of $A+B$ is the sum of spectral measures of~$A$ and $B$ which implies that $A+B$ satisfies the first condition.

If $(A+B)^2f=0,$ then $A^2f+B^2f=0;$ this implies $A^3f=0.$ Therefore, $A^2f=0$ and hence, $Af = 0.$
Similarly, $Bf = 0.$ Hence, $A+B$ satisfies the second condition too.

The equality $\Rindex(A+B) = \Rindex(A) + \Rindex(B)$ follows.
\end{proof}
It is easy to check that
$$
  \text{if} \ \ \Im z > 0 \ \ \text{then} \ \ \Im T_z(H) > 0.
$$
\begin{lemma} \label{L: star} Spectral measures of operators $R_z(H)V$ and $T_z(H)J$ coincide.
\end{lemma}
For bounded $F$ this follows from (\ref{F: mu(AB)=mu(BA)}); in general this can be seen from Lemma \ref{L: j-dimensions coincide}.
\begin{lemma} \label{L: RV in clR} If~$H$ is a self-adjoint operator and if~$V$ is a finite-rank self-adjoint operator
then for any non-real number~$z$ the operators $R_{z}(H)V$ and $T_{z}(H)J$ belong to the class $\clR.$
\end{lemma}
\begin{proof} We prove this for the operator $R_{z}(H)V$ only, since proof for $T_{z}(H)J$ is similar.
The operator $R_{z}(H)V$ is finite-rank and it satisfies the first condition of Definition~\ref{D: R-index} according to Lemma~\ref{L: Az Lemma 4.1.4}.
Let $f \in \hilb$ be such that $(R_{z}(H)V)^2 f = 0.$ Since the operator $R_{z}(H)$ has zero kernel, this implies $V R_{z}(H)V f = 0$
and $\scal{Vf}{R_{z}(H)V f} = 0.$ This equality implies $\scal{Vf}{R_{\bar z}(H)V f} = 0$ and thus
$\scal{Vf}{\Im R_{z}(H)V f} = 0.$ The operator $\Im R_{z}(H)$ is strictly positive if $\Im z > 0$ or is strictly negative if $\Im z < 0.$ Hence,
$\scal{Vf}{\Im R_{z}(H)V f} = 0$ implies $Vf=0.$
\end{proof}

The following theorem is proved in \cite{Kr53MS}. We give here a new proof of this theorem which is based on properties of the $R$-index
and which has topological character.
\begin{thm} \cite{Kr53MS} \label{T: Krein's thm}
If~$H$ is a self-adjoint operator and if~$V$ is a finite-rank self-adjoint operator, then for any $y = \Im z > 0$
the operator $R_z(H)V$ has exactly $\rank(V_\pm)$ eigenvalues in~$\mbC_\pm,$ where $V_+$ is the positive part of~$V$
and $V_-$ is the negative part of~$V.$ In particular,
$$
  \Rindex(R_{\lambda\pm iy}(H)V) = \pm \sign (V).
$$
\end{thm}
\begin{proof} By Lemma~\ref{L: RV in clR} the operator $R_z(H)V$ belongs to the class $\clR.$

(A) Assume first that either~$V$ or $-V$ is non-negative. Let~$N$ be the rank of~$V.$
Since the operator $R_{\lambda+iy}(H)$ has trivial kernel, dimension of the image of the product
$R_{\lambda+iy}(H)V$ is also equal to~$N.$ Hence, the product $R_{\lambda+iy}(H)V$ has~$N$
non-zero eigenvalues (counting multiplicities). That all these non-zero eigenvalues belong either to~$\mbC_+$
in the case of $V \geq 0$ or to $\mbC_-$ in the case of $V \leq 0$ follows from Lemma~\ref{L: Az Lemma 4.1.4}.

(B) If a finite rank self-adjoint operator~$V$ has at least one positive eigenvalue, then one of the positive eigenvalues of~$V$
can be continuously deformed so that it crosses through $0$ from $\mbR_+$ to $\mbR_-.$ For instance, if
$$
  V = \sum_{j=1}^N \alpha_j \scal{v_j}{\cdot}v_j
$$
is the Schmidt representation of~$V$ and $\alpha_1>0$ then the path of operators
$$
  V_t = (1-2t)\alpha_1 \scal{v_1}{\cdot}v_1 + \sum_{j=2}^N \alpha_j \scal{v_j}{\cdot}v_j, \ 0 \leq t \leq 1
$$
deforms the positive eigenvalue $\alpha_1$ to $-\alpha_1.$
By definition, the $R$-index of $R_z(H)V_t$ is constant before and after the moment the eigenvalue being deformed reaches zero.
According to item (v) of Lemma~\ref{L: Rindex(AB)=Rindex(BA)}, when the eigenvalue of~$V$ being deformed crosses through~$0$ to the other half-line,
the $R$-index of $R_z(H)V$ can change by no more than $2.$
According to part (A), if~$V$ is non-negative, then the $R$-index of $R_z(H)V$ is equal to~$N.$
When all eigenvalues of~$V$ become negative one by one as the operator~$V$ is deformed to a non-positive operator $-V,$ the $R$-index of $R_z(H)V$ has to become $-N.$
From this one can infer, that every time one positive eigenvalue of~$V$ crosses $0$ from $\mbR_+$ to $\mbR_-,$
the $R$-index of $R_z(H)V$ has to change by $-2.$ This completes the proof.
\end{proof}
\begin{cor} \label{C: to Krein's thm}
If~$H$ is a self-adjoint operator and if~$V$ is a finite-rank self-adjoint operator, then
for any $z$ with $\Im z >0$ and for any real $s$
$$
  \sign(J) = \sign(V) = \Rindex(T_z(H_s)J).
$$
\end{cor}
\begin{proof} The equality $\sign(J) = \sign(V)$ follows from Lemma \ref{L: sign(M)=sign(FMF)}.
The equality $\Rindex(T_z(H_s)J) = \Rindex(R_z(H_s)V)$ follows from Lemma \ref{L: star}.
Combining these equalities with Theorem \ref{T: Krein's thm} completes the proof.
\end{proof}

\subsection{Idempotents $P_z(r_\lambda)$ and $Q_z(r_\lambda)$}
\label{SS: P(z)(r lambda)}
Given a set of resonance points
$$
  \Gamma = \set{r_z^1, \ldots, r_z^n}
$$
corresponding to~$z\in \Pi,$ let
$$
  P_z(\Gamma) = P_z(r_z^1)+\ldots+P_z(r_z^n).
$$
It follows from~(\ref{F: Pz(1)Pz(2)=0}) that the operator $P_z(\Gamma)$ is an idempotent.
The operator $P_z(\Gamma)$ will be called the \emph{idempotent of a group of resonance points $\Gamma$}.
Similarly, one defines~$Q_z(\Gamma).$ The range of the operator $P_z(\Gamma)$ (respectively, $Q_z(\Gamma)$) will be denoted by $\Upsilon_z(\Gamma)$
(respectively, $\Psi_z(\Gamma)$).

We are mainly interested in the case when the number $z = \lambda\pm iy$ belongs to $\partial \Pi$
and the corresponding resonance point~$r_z = r_{\lambda}$ is real. If the point $z = \lambda+i0$ is slightly shifted
off the real axis, then the pole $s = r_\lambda$ of the meromorphic function $A_{z}(s)$ in general splits into several poles
\label{Page: rz1...rzN}
\begin{equation} \label{F: rz(1),...}
  r_{z}^1, \ldots, r_{z}^N,
\end{equation}
as schematically shown in the figure below.

\begin{picture}(400,90)
\put(60,40){\vector(1,0){140}}
\put(135,40){\circle*{4}}
\put(135,28){{$r_\lambda$}}
\put(30,73){{$s$-plane at $z=\lambda+i0$}}

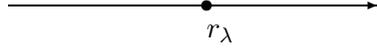
\put(290,40){\vector(1,0){140}}
\put(235,73){{$s$-plane at $z=\lambda+iy$}}
\put(235,59){{with $\abs{y}<\!\!\!<1$}}
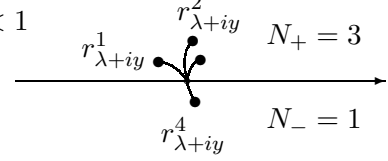
\put(385,54){{$N_+=3$}}
\put(385,22){{$N_-=1$}}
\put(355,40){\circle{2}}
\put(344,47){\circle*{4}}   \qbezier(355,40)(350,47)(344,47) 
\put(315,49){{$r^1_{\lambda+iy}$}}
\put(357,55){\circle*{4}}   \qbezier(355,40)(352,50)(357,55) 
\put(351,63){{$r^2_{\lambda+iy}$}}
\put(360,48){\circle*{4}}   \qbezier(355,40)(356,47)(360,48) 
\put(358,32){\circle*{4}}   \qbezier(355,40)(355,37)(358,32) 
\put(345,17){{$r^4_{\lambda+iy}$}}
\end{picture}
\vskip -0.5cm

\noindent
In these kind of figures the word $s$-plane means that the plane of the figure is the domain of values of the variable~$s.$
The poles (\ref{F: rz(1),...}) will be said to belong to the group of~$r_\lambda;$ the number of these poles (counted with their multiplicities) will be denoted by $N = N_+ + N_-,$
\label{Page: N pm} where $N_\pm$ is the number of poles in~$\mbC_\pm,$ --- for numbers~$z$ outside of $\partial \Pi$
the poles~$r_z^\nu, \ \nu=1,\ldots,N,$ cannot be real, according to Lemma~\ref{L: Az Lemma 4.1.4}.
We denote by $P_z(r_\lambda)=P_{\lambda+iy}(r_\lambda)$ the idempotent of the group of resonance points~(\ref{F: rz(1),...}):
\label{Page: Pz(rl)}
\begin{equation} \label{F: Pz(rl) = Pz(r1)+Pz(r2)+...}
  P_{z}(r_\lambda) = P_{z}(r_{z}^1) + \ldots + P_{z}(r_{z}^N).
\end{equation}
Similarly, $Q_{z}(r_\lambda)$ will denote the sum of idempotents $Q_{z}(r_z^\nu),$ $\nu = 1,\ldots, N:$ \label{Page: Qz(rl)}
\begin{equation} \label{F: Qz(rl)}
  Q_{z}(r_\lambda) = Q_{z}(r_{z}^1) + \ldots + Q_{z}(r_{z}^N).
\end{equation}
The range of the idempotent $P_{z}(r_\lambda)$ will be denoted by~$\Upsilon_z(r_\lambda)$ \label{Page: Upsilon(rl)}
and the range of the idempotent $Q_{z}(r_\lambda)$ will be denoted by~$\Psi_z(r_\lambda).$ \label{Page: Psi(rl)}

By $P^\uparrow_{\lambda+iy}(r_\lambda)$ we denote the sum of idempotents $P_{\lambda+iy}(r_{\lambda+iy}^\nu),$ for which the poles
$r_{\lambda+iy}^\nu$ belong to $\mbC_+,$ and similarly, \label{Page: P(up), P(down)}
the expression $P^\downarrow_{\lambda+iy}(r_\lambda)$ will denote the sum of idempotents $P_{\lambda+iy}(r_{\lambda+iy}^\nu),$ for which the poles
$r_{\lambda+iy}^\nu$ belong to~$\mbC_-.$
Similarly, one defines the operators
$Q^\uparrow_{\lambda+iy}(r_\lambda)$ and $Q^\downarrow_{\lambda+iy}(r_\lambda).$

We remark that a priori the idempotents $P_{\lambda+iy}(r_\lambda),$ $P^\uparrow_{\lambda+iy}(r_\lambda),$ etc,
are defined for small enough values of $y,$ depending on how far away the point $r_z$ as a function of $z$ can be continued analytically
(a hindrance for the analytic continuation of $r_z$ is that it can potentially get absorbed by $\infty$).

Similarly, one defines the operators $P_{\bar z}(r_\lambda),$ $Q_{\bar z}(r_\lambda)$
as idempotents of the group of resonance points of~$r_\lambda$ as $z = \lambda-i0$ is shifted to $z = \lambda-iy.$

In the following figures resonance points will be depicted by dark circles and anti-resonance points by light circles, as shown in the next figure
(see subsection \ref{SS: res index} for definition of anti-resonance points).
This figure shows poles of the group of~$r_\lambda$ for idempotents $P_{\lambda-iy}(r_\lambda)$ and $Q_{\lambda-iy}(r_\lambda).$

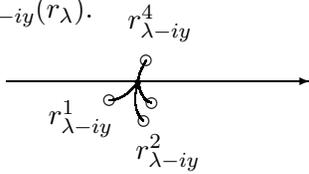
\begin{picture}(140,55)
\put(5,35){\vector(1,0){115}}
\put(55,35){\circle{2}}

\put(44,28){\circle{4}}
\qbezier(55,35)(50,28)(44,28) 
\put(21,19){{$r^1_{\lambda-iy}$}}

\put(57,20){\circle{4}}
\qbezier(55,35)(52,25)(57,20) 
\put(51,55){{$r^4_{\lambda-iy}$}}

\put(60,27){\circle{4}}
\qbezier(55,35)(56,28)(60,27) 

\put(58,43){\circle{4}}
\qbezier(55,35)(55,38)(58,43) 
\put(54,6){{$r^2_{\lambda-iy}$}}
\end{picture}

The following Proposition is Theorem 3.3 from \cite{Az7}. Here another proof of this proposition is given.
\begin{prop} \label{P: Th 3.3 of Az7} For any $z = \lambda\pm i0 \in \partial \Pi$ and any real resonance point~$r_\lambda$ corresponding to~$\lambda\pm i0,$
we have
$$
  \frac 1\pi \oint_{C(r_\lambda)} \Im T_{\lambda+iy}(H_s)J\,ds = P_{\lambda+iy}(r_{\lambda}) - P_{\lambda-iy}(r_{\lambda}),
$$
where $C(r_\lambda)$ is a contour which encloses all poles $r_{\lambda+iy}^1,\ldots,r_{\lambda+iy}^N$ of the group of~$r_\lambda$
and their conjugates $\bar r_{\lambda+iy}^1,\ldots,\bar r_{\lambda+iy}^N.$
\end{prop}
\begin{proof} Since
$$
  \frac 1\pi \Im T_{\lambda+iy}(H_s)J = \frac 1{2\pi i} \brs{A_{\lambda+iy}(s) - A_{\lambda-iy}(s)},
$$
the equality to be proved follows from the Laurent expansion~(\ref{F: Laurent for A+(s)}) of the function~$A_z(s)$ of~$s.$
\end{proof}

\begin{lemma} \label{L: P(l+iy) to P(l+i0)} For any real resonance point~$r_\lambda$
$$
  P_{\lambda \pm i0}(r_{\lambda}) = \lim_{y \to 0^+} P_{\lambda\pm iy}(r_{\lambda}) \quad \text{and} \quad
  Q_{\lambda\pm i0}(r_\lambda) = \lim_{y \to 0^+} Q_{\lambda\pm iy}(r_{\lambda}),
$$
where the limits are taken in the trace-class norm.
\end{lemma}
\begin{proof} It follows from the definition of the idempotent operator $P_{z}(r_{\lambda})$
that the operator $P_{\lambda\pm iy}(r_{\lambda})$ converges to $P_{\lambda\pm i0}(r_{\lambda})$ in the uniform norm.
By a well-known stability property of isolated eigenvalues,
for small enough $y$ the rank of the idempotent operator $P_{\lambda\pm iy}(r_{\lambda})$ is constant and is equal to the rank~$N$
of $P_{\lambda\pm i0}(r_{\lambda}).$ It follows that only the first~$N$ singular values of $P_{\lambda\pm iy}(r_{\lambda})$ can be non-zero.
Hence, the only first $2N$ $s$-numbers of the compact operator $P_{\lambda\pm iy}(r_{\lambda}) - P_{\lambda\pm i0}(r_{\lambda})$
can be non-zero.
This implies the estimate
\begin{equation*}
  \begin{split}
    \norm{P_{\lambda\pm iy}(r_{\lambda}) - P_{\lambda\pm i0}(r_{\lambda})}_1 & \leq \sum_{j=1}^{2N} s_j(P_{\lambda\pm iy}(r_{\lambda}) - P_{\lambda\pm i0}(r_{\lambda}))
     \\ & \leq 2N s_1(P_{\lambda\pm iy}(r_{\lambda}) - P_{\lambda\pm i0}(r_{\lambda}))
     \\ & \leq 2N \norm{P_{\lambda\pm iy}(r_{\lambda}) - P_{\lambda\pm i0}(r_{\lambda})},
  \end{split}
\end{equation*}
which shows that the trace class norm on the left hand side also converges to zero as $y \to 0.$
\end{proof}

Similarly to the definition of idempotents $P_z(r_\lambda)$ one can introduce nilpotent operators \label{Page: bfA(rl)}
\begin{equation} \label{F: bfAz(rl)}
  \bfA_{z}(r_\lambda) = \bfA_{z}(r_{z}^1) + \ldots + \bfA_{z}(r_{z}^N)
\end{equation}
and \label{Page: bfB(rl)}
\begin{equation} \label{F: bfBz(rl)}
  \bfB_{z}(r_\lambda) = \bfB_{z}(r_{z}^1) + \ldots + \bfB_{z}(r_{z}^N),
\end{equation}
where $r_{z}^1, \ldots, r_{z}^N$ are resonance points of the group of~$r_\lambda$ (see~(\ref{F: rz(1),...})).
It follows from~(\ref{F: bfAz(1)bfAz(2)=0}) and~(\ref{F: bfBz(1)bfBz(2)=0}) that the operators $\bfA_z(r_\lambda)$ and
$\bfB_z(r_\lambda)$ are nilpotent.

\begin{lemma} \label{L: A(l+iy) to A(l+i0)} The equalities
$$  \bfA_{\lambda \pm i0}(r_{\lambda}) = \lim_{y \to 0^+} \bfA_{\lambda\pm iy}(r_{\lambda}) \quad \text{and} \quad
  \bfB_{\lambda\pm i0}(r_\lambda) = \lim_{y \to 0^+} \bfB_{\lambda\pm iy}(r_{\lambda})
$$
hold, where the limits converge in trace-class norm.
\end{lemma}
\begin{proof} Since $A_z(s)$ converges to $A_{\lambda+i0}(s)$ in the uniform norm,
it follows from definitions~(\ref{F: bfAz(rl)}) and~(\ref{F: def of bfA})
that the limits above converge in the uniform norm. Hence, the claim
follows from the equality~(\ref{F: PA=AP=A}), Lemma~\ref{L: P(l+iy) to P(l+i0)} and the joint continuity of the mapping
$\clL_\infty \times \clL_1 \ni (A,B) \mapsto AB \in \clL_1.$
\end{proof}

From now on by $P_{\lambda \pm i0}(r_{\lambda}),$ $Q_{\lambda \pm i0}(r_{\lambda}),$ $\bfA_{\lambda \pm i0}(r_{\lambda})$ and
$\bfB_{\lambda \pm i0}(r_{\lambda})$ we mean operators defined in Lemmas \ref{L: P(l+iy) to P(l+i0)} and \ref{L: A(l+iy) to A(l+i0)}.

\subsection{Resonance index}
\label{SS: res index}
Let $z \in \Pi$ and let~$H_0$ and~$V$ be as usual. A resonance point~$r_z$ (see Definition~\ref{D: res point rz}) corresponding to~$z$ will be said to be
an \emph{up-point} \label{Page: up-point} (respectively, \emph{down-point}), if $\Im r_z >0$ (respectively, $\Im r_z < 0$).
Further, if~$r_z$ is an up-point corresponding to $z,$ then $\bar r_z$ will be called an \emph{anti-down-point}, corresponding to~$z$;
similarly, if~$r_z$ is a down-point of $z,$ then $\bar r_z$ will be called an \emph{anti-up-point} of~$z.$
Anti-up-points and anti-down-points of~$z$ will be called \emph{anti-resonance points} of~$z.$
By Corollary~(\ref{C: r(bar z)=bar r(z)}), for any $z \in \Pi$ resonance points of $\bar z$ are anti-resonance points of~$z$ and vice-versa.
In figures resonance points will be depicted by dark circles and anti-resonance points will be depicted by light circles.

If $z = \lambda+i0 \in \partial \Pi$ is an essentially regular point and if~$r_\lambda$ is a corresponding real resonance point,
then \emph{resonance index} \label{Page: res index} of a triple $(\lambda, H_{r_\lambda},V)$ will be defined as the difference of the number~$N_+$ of up-points and the number
$N_-$ of down-points, which belong to the group of~$r_\lambda,$ corresponding to $z = \lambda+iy$ with small enough $y>0.$
Resonance index of a triple $(\lambda, H_{r_\lambda},V)$ will be denoted by \label{Page: res(ind)}
\begin{equation} \label{F: res(ind)}
  \ind_{res}(\lambda; H_{r_\lambda},V).
\end{equation}
Given a real number $s,$ resonance index can also be defined as the difference of the number of eigenvalues $\sigma^\nu_{\lambda+iy}(s)$ in~$\mbC_+$
and the number of eigenvalues $\sigma^\nu_{\lambda+iy}(s)$ in $\mbC_-$  of the operator $A_{\lambda+iy}(s),$ which are obtained from the resonance points
of the group of~$r_\lambda$ for $z = \lambda+iy$ after the transformation $\sigma_z(s) = (s-r_z)^{-1},$ since this transformation maps the upper-plane to the upper-half plane
for any real $s.$ This is demonstrated by the following figure, where the label ``$s$-plane'' respectively ``$\sigma$-plane'',
means that the plane of the figure represents the range of values of the variable~$s$ respectively $\sigma.$
Thus, to calculate the resonance index $N_+-N_-$ one can use either of these two figures.

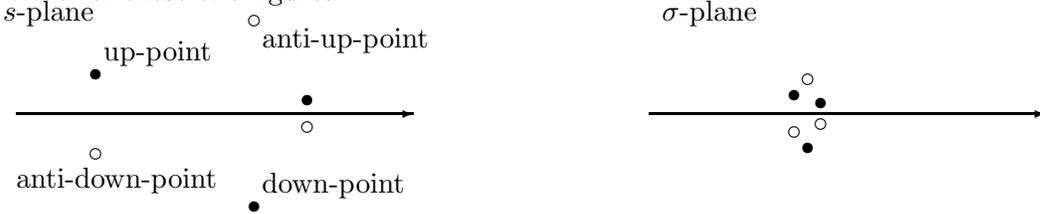
\begin{picture}(170,80)
\put(10,40){\vector(1,0){150}}
\put(5,75){{$s$-plane}}
\put(40,55){\circle*{4}}
\put(43,60){up-point}
\put(40,25){\circle{4}}
\put(10,12){anti-down-point}

\put(100,75){\circle{4}}
\put(103,65){anti-up-point}
\put(100,5){\circle*{4}}
\put(103,10){down-point}

\put(120,45){\circle*{4}}
\put(120,35){\circle{4}}
\end{picture}
\qquad\qquad\qquad
\begin{picture}(170,80)
\put(10,40){\vector(1,0){150}}
\put(15,75){{$\sigma$-plane}}
\put(65,47){\circle*{4}}
\put(65,33){\circle{4}}
\put(75,44){\circle*{4}}
\put(75,36){\circle{4}}

\put(70,53){\circle{4}}
\put(70,27){\circle*{4}}
\end{picture}

\begin{lemma} \label{L: ind-res=R(TJP)}
For any real resonance point~$r_\lambda,$
for any real number~$s$ and for all small enough $y>0$ the following equality holds:
\begin{equation} \label{F: ind-res=R(TJP) (+)}
  \ind_{res}(\lambda; H_{r_\lambda},V) = \Rindex(A_{\lambda+iy}(s)P_{\lambda+iy}(r_\lambda)).
\end{equation}
\end{lemma}
\begin{proof}
Let $\sigma^\nu_{\lambda+iy}(s) = (s-r^\nu_{\lambda+iy})^{-1}$ be an eigenvalue of $A_{\lambda+iy}(s),$
corresponding to a resonance point $r^\nu_{\lambda+iy}$ of the group of~$r_\lambda$ for $z = \lambda+iy.$
Let $u_{\lambda+iy,+}^1,\ldots, u_{\lambda+iy,+}^{N_+}$ and $u_{\lambda+iy,-}^1,\ldots, u_{\lambda+iy,-}^{N_-}$
be linearly independent root vectors of the operator $A_{\lambda+iy}(s) = T_{\lambda+iy}(H_s)J,$
such that the eigenvalue $\sigma^\nu_{\lambda+iy,\pm}$ corresponding to the vector $u^\nu_{\lambda+iy,\pm}$ lies in~$\mbC_\pm.$
Since a resonance point~$r_z$ and the point $\sigma_z(s) = (s-r_z)^{-1}$ belong to the same half-plane,
by definition of the resonance index it follows that
\begin{equation} \label{F: ind res=N-N}
  \ind_{res}(\lambda; H_{r_\lambda},V) = N_+-N_-.
\end{equation}
On the other hand, using~(\ref{F: Pz(1)Pz(2)=0}), we have $P_{\lambda+iy}(r_\lambda) u^\nu_{\lambda+iy,\pm} = u^\nu_{\lambda+iy,\pm},$
and therefore
$$
  \sigma^\nu_{\lambda+iy,\pm} u^\nu_{\lambda+iy,\pm} = A_{\lambda+iy}(s)u^\nu_{\lambda+iy,\pm} = A_{\lambda+iy}(s)P_{\lambda+iy}(r_\lambda) u^\nu_{\lambda+iy,\pm}.
$$
It follows that the operator $A_{\lambda+iy}(s)P_{\lambda+iy}(r_\lambda)$ has $N_\pm$ eigenvalues in~$\mbC_\pm,$ which implies that
$\Rindex(A_{\lambda+iy}(s)P_{\lambda+iy}(r_\lambda)) = N_+-N_-.$
Combining this with~(\ref{F: ind res=N-N}) completes the proof.
\end{proof}
Since resonance points~$r_z$ corresponding to~$z$ are anti-resonance points corresponding to $\bar z,$
the same argument shows that if $y>0,$ then
\begin{equation} \label{F: ind-res=R(TJP) (-)}
  \ind_{res}(\lambda; H_{r_\lambda},V) = - \Rindex(A_{\lambda-iy}(s)P_{\lambda-iy}(r_\lambda)).
\end{equation}
\noindent Further, Lemma~\ref{L: ind-res=R(TJP)}, combined with~(\ref{F: JP=QJ}) and Lemma~\ref{L: Rindex(AB)=Rindex(BA)}(i), imply that for $y>0$
\begin{equation*}
 \begin{split}
  \ind_{res}(\lambda; H_{r_\lambda},V) & = \Rindex(B_{\lambda+iy}(s)Q_{\lambda+iy}(r_\lambda))
  \\ & = - \Rindex(B_{\lambda-iy}(s)Q_{\lambda-iy}(r_\lambda)).
 \end{split}
\end{equation*}
\noindent
Definition of the resonance index can also be written in the form
\begin{equation*}
  \begin{split}
    \ind_{res}(\lambda; H_{r_\lambda},V) & = \Tr(P^\uparrow_{\lambda+iy}(r_{\lambda}) - P^\downarrow_{\lambda+iy}(r_{\lambda})) \\
    & = \rank(P^\uparrow_{\lambda+iy}(r_{\lambda})) - \rank(P^\downarrow_{\lambda+iy}(r_{\lambda})).
  \end{split}
\end{equation*}
\noindent
From Lemma~\ref{L: j-dimensions coincide} one can infer that $\Tr(P^\downarrow_{\lambda+iy}(r_{\lambda})) = \Tr(P^\uparrow_{\lambda-iy}(r_{\lambda}));$
hence, it also follows that
\begin{equation} \label{F: res.ind = Tr(P+up)-Tr(P-up)}
  \begin{split}
    \ind_{res}(\lambda; H_{r_\lambda},V) & = \Tr(P^\uparrow_{\lambda+iy}(r_{\lambda}) - P^\uparrow_{\lambda-iy}(r_{\lambda})) \\
     & = \rank(P^\uparrow_{\lambda+iy}(r_{\lambda})) - \rank(P^\uparrow_{\lambda-iy}(r_{\lambda})).
  \end{split}
\end{equation}
\noindent
According to Corollary~\ref{C: r(bar z)=bar r(z)}, up-points of~$z$ are anti-up-points of $\bar z$ and down-points of~$z$ are anti-down-points of~$\bar z.$
Let $C_+(r_\lambda)$ be a contour, which encloses in anticlockwise direction only up-points and anti-up-points of the group of~$r_\lambda,$
and, similarly, let $C_-(r_\lambda)$ be a contour, which encloses in anticlockwise direction only down-points and anti-down-points of the group of~$r_\lambda,$
as shown in the figure below.

\hskip 9 cm
\begin{picture}(180,50)
\put(113,30){{\small $C_+(r_\lambda)$}}
\put(10,20){\vector(1,0){160}}  
\put(82,32){\circle*{4}}        
\put(93,36){\circle{4}}         
\put(101,29){\circle*{4}}        
\put(92,28){\circle*{4}}        

\put(90,20){\circle*{2}}        

\put(90,23){\oval(40,40)[t]}    
\put(70,23){\line(1,0){40}}     

\put(75,23){\vector(1,0){8}}    
\put(95,23){\vector(1,0){8}}    

\put(82,8){\circle{4}}          
\put(93,4){\circle*{4}}         
\put(92,12){\circle{4}}         
\put(101,11){\circle{4}}         

\put(90,17){\oval(40,40)[b]}    
\put(70,17){\line(1,0){40}}     

\put(85,17){\vector(-1,0){8}}    
\put(105,17){\vector(-1,0){8}}    

\put(113,4){{\small $C_-(r_\lambda)$}}

\end{picture}

\begin{prop} \cite{Az7} \label{P: ind(res)=oint}
If $C_+(r_\lambda)$ and $C_-(r_\lambda)$
are contours as defined above, then for small enough $y>0$
\begin{equation} \label{F: indres = oint Tr(Im TJ)}
  \begin{split}
   \ind_{res}(\lambda; H_{r_\lambda},V) & = \frac 1\pi \Tr\brs{\oint_{C_+(r_\lambda)} \Im T_{\lambda+iy}(H_s)J\,ds}
   \\ & = - \frac 1\pi \Tr\brs{\oint_{C_-(r_\lambda)} \Im T_{\lambda+iy}(H_s)J\,ds}.
  \end{split}
\end{equation}
\end{prop}
\begin{proof}
By Proposition~\ref{P: Pz(rz)=res Az(s)} we have
\begin{equation*}
  \begin{split}
    & \frac 1\pi \oint_{C_+(r_\lambda)} \Im T_{\lambda+iy}(H_s)J\,ds
    \\ & \qquad = \frac 1{2\pi i} \oint_{C_+(r_\lambda)} \brs{A_{\lambda+iy}(H_s) - A_{\lambda-iy}(H_s)}\,ds
    \\ & \qquad = P^\uparrow_{\lambda+iy}(r_\lambda) - P^\uparrow_{\lambda-iy}(r_\lambda).
  \end{split}
\end{equation*}
This equality shows that the integral over $C_+(r_\lambda)$ is trace-class.
After taking traces of both sides of this equality, the first equality of (\ref{F: indres = oint Tr(Im TJ)}) now follows from (\ref{F: res.ind = Tr(P+up)-Tr(P-up)}).
The second equality is proved similarly.
\end{proof}

\section{Total resonance index as singular spectral shift function}
\label{S: res.index as s.ssf}
In this section we give a sketch of the proof of Theorem \ref{T: xis=ind res}, given in my unpublished paper \cite{Az7}.
This subsection is not used in the remaining part of this paper and it may be safely skipped. On the other hand,
results of this subsection provide one of the main motivations for this work.

Theorem \ref{T: xis=ind res} holds under a weaker relatively trace-class assumption which makes it applicable
to one-dimensional Schr\"odinger operators $H_0u(x) = -u''(x) + V_0(x)u(x)$ with bounded potentials~$V_0(x).$ Proof of this more general result relies
on an appropriate modification of the constructive approach to stationary scattering theory discussed in the introduction, see \cite{Az10}.
This modification is lengthy and therefore the proof has not been included here.

In this and only in this section we assume that the perturbation operator $V$ is trace-class.
This is achieved by assuming that the rigging operator $F$ is Hilbert-Schmidt.

Let
$$
  F_z(s) = \frac 1\pi \Tr \brs{\Im R_z(H_s)V} = \frac 1{2\pi i} \Tr\brs{A_z(s) - A_{\bar z}(s)}.
$$
The operator $\Im R_z(H_s)V$ is equal to $\frac 1{2i}(\ulA_z(s) - \ulA_{\bar z}(s))$ but the cyclic property $\Tr(AB) = \Tr(BA)$ of the trace allows to replace
the underlined operators by the non-underlined counter-parts.

\begin{lemma} Let~$\lambda$ be any number from the set $\Lambda(H_0,F)$ of full Lebesgue measure.
Assume that the interval $[a,b]$ of the real axis contains only one resonance point~$r_\lambda$ of the triple $(\lambda; H_0,V).$
Then
\begin{equation} \label{F: int(L2)=int(L1)+int(C+)}
  \int_{L_2} F_{\lambda+iy}(s) \,ds = \int_{L_1} F_{\lambda+iy}(s) \,ds + \ind_{res}(\lambda; H_{r_\lambda},V),
\end{equation}
where $L_1$ and $L_2$ are the contours of integration from~$a$ to~$b$ shown below; namely, the contour $L_2$ goes straight from~$a$ to~$b$
while the contour $L_1$ circumvents the resonance and anti-resonance points of the group of~$r_\lambda$ from above.\\
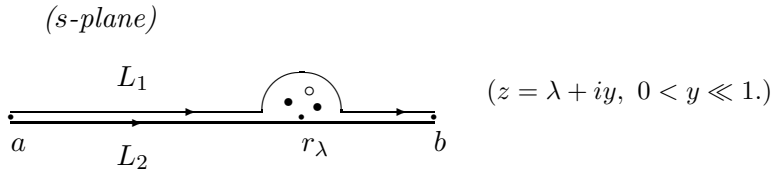
\begin{picture}(400,80)
\put(32,62){($s$-plane)}

\put(200,35){{\small $(z=\lambda+iy, \ 0<y\ll 1.)$}}

\put(60,40){$L_1$}

\put(60,10){$L_2$}

\put(20,26){\line(1,0){160}}
\put(20,26){\vector(1,0){50}}

\put(20,28){\circle*{2}}
\put(20,15){{$a$}}

\put(20,30){\vector(1,0){70}}
\put(20,30){\line(1,0){95}}
\put(145,30){\vector(1,0){25}}
\put(145,30){\line(1,0){35}}

\put(180,28){\circle*{2}}
\put(180,15){{$b$}}

\put(125,34){\circle*{3}}
\put(133,38){\circle{3}}
\put(136,32){\circle*{3}}
\put(130,28){\circle*{2}}

\put(130,30){\oval(30,30)[t]}
\put(130,15){{$r_\lambda$}}
\end{picture}
\end{lemma}
\begin{proof} By Cauchy's theorem, we have the following equality
$$
  \int_{L_2} F_{\lambda+iy}(s) \,ds = \int_{L_1} F_{\lambda+iy}(s) \,ds + \int_{C_+(r_\lambda)} F_{\lambda+iy}(s) \,ds,
$$
where the half-circle $C_+(r_\lambda)$ encloses all and only the resonance and anti-resonance points of the group of~$r_\lambda$ which are in~ $\mbC_+.$
So, the claim follows from Proposition \ref{P: ind(res)=oint}.
\end{proof}

\begin{prop} For a.e. $\lambda \in \mbR$
  $$
    \lim_{y \to 0^+} \int_0^1 F_{\lambda+iy}(s) \,ds = \xi(\lambda; H_1,H_0),
  $$
  where $\xi(\lambda; H_1,H_0)$ is the spectral shift function of the pair $(H_1,H_0).$
\end{prop}
This proposition is in essence the Birman-Solomyak formula (\ref{F: BS formula for xi}) for the spectral shift function.
The difference is that the Birman-Solomyak formula (\ref{F: BS formula for xi}) uses derivative of the distributive function of the spectral shift measure,
while in the formula above it is replaced by $\frac 1 \pi$ times the imaginary part of the limit of the Cauchy transform of the distributive function. By a well-known theorem of complex analysis,
these two functions are equal a.e. Details of the proof can be found in e.g. \cite[\S\S\ 9.5, 9.6]{Az3v6}.

\bigskip Now we discuss the absolutely continuous part of the spectral shift function $\xia(\lambda; H_1,H_0).$
By definition, the function $\xia(\lambda; H_1,H_0)$ is the density of the measure defined by formula
$$
  \Delta \mapsto \int_0^1 \Tr(V E_\Delta^{H_s}P^{(a)}(H_s))\,ds,
$$
where $P^{(a)}(H_s)$ is the (orthogonal) projection onto the absolutely continuous subspace of the self-adjoint operator $H_s.$

It was shown in \cite{Az3v6} that for a.e.~$\lambda$ the number $\xia(\lambda; H_s,H_0)$ is equal to
\begin{equation} \label{F: xia=int Tr(EVE)}
  \int_0^s \Tr_{\hlambda(H_r)}\brs{\euE_\lambda(H_r)V\euE_\lambda^\diamondsuit(H_r)}\,dr,
\end{equation}
where $\euE_\lambda(H_r) \colon \hilb_+ \to \hlambda(H_r)$ is the \emph{evaluation operator} defined by formula (\ref{F: euE(F*psi)=sqrt Im T}).
Since the operator $\euE_\lambda(H_r)$ was introduced in a recent and lengthy paper, the meaning of this formula may need some explanations.
Here $\hilb_+ = F^*\hilb$ is the rigging Hilbert space and $\hlambda(H_r)$ is the subspace of the auxiliary Hilbert space $\clK,$
defined by formula
$$
  \hlambda(H_r) = \overline{\im{\Im T_{\lambda+i0}(H_r)}}.
$$
It was shown in \cite{Az3v6} that $\hlambda(H_r)$ can be treated as the fiber Hilbert space. The operator $\euE_\lambda^\diamondsuit(H_r)$
acts from the Hilbert space $\hlambda(H_r)$ to the Hilbert space $\hilb_-$ which comes from the rigging~$F$; definition of the operator
$\euE_\lambda^\diamondsuit(H_r)$ will follow shortly.
The fact that the trace-class perturbation $V \colon \hilb \to \hilb$
admits factorization $V = F^*JF$ with Hilbert-Schmidt~$F$ and bounded~$J$
allows to treat $V$ as a bounded operator from $\hilb_-$ to $\hilb_+$ since $F$ can be treated as a unitary isomorphism $\hilb_- \stackrel\sim\to \clK$
and $F^*$ can be treated as a unitary isomorphism $\clK \stackrel\sim\to \hilb_+.$ These unitary isomorphisms can be denoted by the same $F$ and $F^*,$
but we will be pedantic for a moment and denote them as $\tilde F \colon \hilb_- \stackrel\sim\to \clK$ and $\tilde F^* \colon \clK \stackrel\sim\to \hilb_+.$
Now, the equality $V= F^*JF$ can be understood in several ways as shown in the following commutative diagram:

\begin{equation*} \label{F: V=tilde F* J tilde F}
  \xymatrix{
   \hilb_- \ar[dd]_{\tilde V} \ar[rr]^{\tilde F} && \clK \ar[dd]_{J} && \hilb \ar[ll]_F \ar[dd]^{V} \ar[rr]^{i_-} && \hilb_- \ar[dd]^{\tilde V} \ar@/_2.5pc/[llllll]_{\id}\\
                    &&                               && \\
   \hilb_+ \ar@/_2.5pc/[rrrrrr]_{\id} && \clK \ar[ll]^{\tilde F^*} \ar[rr]_{F^*} && \hilb && \hilb_+ \ar[ll]^{i_+}
  }
\end{equation*}
Here $i_\pm$ are the Hilbert-Schmidt inclusion operators. In the formula (\ref{F: xia=int Tr(EVE)}) the symbol $V$ denotes the bounded operator $\tilde V \colon \hilb_- \to \hilb_+.$
The operator $\euE_\lambda^\diamondsuit(H_r)$  acts from $\hlambda(H_r)$ to $\hilb_-$ according to the equality
$$
  \scal{\euE_\lambda^\diamondsuit(H_r)g}{f}_{-1,1} = \scal{g}{\euE_\lambda(H_r)f}, \ \ g \in \hlambda(H_r), \ f \in \hilb_+.
$$
This definition of $\euE_\lambda^\diamondsuit$ is equivalent to the equality
$$
  \euE_\lambda^\diamondsuit = \tilde F^{-1} (\tilde F^*)^{-1} \euE_\lambda^*,
$$
where $\euE_\lambda^*\colon \hlambda(H_r) \to \hilb_+$ is the usual adjoint and $\tilde F$ and $\tilde F^*$ are unitary isomorphisms shown in the diagram.
The product $\euE_\lambda(H_r)V\euE_\lambda^\diamondsuit(H_r)$ is trace-class since the operators $\euE_\lambda(H_r)$ and $\euE_\lambda^\diamondsuit(H_r)$
are Hilbert-Schmidt and the operator $V\colon \hilb_- \to \hilb_+$ is bounded.

Note that for any fixed point~$\lambda$ from the set $\Lambda(H_0,F)$ the operator
$\euE_\lambda(H_r)$ is defined for all non-resonant values of $r,$ according to the definition of this operator:
$$
  \euE_\lambda(H_r) \tilde F^* \psi = \frac 1\pi \sqrt{\Im T_{\lambda+i0}(H_r)}\psi, \ \ \psi \in \hilb.
$$
To avoid ambiguity, we write $\tilde V$ instead of~$V,$ when we treat $V$ as an operator $\colon \hilb_- \to \hilb_+.$
Note that, as the left square of the diagram above clearly shows, the operator $\tilde V$ is unitarily equivalent to~$J.$

The following proposition is proved in \cite{Az3v6}, see \cite[Corollary 7.3.5]{Az3v6}.
We give here a sketch of that proof.
\begin{prop} \label{P: Tr(EVE) is analytic} For any $\lambda \in \Lambda(H_0,F),$ the operator-valued function of $s \in \mbR$ defined by formula
$$
  s \mapsto \Tr_{\hlambda(H_r)}\brs{\euE_\lambda(H_r)\tilde V\euE_\lambda^\diamondsuit(H_r)}
$$
is analytic and admits holomorphic continuation to some neighbourhood of $\mbR.$
\end{prop}
\begin{proof} For any $\lambda \in \Lambda(H_0,F)$ and any real non-resonant $r$ the following equality holds:
\begin{equation} \label{F: wEVEw=S'S*}
  w_+^*(\lambda; H_r,H_0)\euE_\lambda(H_r)\tilde V\euE_\lambda^\diamondsuit(H_r)w_+(\lambda; H_r,H_0) = \brs{\frac d{dr} S(\lambda; H_r,H_0)}S^*(\lambda; H_r,H_0),
\end{equation}
where $S(\lambda; H_r,H_0)\colon \hlambda(H_0) \to \hlambda(H_0)$ is the scattering matrix and $w_+(\lambda; H_r,H_0)\colon \hlambda(H_0) \to \hlambda(H_r)$
is the wave matrix. According to \cite[\S 5 and \S 7]{Az3v6}, the right hand side is defined for all non-resonant values of the coupling constant~$r.$
According to \cite[Proposition 7.2.5]{Az3v6}, the scattering matrix $S(\lambda; H_r,H_0)$ is an analytic function of $r$ in the whole real axis $\mbR,$
and therefore so is the right hand side of the equality (\ref{F: wEVEw=S'S*}). It follows that the trace of the left hand side is also analytic.
Since $w_+(\lambda; H_r,H_0)$ is unitary, this trace is equal to
$$
  \Tr_{\hlambda(H_r)}\brs{\euE_\lambda(H_r)\tilde V\euE_\lambda^\diamondsuit(H_r)}
$$
\end{proof}
This proposition should not be surprising in the light of the general coupling constant regularity phenomenon observed first by Aronszajn back in 1957.

\begin{thm} \cite{Az3v6} For a.e.~$\lambda$
the absolutely continuous spectral shift function $\xia(\lambda; H_1, H_0)$ is equal to
\begin{equation} \label{F: int Tr(EVE)}
  \int_0^1\Tr_{\hlambda(H_r)}\brs{\euE_\lambda(H_r)\tilde V\euE_\lambda^\diamondsuit(H_r)}\,dr.
\end{equation}
\end{thm}
For proof see \cite[Lemma 8.2.1, Theorem 8.1.3]{Az3v6}.

Now we return to the equality (\ref{F: int(L2)=int(L1)+int(C+)}). It is not difficult to see that as $y \to 0^+$ the limit
$$
  L_1\mbox{-}\int_0^1 F_{\lambda+i0}(s) \,ds =  L_1\mbox{-}\int_0^1 \Tr_\clK\brs{\frac 1\pi \Im T_{\lambda+i0}(H_r) J} \,ds
$$
of the second integral over the contour $L_1$ exists, where $L_1$ indicates that all resonance points in the interval $[0,1]$
are circumvented in the upper half-plane.
\begin{lemma} For all~$\lambda$ from the set $\Lambda(H_0,F)$ of full Lebesgue measure this limit is equal to~(\ref{F: int Tr(EVE)}).
\end{lemma}
\begin{proof} By definition of $\euE_\lambda(H_r),$ for all non-resonant values of the coupling constant~$r$ the integrand of the integral (\ref{F: int Tr(EVE)})
is equal to
$$
  \Tr_{\hlambda(H_r)}\brs{\euE_\lambda(H_r)\tilde V\euE_\lambda^\diamondsuit(H_r)}
  = \Tr_{\hilb_-}\brs{\euE_\lambda^\diamondsuit(H_r)\euE_\lambda(H_r)\tilde V}
  = \Tr_{\clK}\brs{\frac 1\pi \Im T_{\lambda+i0}(H_r) J}.
$$
So, for small enough $s$ the integrals
$$
  L_1\mbox{-}\int_0^s F_{\lambda+i0}(r) \,dr \ \ \text{and} \ \ \int_0^s\Tr_{\hlambda(H_r)}\brs{\euE_\lambda(H_r)\tilde V\euE_\lambda^\diamondsuit(H_r)}\,dr
$$
are equal since their integrands are equal. We have to show that the integrals are equal for large values of $s$ too, in particular for $s=1.$
The second integral is holomorphic in some neighbourhood of $[0,1],$ since so is its integrand according to Proposition \ref{P: Tr(EVE) is analytic}.
If we show that the first integral is also holomorphic
in some neighbourhood of $[0,1]$ the proof will be complete by the uniqueness theorem for holomorphic functions.

The integrand of the first integral has singularities at resonance points from the interval $[0,1],$ but the integral of it is a single-valued function in
a neighbourhood of $[0,1]$ except maybe the resonance points, since whether we circumvent the resonance points from above or below the result of analytic continuation
will be the same according to the second equality of Proposition \ref{P: ind(res)=oint}.
Hence, the first and the second integrals are both holomorphic single valued functions in some neighbourhood of $[0,1]$  except a finite set of resonance points,
and both integrals coincide for small values of~$s.$ Hence, they coincide everywhere.
\end{proof}

Combining the results of this subsection,
we conclude that after taking the limit $y \to 0^+$ the equality (\ref{F: int(L2)=int(L1)+int(C+)}) with $a=0$ and $b=1$ turns into
$$
  \xi(\lambda; H_1,H_0) = \xia(\lambda; H_1,H_0) + \sum_{r_\lambda} \ind_{res}(\lambda; H_{r_\lambda}, V),
$$
where the sum is taken over all resonance points from $[0,1].$

Since $\xis(\lambda; H_1,H_0) = \xi(\lambda; H_1,H_0) - \xia(\lambda; H_1,H_0),$ this gives the following
\begin{thm} For a.e.~$\lambda$
$$
  \xis(\lambda; H_1,H_0) = \sum_{r_\lambda \in [0,1]} \ind_{res}(\lambda; H_{r_\lambda}, V).
$$
\end{thm}

%
%
%
%
%
%

\section{Signature of resonance matrix}
\label{S: sign of res matrix}
In this section we prove Theorem \ref{T: sign of res matrix for set of res points} which is one of the main technical results of this paper.
This theorem allows to express the signature of the finite-rank self-adjoint operator $Q_{\lambda-i0}(r_\lambda)J P_{\lambda+i0}(r_\lambda)$
which we call resonance matrix, in terms of the $R$-index of the operator $A_{\lambda+iy}(r_\lambda) P_{\lambda+iy}(r_\lambda).$

Assume that we are given a finite set
$$
  \Gamma = \set{r_z^1,\ldots,r_z^n}
$$
of resonance points corresponding to a fixed number $z \in \Pi.$
By $\bar \Gamma$ we denote the set $\set{\bar r_z^1,\ldots,\bar r_z^n}.$
The finite rank self-adjoint operator
$$
  Q_{\bar z}(\bar \Gamma)J P_z(\Gamma)
$$
will be called \emph{resonance matrix} of the set of resonance points $\Gamma.$ \label{Page: res matrix(1)}

Recall that a symmetric matrix $\alpha \in \mbC^{n\times n}$ is positive-definite, if for any non-zero $x \in \mbC^n$
$\scal{x}{\alpha x}>0.$ In particular, rank of a positive-definite matrix is equal to the dimension of vector space on which it acts.

\begin{lemma} \label{L: a bit technical lemma} Let $y>0.$ Let $M$ be a positive integer and let $d_1, \ldots, d_M$
be $M$ positive integers. Assume that we are given $M$ sets of vectors
\begin{equation} \label{F: system of sets}
  \chi_\mu^{(1)}, \ldots, \chi_\mu^{(d_\mu)}, \ \ \mu=1,\ldots,M
\end{equation}
from a pre-Hilbert space, such that all
$D:= d_1+d_2+\ldots+d_M$ vectors $\chi_\mu^{(j)}$ are linearly independent.
Assume further that we are given $M$ complex numbers
$$
  r_1, \ldots, r_M
$$
with positive imaginary parts.
Let $\beta$ be the positive-definite $D\times D$ matrix
\begin{equation} \label{b(mn)(kj)(0)=}
  \beta_{\mu\nu}^{kj} = \scal{\chi_\mu^{(k)}}{\chi_\nu^{(j)}}
\end{equation}
and define another $D\times D$ matrix $\alpha$ by recurrent formula
\begin{equation} \label{a(mn)(kj)(0)=}
  \alpha_{\mu\nu}^{kj} = \frac{2iy}{r_\nu - \bar r_\mu}\beta_{\mu\nu}^{kj} + \frac{1}{r_\nu - \bar r_\mu}\brs{\alpha_{\mu\nu}^{k-1,j} - \alpha_{\mu\nu}^{k,j-1}},
\end{equation}
where it is assumed that $\alpha_{\mu\nu}^{kj}=0$ if at least one of the indices $k$ or $j$ is equal to 0.
Then the matrix $\alpha$ is positive-definite.
\end{lemma}
\begin{proof}
Plainly, the matrix $\alpha$ is symmetric.

(A) Define recurrently a $D \times D$ matrix $\gamma$ with matrix elements
\begin{equation} \label{gamma(mn)(kj)=}
  \gamma_{\mu\nu}^{kj} = 2y\beta_{\mu\nu}^{kj} - i\brs{\gamma_{\mu\nu}^{k-1,j} - \gamma_{\mu\nu}^{k,j-1}},
\end{equation}
where it is assumed that $\gamma_{\mu\nu}^{kj}=0$ if at least one of the indices $k$ or $j$ is equal to 0.
We claim that $\gamma$ is positive definite. We prove this claim using induction on the positive integer
$$
  d = \max \set{d_1,\ldots,d_M}
$$
which will be called order. If $d=1$ then the second term in~(\ref{gamma(mn)(kj)=}) is zero and so in this case the claim
follows from positive definiteness of the matrix~(\ref{b(mn)(kj)(0)=}).
Now assuming that the claim holds for orders $<d$ we show that it holds for order $d.$

Rows of a $D\times D$ matrix will be enumerated by a pair of indices $(\mu,k)$ so that $(\mu,k) < (\nu,j)$
if and only if $\mu<\nu,$ or both $\mu=\nu$ and $k<j.$ A $D\times D$  matrix $\gamma$ can be looked at as composed of $M\times M$ cells, so that $(\mu,\nu)$ indicates a cell
and $(k,j)$ indicates an element of the cell. The second index $k$ in the pair $(\mu,k)$ denoting a row/column will be called order of the row/column.
The following figure shows the structure of a $D\times D$ matrix $\gamma.$

\newcommand{\pp}{\phantom{$\int\limits_0^1 f(x)\,dx$}}
\newcommand{\qq}[2]{
 $\left[
    \begin{matrix}
       \gamma_{#1#2}^{11} & \ldots & \gamma_{#1#2}^{1d_{#2}} \\
       \ldots & \ldots & \ldots \\
       \gamma_{#1#2}^{d_{#1}1} & \ldots & \gamma_{#1#2}^{d_{#1}d_{#2}}
    \end{matrix}
  \right]$
}
\newcommand{\rr}[4]{
 $\left[
    \begin{matrix}
       \gamma_{#1#2}^{11} & \ldots & \gamma_{#1#2}^{1#4} & \ldots & \gamma_{#1#2}^{1d_{#2}} \\
                   \ldots & \ldots & \ldots & \ldots & \ldots \\
                   \gamma_{#1#2}^{#31} & \ldots & \gamma_{#1#2}^{#3#4} & \ldots & \gamma_{#1#2}^{#3d_#2}  \\
                   \ldots & \ldots & \ldots & \ldots & \ldots \\
       \gamma_{#1#2}^{d_{#1}1} & \ldots & \gamma_{#1#2}^{d_#1#4} & \ldots & \gamma_{#1#2}^{d_{#1}d_{#2}}
    \end{matrix}
  \right]$
}

$$\small
\gamma =
\left[
\begin{tabular}{ccccc}
\qq{1}{1} & \ldots  & \ldots  & \ldots  & \qq{1}{M}  \\
  \ldots  & \ldots  & \ldots  & \ldots  & \ldots \\
  \ldots  & \ldots  & \rr{\mu}{\nu}{k}{j}  & \ldots  & \ldots \\
  \ldots  & \ldots  & \ldots  & \ldots  & \ldots \\
\qq{M}{1} & \ldots  & \ldots  & \ldots  & \qq{M}{M}  \\
\end{tabular}
\right]
$$

We apply to the matrix $\gamma$ the following elementary row and column operations:
if a row $(\mu,k)$ has order $k\geq 2,$ then we add to this row the previous row $(\mu,k-1)$ multiplied by~$i$
and if a column $(\nu,j)$ is such that its order $j\geq 2,$ then we add to this column the previous column $(\nu,j-1)$ multiplied by~$-i.$
We still have to specify in which order to execute these row and column operations. The rule is this: we start
with rows of largest orders~$d_\mu$ and finish with rows of order 2; the same rule applies to column operations.
If two rows have the same order then the corresponding row operations are interchangeable and so in this case
we don't need to specify order of these operations. Also, a row and a column operation are always interchangeable.
The following line explains what happens to a $2\times 2$ submatrix of the matrix $\gamma$ after a pair of a row and a column operations
(here for convenience the indices $k$ and $j$ are replaced by integers $3,3$):
$$
  \left(\begin{matrix}
     \gamma_{\mu\nu}^{22} & \gamma_{\mu\nu}^{23} \\
     \gamma_{\mu\nu}^{32} & \gamma_{\mu\nu}^{33}
  \end{matrix}
  \right)
  \rightarrow
  \left(\begin{matrix}
     \gamma_{\mu\nu}^{22} & \gamma_{\mu\nu}^{23} \\
     \gamma_{\mu\nu}^{32} + i\gamma_{\mu\nu}^{22} & \gamma_{\mu\nu}^{33} + i\gamma_{\mu\nu}^{23}
  \end{matrix}
  \right)
  \rightarrow
  \left(\begin{matrix}
     \gamma_{\mu\nu}^{22} & \gamma_{\mu\nu}^{23} - i \gamma_{\mu\nu}^{22}\\
     \gamma_{\mu\nu}^{32} + i\gamma_{\mu\nu}^{22} & \gamma_{\mu\nu}^{33} + i\gamma_{\mu\nu}^{23} - i\gamma_{\mu\nu}^{32} + \gamma_{\mu\nu}^{22}
  \end{matrix}
  \right).
$$
After performing other row and column operations this $2 \times 2$ block of the matrix $\gamma$ takes the form
$$
  \left(\begin{matrix}
     \gamma_{\mu\nu}^{22} + i\gamma_{\mu\nu}^{12} - i\gamma_{\mu\nu}^{21} + \gamma_{\mu\nu}^{11} & \gamma_{\mu\nu}^{23} + i\gamma_{\mu\nu}^{13} - i \gamma_{\mu\nu}^{22} + \gamma_{\mu\nu}^{12}\\
     \gamma_{\mu\nu}^{32} + i\gamma_{\mu\nu}^{22} - i\gamma_{\mu\nu}^{31} + \gamma_{\mu\nu}^{21} & \gamma_{\mu\nu}^{33} + i\gamma_{\mu\nu}^{23} - i \gamma_{\mu\nu}^{32} + \gamma_{\mu\nu}^{22}
  \end{matrix}
  \right).
$$
Now the formula~(\ref{gamma(mn)(kj)=}) implies that after these row and column operations having been performed in the specified order on the matrix
$\gamma$ it will take the form
$$
  2y\beta + \tilde \gamma,
$$
where the matrix $\tilde\gamma$ is obtained from $\gamma$ by the rule
$$
  \tilde\gamma_{\mu\nu}^{kj} = \left\{\begin{matrix} \gamma_{\mu\nu}^{k-1,j-1} & \text{if} \ k,j\geq 2, \\ 0 & \text{if otherwise.}\end{matrix}\right.
$$
This definition shows that after removing zero rows and columns the matrix $\tilde \gamma$ can be deemed as having been obtained by the same formula
(\ref{gamma(mn)(kj)=}) but using the system of sets of vectors
$$
  \chi_\mu^{(1)}, \ldots, \chi_\mu^{(d_\mu-1)}, \ \ \mu=1,\ldots,M.
$$
The order of this system of vectors is $d-1$ and therefore by induction assumption the (original with zero rows and columns) matrix $\tilde \gamma$ is non-negative.
Hence, the matrix $2y\beta + \tilde \gamma$ is positive definite, since so is the matrix $\beta.$
Finally, since the matrix $\gamma$ can be represented as $C(2y\beta + \tilde \gamma)C^*,$ where $C$ is the matrix corresponding to the row operations,
it follows that the matrix $\gamma$ itself is also positive definite.

(B) We have shown that for any system of sets of vectors~(\ref{F: system of sets}), the matrix $\gamma$ defined recurrently by formula~(\ref{gamma(mn)(kj)=})
is positive-definite.

Let $r_\mu = \rho_\mu+i\tau_\mu,$ where $\rho_\mu \in \mbR$ and $\tau_\mu>0.$
Now, given a positive number $p>0$ we define a system of sets of vectors
\begin{equation} \label{F: system of sets(p)}
  e^{-p(\tau_\mu-i\rho_\mu)}\chi_\mu^{(1)}, \ldots, e^{-p(\tau_\mu-i\rho_\mu)}\chi_\mu^{(d_\mu)}, \ \ \mu=1,\ldots,M.
\end{equation}
Using this system of sets of vectors, we construct the matrices $\beta(p)$ and $\gamma(p)$ by formulas
(\ref{b(mn)(kj)(0)=}) and~(\ref{gamma(mn)(kj)=}). According to part (A), the matrices $\gamma(p)$ are positive-definite for all $p>0.$
Hence, so is the matrix
$$
  \omega = \int_0^\infty \gamma(p)\,dp.
$$
Since $p>0$ and $\tau_\mu>0,$ this integral converges absolutely.
From~(\ref{gamma(mn)(kj)=}) it can be seen that
$$
  \gamma_{\mu\nu}^{kj}(p) = e^{-p(\tau_\mu+i\rho_\mu)} e^{-p(\tau_\nu-i\rho_\nu)} \gamma_{\mu\nu}^{kj}.
$$
Using this, we calculate the matrix element $\omega_{\mu\nu}^{kj}:$
$$
  \omega_{\mu\nu}^{kj} = \int_0^\infty \gamma_{\mu\nu}^{kj}(p)\,dp = \gamma_{\mu\nu}^{kj}\int_0^\infty e^{-p(\tau_\mu+\tau_\nu+i\rho_\mu-i\rho_\nu)}\,dp
  = \frac {\gamma_{\mu\nu}^{kj}}{\tau_\mu + \tau_\nu + i\rho_\mu - i\rho_\nu} = \frac{i\gamma_{\mu\nu}^{kj}}{r_\nu-\bar r_\mu}.
$$
Hence, a matrix with matrix elements $\omega_{\mu\nu}^{kj} = \frac{i\gamma_{\mu\nu}^{kj}}{r_\nu-\bar r_\mu}$ is positive definite.
Now comparing the recurrent formulas~(\ref{a(mn)(kj)(0)=}) and~(\ref{gamma(mn)(kj)=}) shows that $\alpha$ and $\omega$ are equal.
Hence, $\alpha$ is positive definite.
\end{proof}
As it can be seen from the proof, if the number $y$ is negative then the matrix $\alpha$ is negative-definite.

\begin{thm} \label{T: res matrix is positive for set of up-points} If $\Gamma = \set{r_z^1, \ldots, r_z^M}$ is
a finite set of resonance up-points corresponding to a non-real number~$z,$ then the operator
$$
  \Im z \, Q_{\bar z}(\bar \Gamma) J P_z(\Gamma)
$$
is non-negative and its rank is equal to the rank of $P_z(\Gamma).$
\end{thm}
\begin{proof} Without loss of generality we assume that $y = \Im z >0.$
By Lemma \ref{L: sign(M)=sign(FMF)}, the operators
$$
  Q_{\bar z}(\bar \Gamma) J P_z(\Gamma) \ \ \text{and} \ \ \ulQ_{\bar z}(\bar \Gamma) V \ulP_z(\Gamma)
$$
have equal ranks and signatures. So, it is sufficient to prove the claim for the latter operator.

(A) For notational convenience we assume that the same point $r_z^\mu$ may appear in the list
$r_z^1, \ldots, r_z^M$ more than one time. More exactly, each point $r_z^\mu$ appears in the list $m^\mu$ times,
where $m^\mu$ is the geometric multiplicity of $r_z^\mu.$
In what follows we often write $r_\mu$ instead of $r_z^\mu.$
For each point $r_z^\mu \in \Gamma$ let
$$
  \chi_{\mu}^{(j)}, \ j = 1,\ldots, d_\mu
$$
be a basis of $\Upsilon_z(r_z^\mu)$ such that $\ubfA_z(r_z^\mu) \chi_{\mu}^{(j)} = \chi_{\mu}^{(j-1)}.$
We can assume existence of such a basis since, as mentioned above, resonance points $r_z^\mu$ appear in the list according to their geometric multiplicities.
Let
\begin{equation*} 
  \alpha_{\mu\nu}^{kj} = \scal{\chi_\mu^{(k)}}{V\chi_\nu^{(j)}}
\end{equation*}
and
$$
  \beta_{\mu\nu}^{kj} = \scal{\chi_\mu^{(k)}}{\chi_\nu^{(j)}}.
$$
The following equality holds:
\begin{equation} \label{a(mn)(kj)=}
  \alpha_{\mu\nu}^{kj} = \frac{2iy}{r_\nu - \bar r_\mu}\beta_{\mu\nu}^{kj} + \frac{1}{r_\nu - \bar r_\mu}\brs{\alpha_{\mu\nu}^{k-1,j} - \alpha_{\mu\nu}^{k,j-1}}.
\end{equation}
Proof of~(\ref{a(mn)(kj)=}):

By Corollary~\ref{C: nasty proposition}, 
$$
  (H_{r_\nu}-z)\chi_\nu^{(j)} = - V\chi_\nu^{(j-1)}.
$$
It follows that
$$
  \scal{\chi_\mu^{(k)}}{(H_{r_\nu}-z)\chi_\nu^{(j)}} = - \scal{\chi_\mu^{(k)}}{V\chi_\nu^{(j-1)}}.
$$
In this equality we swap pairs of indices~$(\mu,k)$ and $(\nu,j)$ and then take conjugates of both sides of the resulting equality:
$$
  \scal{\chi_\mu^{(k)}}{(H_{\bar r_\mu}-\bar z)\chi_\nu^{(j)}} = - \scal{\chi_\mu^{(k-1)}}{V\chi_\nu^{(j)}}.
$$
Subtracting from this equality the previous one gives
$$
  \scal{\chi_\mu^{(k)}}{(-r_\nu V + \bar r_\mu V + z-\bar z)\chi_\nu^{(j)}} = - \scal{\chi_\mu^{(k-1)}}{V\chi_\nu^{(j)}} + \scal{\chi_\mu^{(k)}}{V\chi_\nu^{(j-1)}}.
$$
This can be written as
$$
  (r_\nu - \bar r_\mu)\scal{\chi_\mu^{(k)}}{V\chi_\nu^{(j)}} =
  (z-\bar z)\scal{\chi_\mu^{(k)}}{\chi_\nu^{(j)}} + \scal{\chi_\mu^{(k-1)}}{V\chi_\nu^{(j)}} - \scal{\chi_\mu^{(k)}}{V\chi_\nu^{(j-1)}},
$$
which is equivalent to~(\ref{a(mn)(kj)=}).

(B) Since vectors
$$
  \chi_{\mu}^{(j)}, \ j = 1,\ldots, d_\mu, \ \mu = 1,\ldots,M
$$
form a basis of the range of $\ulP_z(\Gamma),$ to prove the theorem it is enough to prove positive-definiteness of the matrix $\brs{\alpha_{\mu\nu}^{kj}}.$
But positive-definiteness of the matrix $\brs{\alpha_{\mu\nu}^{kj}}$ follows from Lemma~\ref{L: a bit technical lemma} and~(\ref{a(mn)(kj)=}).
\end{proof}

An analogue of Theorem~\ref{T: res matrix is positive for set of up-points} holds also for a set of resonance down-points.
Namely, if $\Gamma$ is a finite set of resonance down-points, then the
operator $\Im z \, Q_{\bar z}(\bar \Gamma) J P_z(\Gamma)$ is non-positive and its rank is equal to the rank of $P_z(\Gamma).$

\begin{thm} \label{T: sign of res matrix for set of res points} If $\Gamma = \set{r_z^1, \ldots, r_z^M}$ is
a finite set of resonance points corresponding to a non-real number~$z,$ then the signature of the finite-rank self-adjoint operator
$
  Q_{\bar z}(\bar \Gamma) J P_z(\Gamma)
$
is equal to the $R$-index of the operator
$
  \Im z \, A_z(s)P_z(\Gamma).
$
\end{thm}
\begin{proof} Without loss of generality we assume that $\Im z >0.$

Let $\Gamma = \Gamma^\uparrow \cup \Gamma^\downarrow,$ where $\Gamma^\uparrow \subset \mbC_+$ and $\Gamma^\downarrow \subset \mbC_-.$
Let $\Upsilon^\uparrow = \im(P_z(\Gamma^\uparrow))$ and $\Upsilon^\downarrow = \im(P_z(\Gamma^\downarrow)).$
The $R$-index of the operator $A_z(s)P_z(\Gamma)$ is equal to $N_+-N_-,$ where $N_+$ (respectively, $N_-$) is the sum of algebraic multiplicities of
all points from $\Gamma^\uparrow$ (respectively, $\Gamma^\downarrow$); that is,
$$
  \Rindex(A_z(s)P_z(\Gamma)) = N_+-N_- := \dim \Upsilon ^\uparrow - \dim \Upsilon ^\downarrow.
$$
For any non-zero $u \in \Upsilon^\uparrow$ we have
$$
  \scal{u}{Q_{\bar z}(\bar \Gamma) J P_z(\Gamma)u} = \scal{P_{z}(\Gamma)u}{ J P_z(\Gamma)u} = \scal{P_{z}(\Gamma^\uparrow)u}{ J P_z(\Gamma^\uparrow)u} > 0,
$$
where the last inequality follows from Theorem~\ref{T: res matrix is positive for set of up-points}.
Similarly, for any non-zero $u \in \Upsilon^\downarrow$ we have
$$
  \scal{u}{Q_{\bar z}(\bar \Gamma) J P_z(\Gamma)u} = \scal{P_{z}(\Gamma)u}{ J P_z(\Gamma)u} = \scal{P_{z}(\Gamma^\downarrow)u}{ J P_z(\Gamma^\downarrow)u} < 0.
$$
Hence, by Lemma~\ref{L: finite-rank lemma}, rank of the positive (respectively, negative) part of $Q_{\bar z}(\bar \Gamma) J P_z(\Gamma)$ is at least $N_+$ (respectively, $N_-$).
Hence, the rank of $Q_{\bar z}(\bar \Gamma) J P_z(\Gamma)$ is at least $N_++N_-=N:=\rank(P_z(\Gamma)),$ and therefore the rank of $Q_{\bar z}(\bar \Gamma) J P_z(\Gamma)$ is equal to $N.$
It follows that in fact the rank of the positive (respectively, negative) part of $Q_{\bar z}(\bar \Gamma) J P_z(\Gamma)$ is equal to $N_+$ (respectively, $N_-$).
Thus, the signature of $Q_{\bar z}(\bar \Gamma) J P_z(\Gamma)$ is equal to $N_+- N_-.$
\end{proof}

Theorem~\ref{T: sign of res matrix for set of res points} is the main ingredient of the proof of Theorem \ref{T: res.ind=sign res.matrix},
which asserts that the resonance index can be treated as signature of a certain finite-rank self-adjoint operator.

We remark that Theorems~\ref{T: res matrix is positive for set of up-points}
and \ref{T: sign of res matrix for set of res points} hold also in a finite-dimensional Hilbert space,
that is, for a pair of self-adjoint matrices $H_0$ and $V.$ Still, even this special case
of these theorems is non-trivial. The finite-dimensional versions of
Theorems~\ref{T: res matrix is positive for set of up-points} and \ref{T: sign of res matrix for set of res points}
can be tested in numerical experiments. Such a testing was carried out by the author using MATLAB and it confirms both theorems.

Theorem~\ref{T: res matrix is positive for set of up-points} has the following corollary.
\begin{cor} \label{C: Qbar is non-degen} Let $z$ be a non-real number. For any finite set of resonance up-points $\Gamma$
corresponding to $z$ the mapping
$$
  Q_{\bar z}(\bar \Gamma)\colon \Psi_z(\Gamma) \to \Psi_{\bar z}(\bar \Gamma)
$$
is a linear isomorphism.
\end{cor}
\begin{proof} Assume the contrary. Then, since dimensions of vectors spaces $\Psi_z(\Gamma)$ and $\Psi_{\bar z}(\bar \Gamma)$
are finite and coincide, for some non-zero $\psi \in \Psi_z(\Gamma)$ we have $Q_{\bar z}(\bar \Gamma)\psi = 0.$
By Lemma~\ref{L: j-dimensions coincide}, there exists a non-zero $u \in \Upsilon_z(\Gamma)$ such that $\psi = J u.$ It follows that
$$
  \scal{u}{Q_{\bar z}(\bar \Gamma)JP_z(\Gamma)u} = \scal{u}{Q_{\bar z}(\bar \Gamma)Ju} = 0.
$$
This contradicts Theorem~\ref{T: res matrix is positive for set of up-points}.
\end{proof}
\begin{cor} \label{C: Pbar is non-degen} Let $z$ be a non-real number.
For any finite set of resonance up-points $\Gamma$ corresponding to $z$ the mapping
$$
  P_{\bar z}(\bar \Gamma)\colon \Upsilon_z(\Gamma) \to \Upsilon_{\bar z}(\bar \Gamma)
$$
is a linear isomorphism.
\end{cor}
\begin{proof} This follows from Lemma~\ref{L: j-dimensions coincide} and previous corollary.
\end{proof}
\noindent These corollaries hold for a finite set of down-points too, of course.
Similarly, for any finite set~$\Gamma$ of resonance points from $\mbC_+$ or $\mbC_-$ the mappings
$$
  Q_z(\Gamma)\colon \Psi_{\bar z}(\bar \Gamma) \to \Psi_z(\Gamma) \ \ \text{and} \ \
  P_z(\Gamma)\colon \Upsilon_{\bar z}(\bar \Gamma) \to \Upsilon_z(\Gamma)
$$
are also linear isomorphisms.

\begin{cor}
For any finite set of resonance up-points $\Gamma$ and for any $j=1,2,\ldots$ the operator
$$
  \bfB^j_{\bar z}(\bar \Gamma) J \bfA^j_z(\Gamma)
$$
is non-negative and its rank is equal to the rank of $\bfA^j_z(\Gamma),$ where $\bfA^j_z(\Gamma) = \sum_{r_z \in \Gamma} \bfA_z^j(r_z).$
A similar inequality also holds with $j$ replaced by a multi-index.
\end{cor}
\noindent Indeed, since in this case $Q_{\bar z}(\bar \Gamma)J P_z(\Gamma) \geq 0,$ we have
$$
  \bfB^j_{\bar z}(\bar \Gamma) J \bfA^j_z(\Gamma) = \brs{\bfA^j_z(\Gamma)}^* \SqBrs{Q_{\bar z}(\bar \Gamma)J P_z(\Gamma)}\bfA^j_z(\Gamma) \geq 0.
$$

One could have suggested that if $\Gamma_1$ and $\Gamma_2$ are two finite sets of resonance up-points, such that
$\Gamma_1 \subset \Gamma_2,$ then
$$
  Q_{\bar z}(\bar \Gamma_1)J P_z(\Gamma_1) \leq Q_{\bar z}(\bar \Gamma_2)J P_z(\Gamma_2),
$$
but this is false.


\section{Vectors of type~I}
\label{S: vectors of type I}
In this section we study a subspace of the vector space $\Upsilon_{\lambda\pm i0}(r_\lambda)$ which consists of vectors with an additional property.
\begin{prop} \label{P: euE (Vf)=0, k>1} Let~$\lambda$ be an essentially regular point,
let $\set{H_0+rV \colon r \in \mbR}$ be a line regular at~$\lambda,$
let~$r_\lambda$ be a real resonance point of the path $\set{H_0+rV \colon r \in \mbR}$ at~$\lambda$ and
let~$k$ be a positive integer. If $u_{\lambda\pm i0}(r_\lambda) \in \Upsilon_{\lambda\pm i0}(r_\lambda)$ is a resonance vector of order
$k\geq 1$ at $\lambda\pm i0,$ then 
for all non-resonant values of~$s$ the following equality holds:
\begin{equation} \label{F: (J psi,Im TJ psi)=c(-2)/s2+...}
  \scal{J u_{\lambda\pm i0}(r_\lambda)}{\Im T_{\lambda \pm i0}(H_s)J u_{\lambda\pm i0}(r_\lambda)}
    = \frac {c_{\pm 2}}{(s-r_\lambda)^2} + \frac {c_{\pm 3}}{(s-r_\lambda)^3} + \ldots + \frac {c_{\pm k}}{(s-r_\lambda)^k},
\end{equation}
where, in case $k\geq 2,$ for $j = 2,\ldots,k$
\begin{equation} \label{F: c(pm j)}
 \begin{split}
   c_{\pm j} & = \Im \scal{u_{\lambda\pm i0}(r_\lambda)}{J\bfA_{\lambda \pm i0}^{j-1}(r_\lambda) u_{\lambda\pm i0}(r_\lambda)}
         \\ & = - \Im \scal{u_{\lambda\pm i0}(r_\lambda)}{J\bfA_{\lambda \mp i0}^{j-1}(r_\lambda) u_{\lambda\pm i0}(r_\lambda)}.
 \end{split}
\end{equation}
In particular, if $u_{\lambda\pm i0}(r_\lambda) \in \Upsilon_{\lambda\pm i0}(r_\lambda)$ is a resonance vector of order 1, then
\begin{equation} \label{F: (J psi,Im TJ psi)=0}
  \scal{J u_{\lambda\pm i0}(r_\lambda)}{\Im T_{\lambda \pm i0}(H_s)J u_{\lambda\pm i0}(r_\lambda)} = 0.
\end{equation}
\end{prop}
\begin{proof}
We give two proofs of~(\ref{F: (J psi,Im TJ psi)=c(-2)/s2+...}) but only in the second proof the formula~(\ref{F: c(pm j)})
for $c_{\pm j}$ will be derived. For brevity we write $u_\pm$ instead of $u_{\lambda\pm i0}(r_\lambda).$
Let
$$
  f_\pm(s) = \scal{J u_\pm}{A_{\lambda \pm i0}(s)u_\pm} = \scal{J u_\pm}{T_{\lambda \pm i0}(H_s)Ju_\pm}.
$$
By Theorem~\ref{T: sum prod psi = 0}, the vector $u_\pm$
satisfies~(\ref{F: boring formula}) with $z = \lambda \pm i0.$ Multiplying both sides of~(\ref{F: boring formula}) by $\scal{Ju_\pm}{\cdot},$
one finds that (recall that $\scal{\cdot}{\cdot}$ is linear in the second argument)
$$
  \sum_{j=1}^k (s_j-r_\lambda)^{k-1}\brs{\scal{Ju_\pm}{u_\pm} + (r_\lambda-s_j) f_\pm(s_j)} \prod_{i=1, i\neq j}^k (s_j-s_i)^{-1} = 0
$$
for all sets $s_1, \ldots, s_k$ of distinct real non-resonance points.
Taking the imaginary parts of both sides of this equality gives
$$
  \sum_{j=1}^k (s_j-r_\lambda)^{k}\Im f_\pm(s_j) \prod_{i=1, i\neq j}^k (s_j-s_i)^{-1} = 0.
$$
By Lemma~\ref{L: div-d diff-nce I}, the left hand side is the divided difference of order $k-1$ of the function $h(s) = (s-r_\lambda)^k \Im f_\pm(s).$
It follows from this and Lemma~\ref{L: div-d diff-nce II} that the function $h(s)$ is a polynomial of degree less or equal to $k-2.$
Hence, the function
$$
  \Im f_\pm(s) = \scal{J u_\pm}{\Im T_{\lambda \pm i0}(H_s)Ju_\pm}
$$
has the form~(\ref{F: (J psi,Im TJ psi)=c(-2)/s2+...}) with some numbers $c_{\pm 2},\ldots,c_{\pm k}.$
Here it is assumed that the function $\Im f_\pm(s)$ is defined by the right hand side of the equality above for real values of $s,$
and only after that it is continued analytically to the complex $s$-plane.

Second proof. We have
\begin{equation*}
 \begin{split}
    2i\Im f_\pm (s) & = \scal{J u_\pm }{T_{\lambda\pm  i0}(H_s)Ju_\pm } - \scal{T_{\lambda\pm i0}(H_s)J u_\pm }{Ju_\pm }.
 \end{split}
\end{equation*}
The Laurent expansion~(\ref{F: A psik=...}) of the function $T_{\lambda\pm  i0}(H_s)J$
implies that for real values of $s$ the Laurent expansion of the function $\Im f_\pm(s)$ at $s = r_\lambda$ is
\begin{equation*}
  \begin{split}
    \Im f_\pm (s) & = \frac 1{2i} \scal{J u_\pm }{\sum_{j=0}^{k-1} \frac {1}{(s-r_\lambda)^{j+1}} \bfA_{\lambda \pm i0}^{j}(r_\lambda) u_\pm }
      - \scal{\sum_{j=0}^{k-1} \frac {1}{(s-r_\lambda)^{j+1}} \bfA_{\lambda \pm i0}^{j}(r_\lambda) u_\pm }{Ju_\pm }
    \\ & = \frac 1{2i}\sum_{j=0}^{k-1} \frac {1}{(s-r_\lambda)^{j+1}} \SqBrs{\scal{Ju_\pm}{\bfA_{\lambda \pm i0}^{j}(r_\lambda) u_\pm} - \scal{\bfA_{\lambda \pm i0}^{j}(r_\lambda) u_\pm}{Ju_\pm}}
    \\ & = \sum_{j=1}^{k-1}  \frac {1}{(s-r_\lambda)^{j+1}} \Im \scal{Ju_\pm}{\bfA_{\lambda \pm i0}^{j}(r_\lambda) u_\pm}.
  \end{split}
\end{equation*}
Comparing the coefficients of $(s-r_\lambda)^{-j}$
in this Laurent series and in~(\ref{F: (J psi,Im TJ psi)=c(-2)/s2+...}) gives the equality
$$
  c_{\pm j} = \Im \scal{u_\pm}{J\bfA_{\lambda \pm i0}^{j-1}(r_\lambda) u_\pm}.
$$
To derive the second formula for $c_{\pm j}$ we note that
(\ref{F: A*(z)=B(bar z)}) and~(\ref{F: JA=BJ}) imply that for all $j=0,1,2,\ldots$
\begin{equation} \label{F: (*,*)=conj (*,*)}
 \begin{split}
  \scal{u_\pm}{J\bfA_{\lambda \pm i0}^{j}(r_\lambda) u_\pm} & =
  \scal{\bfB_{\lambda \mp i0}^{j}(r_\lambda)J u_\pm}{ u_\pm}
  \\ & = \scal{J\bfA_{\lambda \mp i0}^{j}(r_\lambda) u_\pm}{ u_\pm}
     = \overline{\scal{u_\pm}{J\bfA_{\lambda \mp i0}^{j}(r_\lambda) u_\pm}}.
 \end{split}
\end{equation}
Hence, $\Im \scal{u_\pm}{J\bfA_{\lambda \pm i0}^{j}(r_\lambda) u_\pm} = - \Im \scal{u_\pm}{J\bfA_{\lambda \mp i0}^{j}(r_\lambda) u_\pm}$ and therefore
$$
  c_{\pm j} = - \Im \scal{u_\pm}{J\bfA_{\lambda \mp i0}^{j-1}(r_\lambda) u_\pm}.
$$
\end{proof}
Since $\Im T_{\lambda - i0}(H_s) = - \Im T_{\lambda + i0}(H_s),$
it follows from~(\ref{F: (J psi,Im TJ psi)=c(-2)/s2+...}) that if $u \in \Upsilon_{\lambda+i0}^k(r_\lambda)$ or $u \in \Upsilon_{\lambda-i0}^k(r_\lambda),$
then
$$
  \scal{J u_{\lambda\pm i0}(r_\lambda)}{\Im T_{\lambda + i0}(H_s)J u_{\lambda\pm i0}(r_\lambda)}
  = \sum _{j=2}^k \Im \scal{u_{\lambda\pm i0}(r_\lambda)}{J\bfA_{\lambda + i0}^{j-1}(r_\lambda) u_{\lambda\pm i0}(r_\lambda)} (s-r_\lambda)^{-j}
$$

\begin{rems} \rm Since the left hand side of~(\ref{F: (J psi,Im TJ psi)=c(-2)/s2+...}) is non-negative (for plus sign)
or non-positive (for minus sign), it follows from~(\ref{F: (J psi,Im TJ psi)=c(-2)/s2+...}) that the largest $j$ for which $c_{\pm j} \neq 0$ must be even
and that
$$
  \Im \scal{u_{\lambda\pm i0}(r_\lambda)}{J\bfA_{\lambda + i0}(r_\lambda) u_{\lambda\pm i0}(r_\lambda)} \geq 0.
$$
\end{rems}

\begin{defn}
A vector $u \in \Upsilon_{\lambda\pm i0}(r_\lambda)$ will be said to be of \emph{type~I}, \label{Page: type I vector}
if for any non-resonant $s \in \mbR$
\begin{equation} \label{F: euE (Vf)=0, k=1}
  \sqrt{\Im T_{\lambda+i0}(H_s)}Ju = 0.
\end{equation}
\end{defn}
The equality~(\ref{F: euE (Vf)=0, k=1}) is equivalent to
$$
  \Im T_{\lambda+i0}(H_s) J u = 0.
$$
Since $\Im T_{\lambda+i0}(H_s) J = A_{\lambda+i0}(s) - A_{\lambda-i0}(s),$ this is also equivalent to
\begin{equation} \label{F: A(+)u=A(-)u}
  A_{\lambda+i0}(s)u = A_{\lambda-i0}(s)u.
\end{equation}

\begin{prop} \label{P: euE (Vf)=0, k=1} Every vector of order~$1$ is of type~I.
\end{prop}
\begin{proof}
This follows from Proposition~\ref{P: euE (Vf)=0, k>1},~(\ref{F: (J psi,Im TJ psi)=0}).
\end{proof}
\begin{lemma} \label{L: type I vectors} If an element $u$ of one of the two vector spaces $\Upsilon_{\lambda\pm i0}(r_\lambda)$ is a vector of type~I, then $u$
is also an element of the other vector space, that is, $u \in \Upsilon_{\lambda\mp i0}(r_\lambda),$
and orders of $u$ as an element of $\Upsilon_{\lambda+ i0}(r_\lambda)$ and $\Upsilon_{\lambda- i0}(r_\lambda)$ are the same.
\end{lemma}
\begin{proof} If for instance $u \in \Upsilon^k_{\lambda + i0}(r_\lambda),$ then by equality~(\ref{F: boring formula}) of Theorem~\ref{T: sum prod psi = 0}
one has
$$
  \sum_{j=1}^k (s_j-r_\lambda)^{k-1}\brs{u + (r_\lambda-s_j) A_{\lambda+i0}(s_j)u} \prod_{i=1, i\neq j}^k (s_j-s_i)^{-1} = 0,
$$
where $s_1, \ldots, s_k$ is any set of~$k$ distinct real non-resonance points.
If $u$ is a vector of type~I then the equality~(\ref{F: A(+)u=A(-)u}) holds, and therefore in the above equality
the term $A_{\lambda+i0}(s_j)u$ can be replaced by $A_{\lambda-i0}(s_j)u.$ By Theorem~\ref{T: sum prod psi = 0},
the resulting equality implies that $u$ belongs to $\Upsilon^k_{\lambda - i0}(r_\lambda).$
Similarly one shows that if $u \in \Upsilon^k_{\lambda - i0}(r_\lambda)$ is a vector of type~I, then $u \in \Upsilon^k_{\lambda + i0}(r_\lambda).$
Hence, orders of $u$ as elements of $\Upsilon_{\lambda - i0}(r_\lambda)$ and $\Upsilon_{\lambda + i0}(r_\lambda)$  are the same.
\end{proof}
Lemma~\ref{L: type I vectors} combined with Proposition~\ref{P: euE (Vf)=0, k=1} imply the following
\begin{cor} \label{C: Upsilon 1(+)=Upsilon 1(-)}
\begin{equation*}
  \Upsilon^1_{\lambda+i0}(r_\lambda) = \Upsilon^1_{\lambda-i0}(r_\lambda).
\end{equation*}
\end{cor}
By Lemma~\ref{L: j-dimensions coincide}, it follows that also
\begin{equation} \label{F: Psi 1(+)=Psi 1(-)}
  \Psi^1_{\lambda+i0}(r_\lambda) = \Psi^1_{\lambda-i0}(r_\lambda).
\end{equation}
Vectors of type~I form a vector subspace of both $\Upsilon_{\lambda\pm i0}(r_\lambda).$
It follows from~(\ref{F: A(+)u=A(-)u}) and~(\ref{F: bfA to j}) that if $u$ is a vector of type~I then for all $j=0,1,\ldots$
\begin{equation} \label{F: bfA j(+)=bfA j(-)}
  \bfA^j _{\lambda+i0}(r_\lambda) u = \bfA^j _{\lambda-i0}(r_\lambda) u.
\end{equation}
Therefore, it follows from~(\ref{F: Laurent for A+(s)}) and~(\ref{F: A(+)u=A(-)u})
that for vectors $u$ of type~I we have
$$
  \tilde A _{\lambda+i0,r_\lambda}(r_\lambda) u = \tilde A _{\lambda-i0,r_\lambda}(r_\lambda) u,
$$
which by~(\ref{F: tilde AP=0}) implies that for all $s$
$$
  \tilde A _{\lambda\pm i0, r_\lambda}(s) u = 0.
$$
On the other hand, if an element $u$ of the intersection $\Upsilon_{\lambda - i0}(r_\lambda) \cap \Upsilon_{\lambda + i0}(r_\lambda)$
is such that for all $j=0,1,2,\ldots$ the equality~(\ref{F: bfA j(+)=bfA j(-)}) holds
then by~(\ref{F: tilde AP=0}) we have $\tilde A _{\lambda+i0,r_\lambda}(r_\lambda) u = \tilde A _{\lambda+i0,r_\lambda}(r_\lambda) P_{\lambda+i0}(r_\lambda)u = 0$ and similarly
$\tilde A _{\lambda-i0,r_\lambda}(r_\lambda) u = 0,$ and therefore, it follows from the Laurent expansion~(\ref{F: Laurent for A+(s)}) of $A_z(s)$ that~(\ref{F: A(+)u=A(-)u}) holds.
Thus, the following lemma has been proved.
\begin{lemma} \label{L: type I u iff A(+)u=A(-)u} An element $u$ of $\Upsilon_{\lambda+i0}(r_\lambda)$ or $\Upsilon_{\lambda-i0}(r_\lambda)$
is a vector of type~I if and only if for all $j=0,1,2,\ldots$ the equality~(\ref{F: bfA j(+)=bfA j(-)}) holds.
\end{lemma}
\begin{cor} \label{C: u type I then so ia Au} If $u$ is a vector of type~I then so are the vectors $\bfA_{\lambda\pm i0}^j(r_\lambda)u$ for any $j=0,1,2,\ldots.$
\end{cor}
\vskip -8pt
\noindent In other words, the vector space of vectors of type~I is invariant with respect to $\bfA_{\lambda\pm i0}(r_\lambda).$
\begin{lemma} An element $u$ of $\Upsilon_{\lambda+i0}(r_\lambda)$ or $\Upsilon_{\lambda-i0}(r_\lambda)$
is a vector of type~I if and only if there exists a non-resonant real number $r$ such that for all $j=0,1,2,\ldots$
$$
  (A_{\lambda+i0}(r)-A_{\lambda-i0}(r)) \bfA^j _{\lambda+i0}(r_\lambda) u = 0.
$$
\end{lemma}
\begin{proof} 
(Only if) If $u$ is a vector of type~I then by Corollary \ref{C: u type I then so ia Au}
for any $j=0,1,2,\ldots$ the vectors~(\ref{F: bfA j(+)=bfA j(-)})
are also of type~I. Hence the equality to be proved follows from~(\ref{F: A(+)u=A(-)u}).

(If) It follows from the premise with $j=0$ that~(\ref{F: A(+)u=A(-)u}) holds for one non-resonant real number~$r.$
Let $s$ be any other non-resonant real number. Then by (\ref{F: Az3v6 (4.8)})
$$
  (A_{\lambda+i0}(s)-A_{\lambda-i0}(s)) u = (1+(s-r)A_{\lambda-i0}(r))^{-1} (A_{\lambda+i0}(r)-A_{\lambda-i0}(r)) (1+(s-r)A_{\lambda+i0}(r))^{-1}u.
$$
Using Proposition \ref{P: [1+sAz(r)](-1)Pz(rz)}, we can expand the factor $(1+(s-r)A_{\lambda+i0}(r))^{-1}u$ as a linear combination of $\bfA^j _{\lambda+i0}(r_\lambda) u.$
Hence, it follows from the premise that $(A_{\lambda+i0}(s)-A_{\lambda-i0}(s)) u = 0$ for any non-resonant~$s.$
That is, $u$ is a vector of type~I.
\end{proof}

The vector space of vectors of type~I will be denoted by $\Upsilon_{\lambda}^{\mathrm I}(r_\lambda).$
This notation is not ambiguous since,
according to Lemma~\ref{L: type I vectors}, one can omit the sign in the notation $\Upsilon_{\lambda\pm i0}^{\mathrm I}(r_\lambda)$
and write $\Upsilon_{\lambda}^{\mathrm I}(r_\lambda).$  Further, the vector subspaces
$\Upsilon^{k,\mathrm I}_{\lambda}(r_\lambda)$ are also correctly defined in the sense that
$$
  \Upsilon_{\lambda}^{\mathrm I}(r_\lambda) \cap \Upsilon^k_{\lambda+i0}(r_\lambda)
    = \Upsilon_{\lambda}^{\mathrm I}(r_\lambda) \cap \Upsilon^k_{\lambda-i0}(r_\lambda).
$$
We summarize the results of this section in the following
\begin{thm} \label{T: type I vectors} Let~$r_\lambda$ be a real resonance point of the line $\gamma = \set{H_r \colon r \in \mbR},$ corresponding to a real number $\lambda \in \Lambda(\gamma,F).$
Let $u \in \clK.$ The following assertions are equivalent:
\begin{enumerate}
  \item $u \in \Upsilon_{\lambda+i0}(r_\lambda)$ and for all non-resonant real numbers $s$
  $$
    \sqrt{\Im T_{\lambda+i0}(H_s)}Ju = 0.
  $$
  \item $u \in \Upsilon_{\lambda-i0}(r_\lambda)$ and for all non-resonant real numbers $s$
  $$
    \sqrt{\Im T_{\lambda+i0}(H_s)}Ju = 0.
  $$
  \item $u \in \Upsilon_{\lambda+i0}(r_\lambda)$ and for all non-resonant real numbers $s$
  $$
    A_{\lambda+i0}(s)u = A_{\lambda-i0}(s)u.
  $$
  \item $u \in \Upsilon_{\lambda-i0}(r_\lambda)$ and for all non-resonant real numbers $s$
  $$
    A_{\lambda+i0}(s)u = A_{\lambda-i0}(s)u.
  $$
  \item \label{Item: 5} $u \in \Upsilon_{\lambda+i0}(r_\lambda)$ and for all $j=0,1,2,\ldots,d-1,$ where $d$ is the order of~$r_\lambda,$
  $$
    \bfA^j _{\lambda+i0}(r_\lambda) u = \bfA^j _{\lambda-i0}(r_\lambda) u.
  $$
  \item \label{Item: 5'} $u \in \Upsilon_{\lambda-i0}(r_\lambda)$ and for all $j=0,1,2,\ldots,d-1,$ where $d$ is the order of~$r_\lambda,$
  $$
    \bfA^j _{\lambda+i0}(r_\lambda) u = \bfA^j _{\lambda-i0}(r_\lambda) u.
  $$
  \item $u \in \Upsilon_{\lambda+i0}(r_\lambda)$ and there exists a non-resonant real number $r$ such that for all $j=0,1,2,\ldots$
  $$
    (A_{\lambda+i0}(r)-A_{\lambda-i0}(r)) \bfA^j _{\lambda+i0}(r_\lambda) u = 0.
  $$
  \item  $u \in \Upsilon_{\lambda-i0}(r_\lambda)$ and there exists a non-resonant real number $r$ such that for all $j=0,1,2,\ldots$
  $$
    (A_{\lambda+i0}(r)-A_{\lambda-i0}(r)) \bfA^j _{\lambda-i0}(r_\lambda) u = 0.
  $$
  \item \label{Item: 8} $u \in \Upsilon_{\lambda+i0}(r_\lambda)$ and all the coefficients $c_{+j}$ from the equality (\ref{F: (J psi,Im TJ psi)=c(-2)/s2+...})
  are equal to zero.
  \item $u \in \Upsilon_{\lambda-i0}(r_\lambda)$ and all the coefficients $c_{-j}$ from the equality (\ref{F: (J psi,Im TJ psi)=c(-2)/s2+...})
  are equal to zero.
\end{enumerate}
The set $\Upsilon_{\lambda+ i0}^{\mathrm I}(r_\lambda)$ of vectors which satisfy any of these equivalent conditions is a vector subspace
of the vector space $\Upsilon_{\lambda+ i0}(r_\lambda) \cap \Upsilon_{\lambda - i0}(r_\lambda)$ and the vector space
$\Upsilon_{\lambda+ i0}^{\mathrm I}(r_\lambda)$ is invariant with respect to both $\bfA^j _{\lambda+i0}(r_\lambda)$ and $\bfA^j _{\lambda-i0}(r_\lambda).$
\end{thm}

\bigskip
It is an open question whether $\Upsilon_{\lambda+ i0}^{\mathrm I}(r_\lambda) = \Upsilon_{\lambda+ i0}(r_\lambda) \cap \Upsilon_{\lambda - i0}(r_\lambda).$

\begin{thm} \label{T: on vectors with property L} If a resonance vector $u^{(k)} \in \Upsilon_{\lambda\pm i0}(r_\lambda)$ has order $k$ then the vectors
$$
  u^{(1)}, \ldots, u^{(\lceil k/2\rceil)}
$$
are of type~I, where $\lceil k/2\rceil$ is the smallest integer not less than $k/2.$
\end{thm}
\begin{proof} We prove that $u^{(n)}$ is of type I for $n=1,2,\ldots, \lceil k/2\rceil,$ using induction on~$n.$
For $n=1$ this follows from Corollary \ref{C: Upsilon 1(+)=Upsilon 1(-)}.
Assume that all vectors $u^{(1)}, \ldots, u^{(n-1)},$ where $n \leq \lceil k/2\rceil,$ are of type~I.
We have to prove the claim for $u^{(n)}.$ Since $n \leq \lceil k/2\rceil,$ we have $2n-1 \leq k,$ so that
$$
  u^{(n)} = \bfA^{n-1}_{\lambda+i0}(r_\lambda) u^{(2n-1)}.
$$
For any $j=1,2,\ldots$ we have
\begin{equation*}
  \begin{split}
    \scal{Ju^{(n)}}{\bfA^j_{\lambda+i0}(r_\lambda)u^{(n)}} & = \scal{Ju^{(n)}}{u^{(n-j)}}
    \\ & = \scal{J\bfA^{n-1}_{\lambda+i0}(r_\lambda) u^{(2n-1)}}{u^{(n-j)}}
    \\ & = \scal{\bfB^{n-1}_{\lambda+i0}(r_\lambda) Ju^{(2n-1)}}{u^{(n-j)}}
    \\ & = \scal{Ju^{(2n-1)}}{\bfA^{n-1}_{\lambda-i0}(r_\lambda) u^{(n-j)}}.
  \end{split}
\end{equation*}
By the induction assumption, all the vectors $u^{(n-j)}, \ j =1,2,\ldots$ are of type~I and therefore,
according to items (\ref{Item: 5}) and (\ref{Item: 5'}) of Theorem \ref{T: type I vectors},
in the expression $\bfA^{n-1}_{\lambda-i0}(r_\lambda) u^{(n-j)}$
we can replace $\bfA^{n-1}_{\lambda-i0}(r_\lambda)$ by $\bfA^{n-1}_{\lambda+i0}(r_\lambda),$ and this shows that $\bfA^{n-1}_{\lambda-i0}(r_\lambda) u^{(n-j)} = 0.$
This means that for the vector $u^{(n)}$ the item (\ref{Item: 8}) of Theorem \ref{T: type I vectors} holds, and therefore it is of type~I.
\end{proof}

A resonance vector~$u \in \Upsilon_z(r_z)$ will be said to have \emph{depth} $k,$
\label{Page: depth of vector}
if~$u$ belongs to the image of the operator $\bfA_z^k(r_z),$ but not to the image of $\bfA_{z}^{k+1}(r_z).$
The depth of a vector~$u$ will be denoted by $\depth_z(u)$ or by $\depth(u)$ if there is no ambiguity.
In other words,
$$
  \depth_z(u) = \max\set{k \in \mbZ_+ \colon \exists\,\phi \in \clK \ \bfA_z^k \phi = u}.
$$
We say that a vector $u \in \Upsilon_{z}(r_z)$ has \emph{property}~$L,$ if
$$
  \order(u) \leq  \depth(u), \ \text{if} \ \order(u) + \depth(u) \ \text{is even}
$$
or
$$
  \order(u) \leq  \depth(u) + 1, \ \text{if} \ \order(u) + \depth(u) \ \text{is odd}.
$$
By $\clLs_{z}(r_z)$ we denote the linear span of all vectors $u$ with property~$L.$ \label{Page: clLs}

For example, if the Young diagram of the operator $\bfA_z(r_z)$ is as in the left figure, then one can easily prove that the vector space
$\clLs_{z}(r_z)$ is the linear span of those vectors in the right figure which are marked by bullets.

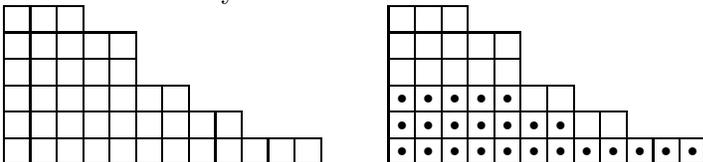
\begin{picture}(120,60)
\put(0,0){\line(0,1){60}}\put(10,0){\line(0,1){60}} \put(20,0){\line(0,1){60}} \put(30,0){\line(0,1){60}}
\put(40,0){\line(0,1){50}} \put(50,0){\line(0,1){50}} \put(60,0){\line(0,1){30}} \put(70,0){\line(0,1){30}}
\put(80,0){\line(0,1){20}} \put(90,0){\line(0,1){20}} \put(100,0){\line(0,1){10}} \put(110,0){\line(0,1){10}}
\put(120,0){\line(0,1){10}}
\put(0,0){\line(1,0){120}}\put(0,10){\line(1,0){120}}\put(0,20){\line(1,0){90}}\put(0,30){\line(1,0){70}}
\put(0,40){\line(1,0){50}}\put(0,50){\line(1,0){50}}\put(0,60){\line(1,0){30}}
\end{picture}
\qquad
\begin{picture}(120,60)
\put(0,0){\line(0,1){60}}\put(10,0){\line(0,1){60}} \put(20,0){\line(0,1){60}} \put(30,0){\line(0,1){60}}
\put(40,0){\line(0,1){50}} \put(50,0){\line(0,1){50}} \put(60,0){\line(0,1){30}} \put(70,0){\line(0,1){30}}
\put(80,0){\line(0,1){20}} \put(90,0){\line(0,1){20}} \put(100,0){\line(0,1){10}} \put(110,0){\line(0,1){10}}
\put(120,0){\line(0,1){10}}
\put(0,0){\line(1,0){120}}\put(0,10){\line(1,0){120}}\put(0,20){\line(1,0){90}}\put(0,30){\line(1,0){70}}
\put(0,40){\line(1,0){50}}\put(0,50){\line(1,0){50}}\put(0,60){\line(1,0){30}}
\put(5,5){\circle*{3}}\put(5,15){\circle*{3}}\put(5,25){\circle*{3}}
\put(15,5){\circle*{3}}\put(15,15){\circle*{3}}\put(15,25){\circle*{3}}
\put(25,5){\circle*{3}}\put(25,15){\circle*{3}}\put(25,25){\circle*{3}}
\put(35,5){\circle*{3}}\put(35,15){\circle*{3}}\put(35,25){\circle*{3}}
\put(45,5){\circle*{3}}\put(45,15){\circle*{3}}\put(45,25){\circle*{3}}
\put(55,5){\circle*{3}}\put(55,15){\circle*{3}}
\put(65,5){\circle*{3}}\put(65,15){\circle*{3}}
\put(75,5){\circle*{3}}
\put(85,5){\circle*{3}}
\put(95,5){\circle*{3}}
\put(105,5){\circle*{3}}
\put(115,5){\circle*{3}}
\end{picture}

Theorem \ref{T: on vectors with property L} implies that every vector with property~$L$ is of type~I.
Hence, we have the following
\begin{cor} The vector space $\clLs_{\lambda+i0}(r_\lambda)$ of vectors spanned by vectors with property~$L,$
is a subspace of the vector space $\Upsilon_{\lambda}^I(r_\lambda)$ of vectors of type~I:
$$
  \clLs_{\lambda+i0}(r_\lambda) \subset \Upsilon_{\lambda}^I(r_\lambda).
$$
\end{cor}
Similarly, one can define the vector space $\clLs_{\lambda-i0}(r_\lambda)$ which is also a subspace of $\Upsilon_{\lambda}^I(r_\lambda).$
The vector spaces $\clLs_{\lambda+i0}(r_\lambda)$ and $\clLs_{\lambda-i0}(r_\lambda)$ coincide.
Proof of this assertion will be given later elsewhere. The main part of the proof is a statement which asserts that $\lambda+i0$-depth of any vector of order 1
from $\Upsilon_\lambda(r_\lambda)$ coincides with $\lambda-i0$-depth of that vector.

If~$r_\lambda$ has order $\mbd = (d_1,\ldots,d_m),$ then
the dimension of $\clLs_{\lambda+i0}(r_\lambda)$ is equal to
$$
  \lceil d_1/2 \rceil + \ldots + \lceil d_m/2 \rceil.
$$

\section{Resonance index and signature of resonance matrix}
\label{S: res.index and sign res matrix}
In this section we prove one of the main results of this paper: equality of the resonance index and the signature of the resonance matrix.

The following theorem is one of the key properties of the idempotents $P_{\lambda\pm i0}(r_\lambda)$ which plays an important role in what follows.
Another proof of this theorem is given in Remark~\ref{R: another proof of property M}.
\begin{thm} \label{T: property M} The idempotents $P_{\lambda\pm i0}(r_\lambda)$ are linear isomorphisms of the vector spaces $\Upsilon_{\lambda\mp i0}(r_\lambda)$
and $\Upsilon_{\lambda\pm i0}(r_\lambda).$
\end{thm}
\begin{proof} Since by Lemma~\ref{L: j-dimensions coincide} the dimensions of the vector spaces $\Upsilon_{\lambda+ i0}(r_\lambda)$ and $\Upsilon_{\lambda- i0}(r_\lambda)$
coincide, it is enough to show that kernels of linear mappings $P_{\lambda\pm i0}(r_\lambda) \colon \Upsilon_{\lambda\mp i0}(r_\lambda) \to \Upsilon_{\lambda\pm i0}(r_\lambda)$
are zero. Assume the contrary, for example, there exists a non-zero $u \in \Upsilon_{\lambda + i0}(r_\lambda)$ such that
\begin{equation} \label{F: P-(u)=0}
  P_{\lambda - i0}(r_\lambda)u=0.
\end{equation}
Then it follows from~(\ref{F: (*,*)=conj (*,*)}) and~(\ref{F: PA=AP=A}) that for all $j=0,1,\ldots$
\begin{equation*}
  \begin{split}
     \scal{u}{J\bfA_{\lambda+i0}^j(r_\lambda) u} & = \overline{\scal{u}{J\bfA_{\lambda-i0}^j(r_\lambda) u}}
     \\ & = \overline{\scal{u}{J\bfA_{\lambda-i0}^j(r_\lambda) P_{\lambda - i0}(r_\lambda)u}} = 0.
  \end{split}
\end{equation*}
This equality combined with~(\ref{F: (J psi,Im TJ psi)=c(-2)/s2+...}) implies that $u$ is a vector of type~I. It follows from this
and Lemma~\ref{L: type I vectors} that $u \in \Upsilon_{\lambda - i0}(r_\lambda)$ and therefore $u = P_{\lambda - i0}(r_\lambda)u \neq 0.$
This contradicts~(\ref{F: P-(u)=0}).
\end{proof}
\noindent Thus, for any real resonance point~$r_\lambda$
$$
  P_{\lambda \pm i0}(r_\lambda)\Upsilon_{\lambda \mp i0}(r_\lambda) = \Upsilon_{\lambda\pm i0}(r_\lambda).
$$
This equality is equivalent to any of the following, which therefore also hold:
\begin{gather}
\label{F: Q Psi=Psi} Q_{\lambda \pm i0}(r_\lambda)\Psi_{\lambda \mp i0}(r_\lambda) = \Psi_{\lambda\pm i0}(r_\lambda), \\
\rank(P_{\lambda \pm i0}(r_\lambda)P_{\lambda \mp i0}(r_\lambda)) = N, \\
\rank(Q_{\lambda \mp i0}(r_\lambda)Q_{\lambda \pm i0}(r_\lambda)) = N,
\end{gather}
where $N=\rank(P_{\lambda \pm i0}(r_\lambda)) = \rank(Q_{\lambda \pm i0}(r_\lambda)).$

Lemma~\ref{L: j-dimensions coincide} and Theorem~\ref{T: property M} imply the following proposition.
\begin{prop} \label{P: rank M=rank P}
If $z = \lambda\pm i0 \in \partial \Pi$ and if~$r_\lambda$ is a real resonance point
corresponding to~$z,$ then
$$
  \rank Q_{\lambda\mp i0}(r_\lambda)JP_{\lambda\pm i0}(r_\lambda) = N,
$$
where~$N$ is the rank of (any of) operators $P_{\lambda\pm i0}(r_\lambda)$ and $Q_{\lambda\pm i0}(r_\lambda).$
\end{prop}

Note that Theorem~\ref{T: property M} is similar to Corollary~\ref{C: Pbar is non-degen}, but with an essential difference: while in Theorem~\ref{T: property M}
$z$ belongs to the boundary of $\Pi,$ in Corollary~\ref{C: Pbar is non-degen} it does not. At the same time, the methods of proof of these two assertions are completely different.

\begin{lemma} \label{L: sign(M)=sign(M(y))} If~$r_\lambda$ is a real resonance point then for all small enough $y > 0,$
$$
  \rank Q_{\lambda-i0}(r_\lambda)JP_{\lambda+i0}(r_\lambda) = \rank Q_{\lambda-iy}(r_\lambda)JP_{\lambda+iy}(r_\lambda)
$$
and
$$
  \sign Q_{\lambda-i0}(r_\lambda)JP_{\lambda+i0}(r_\lambda) = \sign Q_{\lambda-iy}(r_\lambda)JP_{\lambda+iy}(r_\lambda).
$$
\end{lemma}
\begin{proof} Sufficiently small (in norm) perturbations cannot decrease the rank of $Q_{\lambda-i0}(r_\lambda)JP_{\lambda+i0}(r_\lambda).$
Since the rank of the idempotent $P_{\lambda+i0}(r_\lambda)$
is stable under small enough perturbations, it follows from Proposition~\ref{P: rank M=rank P}
that the rank of the resonance matrix $Q_{\lambda-i0}(r_\lambda)JP_{\lambda+i0}(r_\lambda)$
cannot increase too. Thus, the first equality follows.
The second equality follows from the first one and continuity considerations, since in order to change the signature of $Q_{\lambda-iy}(r_\lambda)JP_{\lambda+iy}(r_\lambda)$
some non-zero eigenvalue of this operator has to be deformed to zero, which would violate the constancy of the rank.
\end{proof}

The following theorem is one of the main results of this paper.
\begin{thm} \label{T: res.ind=sign res.matrix} For any real resonance point~$r_\lambda$
the signatures of the resonance matrices $\sign(Q_{\lambda\mp i0}(r_\lambda)JP_{\lambda\pm i0}(r_\lambda))$
of~$r_\lambda$ are the same and are equal to the resonance index
of the triple $(\lambda,H_{r_\lambda},V);$ that is,
$$
  \sign(Q_{\lambda\mp i0}(r_\lambda)JP_{\lambda\pm i0}(r_\lambda)) = \ind_{res}(\lambda; H_{r_\lambda},V).
$$
\end{thm}
\begin{proof}
By Lemma \ref{L: sign(M)=sign(M(y))}  
for small enough $y>0$ we have the equality
$$
  \sign(Q_{\lambda\mp i0}(r_\lambda)JP_{\lambda\pm i0}(r_\lambda)) = \sign(Q_{\lambda\mp iy}(r_\lambda)JP_{\lambda\pm iy}(r_\lambda)).
$$
Hence, the claim follows from Theorem~\ref{T: sign of res matrix for set of res points},~(\ref{F: ind-res=R(TJP) (+)}) and~(\ref{F: ind-res=R(TJP) (-)}).
\end{proof}

Theorem~\ref{T: res.ind=sign res.matrix} implies the following corollary. Nonetheless, we give here another proof of it.
\begin{cor} \label{C: sign M-=sign M+}
For any real resonance point~$r_\lambda,$
the signatures of the finite-rank self-adjoint operators $Q_{\lambda- i0}(r_\lambda)JP_{\lambda+ i0}(r_\lambda)$
and $Q_{\lambda+ i0}(r_\lambda)JP_{\lambda- i0}(r_\lambda)$ coincide.
\end{cor}
\begin{proof} For any $y>0$ and any real $s$ by Corollary \ref{C: to Krein's thm}
$$
  (E) := \sign(Q_{\lambda- i0}(r_\lambda)JP_{\lambda+ i0}(r_\lambda)) = \Rindex(T_{\lambda+iy}(H_s) Q_{\lambda-i0}(r_\lambda)JP_{\lambda+i0}(r_\lambda)).
$$
By the stability of the $R$-index (Lemma~\ref{L: Rindex(AB)=Rindex(BA)}(iv)), for small enough $y'>0$ we get
$$
  (E) = \Rindex(T_{\lambda+iy}(H_s) Q_{\lambda-iy'}(r_\lambda)JP_{\lambda+iy'}(r_\lambda)).
$$
Since this $R$-index does not depend on $y>0,$ the number $y$ in the above equality can be replaced by $y'$ giving
\begin{equation*}
 \begin{split}
  (E) & =  \Rindex(T_{\lambda+iy'}(H_s) Q_{\lambda-iy'}(r_\lambda)JP_{\lambda+iy'}(r_\lambda))
   \\ & = \Rindex(P_{\lambda+iy'}(r_\lambda)T_{\lambda+iy'}(H_s) Q_{\lambda-iy'}(r_\lambda)J)
   \\ & = \Rindex(T_{\lambda+iy'}(H_s) Q_{\lambda+iy'}(r_\lambda)Q_{\lambda-iy'}(r_\lambda)J)
   \\ & = \Rindex(T_{\lambda+iy'}(H_s) Q_{\lambda+iy'}(r_\lambda)JP_{\lambda-iy'}(r_\lambda)).
 \end{split}
\end{equation*}
For small enough $y'$ we also have
\begin{equation*}
  \begin{split}
    \sign(Q_{\lambda+ i0}(r_\lambda)JP_{\lambda- i0}(r_\lambda)) & = \Rindex(T_{\lambda+iy'}(H_s) Q_{\lambda+ i0}(r_\lambda)JP_{\lambda- i0}(r_\lambda))
    \\ & = \Rindex(T_{\lambda+iy'}(H_s)Q_{\lambda+ iy'}(r_\lambda)JP_{\lambda- iy'}(r_\lambda)),
  \end{split}
\end{equation*}
so the proof is complete.
\end{proof}

\begin{lemma} \label{L: not unused lemma} Let $s$ be any real number. If~$r_\lambda$ is a real resonance point, then
there exists $\eps>0$ such that for all $y>0$ and for all $y' \in [0,\eps)$
the operator $T_{\lambda\pm iy}(H_s)Q_{\lambda\mp iy'}(r_\lambda)JP_{\lambda\pm iy'}(r_\lambda)$ belongs to the class~$\clR$ and
the following equality holds:
$$
  \sign Q_{\lambda-i0}(r_\lambda)JP_{\lambda+i0}(r_\lambda) = \Rindex(T_{\lambda+iy}(H_s)Q_{\lambda-iy'}(r_\lambda)JP_{\lambda+iy'}(r_\lambda)).
$$
\end{lemma}
\begin{proof} (A) Let~$z$ be a complex number from the upper complex half-plane $\mbC_+.$
For any finite-rank self-adjoint operator~$M$
\begin{equation} \label{F: 39852}
  \begin{split}
    \sign M & = \sign(F^*M F)       \hskip 2.9 cm \text{by Lemma~\ref{L: sign(M)=sign(FMF)}}
     \\ & = \Rindex(R_z(H_s)F^*M F) \hskip 2.1 cm \text{by \Krein's Theorem~\ref{T: Krein's thm}}
     \\ & = \Rindex(T_z(H_s)M).     \hskip 2.8 cm \text{by Lemma~\ref{L: Rindex(AB)=Rindex(BA)}(i)}
  \end{split}
\end{equation}

(B) By Lemma~\ref{L: sign(M)=sign(M(y))},
for all small enough $y>0$ the rank and signature of the operator
$Q_{\lambda-iy}(r_\lambda)JP_{\lambda+iy}(r_\lambda)$ are the same as those of $Q_{\lambda-i0}(r_\lambda)JP_{\lambda+i0}(r_\lambda).$
Combining this with~(\ref{F: 39852}), it can be concluded that for all small enough $y>0$
\begin{equation} \label{F: R(TM(l+i0))=R(TM(l+iy))}
  \Rindex(T_z(H_s)Q_{\lambda-i0}(r_\lambda)JP_{\lambda+i0}(r_\lambda)) = \Rindex(T_z(H_s)Q_{\lambda-iy}(r_\lambda)JP_{\lambda+iy}(r_\lambda)).
\end{equation}
Further, once $y$ is shifted away from $0,$ the variable~$z$ in the left hand side can be replaced by $\lambda+iy$ in $\mbC_+$ without changing
the $R$-index on the left hand side, since by Theorem~\ref{T: Krein's thm} both $R$-indices are equal to the signature of $Q_{\lambda-i0}(r_\lambda)JP_{\lambda+i0}(r_\lambda).$
For the same reason,~$z$ in the right hand side of~(\ref{F: R(TM(l+i0))=R(TM(l+iy))}) can also be replaced by $\lambda+iy$ in~$\mbC_+$ without changing the value of the right hand side.
Hence,
$$
  \Rindex(T_{\lambda+iy}(H_s)Q_{\lambda-i0}(r_\lambda)JP_{\lambda+i0}(r_\lambda)) = \Rindex(T_{\lambda+iy}(H_s)Q_{\lambda-iy}(r_\lambda)JP_{\lambda+iy}(r_\lambda)).
$$
Finally,~(\ref{F: 39852}) implies that the left hand side of this equality is equal to the signature of
the operator $Q_{\lambda-i0}(r_\lambda)JP_{\lambda+i0}(r_\lambda).$
\end{proof}


In the following theorem we collect together different descriptions of the resonance index.
\begin{thm} \label{T: numbers equal to res index}
Let~$r_\lambda$ be a real resonance point. The following numbers are all equal to each other.
\begin{enumerate}
  \item The resonance index $\ind_{res}(\lambda; H_{r_\lambda},V).$
  \item The signatures of the operators $Q_{\lambda\mp i0}(r_\lambda)JP_{\lambda\pm i0}(r_\lambda).$
  \item The $R$-index of the operator $T_{\lambda+ iy}(H_s)Q_{\lambda-iy}(r_\lambda)JP_{\lambda+iy}(r_\lambda)$ for all $s$ and for all small enough $y>0.$
  \item The $R$-index of the operator $-T_{\lambda- iy}(H_s)Q_{\lambda+iy}(r_\lambda)JP_{\lambda-iy}(r_\lambda)$ for all $s$ and for all small enough $y>0.$
  \item The $R$-index of the operator $A_{\lambda+ iy}(s)P_{\lambda + iy}(r_\lambda)$ for all $s$ and for all small enough $y>0.$
  \item The $R$-index of the operator $-A_{\lambda- iy}(s)P_{\lambda - iy}(r_\lambda)$ for all $s$ and for all small enough $y>0.$
\end{enumerate}
\end{thm}
\begin{proof}
Equality of the first two numbers (1) and (2) is the statement of Theorem~\ref{T: res.ind=sign res.matrix}.
Equality of the second and the third and the fourth numbers follows from Lemma~\ref{L: not unused lemma}.
The equalities (1) = (5) and (1) = (6) follow from~(\ref{F: ind-res=R(TJP) (+)}) and~(\ref{F: ind-res=R(TJP) (-)}) respectively.
\end{proof}

\section{U-turn theorem}
\label{S: U-turn theorem}
According to Lemma~\ref{L: j-dimensions coincide}, the four vector
spaces~$\Upsilon^1_{\lambda\pm i0}(r_\lambda)$ and~$\Psi^1_{\lambda\pm i0}(r_\lambda)$ have the same dimension.
It was noted in \S~\ref{SS: mult of s.c. spectrum} that dimension of the vector space
$\Upsilon^1_{\lambda+ i0}(r_\lambda)$ can be interpreted as multiplicity of a point~$\lambda$ of singular spectrum of a resonant at~$\lambda$
operator~$H_{r_\lambda}.$
Theorem~\ref{T: U-turn for res index} and Corollary~\ref{C: U-turn for positive V}, proved in this section,
provide another rationale towards this interpretation of the dimension of $\Upsilon^1_{\lambda+ i0}(r_\lambda).$
Given this definition of multiplicity of singular spectrum, how should one interpret the case when, for example, the dimension of
$\Upsilon^1_{\lambda\pm i0}(r_\lambda)$ is equal to~1, while the dimension~$N$ of~$\Upsilon_{\lambda\pm i0}(r_\lambda)$ is equal to~2?
Since $N = 2,$ there are two resonance points in the group of~$r_\lambda$ for small~$y.$
It is reasonable to suggest that these two poles should not belong to the same half-plane~$\mbC_\pm,$
since this would mean that the resonance index (=jump of spectral flow) is equal to two, while the multiplicity of the point~$\lambda$ of singular spectrum is~one.
That is, in this case we expect one up-pole and one down-pole, resulting in zero resonance index.
Outside of the essential spectrum, this scenario has an obvious geometric interpretation: a point of singular spectrum
(that is, an eigenvalue) reaches~$\lambda,$ but instead of crossing~$\lambda$ it turns back. Thus,
existence of vectors of order two or more should be interpreted as an indicator of the fact that some points of singular spectrum make a
``U-turn'' at~$\lambda.$ More generally,
if~$\dim \Upsilon^1_{\lambda+i0}(r_\lambda) = m,$ then it is natural to suggest that the
jump of spectral flow at $r = r_\lambda$ should not be larger than $m,$ since there are
only~$m$ ``eigenvalues'' which can cross the point~$\lambda$ as~$r$ crosses~$r_\lambda$ in the positive direction.

In other words, one may expect that the inequality
$$
  \abs{N_+-N_-} \leq \dim \Upsilon^1_{\lambda+ i0}(r_\lambda)
$$
should hold. This inequality (the U-turn theorem) turns out to be true
for all real resonance points~$r_\lambda,$ and is the main result of this section.

The U-turn theorem is non-trivial even for points~~$\lambda$ which do not belong to the essential spectrum.
For instance, a resonance with $N_+=5$ up-points and $N_-=2$ down-points depicted
below, may correspond to either of the following eight possible scenarios.
\begin{enumerate}
  \item As~$r$ crosses a real resonance point~$r_\lambda$ in the positive direction,
  five eigenvalues of~$H_r$ cross~$\lambda$ in the positive direction and two eigenvalues cross~$\lambda$ in the negative direction.
  Each of the five eigenvalues crossing~$\lambda$ in the positive direction create one up-point, and each of the two
  eigenvalues crossing~$\lambda$ in the negative direction create one down-point.
  \item Four eigenvalues cross~$\lambda$ in the positive direction, one eigenvalue cross~$\lambda$ in the negative direction, and one eigenvalue makes a U-turn at~$\lambda.$
  The eigenvalue making a U-turn, creates one up-point and one down-point.
  \item Three eigenvalues cross~$\lambda$ in the positive direction and two eigenvalues make a U-turn at~$\lambda.$ Each of the two eigenvalues making a U-turn, create
  one up-point and one down-point.
  \item Three eigenvalues cross~$\lambda$ in the positive direction, one eigenvalue crosses~$\lambda$ in the negative direction and one eigenvalue makes a double U-turn at~$\lambda.$
  The eigenvalue making a double U-turn, creates two up-points and one down-point.
  \item Three eigenvalues cross~$\lambda$ in the positive direction and one eigenvalue makes a triple U-turn at~$\lambda,$ which results in appearance of two up-points and two down-points.
  \item One eigenvalue crosses~$\lambda$ in the positive direction and two eigenvalues make a double U-turn at~$\lambda.$
  \item Two eigenvalues cross~$\lambda$ in the positive direction and one eigenvalue makes a quadruple U-turn at~$\lambda.$ An eigenvalue making a quadruple U-turn creates
  three up-points and two down-points.
  \item Four eigenvalues cross~$\lambda$ in the positive direction and one eigenvalue makes a triple U-turn at~$\lambda$ and crosses it in the negative direction.
  The eigenvalue making a U-turn, creates one up-point and two down-points.
\end{enumerate}
In these eight possible scenarios the dimension of the vector space~$\Upsilon^1_{\lambda+i0}(r_\lambda)$ is equal to, respectively, $7,6,5,5,4,3,3$ and $5.$

\vskip -1cm
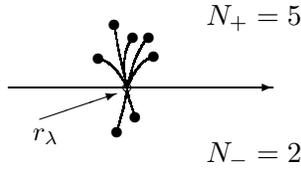
\label{Page: figures}
\begin{picture}(125,90)
\put(10,40){\vector(1,0){100}}
\put(85,64){{$N_+=5$}}
\put(85,12){{$N_-=2$}}
\put(55,40){\circle{3}}    
\put(44,51){\circle*{4}}   \qbezier(55,40)(50,50)(44,51) 
\put(50,64){\circle*{4}}   \qbezier(55,40)(52,52)(50,64) 
\put(57,59){\circle*{4}}   \qbezier(55,40)(52,52)(57,59) 
\put(63,59){\circle*{4}}   \qbezier(55,40)(58,54)(63,59) 
\put(65,52){\circle*{4}}   \qbezier(55,40)(59,50)(65,52) 
\put(58,29){\circle*{4}}   \qbezier(55,40)(55,37)(58,29) 
\put(51,23){\circle*{4}}   \qbezier(55,40)(55,37)(51,23) 

\put(22,28){\vector(3,1){29}}
\put(16,20){{\small~$r_\lambda$}}
\end{picture}
\hskip 3,5 cm
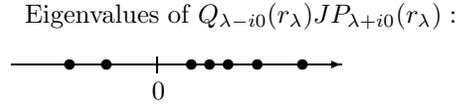
\begin{picture}(140,90)
\put(15,56){{\small Eigenvalues of $Q_{\lambda-i0}(r_\lambda)JP_{\lambda+i0}(r_\lambda):$}}
\put(10,40){\vector(1,0){125}}
\put(65,37){\line(0,1){6}}    
\put(63,27){{\small $0$}}
\put(32,40){\circle*{4}}
\put(46,40){\circle*{4}}
\put(78,40){\circle*{4}}
\put(85,40){\circle*{4}}
\put(92,40){\circle*{4}}
\put(103,40){\circle*{4}}
\put(120,40){\circle*{4}}
\end{picture}

\noindent A typical motion of the eigenvalues of the operator~$H_r$ as~$r$ passes through~$r_\lambda$ in each of these eight possible scenarios,
are given below.

\begin{picture}(140,85)
\put(15,76){\small 1st scenario: $\dim \Upsilon^1_{\lambda+i0}(r_\lambda) = 7$}
\put(10,40){\vector(1,0){125}}
\put(65,40){\circle{3}}    
\put(65,40){\line(0,-1){20}}
\put(63,9){$\lambda$}

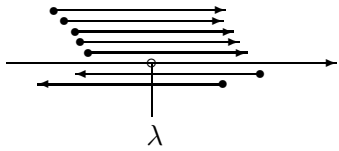
\put(106,36){\circle*{3}} \put(106,36){\vector(-1,0){70}}
\put(92,32){\circle*{3}} \put(92,32){\vector(-1,0){70}}

\put(41,44){\circle*{3}} \put(41,44){\vector(1,0){60}}
\put(38,48){\circle*{3}} \put(38,48){\vector(1,0){60}}
\put(36,52){\circle*{3}} \put(36,52){\vector(1,0){60}}
\put(32,56){\circle*{3}} \put(32,56){\vector(1,0){60}}
\put(28,60){\circle*{3}} \put(28,60){\vector(1,0){65}}
\end{picture}
\hskip 3 cm
\begin{picture}(140,85)
\put(15,76){{\small 2nd scenario: $\dim \Upsilon^1_{\lambda+i0}(r_\lambda) = 6$}}
\put(10,40){\vector(1,0){125}}
\put(65,40){\circle{3}}    
\put(65,40){\line(0,-1){20}}
\put(63,9){$\lambda$}

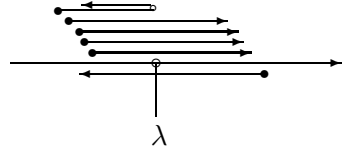
\put(106,36){\circle*{3}} \put(106,36){\vector(-1,0){70}}

\put(41,44){\circle*{3}} \put(41,44){\vector(1,0){60}}
\put(38,48){\circle*{3}} \put(38,48){\vector(1,0){60}}
\put(36,52){\circle*{3}} \put(36,52){\vector(1,0){60}}
\put(32,56){\circle*{3}} \put(32,56){\vector(1,0){60}}
\put(28,60){\circle*{3}} \put(28,60){\line(1,0){36}}
                         \put(63,62){\vector(-1,0){26}}
                         \put(64,61){\circle{2}}
\end{picture}

\begin{picture}(140,85)
\put(15,76){{\small 3rd scenario: $\dim \Upsilon^1_{\lambda+i0}(r_\lambda) = 5$}}
\put(10,40){\vector(1,0){125}}
\put(65,40){\circle{3}}    
\put(65,40){\line(0,-1){20}}
\put(63,9){$\lambda$}

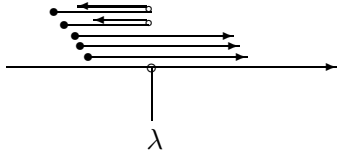
\put(41,44){\circle*{3}} \put(41,44){\vector(1,0){60}}
\put(38,48){\circle*{3}} \put(38,48){\vector(1,0){60}}
\put(36,52){\circle*{3}} \put(36,52){\vector(1,0){60}}
\put(32,56){\circle*{3}} \put(32,56){\line(1,0){32}}
                         \put(63,58){\vector(-1,0){20}}
                         \put(64,57){\circle{2}}

\put(28,61){\circle*{3}} \put(28,61){\line(1,0){37}}
                         \put(63,63){\vector(-1,0){26}}
                         \put(64,62){\circle{2}}
\end{picture}
\hskip 3 cm
\begin{picture}(140,85)
\put(15,76){{\small 4th scenario: $\dim \Upsilon^1_{\lambda+i0}(r_\lambda) = 5$}}
\put(10,40){\vector(1,0){125}}
\put(65,40){\circle{3}}    
\put(65,40){\line(0,-1){20}}
\put(63,9){$\lambda$}

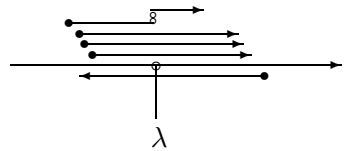
\put(106,36){\circle*{3}} \put(106,36){\vector(-1,0){70}}
\put(41,44){\circle*{3}} \put(41,44){\vector(1,0){60}}
\put(38,48){\circle*{3}} \put(38,48){\vector(1,0){60}}
\put(36,52){\circle*{3}} \put(36,52){\vector(1,0){60}}
\put(32,56){\circle*{3}} \put(32,56){\line(1,0){32}}
                         \put(64,57){\circle{2}}
                         \put(64,59){\circle{2}}
                         \put(63,61){\vector(1,0){20}}
\end{picture}

\begin{picture}(140,85)
\put(15,76){{\small 5th scenario: $\dim \Upsilon^1_{\lambda+i0}(r_\lambda) = 4$}}
\put(10,40){\vector(1,0){125}}
\put(65,40){\circle{3}}    
\put(65,40){\line(0,-1){20}}
\put(63,9){$\lambda$}

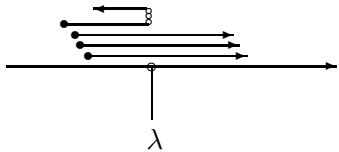
\put(41,44){\circle*{3}} \put(41,44){\vector(1,0){60}}
\put(38,48){\circle*{3}} \put(38,48){\vector(1,0){60}}
\put(36,52){\circle*{3}} \put(36,52){\vector(1,0){60}}
\put(32,56){\circle*{3}} \put(32,56){\line(1,0){32}}
                         \put(64,57){\circle{2}}
                         \put(64,59){\circle{2}}
                         \put(64,61){\circle{2}}
                         \put(63,62){\vector(-1,0){20}}
\end{picture}
\hskip 3 cm
\begin{picture}(140,85)
\put(15,76){{\small 6th scenario: $\dim \Upsilon^1_{\lambda+i0}(r_\lambda) = 3$}}
\put(10,40){\vector(1,0){125}}
\put(65,40){\circle{3}}    
\put(65,40){\line(0,-1){20}}
\put(63,9){$\lambda$}

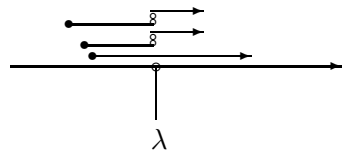
\put(41,44){\circle*{3}} \put(41,44){\vector(1,0){60}}
\put(38,48){\circle*{3}} \put(38,48){\line(1,0){26}}
                         \put(64,49){\circle{2}}
                         \put(64,51){\circle{2}}
                         \put(63,53){\vector(1,0){20}}

\put(32,56){\circle*{3}} \put(32,56){\line(1,0){32}}
                         \put(64,57){\circle{2}}
                         \put(64,59){\circle{2}}
                         \put(63,61){\vector(1,0){20}}
\end{picture}

\begin{picture}(140,85)
\put(15,76){{\small 7th scenario: $\dim \Upsilon^1_{\lambda+i0}(r_\lambda) = 3$}}
\put(10,40){\vector(1,0){125}}
\put(65,40){\circle{3}}    
\put(65,40){\line(0,-1){20}}
\put(63,9){$\lambda$}

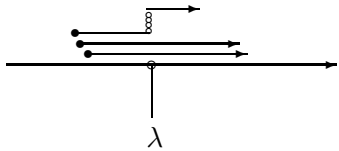
\put(41,44){\circle*{3}} \put(41,44){\vector(1,0){60}}
\put(38,48){\circle*{3}} \put(38,48){\vector(1,0){60}}
\put(36,52){\circle*{3}} \put(36,52){\line(1,0){28}}
                         \put(64,53){\circle{2}}
                         \put(64,55){\circle{2}}
                         \put(64,57){\circle{2}}
                         \put(64,59){\circle{2}}
                         \put(63,61){\vector(1,0){20}}
\end{picture}
\hskip 3 cm
\begin{picture}(140,85)
\put(15,76){{\small 8th scenario: $\dim \Upsilon^1_{\lambda+i0}(r_\lambda) = 5$}}
\put(10,40){\vector(1,0){125}}
\put(65,40){\circle{3}}    
\put(65,40){\line(0,-1){20}}
\put(63,9){$\lambda$}

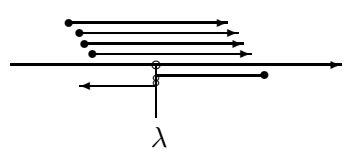
\put(41,44){\circle*{3}} \put(41,44){\vector(1,0){60}}
\put(38,48){\circle*{3}} \put(38,48){\vector(1,0){60}}
\put(36,52){\circle*{3}} \put(36,52){\vector(1,0){60}}
\put(32,56){\circle*{3}} \put(32,56){\vector(1,0){60}}

\put(106,36){\circle*{3}} \put(106,36){\line(-1,0){41}}
                         \put(65,35){\circle{2}}
                         \put(65,33){\circle{2}}
                         \put(65,32){\vector(-1,0){29}}
\end{picture}

For values of~$\lambda$ outside the essential spectrum these scenarios make rigorous sense, since in this case~$\lambda$ depends on~$r$ analytically.
The U-turn theorem allows us to extrapolate this behavior of isolated eigenvalues to points of singular spectrum inside essential spectrum.

\noindent One has to note that for the resonance index $N_+-N_-$ it does not matter which side an eigenvalue making a U-turn approaches the point~$\lambda$ from;
in both cases the eigenvalue increases the number~$N_+$ of up-points and the number~$N_-$ of down-points by~1.
Taking this into account, we do not distinguish for example the 2th scenario above from the following possibility:
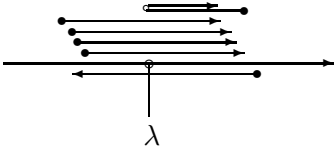
\begin{picture}(140,85)
\put(10,40){\vector(1,0){125}}
\put(65,40){\circle{3}}    
\put(65,40){\line(0,-1){20}}
\put(63,9){$\lambda$}

\put(106,36){\circle*{3}} \put(106,36){\vector(-1,0){70}}

\put(41,44){\circle*{3}} \put(41,44){\vector(1,0){60}}
\put(38,48){\circle*{3}} \put(38,48){\vector(1,0){60}}
\put(36,52){\circle*{3}} \put(36,52){\vector(1,0){60}}
\put(32,56){\circle*{3}} \put(32,56){\vector(1,0){60}}
\put(101,60){\circle*{3}} \put(64,60){\line(1,0){36}}
                         \put(65,62){\vector(1,0){26}}
                         \put(64,61){\circle{2}}
\end{picture}

Let $z \in \Pi,$ let $r_z$ be a resonance point corresponding to $z,$ let $u \in \Upsilon_z(r_z)$ be a resonance vector
and let~$k$ be a non-negative integer. By $\clLw_z(r_z)$ we denote the linear span of all vectors~$u$ from~$\Upsilon_z(r_z)$
such that \label{Page: clLw}
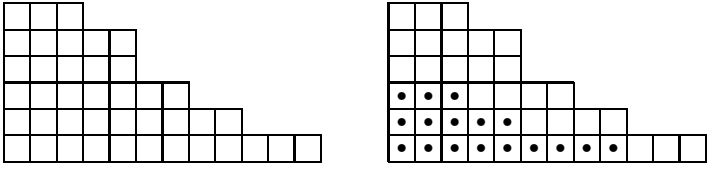
\begin{equation} \label{F: depth geq order}
  \depth(u) \geq \order(u).
\end{equation}

\begin{picture}(120,60)
\put(0,0){\line(0,1){60}}\put(10,0){\line(0,1){60}} \put(20,0){\line(0,1){60}} \put(30,0){\line(0,1){60}}
\put(40,0){\line(0,1){50}} \put(50,0){\line(0,1){50}} \put(60,0){\line(0,1){30}} \put(70,0){\line(0,1){30}}
\put(80,0){\line(0,1){20}} \put(90,0){\line(0,1){20}} \put(100,0){\line(0,1){10}} \put(110,0){\line(0,1){10}}
\put(120,0){\line(0,1){10}}
\put(0,0){\line(1,0){120}}\put(0,10){\line(1,0){120}}\put(0,20){\line(1,0){90}}\put(0,30){\line(1,0){70}}
\put(0,40){\line(1,0){50}}\put(0,50){\line(1,0){50}}\put(0,60){\line(1,0){30}}
\end{picture}
\qquad
\begin{picture}(120,60)
\put(0,0){\line(0,1){60}}\put(10,0){\line(0,1){60}} \put(20,0){\line(0,1){60}} \put(30,0){\line(0,1){60}}
\put(40,0){\line(0,1){50}} \put(50,0){\line(0,1){50}} \put(60,0){\line(0,1){30}} \put(70,0){\line(0,1){30}}
\put(80,0){\line(0,1){20}} \put(90,0){\line(0,1){20}} \put(100,0){\line(0,1){10}} \put(110,0){\line(0,1){10}}
\put(120,0){\line(0,1){10}}
\put(0,0){\line(1,0){120}}\put(0,10){\line(1,0){120}}\put(0,20){\line(1,0){90}}\put(0,30){\line(1,0){70}}
\put(0,40){\line(1,0){50}}\put(0,50){\line(1,0){50}}\put(0,60){\line(1,0){30}}
\put(5,5){\circle*{3}}\put(5,15){\circle*{3}}\put(5,25){\circle*{3}}
\put(15,5){\circle*{3}}\put(15,15){\circle*{3}}\put(15,25){\circle*{3}}
\put(25,5){\circle*{3}}\put(25,15){\circle*{3}}\put(25,25){\circle*{3}}
\put(35,5){\circle*{3}}\put(35,15){\circle*{3}}
\put(45,5){\circle*{3}}\put(45,15){\circle*{3}}
\put(55,5){\circle*{3}}
\put(65,5){\circle*{3}}
\put(75,5){\circle*{3}}
\put(85,5){\circle*{3}}
\end{picture}

\noindent
As an example, if the Young diagram of $\bfA_z(r_z)$ is as shown on the left figure above then
the vector space $\clLw_z(r_z)$ is the linear span of vectors marked by circles.

The following lemma is trivial.
\begin{lemma} \label{L: 2*dim L+m > dim Ups}
For any $z \in \Pi$ and for any resonance point~$r_z$ corresponding to~$r_z$ the following inequality holds:
$$
  2 \dim \clLw_z(r_z) + \dim \Upsilon_{z}^1(r_z) \geq \dim \Upsilon_{z}(r_z).
$$
\end{lemma}

\begin{prop} \label{P: (psi,J psi)=0}
If $z = \lambda \pm i0 \in \partial \Pi$ and if~$r_\lambda$ is a real resonance point
corresponding to $\lambda\pm i0,$ then for any $u_1, u_2 \in \clLw_z(r_\lambda)$
$$
  \scal{u_1}{Ju_2} = 0.
$$
\end{prop}
\begin{proof} Assume that $z = \lambda + i0.$ By linearity, it is enough to prove the claim for vectors~$u_1$ and~$u_2$
from $\clLw_{\lambda+i0}(r_\lambda),$ which satisfy the inequality~(\ref{F: depth geq order}).
By Theorem \ref{T: on vectors with property L}, the vectors $u_1$ and $u_2$ are vectors of type~I; in particular,
their $\bfA_{\lambda+i0}(r_\lambda)$ and $\bfA_{\lambda-i0}(r_\lambda)$ orders are equal:
\begin{equation} \label{F: u1, u2 are of type I}
  \order_+(u_1) = \order_-(u_1) \quad \text{and} \quad \order_+(u_2) = \order_-(u_2).
\end{equation}
Let $k = \depth_+(u_1)$ and $j = \depth_+(u_2)$
and assume, without loss of generality, that $k \geq j.$
By definition of depth, $u_1 = \bfA_{\lambda+i0}^k(r_\lambda) \phi$ for some $\phi.$
Since $k \geq j \geq \order(u_2)$ we have $\bfA_{\lambda+i0}^k(r_\lambda) u_2 = 0.$
By (\ref{F: u1, u2 are of type I}), this implies that $\bfA_{\lambda-i0}^k(r_\lambda) u_2 = 0.$
It follows from this equality and~(\ref{F: A*(z)=B(bar z)}),~(\ref{F: JA=BJ}) that
$$
  \scal{u_1}{Ju_2} = \scal{\bfA_{\lambda+i0}^k(r_\lambda) \phi}{Ju_2} = \scal{\phi}{\bfB_{\lambda-i0}^k(r_\lambda) Ju_2} = \scal{\phi}{J\bfA_{\lambda-i0}^k(r_\lambda)u_2} = 0.
$$
\end{proof}

\begin{prop} \label{P: J>0 then res points has order 1}
If $z = \lambda \pm i0 \in \partial \Pi$ and if the perturbation $J$ is non-negative (or non-positive), then every real resonance point has order~1.
\end{prop}
\begin{proof} Assume the contrary: a real resonance point~$r_\lambda$ has order larger than~$1.$
In this case there exists a vector $\phi \in \Upsilon^2_{\lambda + i0}(r_\lambda)$ of order~$2.$
Hence, by Theorem~\ref{T: Laurent for A(s)psi}, the vector $u = \bfA_{\lambda+i0}(r_\lambda)\phi$
is of order~$1$ (and therefore is non-zero) and has depth $\geq~1.$ It follows that
\begin{equation} \label{F: an argument}
  \scal{u}{Ju} = \scal{\bfA_{\lambda+i0}(r_\lambda)\phi}{Ju} = \scal{\phi}{\bfB_{\lambda-i0}(r_\lambda)Ju} = \scal{\phi}{J\bfA_{\lambda-i0}(r_\lambda)u} = 0,
\end{equation}
where the last equality follows from Corollary~\ref{C: Upsilon 1(+)=Upsilon 1(-)}.
Since $J\geq 0$ (or $J\leq 0$), it follows that $J u = 0.$ But this contradicts Lemma~\ref{L: j-dimensions coincide}.
\end{proof}

Even if the operator $J$ is not sign-definite, the resonance matrix $Q_{\lambda-i0}(r_\lambda)JP_{\lambda+i0}(r_\lambda)$ may be sign-definite for some
resonance points~$r_\lambda.$ If this is the case, one may ask whether the conclusion of Proposition~\ref{P: J>0 then res points has order 1} still holds.
In fact, the same argument shows that if the resonance matrix $Q_{\lambda-i0}(r_\lambda)JP_{\lambda+i0}(r_\lambda)$ is non-negative,
then the point~$r_\lambda$ is of type~I.
\begin{prop} \label{P: M>0 then type I} Let $z = \lambda \pm i0 \in \partial \Pi$
and let~$r_\lambda$ be a real resonance point 
corresponding to~$z.$ If the resonance matrix $Q_{\lambda-i0}(r_\lambda)JP_{\lambda+i0}(r_\lambda)$
or $Q_{\lambda+i0}(r_\lambda)JP_{\lambda-i0}(r_\lambda)$ is non-negative or non-positive, then~$r_\lambda$ has order~1.
\end{prop}
\begin{proof} Let for instance $z = \lambda+i0$
and assume the contrary:~$r_\lambda$ has order not less than two.
Then there exists a vector~$u$ of order~$1$ and of depth at least 1.
Since the vector~$u$ has order one, by Corollary~\ref{C: Upsilon 1(+)=Upsilon 1(-)}
we have $$P_{\lambda+i0}(r_\lambda)u = P_{\lambda-i0}(r_\lambda)u = u.$$ Further, we have
\begin{equation*}
  \begin{split}
    \scal{u}{Q_{\lambda-i0}(r_\lambda)JP_{\lambda+i0}(r_\lambda) u} & = \scal{P_{\lambda+i0}(r_\lambda)u}{JP_{\lambda+i0}(r_\lambda)u}
    \\ & = \scal{u}{Ju}.
  \end{split}
\end{equation*}
From the last two equalities, using the argument~(\ref{F: an argument}) of Proposition~\ref{P: J>0 then res points has order 1},
one can infer that $\scal{u}{Q_{\lambda-i0}(r_\lambda)JP_{\lambda+i0}(r_\lambda) u} = 0.$
Since $Q_{\lambda-i0}(r_\lambda)JP_{\lambda+i0}(r_\lambda)$ is non-negative (or non-positive), this implies that $Q_{\lambda-i0}(r_\lambda)JP_{\lambda+i0}(r_\lambda) u = 0.$
On the other hand, by Proposition~\ref{P: rank M=rank P} for all real resonance points, the
restriction of $Q_{\lambda-i0}(r_\lambda)JP_{\lambda+i0}(r_\lambda)$ to the resonance vector space~$\Upsilon_{\lambda+i0}(r_\lambda)$ has zero kernel.
This gives a contradiction.
\end{proof}

\bigskip The following theorem and its corollary Theorem~\ref{T: U-turn for res index} are one of the main results of this paper.

\begin{thm} \label{T: U-turn for res matrix} If~$r_\lambda$ is a real resonance point 
corresponding to~$z = \lambda \pm i0,$ then the absolute value of the signature of the resonance
matrices $Q_{\lambda\mp i0}(r_\lambda)JP_{\lambda\pm i0}(r_\lambda)$ is less or equal to the dimension of the vector space
$\Upsilon_{\lambda+i0}^1(r_\lambda):$
$$
  \abs{\sign Q_{\lambda\mp i0}(r_\lambda)JP_{\lambda\pm i0}(r_\lambda)} \leq \dim \Upsilon_{\lambda+i0}^1(r_\lambda).
$$
\end{thm}
\begin{proof}
We prove this for the operator $Q_{\lambda-i0}(r_\lambda)JP_{\lambda+i0}(r_\lambda).$
Let $\mu_+$ respectively, $\mu_-$ be the rank of the positive respectively, negative part of $Q_{\lambda-i0}(r_\lambda)JP_{\lambda+i0}(r_\lambda).$
Assume contrary to the claim, that is,
$$
  \abs{\mu_+ - \mu_-} > m,
$$ where $m = \dim \Upsilon_{\lambda+i0}^1(r_\lambda).$
By Proposition~\ref{P: rank M=rank P},
$$
  \mu_+ + \mu_- = N \:= \dim \Upsilon_{\lambda+i0}(r_\lambda).
$$
This equality combined with the previous inequality imply that either $\mu_+$ or $\mu_-$
is less than $(N-m)/2.$
Since by Lemma~\ref{L: 2*dim L+m > dim Ups}
$$
  (N-m)/2 \leq \dim \clLw_{\lambda+i0}(r_\lambda)
$$ we conclude that either $\mu_+$ or $\mu_-$ is less than
$$
  \dim \clLw_{\lambda+i0}(r_\lambda)=:p.
$$
Without loss of generality it can be assumed that it is the rank $\mu_+$ of the positive part of $Q_{\lambda-i0}(r_\lambda)JP_{\lambda+i0}(r_\lambda)$ which is less than~$p.$
Let $u_1, \ldots, u_p$ be a basis of the vector space $\clLw_{\lambda+i0}(r_\lambda).$
Since $\mu_+ < p,$ there exists a non-zero linear combination $u = \alpha_1u_1 + \ldots + \alpha_pu_p \in \clLw_{\lambda+i0}(r_\lambda)$
whose positive part with respect to $Q_{\lambda-i0}(r_\lambda)JP_{\lambda+i0}(r_\lambda)$ is zero.
Since $u_1,\ldots,u_p \in \clLw_{\lambda+i0}(r_\lambda),$ 
it follows from Proposition~\ref{P: (psi,J psi)=0} that
\begin{equation} \label{F: (u,Ju)=0}
  \scal{u}{Ju} = \sum_{i=1}^p \sum_{j=1}^p \bar \alpha_i \alpha_j \scal{u_i}{Ju_j} = 0.
\end{equation}
Since, by Proposition~\ref{P: rank M=rank P}, the restriction
of the operator $Q_{\lambda-i0}(r_\lambda)JP_{\lambda+i0}(r_\lambda)$ to~$\Upsilon_{\lambda+i0}$ has zero kernel
and since the positive part of~$u$ with respect to $Q_{\lambda-i0}(r_\lambda)JP_{\lambda+i0}(r_\lambda)$ is zero,
it follows that the negative part of~$u$ with respect to $Q_{\lambda-i0}(r_\lambda)JP_{\lambda+i0}(r_\lambda)$ is non-zero.
Hence,
$$
  \scal{u}{Ju} = \scal{P_{\lambda+i0}(r_\lambda)u}{JP_{\lambda+i0}(r_\lambda)u} = \scal{u}{Q_{\lambda-i0}(r_\lambda)JP_{\lambda+i0}(r_\lambda)u} < 0.
$$
This contradicts~(\ref{F: (u,Ju)=0}).
\end{proof}

The following theorem follows immediately from Theorems~\ref{T: U-turn for res matrix} and~\ref{T: res.ind=sign res.matrix}.
\begin{thm} \label{T: U-turn for res index} (U-turn theorem) For any real resonance point~$r_\lambda$ with property $L$ the
absolute value of the resonance index is less or equal to the dimension of the vector space~$\Upsilon^1_{\lambda+i0}(r_\lambda):$
$$
  \abs{\ind_{res}(\lambda; H_{r_\lambda}, V)} \leq \dim \Upsilon^1_{\lambda+i0}(r_\lambda).
$$
\end{thm}

\begin{cor} \label{C: U-turn for positive V} If the perturbation~$V$ is non-negative or non-positive, then
the absolute value of the resonance index $\ind_{res}(\lambda; H_{r_\lambda}, V)$ is equal to the dimension
of the vector space~$\Upsilon^1_{\lambda+i0}(r_\lambda).$
\end{cor}
\begin{proof} By Theorem~\ref{T: res.ind=sign res.matrix}, the resonance index $\ind_{res}(\lambda; H_{r_\lambda}, V)$
is equal to the signature of the resonance matrix $Q_{\lambda-i0}(r_\lambda)JP_{\lambda+i0}(r_\lambda).$
By Proposition~\ref{P: rank M=rank P}, dimension~$N$ of $\Upsilon_{\lambda+i0}(r_\lambda)$ is equal to the rank of the resonance matrix.
Since the resonance matrix is also non-negative or non-positive
it follows that the signature of the resonance matrix is equal to~$N$ or $-N.$
Finally, since~$V$ is non-negative or non-positive, by Proposition~\ref{P: J>0 then res points has order 1}
the vector space $\Upsilon_{\lambda+i0}(r_\lambda)$ coincides with $\Upsilon^1_{\lambda+i0}(r_\lambda),$
and therefore
$$
  \abs{\sign Q_{\lambda-i0}(r_\lambda)JP_{\lambda+i0}(r_\lambda)} = N = \dim \Upsilon_{\lambda+i0}(r_\lambda) = \dim \Upsilon^1_{\lambda+i0}(r_\lambda).
$$
\end{proof}

Theorem~\ref{T: U-turn for res matrix} and Proposition~\ref{P: rank M=rank P} imply the following.
\begin{cor} 
Let $z= \lambda\pm i0 \in \partial \Pi.$ Assume that a real resonance point~$r_\lambda$
corresponding to~$z$ has the geometric multiplicity~$m=1.$
If the order of~$r_\lambda$ is even, then the signature of the resonance matrix
$Q_{\lambda-i0}(r_\lambda)JP_{\lambda+i0}(r_\lambda)$
is equal to zero; if the order of~$r_\lambda$ is odd, then
the signature of the resonance matrix is equal to $+1$ or $-1.$
\end{cor}

\begin{cor} Let~$r_\lambda$ be a real resonance point.
If one of the numbers $N_+$ or $N_-$ from the definition~(\ref{F: ind res=N-N}) of resonance index is zero, then~$r_\lambda$ has order 1.
\end{cor}
\begin{proof} By Theorem~\ref{T: res.ind=sign res.matrix},
the resonance index $N_+-N_-$ of~$r_\lambda$ is equal to the signature of the self-adjoint operator $Q_{\lambda-i0}(r_\lambda)JP_{\lambda+i0}(r_\lambda).$
Hence, since one of the numbers $N_+$ or $N_-$ is zero, the signature of the self-adjoint operator $Q_{\lambda-i0}(r_\lambda)JP_{\lambda+i0}(r_\lambda)$
is equal to either~$N$ or $-N,$ where~$N$ is equal to the rank of $Q_{\lambda-i0}(r_\lambda)JP_{\lambda+i0}(r_\lambda).$
Hence, either $Q_{\lambda-i0}(r_\lambda)JP_{\lambda+i0}(r_\lambda)$ is non-positive or it is non-negative.
Therefore, Proposition~\ref{P: M>0 then type I} implies that~$r_\lambda$ has order 1.
\end{proof}


\section{Order preserving property of $P_{\lambda \pm i0}(r_\lambda) \colon \Upsilon_{\lambda \mp i0}(r_\lambda) \to \Upsilon_{\lambda \pm i0}(r_\lambda)$}
\label{S: property U}


The main result of this subsection is Theorem \ref{T: Property U} which asserts that if the geometric multiplicity of a real resonance point is equal to 1 then the mappings
$P_{\lambda \pm i0}(r_\lambda) \colon \Upsilon_{\lambda \mp i0}(r_\lambda) \to \Upsilon_{\lambda \pm i0}(r_\lambda)$
preserve order of resonance vectors.
Along the way we prove some properties of operators $P_{\lambda \pm i0}(r_\lambda)$ and $\bfA_{\lambda \pm i0}(r_\lambda)$ which seem to be interesting on their own.

%
\begin{prop} \label{P: delta JP delta=0} For any non-resonance point $r\in\mbR$ and any real resonance point~$r_\lambda \in \mbR$
\begin{equation} \label{F: delta JP delta=0}
  \sqrt{\Im T_{\lambda+i0}(H_r)} J P_{\lambda\pm i0}(r_\lambda) \sqrt{\Im T_{\lambda+i0}(H_r)} = 0
\end{equation}
and for all $j=1,2,\ldots$
\begin{equation} \label{F: delta JA(j)delta=0}
  \sqrt{\Im T_{\lambda+i0}(H_r)} J \bfA_{\lambda\pm i0}^j(r_\lambda) \sqrt{\Im T_{\lambda+i0}(H_r)} = 0.
\end{equation}
\end{prop}
\begin{proof} We prove these equalities for the upper plus sign. The equalities for the lower sign
can be derived from the upper sign equalities after taking adjoint and using (\ref{F: Pz*=Q(bar z)}), (\ref{F: JP=QJ}), (\ref{F: A*(z)=B(bar z)}), (\ref{F: JA=BJ}).

It is well-known (see e.g. \cite{Pu01FA}) that the operator
\begin{equation} \label{F: tilde S}
  \tilde S(\lambda+i0; H_s,H_r) = 1 - 2i \sqrt{\Im T_{\lambda+i0}(H_r)} (s-r)J (1 + (s-r)T_{\lambda+i0}(H_r)J)^{-1} \sqrt{\Im T_{\lambda+i0}(H_r)}
\end{equation}
is unitary for real non-resonant~$r$ and~$s;$ proof of this fact is a direct calculation.
Since the right hand side of~(\ref{F: tilde S}) makes sense for complex values of~$s,$ the operator $\tilde S(\lambda+i0; H_s,H_r)$
will be treated as a function of complex variable~$s.$
By the analytic Fredholm alternative (Theorem~\ref{T: Analytic Fredholm alternative}) the operator-function
$\tilde S(\lambda+i0; H_s,H_r)$ is a meromorphic function of $s.$ Since this function is also unitary
and therefore is bounded for real $s,$ it cannot have poles on the real axis $\mbR.$
Hence, $\tilde S(\lambda+i0; H_s,H_r)$ as a function of $s$ is holomorphic in a neighbourhood of $\mbR.$
Making the change of variables $\sigma = (r-s)^{-1}$ one infers that the function of~$\sigma$
$$
  \tilde S(\lambda+i0; H_{s(\sigma)},H_r) = 1 + 2i \sqrt{\Im T_{\lambda+i0}(H_r)} J (\sigma - T_{\lambda+i0}(H_r)J)^{-1} \sqrt{\Im T_{\lambda+i0}(H_r)}
$$
is holomorphic in a neighbourhood of $\mbR.$
Hence, the residue of this function at
$$
  \sigma _0 := (r-r_\lambda)^{-1}
$$
is equal to zero. By definition~(\ref{F: Pz(rz)=oint (sigma-Az)(-1)d sigma}) of the idempotent $P_z(r_z),$
this residue is equal (up to a constant) to the left hand side of~(\ref{F: delta JP delta=0}), which therefore is equal to zero too.
This completes proof of~(\ref{F: delta JP delta=0}).

Further, since the function $\tilde S(\lambda+i0; H_{s(\sigma)},H_0)$ of~$\sigma$ is holomorphic in a neighbourhood of $\mbR,$
it follows that all the other terms $(\sigma-\sigma_0)^{-j}$ with negative powers in the Laurent expansion
of $\tilde S(\lambda+i0; H_{s(\sigma)},H_0)$ at $\sigma = \sigma_0$ also vanish.
Combining this with equality~(\ref{F: d-th Laurent coef for A+(s)}) of Proposition~\ref{P: (sigma-Az(r))to(-1)=...}
implies the equality
$$
  \sqrt{\Im T_{\lambda+i0}(H_r)} J \bfA_{\lambda+i0}^{d-1}(r_\lambda) \sqrt{\Im T_{\lambda+i0}(H_r)} = 0.
$$
Further, using this equality and~(\ref{F: k-th Laurent coef for A+(s)}) with $k=d-2,$ we infer that
$$
  \sqrt{\Im T_{\lambda+i0}(H_r)} J \bfA_{\lambda+i0}^{d-2}(r_\lambda) \sqrt{\Im T_{\lambda+i0}(H_r)} = 0.
$$
Continuing in this way gives equalities~(\ref{F: delta JA(j)delta=0}) for all $j=d-1,d-2,\ldots,1.$
\end{proof}
Proposition~\ref{P: delta JP delta=0} implies that for all $j=0,1,2,\ldots$ and for all $s$
\begin{equation} \label{F: [A+(s)-A-(s)]bfA+[A+(s)-A-(s)]=0}
  (A_{\lambda+i0}(s) - A_{\lambda-i0}(s)) \bfA_{\lambda+i0}^j(r_\lambda) (A_{\lambda+i0}(s) - A_{\lambda-i0}(s)) = 0.
\end{equation}
This equality itself is not useful but its modification which follows is.
\begin{lemma} For any non-resonant real numbers~$r$ and $s$ and for all $j=0,1,2,\ldots$
\begin{equation} \label{F: [A+(r)-A-(r)]bfA+[A+(s)-A-(s)]=0}
  (A_{\lambda+i0}(r) - A_{\lambda-i0}(r)) \bfA_{\lambda+i0}^j(r_\lambda) (A_{\lambda+i0}(s) - A_{\lambda-i0}(s)) = 0.
\end{equation}
\end{lemma}
\begin{proof} Using~(\ref{F: Az3v6 (4.8)}) we have
\begin{equation*}
  \begin{split}
    A_{\lambda+i0}(r) - A_{\lambda-i0}(r) & = 2i\Im T_{\lambda+i0}(r)J
    \\ & = 2i(1+(r-s)A_{\lambda-i0}(s))^{-1} \Im T_{\lambda+i0}(H_s) (1+(r-s)B_{\lambda+i0}(s))^{-1}J
    \\ & = 2i(1+(r-s)A_{\lambda-i0}(s))^{-1} \Im T_{\lambda+i0}(H_s)J (1+(r-s)A_{\lambda+i0}(s))^{-1}
    \\ & = (1+(r-s)A_{\lambda-i0}(s))^{-1} (A_{\lambda+i0}(s) - A_{\lambda-i0}(s)) (1+(r-s)A_{\lambda+i0}(s))^{-1}.
  \end{split}
\end{equation*}
It follows that
\begin{multline*}
 \SqBrs{A_{\lambda+i0}(r) - A_{\lambda-i0}(r)} P_{\lambda+i0}(r_\lambda) \\ = (1+(r-s)A_{\lambda-i0}(s))^{-1} (A_{\lambda+i0}(s) - A_{\lambda-i0}(s)) (1+(r-s)A_{\lambda+i0}(s))^{-1}P_{\lambda+i0}(r_\lambda).
\end{multline*}
Expanding the factor $(1+(r-s)A_{\lambda+i0}(s))^{-1}P_{\lambda+i0}(r_\lambda)$ by~(\ref{F: weird f-n}) and
multiplying both sides of this equality on the right by $\bfA_{\lambda+i0}^j(r_\lambda)(A_{\lambda+i0}(s) - A_{\lambda-i0}(s)),$
one can see from~(\ref{F: [A+(s)-A-(s)]bfA+[A+(s)-A-(s)]=0})
that the left hand side of~(\ref{F: [A+(r)-A-(r)]bfA+[A+(s)-A-(s)]=0}) is zero.
\end{proof}
The left hand side of~(\ref{F: [A+(r)-A-(r)]bfA+[A+(s)-A-(s)]=0}) is a meromorphic function of two variables~$r$ and $s.$
Using~(\ref{F: Laurent for A+(s)}), one can expand this function into Laurent series at $r = r_\lambda,$ $s = r_\lambda.$ Since the function is zero,
all coefficients of terms $(r-r_\lambda)^k(s-r_\lambda)^l,$\, $k,l=0,\pm 1, \pm 2,\ldots,$ in the Laurent expansion are also zero. This gives some relations between operators
$\tilde A_{\lambda\pm i0,r_\lambda}(r_\lambda),$ $P_{\lambda\pm i0}(r_\lambda)$ and $\bfA_{\lambda\pm i0}(r_\lambda),$
such as
\begin{equation} \label{F: (A-A)A(A-A)=0}
  (\bfA_{\lambda+i0}^k(r_\lambda) - \bfA_{\lambda-i0}^k(r_\lambda))\bfA^j_{\lambda+i0}(r_\lambda)(\bfA^l_{\lambda+i0}(r_\lambda) - \bfA^l_{\lambda-i0}(r_\lambda)) = 0.
\end{equation}
The one which will be used shortly is obtained by setting to zero
the coefficient of $(r-r_\lambda)^{-1}(s-r_\lambda)^{-1}$ from the Laurent expansion of the left hand side of~(\ref{F: [A+(r)-A-(r)]bfA+[A+(s)-A-(s)]=0}).
Taking $j=0$ in the resulting relation gives the following equality.
\begin{lemma} \label{L: (P+-P-)P+(P+-P-)=0} For any real resonance point~$r_\lambda$
\begin{equation} \label{F: (P+-P-)P+(P+-P-)=0}
  (P_{\lambda+i0}(r_\lambda) - P_{\lambda-i0}(r_\lambda))P_{\lambda+i0}(r_\lambda)(P_{\lambda+i0}(r_\lambda) - P_{\lambda-i0}(r_\lambda)) = 0.
\end{equation}
\end{lemma}
\begin{thm} \label{T: spec P(+)P(-)} For any real resonance point~$r_\lambda$ the spectrum
of the product $P_{\lambda+i0}(r_\lambda) P_{\lambda-i0}(r_\lambda)$ consists of only $0$ and $1.$
Moreover, algebraic multiplicity of~$1$ is equal to $N = \dim \Upsilon_{\lambda+i0}(r_\lambda).$
\end{thm}
\begin{proof} For brevity, we write $P_+$ instead of $P_{\lambda+i0}(r_\lambda)$ and $P_-$ instead of $P_{\lambda-i0}(r_\lambda).$
Expanding~(\ref{F: (P+-P-)P+(P+-P-)=0}) we obtain
\begin{equation} \label{F: P(+)-P(-)P(+)-P(+)P(-)+P(-)P(+)P(-)=0}
  P_+ - P_-P_+ - P_+P_- + P_-P_+P_- = 0.
\end{equation}
Taking traces of both sides of this equality and using $\Tr(P_+P_-) = \Tr(P_-P_+P_-)$ give
\begin{equation} \label{F: Tr(P(-)P(+))=N}
  \Tr(P_-P_+) = \Tr(P_+) = N.
\end{equation}
Multiplying both sides of~(\ref{F: P(+)-P(-)P(+)-P(+)P(-)+P(-)P(+)P(-)=0}) by $P_+$ on the right gives
\begin{equation} \label{F: P(+)-P(-)P(+)-P(+)P(-)+P(-)P(+)P(-)=0(2)}
  P_+ - P_-P_+ - P_+P_-P_+ + P_-P_+P_-P_+ = 0.
\end{equation}
Taking trace of this equality and using~(\ref{F: Tr(P(-)P(+))=N}) one gets
$$
  \Tr(P_-P_+P_-P_+) = N.
$$
Multiplying~(\ref{F: P(+)-P(-)P(+)-P(+)P(-)+P(-)P(+)P(-)=0(2)}) on the right by $P_-P_+$ and taking the trace of the equality obtained implies
$$
  \Tr((P_-P_+)^3) = N.
$$
Continuing in this manner, it can be shown that for any $k=1,2,3,\ldots$
\begin{equation} \label{F: Tr[(P-P+)(k)]=N}
  \Tr((P_-P_+)^k) = N.
\end{equation}
Since $P_-P_+$ has rank $\leq N$ (in fact this rank is equal to~$N$ by Theorem~\ref{T: property M}, but we don't need this),
if $x_1, \ldots, x_N$ is the list containing all non-zero eigenvalues of $P_-P_+$ counting multiplicities,
then it follows from the spectral mapping theorem, the Lidskii theorem~(\ref{F: Lidskii Thm}) and
(\ref{F: Tr[(P-P+)(k)]=N}) that for all $k=1,2,\ldots$
$$
  x_1^k + \ldots + x_N^k = N.
$$
This is possible only if all the~$N$ numbers $x_1,\ldots,x_N$ are equal to $1.$
\end{proof}
\begin{rems} \label{R: another proof of property M} \rm Theorem~\ref{T: spec P(+)P(-)} implies that the ranks of the products $P_{\lambda+i0}(r_\lambda) P_{\lambda-i0}(r_\lambda)$
and $P_{\lambda-i0}(r_\lambda) P_{\lambda+i0}(r_\lambda)$ are the same as that of $P_{\lambda+i0}(r_\lambda)$ and $P_{\lambda-i0}(r_\lambda)$
and thus it gives another proof of Theorem~\ref{T: property M}.
\end{rems}

\label{Page: point with property C(2)}
\begin{defn} We say that a real resonance point~$r_\lambda$ of geometric multiplicity~$m$ has \emph{property~$C,$} if the vector spaces $\Upsilon_{\lambda+i0}(r_\lambda)$
and $\Upsilon_{\lambda-i0}(r_\lambda)$ admit Jordan decompositions (see p.\pageref{Page: Jordan decomp-n} for definition of a Jordan decomposition)
\begin{equation} \label{F: Property C, (+)decomp}
  \Upsilon_{\lambda+i0}(r_\lambda) = \Upsilon^{[1]}_{\lambda+i0}(r_\lambda) \dotplus \Upsilon^{[2]}_{\lambda+i0}(r_\lambda) \dotplus \ldots \dotplus \Upsilon^{[m]}_{\lambda+i0}(r_\lambda)
\end{equation}
and
\begin{equation} \label{F: Property C, (-)decomp}
  \Upsilon_{\lambda-i0}(r_\lambda) = \Upsilon^{[1]}_{\lambda-i0}(r_\lambda) \dotplus \Upsilon^{[2]}_{\lambda-i0}(r_\lambda) \dotplus \ldots \dotplus \Upsilon^{[m]}_{\lambda-i0}(r_\lambda)
\end{equation}
such that for all $j=1,2,\ldots,m$ the following equalities hold:
\begin{equation} \label{F: Property C}
  P_{\lambda + i0}(r_\lambda) \Upsilon^{[\nu]}_{\lambda - i0}(r_\lambda) = \Upsilon^{[\nu]}_{\lambda + i0}(r_\lambda)
  \quad \text{and} \quad P_{\lambda - i0}(r_\lambda) \Upsilon^{[\nu]}_{\lambda + i0}(r_\lambda) = \Upsilon^{[\nu]}_{\lambda - i0}(r_\lambda).
\end{equation}
\end{defn}

The goal of this subsection is to prove Theorem~\ref{T: Property U}. 
The proof starts with the following lemma.
\begin{lemma} \label{L: trivial lemma about type I vectors}
Let~$r_\lambda$ be a real resonance point with property~$C$ and let $j,k,l$ be three non-negative integers.
If the operator
$$
  \bfA^k_{\lambda \pm i0}(r_\lambda)\bfA^j_{\lambda \mp i0}(r_\lambda)\bfA^l_{\lambda \pm i0}(r_\lambda)
$$
sends all vectors from $\Upsilon_{\lambda \pm i0}(r_\lambda)$ to vectors of type~I and if it sends all vectors of type $I$ to zero, then
this operator decreases the order of vectors from $\Upsilon_{\lambda \pm i0}(r_\lambda).$
\end{lemma}
\begin{proof} We prove this only for the upper signs.

Since~$r_\lambda$ has property~$C,$ the vector spaces $\Upsilon_{\lambda+i0}(r_\lambda)$ and $\Upsilon_{\lambda-i0}(r_\lambda)$
admit decompositions (\ref{F: Property C, (+)decomp}) and (\ref{F: Property C, (-)decomp}) into direct sums
of vectors spaces $\Upsilon^{[\nu]}_{\lambda\pm i0}(r_\lambda)$ such that for any $k\geq 0$
$$
  \bfA^k_{\lambda \pm i0}(r_\lambda) \Upsilon^{[\nu]}_{\lambda\pm i0}(r_\lambda) \subset \Upsilon^{[\nu]}_{\lambda\pm i0}(r_\lambda)
$$
and the relations (\ref{F: Property C}) hold.

Each vector space $\Upsilon^{[\nu]}_{\lambda\pm i0}(r_\lambda)$ has a basis
$$
  u_{\nu\pm}^{(1)}, \ldots, u_{\nu\pm}^{(d_\nu)}
$$
such that $\bfA^k_{\lambda \pm i0}(r_\lambda)u_{\nu\pm}^{(j)} = u_{\nu\pm}^{(j-1)}.$
Therefore, it is enough to show that the operator $\bfA^k_{\lambda+ i0}(r_\lambda)\bfA^j_{\lambda- i0}(r_\lambda)\bfA^l_{\lambda+ i0}(r_\lambda)$
decreases order of each of the vectors $u_{\nu+}^{(j)}.$ We shall prove this assertion.

For each $\nu = 1,\ldots,m,$ there exists the largest index $\alpha$ such that $u_{\nu+}^{(\alpha)}$ is a vector of type~I.
Corollary~\ref{C: u type I then so ia Au} implies that
$$
  \underbrace{u_{\nu+}^{(1)}, \ \ldots, \ u_{\nu+}^{(\alpha)}}_{\text{are of type I}}, \ \ \underbrace{u_{\nu+}^{(\alpha+1)}, \ \ldots, \ u_{\nu+}^{(d_\nu)}}_{\text{are not of type I}}.
$$
The operator $\bfA^k_{\lambda+ i0}(r_\lambda)\bfA^j_{\lambda- i0}(r_\lambda)\bfA^l_{\lambda+ i0}(r_\lambda)$ decreases order of the vectors
$u_{\nu+}^{(1)}, \ldots, u_{\nu+}^{(\alpha)},$ since by the premise these vectors belong to the kernel of the operator.
Now we show that the image of each of the vectors $u_{\nu+}^{(\alpha+1)}, \ldots, u_{\nu+}^{(d_\nu)}$ is a linear combination of
$u_{\nu+}^{(1)}, \ldots, u_{\nu+}^{(\alpha)}$ and this will complete the proof. 
By (\ref{F: Property C}), for this it is enough to show that any vector of type I
from $\Upsilon^{[\nu]}_{\lambda+i0}(r_\lambda)$ is a linear combination of $u_{\nu+}^{(1)}, \ldots, u_{\nu+}^{(\alpha)}.$
Assume the contrary. Then there exists a vector $f$ of type I and of order $>\alpha.$ Using Corollary~\ref{C: u type I then so ia Au},
we can assume that this vector has order $\alpha+1.$ Since $f$ is a linear combination of $u_{\nu+}^{(1)}, \ldots, u_{\nu+}^{(\alpha+1)},$
it follows that $u_{\nu+}^{(\alpha+1)}$ is a vector of type~I. This contradicts definition of $\alpha.$
\end{proof}

Let
$$
  D_{\lambda+i0}(r_\lambda) = P_{\lambda+i0}(r_\lambda) - P_{\lambda+i0}(r_\lambda)P_{\lambda-i0}(r_\lambda)P_{\lambda+i0}(r_\lambda)
$$
and
$$
  D_{\lambda-i0}(r_\lambda) = P_{\lambda-i0}(r_\lambda) - P_{\lambda-i0}(r_\lambda)P_{\lambda+i0}(r_\lambda)P_{\lambda-i0}(r_\lambda).
$$
\begin{lemma} $D_{\lambda+i0}(r_\lambda) = D_{\lambda-i0}(r_\lambda).$
\end{lemma}
\begin{proof} By Lemma \ref{L: (P+-P-)P+(P+-P-)=0} we have $P_-D_+ = D_+$ and similarly $D_-P_+ = D_-.$
It is left to note that $P_-D_+ = D_-P_+.$
\end{proof}
This lemma allows us to write $D_\lambda(r_\lambda)$ instead of $D_{\lambda-i0}(r_\lambda)$ and $D_{\lambda+i0}(r_\lambda).$

\begin{lemma} The operator $D_\lambda(r_\lambda)$ has the following properties:
\begin{enumerate}
 \item $D^2_{\lambda}(r_\lambda) = 0.$
 \item The image of $D_\lambda(r_\lambda)$ consists of vectors of type~I.
 \item The kernel of $D_\lambda(r_\lambda)$ contains all vectors of type~I.
\end{enumerate}
\end{lemma}
\begin{proof}
Multiplying the left hand side of the equality~(\ref{F: (P+-P-)P+(P+-P-)=0}) on both sides by $P_{\lambda+i0}(r_\lambda)$ gives~$D^2_{\lambda}(r_\lambda) = 0.$
It follows from~(\ref{F: (A-A)A(A-A)=0}) with $j=l=0$ that for all $k=0,1,2,\ldots$
$$
  (\bfA_+^k - \bfA_-^k)D_\lambda(r_\lambda) = (\bfA_+^k - \bfA_-^k)P_+(P_+-P_-)P_+ = 0.
$$
Hence, by Lemma~\ref{L: type I u iff A(+)u=A(-)u}, the image of the operator $D_{\lambda}(r_\lambda)$ consists only of vectors of type~I.
The third assertion is obvious from Theorem \ref{T: type I vectors}.
\end{proof}
\begin{lemma} \label{L: P(+)P(-)P(+) preserves order}
If a real resonance point~$r_\lambda$ has property~$C$ then
the operator $P_{\lambda\pm i0}(r_\lambda)P_{\lambda\mp i0}(r_\lambda)P_{\lambda\pm i0}(r_\lambda)$
preserves the order of vectors
from $\Upsilon_{\lambda\pm i0}(r_\lambda),$ that is, for all $j=1,2,\ldots$
$$
  P_{\lambda\pm i0}(r_\lambda)P_{\lambda\mp i0}(r_\lambda)\Upsilon^j_{\lambda\pm i0}(r_\lambda) = \Upsilon^j_{\lambda\pm i0}(r_\lambda).
$$
\end{lemma}
\begin{proof} We prove this for the upper signs. In the proof we will use properties of the operator $D = D_\lambda(r_\lambda)$ from previous lemma.

If a vector $u \in \Upsilon_{\lambda +i0}(r_\lambda)$ is of type~I, then
$$
  P_+P_- u = P_+P_-P_+ u = (P_+ - D) u = u - Du = u,
$$
so the operator $P_+P_-P_+$ preserves order of type I vectors.
For any vector $u \in \Upsilon_{\lambda +i0}(r_\lambda)$ the vector $Du$ is a vector of type~I and therefore it
follows from Lemma~\ref{L: trivial lemma about type I vectors} that order of $Du$ is less than the order of~$u.$
Hence, the operator $P_+P_-P_+=P_+-D$ preserves order.
\end{proof}

\begin{lemma} \label{L: P(+)A(-)P(+) decreases order}
If a real resonance point~$r_\lambda$ has property~$C$ then
the operator $P_{\lambda\pm i0}(r_\lambda)\bfA_{\lambda \mp i0}(r_\lambda)P_{\lambda\pm i0}(r_\lambda)$
decreases order of vectors from $\Upsilon_{\lambda\pm i0}(r_\lambda).$
\end{lemma}
\begin{proof} We prove this for the upper signs.
Let
$$
  E_+ := \bfA_+ - P_+\bfA_-P_+.
$$
It follows from~(\ref{F: (A-A)A(A-A)=0}) with $k=l=1$ and $j=0$ that
$$
  E_+^2 = (\bfA_+ - P_+\bfA_-P_+)(\bfA_+ - P_+\bfA_-P_+) = P_+(\bfA_+-\bfA_-)P_+(\bfA_+-\bfA_-)P_+ = 0.
$$
It follows from~(\ref{F: (A-A)A(A-A)=0}) with $l=1$ and $j=0$ that for all $k=0,1,2,\ldots$
$$
  (\bfA_+^k - \bfA_-^k)E_+ = (\bfA_+^k - \bfA_-^k)P_+(\bfA_+-\bfA_-)P_+ = 0.
$$
It follows from this and Lemma~\ref{L: type I u iff A(+)u=A(-)u} that the image of $E_+$ is a subspace of $\Upsilon^I_\lambda(r_\lambda).$
So, on one hand, the operator $E_+$ obviously maps all vectors of type~I to zero, on the other hand the image of $E_+$ consists of only vectors of type~I.
By Lemma~\ref{L: trivial lemma about type I vectors}, this implies that $E_+$ decreases order.
Since $P_+\bfA_-P_+ = \bfA_+-E_+$ and since $\bfA_+$ also decreases order, it follows that $P_+\bfA_-P_+$ decreases order too.
\end{proof}

\begin{thm} \label{T: Property U}
For any $z = \lambda \pm i0 \in \partial \Pi,$ for any real resonance point~$r_\lambda$ with property~$C$
corresponding to~$z$ and for any $j=1,2,3,\ldots$ restriction of the idempotent
operator $P_{\lambda\pm i0}(r_\lambda)$ to $\Upsilon^j_{\lambda\mp i0}(r_\lambda)$
is a linear isomorphism of the vector spaces $\Upsilon^j_{\lambda\mp i0}(r_\lambda)$ and $\Upsilon^j_{\lambda\pm i0}(r_\lambda).$
\end{thm}
\begin{proof} As usual, only the statement for upper signs is proved.
Since by Theorem \ref{T: property M} the idempotent $P_+$ is a linear isomorphism of the vector spaces $\Upsilon_-$ and $\Upsilon_+,$
the claim is equivalent to $P_+(\Upsilon^j_-) \subset \Upsilon^j_+$ for all $j.$
Since $u \in \Upsilon^j_\pm$ if and only if $\bfA^j_\pm u = 0,$
the last assertion in its turn is equivalent to
$$
  \forall u \in \Upsilon_- \quad  \bfA_-^j u = 0 \quad \then \quad \bfA_+^j u = 0.
$$
For $j=1$ this follows from Corollary~\ref{C: Upsilon 1(+)=Upsilon 1(-)}. Assume that the claim holds for $j=k$
and let $u \in \Upsilon_-^{k+1}.$ Then, since by Lemma~\ref{L: P(+)A(-)P(+) decreases order} the operator $P_-\bfA_+P_-$ decreases order,
the inclusion $P_-\bfA_+P_-u \in \Upsilon_-^{k}$ holds which implies the equality
$$
  \bfA_-^k (P_-\bfA_+P_- u) = 0.
$$
By induction assumption, this implies
$$
  \bfA_+^k (P_-\bfA_+P_- u) = 0.
$$
The left hand side can be written as $\bfA_+^k (P_+P_-P_+)\bfA_+u$ and so
$$
  \bfA_+^k (P_+P_-P_+)\bfA_+u = 0.
$$
It follows that $(P_+P_-P_+)\bfA_+u$ is a $+$-vector of order $\leq k.$ Since by Lemma~\ref{L: P(+)P(-)P(+) preserves order} the operator $P_+P_-P_+$ preserves order,
it follows that $\bfA_+u$ is a $+$-vector of order $\leq k$ too. Hence,
$$
  \bfA_+^{k+1} u = 0.
$$
\end{proof}

\begin{thm} \label{T: Q(+) is a lin. isom-m}
Assume that a real resonance point~$r_\lambda$ has property~$C.$
For any $z = \lambda \pm i0 \in \partial \Pi$ and for any real resonance point~$r_\lambda \in \mbR,$
corresponding to $z,$ the idempotent $Q_{\lambda\pm i0}(r_\lambda)$ is a linear isomorphism
of the vector spaces~$\Psi^j_{\lambda\mp i0}(r_\lambda)$ and~$\Psi^j_{\lambda\pm i0}(r_\lambda)$
for all $j=1,2,\ldots$
\end{thm}
\begin{proof} We prove this assertion for the upper sign.
Using successively Lemma~\ref{L: j-dimensions coincide}, the equality~(\ref{F: JP=QJ}), Theorem~\ref{T: Property U} and Lemma~\ref{L: j-dimensions coincide} again,
one has the following chain of linear isomorphisms:
\begin{equation*}
  \begin{split}
    Q_{\lambda+i0}(r_\lambda) \Psi^j_{\lambda-i0}(r_\lambda) & = Q_{\lambda+i0}(r_\lambda) J \Upsilon^j_{\lambda-i0}(r_\lambda)
     \\ & = J P_{\lambda+i0}(r_\lambda)  \Upsilon^j_{\lambda-i0}(r_\lambda)
        \simeq J \Upsilon^j_{\lambda+i0}(r_\lambda)
        = \Psi^j_{\lambda+i0}(r_\lambda).
  \end{split}
\end{equation*}
\end{proof}
\bigskip
For real resonance points with property~$C$ the following two commutative diagrams of linear isomorphisms of vector spaces
summarize Theorems~\ref{T: Property U},~\ref{T: Q(+) is a lin. isom-m} and Lemma~\ref{L: j-dimensions coincide},
\begin{equation*} \label{F: g*=g phi}
  \xymatrix{
    \Psi_{\lambda+i0}^j(r_\lambda)  && \Upsilon_{\lambda+i0}^j(r_\lambda) \ar[ll]_J  \\
                    &&                               \\
    \Psi_{\lambda-i0}^j(r_\lambda) \ar[uu]^{Q_{\lambda+i0}(r_\lambda)} && \Upsilon_{\lambda-i0}^j(r_\lambda) \ar[ll]^J \ar[uu]_{P_{\lambda+i0}(r_\lambda)}
  }
\qquad\qquad
  \xymatrix{
    \Psi_{\lambda+i0}^j(r_\lambda) \ar[dd]_{Q_{\lambda-i0}(r_\lambda)} && \Upsilon_{\lambda+i0}^j(r_\lambda) \ar[ll]_J \ar[dd]^{P_{\lambda-i0}(r_\lambda)} \\
                    &&                               \\
    \Psi_{\lambda-i0}^j(r_\lambda)  && \Upsilon_{\lambda-i0}^j(r_\lambda) \ar[ll]^J
  }
\end{equation*}

\label{Page: point with property U}
We say that a real resonance point~$r_\lambda$ has \emph{property~$U$} if the operators $P_{\lambda\pm i0}(r_\lambda) \colon \Upsilon_{\lambda\mp i0}(r_\lambda)
\to \Upsilon_{\lambda\pm i0}(r_\lambda)$ preserve order of vectors. Thus, Theorem \ref{T: Property U} asserts that property~$C$ implies property~$U.$

\section{Questions of independence from the rigging~$F$}
\label{S: questions of indend from F}
Here we discuss some questions of independence from the rigging~$F$ for some of the notions which have been studied so far.

\begin{lemma} \label{L: R=indices of AP and AP} The $R$-indices of operators $A_{\lambda+ iy}(s)P_{\lambda + iy}(r_\lambda)$
and $\ulA_{\lambda+ iy}(s)\ulP_{\lambda + iy}(r_\lambda)$ coincide for all $s$ and for all small enough $y>0.$
\end{lemma}
\begin{proof}
If the rigging operator~$F$ is bounded then this follows directly from (\ref{F: mu(AB)=mu(BA)}).
In general, it is not difficult to see that if $u$ is a solution of the equation
$$
  (1+(r_z-s)A_z(s))^ku = 0,
$$
then for some unique $\chi$ we have $u=F\chi$ where $\chi$ is a solution of the equation
$$
  (1+(r_z-s)\ulA_z(s))^k \chi = 0,
$$
and vice versa, if a vector $\chi$ is a solution of this equation then $u=F\chi$ is a solution of the previous one.
It follows that spectral measures of operators $A_{\lambda+ iy}(s)P_{\lambda + iy}(r_\lambda)$ and
$\ulA_{\lambda+ iy}(s)\ulP_{\lambda + iy}(r_\lambda)$ coincide.
That is, eigenvalues of operators $A_{\lambda+ iy}(s)P_{\lambda + iy}(r_\lambda)$ and $\ulA_{\lambda+ iy}(s)\ulP_{\lambda + iy}(r_\lambda)$
are the same and their algebraic multiplicities are also the same. Hence, their $R$-indices are also equal.
\end{proof}

\begin{thm} \label{T: res.ind independent of F} The resonance index $\ind_{res}(\lambda; H,V)$ does not depend on the choice of the rigging operator~$F$
as long as~$\lambda$ is essentially regular for the pair $(\clA,F),$ where $\clA = \set{H+rV \colon r \in \mbR}$ and $V$ is
a regularizing direction for an operator $H$ which is resonant at~$\lambda.$
\end{thm}
\begin{proof} Since the operators $\ulA_z(s) = R_z(H_s)V$ and $\ulP_z(r_z)$ do not depend on~$F,$
this follows immediately from Theorem \ref{T: numbers equal to res index} and Lemma \ref{L: R=indices of AP and AP}.
\end{proof}
This theorem raises natural questions of independence of the notions of essentially regular points and regularizing directions from the rigging~$F.$

\begin{cor} \label{C: V>=0 then dim Ups(1) does not depend on F}
If the perturbation $V$ is non-negative (or non-positive) then the dimension of the vector space $\Upsilon_{\lambda+i0}^1(H_{r_\lambda},V)$
does not depend on the choice of rigging~$F.$
\end{cor}
\begin{proof} Since $V$ is non-negative, by Proposition \ref{P: J>0 then res points has order 1},
$\dim \Upsilon_{\lambda+i0}^1(H_{r_\lambda},V)$ is equal to $\dim \Upsilon_{\lambda+i0}(H_{r_\lambda},V).$
Since by $V\geq 0$ there are no resonance down-points, this number is equal to the resonance index $\ind_{res}(\lambda;H_{r_\lambda},V),$
which is independent of~$F$ by Theorem \ref{T: res.ind independent of F}.
\end{proof}
By Lemma \ref{L: sign(M)=sign(M(y))}, for small enough $y$ the signatures of operators $Q_{\lambda\mp i0}(r_\lambda)JP_{\lambda\pm i0}(r_\lambda)$
and $\ulQ_{\lambda\mp iy}(r_\lambda)V \ulP_{\lambda\pm iy}(r_\lambda)$ coincide.
Hence, another way to prove Theorem \ref{T: res.ind independent of F} is to observe that the latter operator does not depend on~$F.$

Combining Corollary \ref{C: V>=0 then dim Ups(1) does not depend on F} with Theorem \ref{T: mult of s.c. spectrum: drastic version}
we obtain the following
\begin{thm}
If the real vector space of self-adjoint perturbation operators $\clA_0(F)$ has at least one non-negative operator~$V,$ then the dimension of the vector space $\Upsilon^1_{\lambda+i0}(r_\lambda)$
is independent of~$F.$
\end{thm}

\section{Resonance points of type~I}
\label{S: res point of type I}
It turns out that real resonance points have a certain generic property, which admits many equivalent reformulations.
A~real resonance point with this property will be called a point of type~I.
As it will be shown, if a point~$\lambda$ on the spectral line lies outside the essential spectrum, then all real resonance points corresponding to $\lambda\pm i0$ are of type~I.
Further, if the perturbation~$V$ is non-negative, then all points are also of type~I for any essentially regular point~$\lambda.$
For a resonance point to be of type~I is a generic property since, as it will be shown, all resonance points of order~$1$ are of type~I.
Resonance points which are not of type~I exist, examples of such points will be given in subsection~\ref{SSS: order 2} of section~\ref{S: Pert embedded eig}.

At the end of this section we introduce a class of real resonance points with the so-called property~$S$ which is strictly larger than
the class of real resonance points of type~I.

Initially, results of section \ref{S: U-turn theorem} were proved for points of type I. At that stage of preparation of this paper
I did not know whether there were real resonance points not of type I. In fact, a significant time was spent in an effort to prove a conjecture that all
real resonance points are of type I. This conjecture was supported by the fact that it holds in several special cases mentioned in the beginning of this section.
However, later an example of a resonance point not of type I was found. This example is given in section \ref{S: Pert embedded eig}.
A similar story was repeated with resonance points with property~$S.$ To prove main results of section \ref{S: U-turn theorem} in the case of
arbitrary real resonance points took another year.

By definition, a real resonance point~$r_\lambda$ is a \emph{point of type~I}, if
for some non-resonance point $s\in \mbR$ the following equality holds: \label{Page: point of type I}
\begin{equation} \label{F: type I resonance point, definition}
  \sqrt{\Im T_{\lambda+i0}(H_s)}\, JP_{\lambda+i0}(r_\lambda) = 0.
\end{equation}
This equality is a strengthened version of~(\ref{F: delta JP delta=0}), and while the equality~(\ref{F: delta JP delta=0})
holds for all resonance points~$r_\lambda,$ it will be shown that not all resonance points are of type~I.
One can also see that definition of a point of type~I is equivalent to requiring that all resonance vectors corresponding to $\lambda+i0$
are of type~I.
\begin{lemma} \label{L: type I iff Im Q=0}
A~real resonance point~$r_\lambda$ is of type~I if and only if for some non-resonant $s \in \mbR$
\begin{equation} \label{F: type I iff Im Q=0}
  \sqrt{\Im T_{\lambda+i0}(H_s)}\, Q_{\lambda+i0}(r_\lambda) = 0.
\end{equation}
\end{lemma}
\begin{proof} By Lemma~\ref{L: j-dimensions coincide},
the range of the operator $Q_{\lambda+i0}(r_\lambda)$ coincides with the range of the product $JP_{\lambda+i0}(r_\lambda).$
The assertion follows.
\end{proof}

In what follows it is assumed for convenience that the point $s,$ for which the equality~(\ref{F: type I resonance point, definition})
holds, is $s = 0.$

\begin{lemma} \label{L: type I iff w(s) is holom-c} A resonance point~$r_\lambda$ is a point of type~I if and only if the function
\begin{equation} \label{F: a wonderful function}
  \mbC \ni s \mapsto w(s):=\sqrt{\Im T_{\lambda+i0}(H_0)}\, (1+sJT_{\lambda+i0}(H_0))^{-1}
\end{equation}
is holomorphic at~$r_\lambda.$
\end{lemma}
\begin{proof} Let $\sigma = -s^{-1}$ and let
$$
  \tilde w(\sigma) = \sqrt{\Im T_{\lambda+i0}(H_0)}\, (\sigma - JT_{\lambda+i0}(H_0))^{-1} = -\frac 1s w(s).
$$
The function $w(s)$ is holomorphic at~$r_\lambda$ if and only if $\tilde w(\sigma)$ is holomorphic at $\sigma_\lambda(0)=-r_\lambda^{-1}.$

($\then$) By the analytic Fredholm alternative, the function $\tilde w(\sigma)$ is meromorphic with a possible pole at $\sigma_\lambda(0).$
It follows from the definition~(\ref{F: Qz(rz)=oint (sigma-Bz)(-1)d sigma}) of the idempotent operator $Q_{\lambda+i0}(r_\lambda)$ and Lemma~\ref{L: type I iff Im Q=0} that
\begin{equation} \label{F: oint of w(s)}
  \begin{split}
    \oint_{C(\sigma_\lambda(0))} \tilde w(\sigma)\,d\sigma & = \sqrt{\Im T_{\lambda+i0}(H_0)} \, \oint_{C(\sigma_\lambda(0))} (\sigma - JT_{\lambda+i0}(H_0))^{-1}\,d\sigma
    \\ & = 2\pi i \sqrt{\Im T_{\lambda+i0}(H_0)} Q_{\lambda+i0}(r_\lambda) = 0,
  \end{split}
\end{equation}
where $C(\sigma_\lambda(0))$ is a small closed contour enclosing $\sigma_\lambda(0)=-r_\lambda^{-1}.$
Hence, the coefficient of $(\sigma-\sigma_\lambda(0))^{-1}$ in the Laurent series of $\tilde w(\sigma)$ is $0.$
Now Proposition~\ref{P: nilpotent terms of (sigma-B)(-1)} and equality~~(\ref{F: QB=BQ=B}) imply
that the coefficients of terms $(\sigma-\sigma_\lambda(0))^{-n}$ with $n>1$ also vanish.

($\Leftarrow$)
If the function $\tilde w(\sigma)$ is holomorphic at $\sigma_\lambda(0),$ then the integral $\oint_C \tilde w(\sigma)\,d\sigma$ vanishes.
On the other hand, this integral is equal to $2\pi i \sqrt{\Im T_{\lambda+i0}(H_0)} Q_{\lambda+i0}(r_\lambda).$
It now follows from Lemma~\ref{L: type I iff Im Q=0} that~$r_\lambda$ has type~I.
\end{proof}
The function $w(s)$ is holomorphic, but the adjoint function $w^*(s)$ is not.
For this reason, instead of $w^*(s),$ the meromorphic continuation $w^\dagger(s)$ of the restriction of $w^*(s)$ to the real axis will be used:
$$
  \mbC \ni s \mapsto w^\dagger(s) := (1+sT_{\lambda-i0}(H_0)J)^{-1} \sqrt{\Im T_{\lambda+i0}(H_0)}.
$$
\begin{lemma} \label{L: w(s) is holo iff w(s)w*(s) is} If $w(s)$ is a meromorphic operator-valued function
in some domain $G \subset \mbC$ which is symmetric with respect to the real axis,
then $w(s)$ is holomorphic at a real point $r_0 \in G$ if and only if so is the function $w(s)w^\dagger(s).$
\end{lemma}
\begin{proof}
If $(s-r_0)^{-k}X_k$ is the term of lowest order in the Laurent series of $w(s)$ at $s=r_0,$ then the lowest order term
in the Laurent series of the function $w(s)w^\dagger(s)$ at $s=r_0$ is $(s-r_0)^{-2k}X_kX_k^*.$ Since $X_k = 0$ if and only if $X_kX_k^* = 0,$
the claim follows.
\end{proof}
\begin{prop} \label{P: type I iff w is analytic}
Let $w(s)$ be the function given by~(\ref{F: a wonderful function}).
The following assertions are equivalent.
\begin{enumerate}
  \item[(i)] \ The point~$r_\lambda$ is of type~I.
  \item[(ii)] \ The meromorphic function $\mbC \ni s \mapsto w(s)$ is holomorphic at~$r_\lambda.$
  \item[(iii)] \ The meromorphic function $\mbC \ni s \mapsto w^\dagger(s)w(s)$ is holomorphic at~$r_\lambda.$
  \item[(iv)] \ The meromorphic function $\mbC \ni s \mapsto w^\dagger(s)$ is holomorphic at~$r_\lambda.$
  \item[(v)] \ The meromorphic function $\mbC \ni s \mapsto \Im T_{\lambda+i0}(H_s)$ is holomorphic at~$r_\lambda.$
\end{enumerate}
\end{prop}
\begin{proof}
The equivalence (i) $\iff$ (ii) is the content of Lemma~\ref{L: type I iff w(s) is holom-c}.
The equivalence (ii) $\iff$ (iv) is obvious.
The equivalence (iii) $\iff$ (v) follows from~(\ref{F: Az3v6 (4.8)}).
The equivalence (ii) $\iff$ (iii) follows from Lemma~\ref{L: w(s) is holo iff w(s)w*(s) is}.
\end{proof}

\begin{obs} \rm
The equality~(\ref{F: type I iff Im Q=0}) is plainly equivalent to the equality
$$
  P_{\lambda-i0}(r_\lambda) \sqrt{\Im T_{\lambda+i0}(H_0)} = 0,
$$
which therefore gives another characterization of points of type~I.
\end{obs}

\begin{lemma}
A resonance point~$r_\lambda$ is of type~I if and only if
\begin{equation*}
  \sqrt{\Im T_{\lambda+i0}(H_s)}\, JP_{\lambda-i0}(r_\lambda) = 0.
\end{equation*}
That is, definition~(\ref{F: type I resonance point, definition}) of a resonance point~$r_\lambda$ of type~I does not depend on the choice of sign in~$P_{\lambda\pm i0}(r_\lambda).$
\end{lemma}
\begin{proof}
Since $\Im T_{\lambda-i0}(H_s) = - \Im T_{\lambda+i0}(H_s),$
the function $\Im T_{\lambda+i0}(H_s)$ is holomorphic at some point~$s$ if and only if so is $\Im T_{\lambda-i0}(H_s).$
Since, by~(\ref{F: Az3v6 (4.8)}),
$$
  \Im T_{\lambda-i0}(H_s) = (1+sT_{\lambda+i0}(H_0)J)^{-1} \Im T_{\lambda-i0}(H_0) (1+sJT_{\lambda-i0}(H_0))^{-1},
$$
it follows from Proposition~\ref{P: type I iff w is analytic}(v) and Lemma~\ref{L: w(s) is holo iff w(s)w*(s) is}
that a resonance point~$r_\lambda$ is a point of type~I if and only if the function
$$
  h(s) = \sqrt{\Im T_{\lambda+i0}(H_0)}\, (1+sJT_{\lambda-i0}(H_0))^{-1}
$$
is holomorphic at~$r_\lambda.$ Hence, making the change of variables $\sigma = -s^{-1}$
and taking the contour integral of the function $s \cdot h(s)$ over a small circle $C$ enclosing $-r_\lambda^{-1}$
shows that if~$r_\lambda$ is a point of type~I, then
$$
  \sqrt{\Im T_{\lambda+i0}(H_0)} Q_{\lambda-i0}(r_\lambda) = 0.
$$
It follows from Lemma~\ref{L: j-dimensions coincide} that
$$
  \sqrt{\Im T_{\lambda+i0}(H_0)} JP_{\lambda-i0}(r_\lambda) = 0.
$$
Now, the argument of Lemma~\ref{L: type I iff w(s) is holom-c} shows that the last equality implies that $h(s)$ is holomorphic at~$r_\lambda;$
hence, the reverse implication is also proved.
\end{proof}

\begin{lemma} \label{L: can remove sqrt} The equality~(\ref{F: type I resonance point, definition}) holds for some value of $s$
if and only if for the same value of $s$
\begin{equation} \label{F: 9847}
  \Im T_{\lambda+i0}(H_s)\, JP_{\lambda+i0}(r_\lambda) = 0.
\end{equation}
\end{lemma}
\begin{proof} Plainly,~(\ref{F: type I resonance point, definition}) implies~(\ref{F: 9847}).
If~(\ref{F: 9847}) holds, then by the $C^*$-equality $\norm{T}^2 = \norm{T^*T}$
$$
  \norm{\sqrt{\Im T_{\lambda+i0}(H_s)}\, JP_{\lambda+i0}(r_\lambda)}^2 = \norm{Q_{\lambda-i0}(r_\lambda)J\Im T_{\lambda+i0}(H_s)\, JP_{\lambda+i0}(r_\lambda)}  = 0.
$$
\end{proof}

\begin{lemma} If~(\ref{F: type I resonance point, definition}) holds for one real non-resonant value of $s,$
then it holds for any other real non-resonant value of $s.$
\end{lemma}
\begin{proof} Assume that~(\ref{F: type I resonance point, definition}) holds for $s = r.$
By Lemma~\ref{L: can remove sqrt}, the square root in~(\ref{F: type I resonance point, definition}) can be removed, so that
\begin{equation} \label{F: T(+)JP=T(-)JP}
  T_{\lambda+i0}(H_r)J P_{\lambda + i0}(r_\lambda) = T_{\lambda-i0}(H_r)J P_{\lambda + i0}(r_\lambda).
\end{equation}
Hence, restrictions of operators $A_{\lambda+i0}(r)=T_{\lambda+i0}(H_r)J$ and $A_{\lambda-i0}(r)=T_{\lambda-i0}(H_r)J$ to
the resonance space~$\Upsilon_{\lambda+i0}(r_\lambda) = \im P_{\lambda+i0}(r_\lambda)$ coincide.
By Corollary~\ref{C: Upsilon is invariant} the resonance vector space~$\Upsilon_{\lambda+i0}(r_\lambda)$ is invariant under the operator
$A_{\lambda+i0}(r)$ and, therefore, by~(\ref{F: T(+)JP=T(-)JP}),
the vector space~$\Upsilon_{\lambda+i0}(r_\lambda)$ is invariant under the operator $A_{\lambda-i0}(r)$ too.
It follows from this and~(\ref{F: A(s)=(1+(s-r)A(r))(-1)A(r)}) that the
restrictions of operators $A_{\lambda+i0}(s)$ and $A_{\lambda-i0}(s)$ to
the resonance space~$\Upsilon_{\lambda+i0}(r_\lambda)$ coincide for all non-resonance~$s.$
Hence, for all such $s$ the equality $A_{\lambda+i0}(s)P_{\lambda+i0}(r_\lambda) = A_{\lambda-i0}(s)P_{\lambda+i0}(r_\lambda)$ holds, which is what is required.
\end{proof}

These results are summarized in the following theorem.
\begin{thm} \label{T: type I thm} Let~$\lambda$ be an essentially regular point for the pair $(\clA,F).$
Let~$H_0 \in \clA$ be an operator regular at~$\lambda$ and let $V \in \clA_0(F).$ Let~$r_\lambda \in \mbR$ be a resonance point of the path $\set{H_0+rV \colon r \in \mbR}.$
The following assertions are all equivalent to~$r_\lambda$ of being of type I.
\begin{enumerate}
\item[(i$_\pm$)] For any regular point~$r$ \ \ $\sqrt{\Im T_{\lambda+i0}(H_r)}J P_{\lambda\pm i0}(r_\lambda) = 0.$
\item[(i$^*_\pm$)] There exists a regular point~$r$ such that $\sqrt{\Im T_{\lambda+i0}(H_r)}J P_{\lambda\pm i0}(r_\lambda) = 0.$
\item[(ii$_\pm$)] For any regular point~$r$ \ \ $\sqrt{\Im T_{\lambda+i0}(H_r)}Q_{\lambda\pm i0}(r_\lambda) = 0.$
\item[(ii$^*_\pm$)] There exists a regular point~$r$ such that \ \ $\sqrt{\Im T_{\lambda+i0}(H_r)} Q_{\lambda\pm i0}(r_\lambda) = 0.$
\item[(iii$_\pm$)] The meromorphic function
$$
  w_\pm(s) := \sqrt{\Im T_{\lambda+i0}(H_0)}[1+sJT_{\lambda\pm i0}(H_0)]^{-1}
$$
is holomorphic at $s = r_\lambda.$
\item[(iii$'_\pm$)] The meromorphic function
$$
  w_\pm(s)J = \sqrt{\Im T_{\lambda+i0}(H_0)} J [1+sT_{\lambda\pm i0}(H_0)J]^{-1}
$$
is holomorphic at $s = r_\lambda.$
\item[(iv$_\pm$)] The meromorphic function
$$
  w^\dagger_\pm(s) = [1+sT_{\lambda\mp i0}(H_0)J]^{-1}\sqrt{\Im T_{\lambda+i0}(H_0)}
$$
is holomorphic at $s = r_\lambda.$
\item[(v$_\pm$)] The residue of the function $w_\pm(s)$ at $s = r_\lambda$ is zero.
\item[(vi$_\pm$)] For all $\pm$-resonance vectors the real numbers $c_{-j}$ from Proposition~\ref{P: euE (Vf)=0, k>1}
are all zero.
\item[(vii)] The function \
$
  s \mapsto \Im T_{\lambda+i0}(H_s)
$ \
is holomorphic at $s = r_\lambda.$
\item[(viii)] The function \
$
  s \mapsto J\Im T_{\lambda+i0}(H_s)J
$ \
is holomorphic at $s = r_\lambda.$
\end{enumerate}
Moreover, assertions obtained from {\rm (i$_\pm$)--(ii$_\pm$)} and {\rm (i$^*_\pm$)--(ii$^*_\pm$)} by removing the square root are also equivalent to these ones.
\end{thm}
\begin{proof}
Equivalence of items (i$_\pm$), (i$^*_\pm$), (ii$_\pm$), (ii$^*_\pm$), (iii$_\pm$), (iv$_\pm$), (v$_\pm$) and (vii) has already been proved.

It is not difficult to see that (iii$_\pm$) implies (iii$'_\pm $).
Now it will be shown that (iii$'_\pm$) implies (i$_\pm$). Making the change $\sigma = -s^{-1}$ and taking
the contour integral over $C(\sigma_\lambda(0))$ (where $\sigma_\lambda(0) = -r_\lambda^{-1}$) of the function
$s w_\pm(s)J$ gives the equality
$$
  0 = \oint_{C(\sigma_\lambda(0))} \sigma^{-1} w_\pm(-\sigma^{-1}) J\,d\sigma = \sqrt{\Im T_{\lambda+i0}(H_0)} JP_{\lambda+i0}.
$$
The item (vii) obviously implies (viii). The item (viii) combined
with Lemma~\ref{L: w(s) is holo iff w(s)w*(s) is} and equality~(\ref{F: Az3v6 (4.8)}) implies (iii$'$).

Finally, the item (vii) obviously implies (vi$_\pm$) and the item (vi$_\pm$) implies (i$_\pm$).
\end{proof}
\begin{cor} If the right hand side of~(\ref{F: (J psi,Im TJ psi)=c(-2)/s2+...}) is non-zero, then it is strictly positive
for all non-resonance points~$s.$
\end{cor}
\begin{proof} If the right hand side of~(\ref{F: (J psi,Im TJ psi)=c(-2)/s2+...}) vanishes at some point $s,$ then by implication (i$_\pm^*$) $\then$ (i$_\pm$)
of Theorem~\ref{T: type I thm} it vanishes at all points $s.$
\end{proof}
\begin{rems} \rm Properties (iii$_\pm$) and (iv$_\pm$) have something in common with the fact that the scattering matrix and $\tilde S$-function
are holomorphic in a neighbourhood of $\mbR.$
One can see this from the stationary formula for the scattering matrix, recalling the relation~(\ref{F: euE(F*psi)=sqrt Im T})
between $\sqrt{\Im T_{\lambda+i0}(H_0)}$ and $\euE_\lambda(H_0).$
\end{rems}
In addition to the equivalent conditions of Theorem~\ref{T: type I thm}, one can add the equivalent conditions
\begin{gather}
 \label{F: A+(s)=A-(s)}  A_{\lambda+i0}(s) = A_{\lambda-i0}(s) \quad \text{on } \Upsilon_{\lambda\pm i0}(r_\lambda), \\
 \label{F: B+(s)=B-(s)} B_{\lambda+i0}(s) = B_{\lambda-i0}(s) \quad \text{on } \Psi_{\lambda\pm i0}(r_\lambda), \\
 \label{F: A+=A-} \bfA_{\lambda+i0}(r_\lambda) = \bfA_{\lambda-i0}(r_\lambda) \quad \text{on } \Upsilon_{\lambda\pm i0}(r_\lambda), \\
 \label{F: B+=B-} \bfB_{\lambda+i0}(r_\lambda) = \bfB_{\lambda-i0}(r_\lambda) \quad \text{on } \Psi_{\lambda\pm i0}(r_\lambda).
\end{gather}
The equality~(\ref{F: A+(s)=A-(s)}) and Lemma~\ref{L: j-dimensions coincide} imply that
restrictions of operators $T_{\lambda+i0}(H_s)$ and $T_{\lambda-i0}(H_s)$ to the vector subspaces~$\Psi_{\lambda\pm i0}(r_\lambda)$
coincide. Hence, it follows that restrictions of operators $B_{\lambda+i0}(s) = JT_{\lambda+i0}(H_s)$
and $B_{\lambda-i0}(s) = JT_{\lambda-i0}(H_s)$ to the vector subspaces~$\Psi_{\lambda\pm i0}(r_\lambda)$ also coincide.
Hence,~(\ref{F: A+(s)=A-(s)}) implies~(\ref{F: B+(s)=B-(s)}).

Further,~(\ref{F: B+(s)=B-(s)}) and Lemma~\ref{L: j-dimensions coincide} imply that $B_{\lambda+i0}(s)J = B_{\lambda-i0}(H_s)J$
on~$\Upsilon_{\lambda\pm i0}(r_\lambda).$ Hence, $JA_{\lambda+i0}(s) = JA_{\lambda-i0}(H_s)$
on~$\Upsilon_{\lambda\pm i0}(r_\lambda),$ and therefore, by Lemma~\ref{L: j-dimensions coincide},
$A_{\lambda+i0}(s) = A_{\lambda-i0}(H_s)$ on~$\Upsilon_{\lambda\pm i0}(r_\lambda).$
Hence,~(\ref{F: B+(s)=B-(s)}) implies~(\ref{F: A+(s)=A-(s)}).

Further, definition~(\ref{F: def of bfA}) of the operator $\bfA_{\lambda+i0}$
and~(\ref{F: Az(s)Pz(rz)=sum...}) imply that~(\ref{F: A+(s)=A-(s)}) and~(\ref{F: A+=A-}) are equivalent.
Similarly, the conditions~(\ref{F: B+(s)=B-(s)}) and~(\ref{F: B+=B-}) are also equivalent.
Finally, the condition~(\ref{F: A+(s)=A-(s)}) is just a reformulation of the item (i$_\pm$) of Theorem~\ref{T: type I thm}.

\bigskip

According to Corollary~\ref{C: Upsilon 1(+)=Upsilon 1(-)} the vector spaces
$\Upsilon^1_{\lambda+i0}(r_\lambda)$ and~$\Upsilon^1_{\lambda-i0}(r_\lambda)$ of $+$-resonance
and $-$-resonance vectors of order~$1$ coincide for any real resonance point~$r_\lambda.$ For $k>1$ the vector spaces
$\Upsilon^k_{\lambda+i0}(r_\lambda)$ and~$\Upsilon^k_{\lambda-i0}(r_\lambda)$ are different in general,
but if~$r_\lambda$ is a type~I point, then these vectors spaces coincide for all $k=1,2,\ldots$ as the following proposition shows.
\begin{prop} \label{P: Upsilon(k)=...} In the conditions of Proposition~\ref{P: euE (Vf)=0, k>1}, if~$r_\lambda$ is a real resonance point of type~I, then
for all $k=1,2,\ldots$ solutions of the resonance equations
$$
  (1+(r_\lambda-r)T_{\lambda+i0}(H_r)J)^ku=0
$$
and
$$
  (1+(r_\lambda-r)T_{\lambda-i0}(H_r)J)^ku = 0
$$ coincide, that is,
\begin{equation} \label{F: Ups(+)(k)=Ups(-)(k), if type I}
  \Upsilon^k_{\lambda+i0}(r_\lambda) = \Upsilon^k_{\lambda-i0}(r_\lambda).
\end{equation}
\end{prop}
\begin{proof} This assertion follows directly from Lemma \ref{L: type I vectors}. Nevertheless, we give another proof.

The case $k=1$ follows from Proposition~\ref{P: euE (Vf)=0, k=1} (and holds for all resonance points).
Assume that the claim holds for $k-1.$
If~$u$ is a solution of the equation
$$
  (1+(r_\lambda-r)T_{\lambda+i0}(H_r)J)^ku=0,
$$
then the vector $(1+(r_\lambda-r)T_{\lambda+i0}(H_r)J)u$
is a solution of the equation
\begin{equation} \label{F: nasty formula}
  (1+(r_\lambda-r)T_{\lambda+i0}(H_r)J)^{k-1}f=0.
\end{equation}
Since the resonance point~$r_\lambda$ is of type~I, we have $\Im T_{\lambda+i0}(H_r)Ju=0.$
It follows that the vector $(1+(r_\lambda-r)T_{\lambda-i0}(H_r)J)u$
is also a solution of the equation~(\ref{F: nasty formula}).
From the induction assumption it follows that $(1+(r_\lambda-r)T_{\lambda-i0}(H_r)J)u$
is a solution of the equation
$$
  (1+(r_\lambda-r)T_{\lambda-i0}(H_r)J)^{k-1}f=0.
$$
It follows that~$u$ is a solution of $(1+(r_\lambda-r)T_{\lambda-i0}(H_r)J)^{k}u=0.$
\end{proof}
\noindent
The same argument shows that for points~$r_\lambda$ of type~I
$$
  \Psi^k_{\lambda+i0}(r_\lambda) = \Psi^k_{\lambda-i0}(r_\lambda).
$$
This equality also follows from~(\ref{F: Ups(+)(k)=Ups(-)(k), if type I}) and Lemma~\ref{L: j-dimensions coincide}.

Proposition~\ref{P: Upsilon(k)=...} implies, in particular, that for points~$r_\lambda$ of type~I the ranges of idempotent
operators $P_{\lambda+i0}(r_\lambda)$ and $P_{\lambda-i0}(r_\lambda)$ coincide.
In fact, for points of type~I these idempotents coincide, as the following theorem shows.
\begin{thm} \label{T: P(+)=P(-)} Let~$H_0$ be a self-adjoint operator from~$\clA,$
let~$\lambda$ be an essentially regular point and let~$V$ be a regularizing direction.
If a real number~$r_\lambda$ is a resonance point of type~$I,$ then
the idempotents $P_{\lambda-i0}(r_\lambda)$ and $P_{\lambda+i0}(r_\lambda)$ coincide.
\end{thm}
\begin{proof} Let $y$ be a small positive number.
Proposition~\ref{P: Th 3.3 of Az7} implies the equality
$$
  \frac 1\pi \oint_{C(r_\lambda)} \Im T_{\lambda+iy}(H_s)J\,ds = P_{\lambda+iy}(r_{\lambda}) - P_{\lambda-iy}(r_{\lambda}),
$$
where $C(r_\lambda)$ is a contour which encloses all poles $r_{\lambda+iy}^1,\ldots,r_{\lambda+iy}^N$ of the group of~$r_\lambda$
and their conjugates $\bar r_{\lambda+iy}^1,\ldots,\bar r_{\lambda+iy}^N$
(see subsection~\ref{SS: P(z)(r lambda)} for definition of poles of the group of~$r_\lambda$).
By Lemmas~\ref{L: T(l+iy) to T(l+i0)} and~\ref{L: P(l+iy) to P(l+i0)},
taking the limit $y \to 0$ in the above equality gives
\begin{equation} \label{F: oint Im T(s)=P(+)-P(-)}
  \frac 1\pi \oint_{C(r_\lambda)} \Im T_{\lambda+i0}(H_s)J\,ds = P_{\lambda+i0}(r_\lambda) - P_{\lambda-i0}(r_\lambda).
\end{equation}
By Proposition~\ref{P: type I iff w is analytic}(v), the integrand of the left hand side is holomorphic in a neighbourhood of~$r_\lambda,$
and therefore the integral vanishes. Hence, $P_{\lambda+i0}(r_\lambda) = P_{\lambda-i0}(r_\lambda).$
%
%
%
\end{proof}
Theorem~\ref{T: P(+)=P(-)} and~(\ref{F: A+=A-}) provide another proof of Proposition~\ref{P: Upsilon(k)=...}.

\begin{prop} \label{P: one more char-n of type I} A point~$r_\lambda$ is of type~I if and only if for some and thus for any non-resonant~$r$
$$
  \mathfrak h_\lambda(H_r) \ \perp \ \Psi_{\lambda+i0}(r_\lambda),
$$
where $\hlambda(H_r)$ is the fiber Hilbert space as defined by~(\ref{F: hlambda def of}).
\end{prop}
\begin{proof} This follows from items (ii$_+$) and (ii$^*_+$) of Theorem~\ref{T: type I thm} and the equality~$\Psi_{\lambda+i0}(r_\lambda) = \im Q_{\lambda+i0}(r_\lambda).$
\end{proof}
By Proposition~\ref{P: euE (Vf)=0, k=1} for any real resonance point~$r_\lambda$ the relation $\mathfrak h_\lambda(H_r) \ \perp \ \Psi^1_{\lambda+i0}(r_\lambda)$
holds. The vector space~$\Psi_{\lambda+i0}(r_\lambda)$ is in fact also the image of the resonance matrix $Q_{\lambda-i0}(r_\lambda)JP_{\lambda+i0}(r_\lambda).$
Hence, this gives another characterization of points of type~I.
\begin{prop} A point~$r_\lambda$ is of type~I if and only if for some and thus for any non-resonant~$r$
$$
  \Im T_{\lambda+i0}(H_r) Q_{\lambda-i0}(r_\lambda)JP_{\lambda+i0}(r_\lambda) = 0.
$$
\end{prop}


\subsection{Examples of points of type~I}
In this subsection we give several conditions which ensure that a resonance point has type~I.

\begin{thm} \label{T: outside ess sp then type I} Let~$\lambda$ be an essentially regular point, let~$H_0 \in \clA$ and let $V \in \clA_0(F)$ be a regularizing direction at~$\lambda.$
  If~$\lambda$ does not belong to the (necessarily common) essential spectrum of operators from~$\clA,$
  then every resonance point of the triple $(\lambda, H_0,V)$ is of type~I.
\end{thm}
\begin{proof}
In this case the function $\mbR \ni r \mapsto \Im T_{\lambda+i0}(H_r)$ is zero.
Thus, the claim follows from, for example, Theorem~\ref{T: type I thm}(vii).
\end{proof}

The following assertion immediately follows from Proposition~\ref{P: euE (Vf)=0, k=1} and definition~(\ref{F: type I resonance point, definition})
of resonance points of type~I.
\begin{thm} \label{T: order 1 then type I}
Let~$\lambda$ be an essentially regular point, let~$H_0 \in \clA,$ and let $V \in \clA_0(F)$ be a regularizing direction at~$\lambda.$
All resonance points of the triple $(\lambda, H_0,V)$ which have order one, are of type~I.
\end{thm}
\noindent
Since resonance points generically have order 1, Theorem~\ref{T: order 1 then type I} shows that points of type~I are in abundance.
An example of a resonance point not of type~I will be given in Section~\ref{S: Pert embedded eig}.

\begin{thm} \label{T: V>0, then type I}
Let~$\lambda$ be an essentially regular point, let~$H_0 \in \clA,$ and let $V \in \clA_0(F)$ be a regularizing direction at~$\lambda.$
If the perturbation~$V$ is non-negative (or non-positive), then every resonance point of the triple $(\lambda, H_0,V)$ is of type~I.
\end{thm}
\begin{proof} This follows from Proposition~\ref{P: J>0 then res points has order 1} and Theorem~\ref{T: order 1 then type I}.
\end{proof}

\begin{cor} If~$r_\lambda$ is not a point of type~I, then $\lambda \in \sigma_{ess}$ and the order of~$r_\lambda$ is not less than 2.
Moreover, in this case the perturbation $J$ is not sign definite.
\end{cor}

\begin{prop} Let~$r_\lambda$ be a real resonance point corresponding to $\lambda \pm i0 \in \partial \Pi.$
If the resonance matrix $Q_{\lambda\mp i0}(r_\lambda)JP_{\lambda\pm i0}(r_\lambda)$ is either non-negative or non-positive, then~$r_\lambda$ is of type~I.
\end{prop}
\begin{proof} This follows from Proposition~\ref{P: M>0 then type I} and Theorem~\ref{T: order 1 then type I}.
\end{proof}

\subsection{Resonance points with property~$S$}
\label{SS: property S}
\noindent
In this subsection a class of real resonance points is introduced which is strictly larger than the class of points of type~I.
Let~$\lambda$ be an essentially regular point.
\label{Page: point with property S}
A~real resonance point~$r_\lambda$ will be said to have \emph{property~$S$}
if and only if
$$
  \ker P_{\lambda+i0}(r_\lambda) = \ker P_{\lambda-i0}(r_\lambda).
$$
\begin{prop} \label{P: property S} Let~$\lambda$ be an essentially regular point and let~$r_\lambda$ be a real resonance point.
The following assertions are equivalent:
\begin{enumerate}
  \item[(i)] \ $r_\lambda$ has property $S.$
  \item[(ii)] \ $P_{\lambda+i0}(r_\lambda) P_{\lambda-i0}(r_\lambda) = P_{\lambda+i0}(r_\lambda)$ and $P_{\lambda-i0}(r_\lambda) P_{\lambda+i0}(r_\lambda) = P_{\lambda-i0}(r_\lambda).$
  \item[(iii)] \ $\im Q_{\lambda+i0}(r_\lambda) = \im Q_{\lambda-i0}(r_\lambda),$ that is, $\Psi_{\lambda+i0}(r_\lambda) = \Psi_{\lambda-i0}(r_\lambda).$
  \item[(iv)] \ $Q_{\lambda+i0}(r_\lambda) Q_{\lambda-i0}(r_\lambda) = Q_{\lambda-i0}(r_\lambda)$ and $Q_{\lambda-i0}(r_\lambda) Q_{\lambda+i0}(r_\lambda) = Q_{\lambda+i0}(r_\lambda).$
  \item[(v)] \ $Q_{\lambda-i0}(r_\lambda)JP_{\lambda+i0}(r_\lambda) = JP_{\lambda+i0}(r_\lambda).$
  \item[(vi)] \ $Q_{\lambda+i0}(r_\lambda)JP_{\lambda-i0}(r_\lambda) = JP_{\lambda-i0}(r_\lambda).$
  \item[(vii)] \ $Q_{\lambda-i0}(r_\lambda)JP_{\lambda+i0}(r_\lambda) = Q_{\lambda-i0}(r_\lambda)J.$
  \item[(viii)] \ $Q_{\lambda+i0}(r_\lambda)JP_{\lambda-i0}(r_\lambda) = Q_{\lambda+i0}(r_\lambda)J.$
  \item[(ix)] \ $Q_{\lambda-i0}(r_\lambda)JP_{\lambda+i0}(r_\lambda) = Q_{\lambda+i0}(r_\lambda) J P_{\lambda-i0}(r_\lambda).$
\end{enumerate}
\end{prop}
\begin{proof} (ii) $\then$ (i). If $P_{\lambda-i0}f = 0,$ then $P_{\lambda+i0}f = P_{\lambda+i0}P_{\lambda-i0}f = 0.$ Similarly,
if $P_{\lambda+i0}f = 0,$ then $P_{\lambda-i0}f = P_{\lambda-i0}P_{\lambda+i0}f = 0.$

(i) $\then$ (ii). Let $f$ be an arbitrary vector from~$\clK$ and let $f = f'+f'',$ where vectors $f'$ and $f''$ satisfy
$P_{\lambda-i0}f'=f'$ and $P_{\lambda-i0}f'' = 0.$ Then the vector $P_{\lambda+i0}f''$ is also zero,
and therefore $$P_{\lambda+i0}P_{\lambda-i0}f=P_{\lambda+i0}f'=P_{\lambda+i0}(f'+f'') = P_{\lambda+i0}f.$$
By the same argument, $P_{\lambda-i0}P_{\lambda+i0}f=P_{\lambda-i0}f.$

The equivalence (i) $\iff$ (iii) follows from~$\overline{\im A^*} = (\ker A)^\perp$ and~(\ref{F: Pz*=Q(bar z)}).

The equivalence (ii) $\iff$ (iv) follows from~(\ref{F: Pz*=Q(bar z)}).

The equivalence (iii) $\iff$ (v) follows from Lemma~\ref{L: j-dimensions coincide},~(\ref{F: JP=QJP}) and~(\ref{F: Q Psi=Psi}).
The equivalence (iii) $\iff$ (vi) is proved by the same argument.

The equivalences (v) $\iff$ (vii) and (vi) $\iff$ (viii) are consequences of self-adjointness of $Q_{\lambda-i0}(r_\lambda)JP_{\lambda+i0}(r_\lambda)$ and $~(\ref{F: Pz*=Q(bar z)}).$

Proof of (ii) $\then$ (ix). We have
\begin{equation*}
  \begin{split}
    Q_{\lambda-i0}(r_\lambda)JP_{\lambda+i0}(r_\lambda) & = Q_{\lambda-i0}(r_\lambda)JP_{\lambda+i0}(r_\lambda) P_{\lambda-i0}(r_\lambda)
    \\ & = Q_{\lambda-i0}(r_\lambda)Q_{\lambda+i0}(r_\lambda)J P_{\lambda-i0}(r_\lambda)
    \\ & = Q_{\lambda+i0}(r_\lambda) J P_{\lambda-i0}(r_\lambda),
  \end{split}
\end{equation*}
where the first equality follows from~(ii), the second equality follows from~(\ref{F: JP=QJ}) and the third equality follows from~(iv).

Proof of (ix) $\then$ (iii). By Proposition~\ref{P: rank M=rank P}, ranks of operators
$Q_{\lambda-i0}(r_\lambda)JP_{\lambda+i0}(r_\lambda)$ and $Q_{\lambda+i0}(r_\lambda) J P_{\lambda-i0}(r_\lambda)$
are both equal to $N=\rank Q_{\lambda\pm i0}(r_\lambda).$ Hence,
$$
  \im Q_{\lambda-i0}(r_\lambda) = \im Q_{\lambda-i0}(r_\lambda)JP_{\lambda+i0}(r_\lambda) = \im Q_{\lambda+i0}(r_\lambda)JP_{\lambda-i0}(r_\lambda) = \im Q_{\lambda+i0}(r_\lambda).
$$
\end{proof}
\noindent According to Theorem~\ref{T: spec P(+)P(-)}, the operators
$$
  P_{\lambda+i0}(r_\lambda)P_{\lambda-i0}(r_\lambda)-P_{\lambda+i0}(r_\lambda) \ \ \text{and} \ \  P_{\lambda-i0}(r_\lambda)P_{\lambda+i0}(r_\lambda)-P_{\lambda-i0}(r_\lambda)
$$
are nilpotent. Hence, a real resonance point has property~$S$
if and only if the nilpotent parts of $P_{\lambda+i0}(r_\lambda)P_{\lambda-i0}(r_\lambda)$ and $P_{\lambda-i0}(r_\lambda)P_{\lambda+0}(r_\lambda)$ are zero.

\begin{prop} \label{P: points without property S exist} Every resonance point of type~I has property~$S.$ There are resonance points which do not have property $S,$
and there are points with property~$S$ which are not of type~I.
\end{prop}
\noindent The first part of this proposition is trivial; to prove it one can note that by Theorem~\ref{T: P(+)=P(-)} for points~$r_\lambda$ of type~I we have
$P_{\lambda+i0}(r_\lambda) = P_{\lambda-i0}(r_\lambda)$ and therefore~$r_\lambda$ has property~$S.$ 
Examples of resonance points with the required properties will be given in part~\ref{SSS: order 2} of section~\ref{S: Pert embedded eig}.

Propositions~\ref{P: property S} and~\ref{P: points without property S exist} give answers to some natural questions,
such as whether the two operators $Q_{\lambda-i0}(r_\lambda)JP_{\lambda+i0}(r_\lambda)$ and $Q_{\lambda+i0}(r_\lambda) J P_{\lambda-i0}(r_\lambda)$
always coincide or not.

\begin{prop} If $\Upsilon_{\lambda+i0}(r_\lambda) = \Upsilon_{\lambda-i0}(r_\lambda),$ then
$P_{\lambda+i0}(r_\lambda) = P_{\lambda-i0}(r_\lambda).$
\end{prop}
\begin{proof} If $\Upsilon_{\lambda+i0}(r_\lambda) = \Upsilon_{\lambda-i0}(r_\lambda),$ then since $\Upsilon_{\lambda\pm i0}(r_\lambda) = \im Q_{\lambda\pm i0}(r_\lambda),$
it follows from Proposition \ref{P: property S}(iii) that the kernels of the idempotents $P_{\lambda+i0}(r_\lambda)$ and $P_{\lambda-i0}(r_\lambda)$ coincide.
Since the ranges $\Upsilon_{\lambda+i0}(r_\lambda)$ and $\Upsilon_{\lambda-i0}(r_\lambda)$ of these idempotents are also equal by the premise, it follows that
$P_{\lambda+i0}(r_\lambda) = P_{\lambda-i0}(r_\lambda).$
\end{proof}
Plainly, the equality $P_{\lambda+i0}(r_\lambda) = P_{\lambda-i0}(r_\lambda)$ is also equivalent to $Q_{\lambda+i0}(r_\lambda) = Q_{\lambda-i0}(r_\lambda),$
but these equalities are not equivalent to $\Psi_{\lambda+i0}(r_\lambda) = \Psi_{\lambda-i0}(r_\lambda),$ which is property~$S.$


\section{Perturbation of an embedded eigenvalue}
\label{S: Pert embedded eig}
In this section we study the behaviour of an eigenvalue of a self-adjoint operator embedded into the essential
spectrum as the operator undergoes a perturbation. This is a classical problem, but in this section some new results will be given.
Not only is the behaviour of embedded eigenvalues under perturbations interesting on its own, but this investigation
will also provide examples and counter-examples to many possible relations which may be posed in regard to the material of previous sections.
In fact, from the point of view of deductive structure, this section is quite independent of previous ones; on the other hand,
this section was written almost in parallel with previous sections, and it is this study of embedded eigenvalues
that gave many suggestions about possible properties of resonance points.

\begin{lemma} \label{L: An converges iff ABCDn converges}
Let~$N$ be a positive integer and let $\hilb = \hat \hilb \oplus \mbC^N$ be a decomposition of a Hilbert space $\hilb$
into the orthogonal direct sum of another Hilbert space $\hat \hilb$ and $\mbC^N.$
If  $$\left(\begin{matrix}
    A_n & B_n \\
    C_n & D_n
  \end{matrix}\right), \ n=1,2,\ldots $$ is a sequence of operators on the Hilbert space $\hilb$ which converges
to an operator $$\left(\begin{matrix} A & B \\ C & D \end{matrix}\right)$$
in the uniform norm, then this convergence holds also in $p$-norm if and only if the sequence~$A_n, \ n=1,2,\ldots$ converges to~$A$ in $p$-norm.
\end{lemma}
\begin{proof}
  The ``only if'' part is trivial. Since the ranks of operators $B,B_1,B_2,\ldots$ and $C,C_1,C_2,\ldots$ are bounded by~$N,$
  the ``if'' part follows from Lemma~\ref{L: An to A iff An to A in p-norm}.
\end{proof}

Usually by $r_z$ we denote a resonance point corresponding to~$z.$ In the following two lemmas we divert from this
agreement. The reason for this is that later in this section we are going to embed operators $H_0$ and $V$ to a slightly larger Hilbert space,
where a non-resonance point $r_z$ will become a resonant one.

\begin{lemma} Let~$r_z$ be a non-resonance point for~$z.$ For any regular points $s$ and $t$
the operator $(1+(r_z-s)A_z(s))^{-1}$ is a linear combination of operators
$1$ and $(1+(r_z-t)A_z(t))^{-1},$ namely,
$$
  (1+(r_z-s)A_z(s))^{-1} = \frac{t-s}{t-r_z} + \frac{s-r_z}{t-r_z}\, (1+(r_z-t)A_z(t))^{-1}.
$$
\end{lemma}
\begin{proof} Proof is a direct calculation based on equalities~(\ref{F: A(s) and A(r) commute}) and~(\ref{F: II resolvent identity}).
\end{proof}
\begin{cor}
Let~$r_z$ be a non-resonance point for~$z.$ For any integer $k\geq 1$ and for any regular points $s$ and $t$
the operator $(1+(r_z-s)A_z(s))^{-k}$ is a linear combination of operators
$$
  1, \ (1+(r_z-t)A_z(t))^{-1}, \ \ldots, \ 1+(r_z-t)A_z(t))^{-k}.
$$
\end{cor}
\begin{proof} This follows from previous lemma and induction.
\end{proof}

\subsection{Vector spaces $\Upsilon^j_{\lambda\pm i0}(r_\lambda)$}

Let~$H_{r_\lambda}$ be a self-adjoint operator on a Hilbert space
$\hilb$ with an eigenvalue~$\lambda$ of multiplicity one.
No assumptions are made about location of this eigenvalue yet: it can
be outside of essential spectrum or inside of it. Let~$\chi$ be
the corresponding eigenvector:
\begin{equation} \label{F: H0 chi=lambda chi}
  H_{r_\lambda} \chi = \lambda \chi.
\end{equation}
The orthogonal complement of~$\chi$ will be denoted by $\hat \hilb.$ The
subspace $\hat \hilb$ reduces~$H_{r_\lambda}$ and the reduction will be
denoted by $\hat H_{r_\lambda}.$ \label{Page: hat H0} Thus, the
Hilbert space $\hilb$ becomes split into a direct product $\hat
\hilb \oplus \mbC,$ and in this representation of $\hilb$ the
operator~$H_{r_\lambda}$ has the form
\begin{equation} \label{F: H0 gen case}
  H_{r_\lambda} = \left(\begin{matrix}
    \hat H_{r_\lambda} & 0 \\
    0 & \lambda
  \end{matrix}\right).
\end{equation}
We have to choose a rigging operator~$F.$ To simplify calculations, the operator~$F$ is chosen to be of the form
\begin{equation} \label{F: rigging (F0,1) gen case}
  F = \left(\begin{matrix}
    \hat F & 0 \\
    0 & 1
  \end{matrix}\right),
\end{equation}
where $\hat F \colon \hat \hilb \to \hat \clK$ is a rigging
operator in the Hilbert space $\hat \hilb,$ so that the
operator~$F$ itself acts from $\hilb$ to $\clK = \hat \clK \oplus
\mbC.$ Since~$\lambda$ is a non-degenerate eigenvalue
of~$H_{r_\lambda},$ it cannot be an eigenvalue of $\hat
H_{r_\lambda},$ but it is still possible that~$\lambda$ does not
belong to $\Lambda(\hat H_{r_\lambda},\hat F).$ By
Proposition~\ref{P: Az 4.1.10}, since~$\lambda$ is an eigenvalue
of~$H_{r_\lambda},$ the number~$\lambda$ does not belong to
$\Lambda(H_{r_\lambda},F),$ but we assume that
\begin{equation} \label{F: lambda in hat (H,F)}
  \lambda \in \Lambda(\hat H_{r_\lambda},\hat F).
\end{equation}
This assumption means that there are no other singularities
of~$H_{r_\lambda}$ at~$\lambda$ except the fact that~(\ref{F: H0
chi=lambda chi}) holds. Let~$V$ be a self-adjoint
operator from the affine space $\clA(H_{r_\lambda},F).$ The operator~$V$ has
the form
$$
  V = \left(\begin{matrix}
    \hat V & \hat v \\
    \scal{\hat v}{\cdot} & \alpha
  \end{matrix}\right),
$$
where $\hat V$ is a self-adjoint operator in $\hat \hilb.$ Since $V \in \clA(H_{r_\lambda},F),$ there exists a bounded self-adjoint operator
$J$ on~$\clK$ such that
$$
  V = F^*JF.
$$
Let
\begin{equation} \label{F: J=(J,uu;uu,alpha)}
  J = \left(\begin{matrix}
    \hat J & \hat \psi \\
    \la \hat \psi, \cdot \ra & \alpha
  \end{matrix}\right)
\end{equation}
be the representation of $J$ in the direct product $\hat \clK \oplus \mbC,$
where $\hat J$ is a bounded self-adjoint operator on $\hat \clK,$ $\hat \psi \in \hat \clK$ and $\alpha \in \mbR.$
Then one can see that
\begin{equation} \label{F: hat V=hat F* J hat F}
  \hat V = \hat F^* \hat J\hat F
\end{equation}
and
$$
  \hat v = \hat F^* \hat \psi.
$$
In particular, the vector $\hat v$ belongs to the Hilbert space $\hat \hilb_+(\hat F)$
and the operator $\hat V$ belongs to the vector space $\clA_0(\hat F).$
The eigenvector $\chi$ of $H_{r_\lambda}$ in $\hat \hilb \oplus \mbC$ has the form {$\const$\small $\cdot\left(\begin{matrix} 0 \\ 1\end{matrix}\right).$}
The matrix components $\hat \psi$ and $\alpha$ of the operator $J$
can be recovered by equalities
$$
  \alpha = \scal{\chi}{V\chi} = \scal{F\chi}{JF\chi}
$$
and
$$
  \hat \psi \oplus 0 = JF\chi - \alpha F\chi.
$$
\begin{lemma} If $\alpha = 0,$ then $\hat \psi \oplus 0$ is a co-resonance vector of order 1.
\end{lemma}
\begin{proof} By Theorem~\ref{T: res eq-n and eigenvectors}, the vector $F \chi$ is a resonance vector of order 1.
Hence, by Lemma~\ref{L: j-dimensions coincide}, it follows that the vector $\hat \psi \oplus 0 = J F \chi$ is a co-resonance vector of order 1.
\end{proof}
\noindent
For a real number $s$ the operator $H_s$ is defined by
\begin{equation} \label{F: Hs=(Hs,sv;sv,l+s alpha)}
  H_s :=  H_{r_\lambda} + (s-r_\lambda)V
  = \left(\begin{matrix}
    \hat H_s &  (s-r_\lambda)\hat v \\
    (s-r_\lambda)\la \hat v, \cdot \ra & \lambda+(s-r_\lambda)\alpha
  \end{matrix}\right),
\end{equation}
where \label{Page: hat Hs}
\begin{equation} \label{F: hat Hs}
  \hat H_s = \hat H_{r_\lambda} + (s-r_\lambda)\hat V.
\end{equation}
A direct but a bit lengthy calculation shows that the operator $T_z(H_s) = F R_z(H_s)F^*$ is given by
\begin{equation} \label{F: Tz(Hs) gen case}
  T_z(H_s) =
    \left( \begin{array}{ll}
       T_z(\hat H_s) + (s-r_\lambda)^2\euD_z(s)\scal{\hat u_{\bar z}(s)}{\cdot} \hat u_z(s) & \ \    (r_\lambda-s)\euD_z(s) \hat u_z(s) \\
       (r_\lambda-s)\euD_z(s)\scal{\hat u_{\bar z}(s)}{\cdot}                                   & \ \    \euD_z(s)
    \end{array} \right),
\end{equation}
where \label{Page: uz(s)}
\begin{equation} \label{F: uz(s)}
  \hat u_z(s) = T_z(\hat H_s) \hat \psi
\end{equation}
and \label{Page: euDz(s)}
\begin{equation} \label{F: def of euA(s)}
  \euD_z(s) = \brs{\lambda - z + (s-r_\lambda)\alpha - (s-r_\lambda)^2\scal{\hat \psi}{\hat u_z(s)}}^{-1}.
\end{equation}

The condition (\ref{F: lambda in hat (H,F)}) means that the operator $\hat H_{r_\lambda}$ is regular at~$\lambda,$ and thus any perturbation operator $\hat V$
of the form (\ref{F: hat V=hat F* J hat F}) is a regularizing direction at~$\lambda$ for the operator $\hat H_{r_\lambda}.$
We wish to find conditions which ensure that the operator $V$ is a regularizing direction at~$\lambda$ for $H_{r_\lambda}.$
Recall that $V$ is a regularizing direction at~$\lambda$ for $H_{r_\lambda}$ if for some real number $s$ the operator $T_z(H_s)$
has the norm limit $T_{\lambda+i0}(H_s).$ Since the norm limit $T_{\lambda+i0}(\hat H_s)$ of $T_{\lambda+iy}(\hat H_{s})$ exists for some $s$ (namely, for $s = r_\lambda$)
by the assumption (\ref{F: lambda in hat (H,F)}), it follows from (\ref{F: Tz(Hs) gen case})
and Lemma~\ref{L: An converges iff ABCDn converges}
that the norm limit $T_{\lambda+i0}(H_s)$ exists for some real~$s$ if and only if the limit $\euD_{\lambda+i0}(s)$
exists for some real~$s.$ From the definition (\ref{F: def of euA(s)}) of $\euD_z(s)$ it is easy to see that the limit $\euD_{\lambda+i0}(s)$ exists if and only if either $\alpha\neq 0$
or both $\alpha = 0$ and $\scal{\hat \psi}{\hat u_{\lambda+i0}(s)} \neq 0.$
Thus, we have proved the following
\begin{lemma} \label{L: V is reg-ing iff alpha neq 0 or ...} The operator $V = F^*JF$ where $F$ and $J$ are defined by (\ref{F: rigging (F0,1) gen case})
and (\ref{F: J=(J,uu;uu,alpha)}), is a regularizing direction for resonant at~$\lambda$ operator
$H_{r_\lambda}$ given by (\ref{F: H0 gen case}), if and only if $\alpha \neq 0$ or both $\alpha = 0$
and
\begin{equation} \label{F: psi u(s)neq 0}
  \text{for some real number} \ s \quad \scal{\hat \psi}{\hat u_{\lambda+i0}(s)} \neq 0.
\end{equation}
\end{lemma}
From now on we shall assume that~$V$ is a regularizing direction for $H_{r_\lambda}.$

Let \label{Page: hat A}
\begin{equation} \label{F: hat Az(s)}
  \hat A_z(s) = T_z(\hat H_s) \hat J \quad \text{and} \quad \hat B_z(s) = \hat J\,T_z(\hat H_s).
\end{equation}
The operator~$A_z(s) = T_z(H_s)J$ is equal to
{\small
\begin{equation*}
 \begin{split}
  & A_z(s) =
    \left( \begin{array}{ll}
       T_z(\hat H_s) + (s-r_\lambda)^2\euD_z(s)\scal{\hat u_{\bar z}(s)}{\cdot} \hat u_z(s)    & \ \    (r_\lambda-s)\euD_z(s) \hat u_z(s) \\
       (r_\lambda-s)\euD_z(s)\scal{\hat u_{\bar z}(s)}{\cdot}                                   & \ \    \euD_z(s)
    \end{array} \right)
  \left(\begin{matrix}
    \hat J  & \hat \psi \\
    \la \hat \psi, \cdot \ra & \alpha
  \end{matrix}\right)
  \\ & = \left(\begin{matrix}
    \hat A_z(s) + (r_\lambda-s)\euD_z(s)\scal{\hat \psi +(r_\lambda-s)\hat J\hat u_{\bar z}(s)}{\cdot} \hat u_z(s)    &    \SqBrs{1 + (s-r_\lambda)\euD_z(s)\big((s-r_\lambda) \scal{\hat u_{\bar z}(s)}{\hat \psi} - \alpha\big)} \hat u_z(s)\\
     \euD_z(s) \scal{\hat \psi +(r_\lambda-s)\hat J\hat u_{\bar z}(s)}{\cdot} & -\euD_z(s)\big((s-r_\lambda)\scal{\hat u_{\bar z}(s)}{\hat \psi} - \alpha\big)
  \end{matrix}\right).
 \end{split}
\end{equation*}
}
\!In what follows the operator $1 + (r_\lambda-s)\hat A_z(s)$ will be encountered very often. For this reason, we introduce a special notation for this operator:
\begin{equation} \label{F: euF}
  \euF_z(s) = 1 + (r_\lambda-s)\hat A_z(s).
\end{equation}\label{Page: euF}
Note that
\begin{equation} \label{F: euF*}
  \euF_z^*(s) = 1 + (r_\lambda-s)\hat B_{\bar z}(s).
\end{equation}

\begin{lemma} \label{L: about euF(-1)}
$$
  \euF^{-1}_{\lambda+i0}(s) = 1 + (s-r_\lambda)\hat A_{\lambda+i0}(r_\lambda).
$$
\end{lemma}
\begin{proof} This equality follows from~(\ref{F: II resolvent identity}).
\end{proof}

Since by~(\ref{F: uz(s)}) and~(\ref{F: hat Az(s)})
$$
  \hat \psi +(r_\lambda-s)\hat J\hat u_{\bar z}(s) = [1+(r_\lambda-s)\hat B_{\bar z}(s)]\hat \psi = \euF^*_z(s)\hat \psi,
$$
the following lemma has been proved.
\begin{lemma} Let~$F$ be given by~(\ref{F: rigging (F0,1) gen case}), let $J$ be given by~(\ref{F: J=(J,uu;uu,alpha)}),
and let $H_s$ be given by~(\ref{F: Hs=(Hs,sv;sv,l+s alpha)}). Then the operator~$A_z(s) = T_z(H_s)J$
is equal to
\begin{equation} \label{F: Az(s)=2x2 matrix}
  \left(\begin{matrix}
    \hat A_z(s) + (r_\lambda-s)\euD_z(s)\scal{\euF_{z}^*(s)\hat \psi}{\cdot} \hat u_z(s)    &    \SqBrs{1 + (s-r_\lambda)\euD_z(s)\big((s-r_\lambda) \scal{\hat u_{\bar z}(s)}{\hat \psi} - \alpha\big)} \hat u_z(s)\\
     \euD_z(s) \scal{\euF_{z}^*(s)\hat \psi}{\cdot} & -\euD_z(s)\big((s-r_\lambda)\scal{\hat u_{\bar z}(s)}{\hat \psi} - \alpha\big)
  \end{matrix}\right),
\end{equation}
where~$\euD_z(s)$ is given by~(\ref{F: def of euA(s)}), $\hat
u_z(s)$ is given by~(\ref{F: uz(s)}), $\euF_z^*(s)$ is given by~(\ref{F: euF*}), and $\hat A_z(s)$ is given by~(\ref{F: hat Az(s)}).
\end{lemma}
Now we study the operator~(\ref{F: Az(s)=2x2 matrix}) when~$z$ belongs to the boundary of $\Pi,$
that is, $z = \lambda \pm i0.$ It will be assumed that $z = \lambda + i0,$ but all the equalities and assertions have appropriate analogues for
$z = \lambda - i0$ too.
\noindent If $z = \lambda+i0,$ then, using definition~(\ref{F: def of euA(s)}) of~$\euD_z(s)$
and noting that
$$
  \scal{\hat u_{\bar z}(s)}{\hat \psi} = \scal{\hat \psi}{\hat u_{z}(s)},
$$
one can see that the $(1,2)$-entry of~(\ref{F: Az(s)=2x2 matrix}) vanishes and therefore this yields the following equality
\begin{equation} \label{F: A+(s)=2 by 2 matrix}
    A_{\lambda+i0}(s) =
     \left(\begin{matrix}
      \hat A_{\lambda+i0}(s) +(r_\lambda-s) \euD_{\lambda+i0}(s) \scal{\euF^*_{\lambda+i0}(s)\hat \psi}{\cdot} \hat u_{\lambda+i0}(s) & 0 \\
       \euD_{\lambda+i0}(s) \scal{\euF^*_{\lambda+i0}(s)\hat \psi}{\cdot} & (s-r_\lambda)^{-1}
    \end{matrix}\right).
\end{equation}
Hence, the resonance equation of order~$k$ (see~(\ref{F: res eq-n}))
$$
  [1+(r_\lambda-s)A_{\lambda+i0}(s)]^k u = 0
$$
takes the form
\begin{equation} \label{F: big res equ-n}
     \left(\begin{matrix}
      \euF_{\lambda+i0}(s) + (s-r_\lambda)^2 \euD_{\lambda+i0}(s) \scal{\euF^*_{\lambda+i0}(s)\hat \psi}{\cdot} \hat u_{\lambda+i0}(s) & 0 \\
       (r_\lambda-s)\euD_{\lambda+i0}(s) \scal{\euF^*_{\lambda+i0}(s)\hat \psi}{\cdot} & 0
    \end{matrix}\right)^k u = 0.
\end{equation}
Hence, the vector space~$\Upsilon^1_{\lambda+i0}(r_\lambda)$ of solutions of this equation when $k=1$ consists of all vectors of the form
$$
  \left(\begin{matrix}
    \hat u \\ b
  \end{matrix}\right),
$$
where $b \in \mbC$ and $\hat u$ is a solution of the equations
\begin{equation} \label{F: two equalities}
  \euF_{\lambda+i0}(s)\hat u = 0 \ \ \text{and} \ \ \scal{\euF^*_{\lambda+i0}(s)\hat \psi}{\hat u} = 0.
\end{equation}
The vector space of resonance vectors of order $\leq k$ for the pair $\hat H_s,\hat V$ at $s = r_\lambda$ will be denoted by
$\hat \Upsilon^k_{\lambda\pm i0}(r_\lambda).$ In particular, a vector $\hat u$ belongs to $\hat \Upsilon^1_{\lambda + i0}(r_\lambda)$
if and only if $\euF_{\lambda+i0}(s)\hat u = 0.$
Since the second of the equalities~(\ref{F: two equalities}) follows from the first one, it follows that
$$
  \Upsilon^1_{\lambda+i0}(r_\lambda) = \hat \Upsilon^1_{\lambda+i0}(r_\lambda) \oplus \mbC.
$$
In fact, the condition~(\ref{F: lambda in hat (H,F)}) which says that the number~$\lambda$ is a
regular point of the pair $(\hat H_{r_\lambda},\hat F),$ is equivalent to the equality
$\Upsilon^1_{\lambda+i0}(\hat H_{r_\lambda}, \hat V) = \set{0},$ and therefore
\begin{equation} \label{F: Upsilon1=0+C}
  \Upsilon^1_{\lambda+i0}(r_\lambda) = \set{0} \oplus \mbC.
\end{equation}

\bigskip
\noindent
We introduce the following notation for convenience.
\\ {\bf Notation.} \ Let $j = -1,0,1,2,\ldots.$ We define a vector $\hatlmbu{j}$ by equality
\begin{equation} \label{F: vj(l+i0)(s)}
  \hatlmbu{j} = \euF^{-j}_{\lambda+i0}(s) \hat u_{\lambda+i0}(r_\lambda).
\end{equation}\label{Page: vj(l+i0)(s)}
The operator $\hat A_{\lambda+i0}(s)$ is compact,
and the assumption~(\ref{F: lambda in hat (H,F)}) means that the operator $$\euF_{\lambda+i0}(s) = 1+(r_\lambda-s)\hat A_{\lambda+i0}(s)$$
has zero kernel. Hence, it is invertible
and therefore the vectors~(\ref{F: vj(l+i0)(s)}) are well-defined.

\begin{lemma} \label{L: [...](-1)uu(s)=uu(0)} The following equality holds:
\begin{equation} \label{F: [...](-1)uu(s)=uu(0)}
  \euF^{-1}_{\lambda+i0}(s)\hat u_{\lambda+i0}(s) = \hat u_{\lambda+i0}(r_\lambda).
\end{equation}
That is, $$\hatlmbu{-1} = \hat u_{\lambda+i0}(s).$$
In particular, the vector $\euF^{-1}_{\lambda+i0}(s)\hat u_{\lambda+i0}(s)$ does not depend on~$s.$
\end{lemma}
\begin{proof} This follows from~(\ref{F: A(s)=(1+(s-r)A(r))(-1)A(r)}) (or rather its proof) and definition~(\ref{F: uz(s)}) of the vector
$\hat u_{\lambda+i0}(s).$
\end{proof}
Plainly, the equality
$$\hatlmbu{0} = \hat u_{\lambda+i0}(r_\lambda)$$
also holds.

\begin{lemma} \label{L: beau vectors} Let~$H_s,$~$V$ and~$F$ be as above. 
For each $j=1,2,3,\ldots$ the resonance vector space~$\Upsilon^j_{\lambda+i0}(r_\lambda)$
is the linear span of the following $j$ vectors
\begin{equation} \label{F: beau vectors to infty}
  \twovector{0}{1}, \ \ \twovector{\hat u_{\lambda+i0}(r_\lambda)}{0}, \ \ \twovector{\hatlmbu{1}}{0}, \ \ \ldots, \ \ \twovector{\hatlmbu{j-2}}{0}.
\end{equation}
In particular, $\dim \Upsilon^j_{\lambda+i0}(r_\lambda) \leq j.$
\end{lemma}
\begin{proof} For $j=1$ this has already been observed, see~(\ref{F: Upsilon1=0+C}).
Assume that $\twovector{\hat \phi}{a}$ is a vector of order two, that is, $\twovector{\hat \phi}{a}$
is a solution of~(\ref{F: big res equ-n}) with $k=2$ and $\hat \phi \neq 0.$
Applying to this vector the operator $[1+(r_\lambda-s)A_{\lambda+i0}(s)]$ gives a vector of order 1.
Since by~(\ref{F: Upsilon1=0+C}) such a vector has the form $\twovector{0}{b}$ with non-zero~$b,$
the first component of the vector $[1+(r_\lambda-s)A_{\lambda+i0}(s)]\twovector{\hat \phi}{a}$ is to be zero:
\begin{equation} \label{F: ga-ga-ga=0}
  \euF_{\lambda+i0}(s)\hat \phi + (s-r_\lambda)^2 \euD_{\lambda+i0}(s) \scal{\euF^*_{\lambda+i0}(s)\hat \psi}{\hat \phi} \hat u_{\lambda+i0}(s) = 0,
\end{equation}
and the second component must be non-zero:
$$
  \scal{\euF^*_{\lambda+i0}(s)\hat \psi}{\hat \phi} \neq 0.
$$
Applying the operator $\euF^{-1}_{\lambda+i0}(s)$ to the equality~(\ref{F: ga-ga-ga=0}) and
using~(\ref{F: [...](-1)uu(s)=uu(0)}) gives the equality
\begin{equation} \label{F: ga-ga-ga=0(2)}
  \hat \phi + (s-r_\lambda)^2 \euD_{\lambda+i0}(s) \scal{\euF^*_{\lambda+i0}(s)\hat \psi}{\hat \phi} \hat u_{\lambda+i0}(r_\lambda) = 0.
\end{equation}
It follows from this that if $\twovector{\hat \phi}{a}$ is a vector of order two, then $\hat \phi$ has to be co-linear with the vector $\hat u_{\lambda+i0}(r_\lambda).$
It follows that the vector space $\Upsilon^2_{\lambda+i0}(r_\lambda)$ has dimension $\leq 2$
and that~$\Upsilon^2_{\lambda+i0}(r_\lambda)$ 
is a subspace of the linear span of $\twovector{0}{1}$ and
$$
  \twovector{\euF^{-1}_{\lambda+i0}(s)\hat u_{\lambda+i0}(s)}{0} = \twovector{\hat u_{\lambda+i0}(r_\lambda)}{0}.
$$
This proves the assertion for $j=2.$
Now, assuming that the assertion holds for $j=k,$ 
it will be shown that it holds for $j=k+1.$
Let $\twovector{\hat \phi}{a}$ be a vector of order $\leq k+1.$ Then the vector
$$[1+(r_\lambda-s)A_{\lambda+i0}(s)]\twovector{\hat \phi}{a}$$ has order $\leq k.$ By the induction assumption,
the first component of this vector given by the left hand side of~(\ref{F: ga-ga-ga=0}), is a linear combination of vectors
\begin{equation*}
  \hat u_{\lambda+i0}(r_\lambda), \ \hatlmbu{1}, \ \ldots, \ \hatlmbu{k-2}.
\end{equation*}
Thus, the vector
\begin{equation*} 
  \hat \phi + (s-r_\lambda)^2 \euD_{\lambda+i0}(s) \scal{\euF^*_{\lambda+i0}(s)\hat \psi}{\hat \phi} \euF^{-1}_{\lambda+i0}(s)\hat u_{\lambda+i0}(s)
\end{equation*}
is a linear combination of vectors
\begin{equation*}
  \euF^{-1}_{\lambda+i0}(s)\hat u_{\lambda+i0}(r_\lambda) = u_{\lambda+i0}^{(1)}(s), \ \hatlmbu{2}, \ \ldots, \ \hatlmbu{k-1}.
\end{equation*}
It follows from this and~(\ref{F: [...](-1)uu(s)=uu(0)}) that $\hat \phi$ is a linear combination of vectors
\begin{equation*}
  \hat u_{\lambda+i0}(r_\lambda), \ \hatlmbu{1}, \ \ldots, \ \hatlmbu{k-1}.
\end{equation*}
Proof is complete.
\end{proof}

\begin{lemma} \label{L: beau vectors(2)} 
Order of the resonance point~$r_\lambda$
is not less than 2 if and only if $\alpha = 0.$
If this is the case, then the vector space~$\Upsilon^2_{\lambda+i0}(r_\lambda)$ is two-dimensional and is generated by vectors
$$
  F\chi = \twovector{0}{1} \quad \text{and} \quad \twovector{\hat u_{\lambda+i0}(r_\lambda)}{0},
$$
which have orders~$1$ and~$2$ respectively.
\end{lemma}
\begin{proof}
By Lemma~\ref{L: beau vectors}, a resonance vector of order $\leq 2$ has the form
$$
  \twovector{ \hat u_{\lambda+i0}(r_\lambda)}{b}.
$$
The vector
$$
  \twovector{\hat u_{\lambda+i0}(r_\lambda)}{0}
$$
is a resonance vector of order~$2$ if and only if
$$
  [1+(r_\lambda-s)A_{\lambda+i0}(s)]\twovector{\hat u_{\lambda+i0}(r_\lambda)}{0}
$$
is a vector of order 1, and thus has the form $\twovector{0}{b}.$
That is, this is equivalent to the first component of this vector being equal to zero:
$$
  \euF_{\lambda+i0}(s)\hat u_{\lambda+i0}(r_\lambda) + (s-r_\lambda)^2 \euD_{\lambda+i0}(s) \scal{\euF^*_{\lambda+i0}(s)\hat \psi}{\hat u_{\lambda+i0}(r_\lambda)} \hat u_{\lambda+i0}(s) = 0.
$$
Applying to this equality the operator $\euF_{\lambda+i0}^{-1}(s)$ and using Lemma~\ref{L: [...](-1)uu(s)=uu(0)} we infer that this equality is equivalent to
$$
  1 + (s-r_\lambda)^2 \euD_{\lambda+i0}(s) \scal{\hat \psi}{\hat u_{\lambda+i0}(s)} = 0.
$$
Definition~(\ref{F: def of euA(s)}) of~$\euD_{\lambda+i0}(s)$ implies that this equality is equivalent to
$$
  (s-r_\lambda)^2  \scal{\hat \psi}{\hat u_{\lambda+i0}(s)} = (s-r_\lambda)^2  \scal{\hat \psi}{\hat u_{\lambda+i0}(s)} - \alpha (s-r_\lambda).
$$
It follows that order~$d$ of the resonance point~$r_\lambda$ is not less than two if and only if $\alpha = 0.$
\end{proof}

Since throughout this section we are assuming that $V$ is a regularizing direction,
Lemma \ref{L: beau vectors(2)} combined with Lemma \ref{L: V is reg-ing iff alpha neq 0 or ...}
imply the following
\begin{cor} If the order $d$ of the resonance point~$r_\lambda$ is not less than two, then for some real number $s$
$$
  \scal{\hat \psi}{\hat u_{\lambda+i0}(s)} \neq 0.
$$
\end{cor}

Since the vector spaces $\Upsilon_z^j(r_z)$ have the stability property $\Upsilon_z^j(r_z) = \Upsilon_z^{j+1}(r_z) \then \Upsilon_z^j(r_z) = \Upsilon_z(r_z),$
Lemma \ref{L: beau vectors} has the following corollary.
\begin{thm} \label{T: beau vectors} Let $d$ be an integer $\geq 2.$ 
The following assertions are equivalent.
\begin{enumerate}
  \item Order of the resonance point~$r_\lambda$ is equal to $d.$
  \item The dimension of the vector space $\Upsilon_{\lambda+i0}(r_\lambda)$ is equal to $d.$
  \item The vectors
\begin{equation*}
  \hat u_{\lambda+i0}(r_\lambda), \ \hatlmbu{1}, \ \ldots, \ \hatlmbu{d-2}
\end{equation*}
are linearly independent and the vector $\hatlmbu{d-1}$ is a linear combination of these vectors.
\end{enumerate}

Further, if the order of~$r_\lambda$ is equal to~$d,$ then for all $j=1,2,\ldots,d$
the vector space~$\Upsilon^j_{\lambda+i0}(r_\lambda)$ is $j$-dimensional and is generated by vectors
$$
  \twovector{0}{1}, \ \ \twovector{\hat u_{\lambda+i0}(r_\lambda)}{0}, \ \ \twovector{\hatlmbu{1}}{0}, \ \ \ldots, \ \
  \twovector{\hatlmbu{j-2}}{0},
$$
which have orders $1,2,\ldots,d$ respectively.
\end{thm}
\noindent
This theorem gives a criterion for the order of~$r_\lambda$ to be equal to $d$ but it is not very tangible.
To get a better criterion, one needs to find out when a vector
$$
  \twovector{\hatlmbu{j-2}}{0}, \ j=1,2,3,\ldots
$$
is a resonance vector of order $j.$ Lemma \ref{L: beau vectors(2)} gives an answer to this question
in the case of $j=2.$

\begin{thm} \label{T: order k case (2)} 
Let~$d$ be an integer $\geq 2.$ 
The order of the real resonance point~$r_\lambda$ is equal to~$d$ if and only if for some real~$s,$ and thus for any real~$s,$ all of the following vectors
\begin{equation} \label{F: (uu,T(0)uu) = 0(n)}
  \hat u_{\lambda+i0}(r_\lambda), \ \ \hatlmbu{1}, \ \ \hatlmbu{2}, \ \ \ldots, \ \ \hatlmbu{d-3}
\end{equation}
are orthogonal to the vector $\hat \psi$ but the vector $\hatlmbu{d-2}$ is not.
\end{thm}
\begin{proof}
It can be seen that it is enough to prove the following assertion:
the order of the real resonance point~$r_\lambda$ is not less than~$d$ if and only if for some $s$ all of the following vectors
\begin{equation*}
  \hat u_{\lambda+i0}(r_\lambda), \ \ \hatlmbu{1}, \ \ \hatlmbu{2}, \ \ \ldots, \ \ \hatlmbu{d-3}
\end{equation*}
are orthogonal to the vector $\hat \psi.$ We prove this using induction on $d=2,3,\ldots.$

%
%
%

According to Theorem \ref{T: beau vectors}, the resonance point~$r_\lambda$ has order $\geq 3$ if and only if
the vector $\twovector{\hatlmbu{1}}{0}$ is a resonance vector of order~$3.$
The vector $\twovector{\hatlmbu{1}}{0}$ is a resonance vector of order~$3$ if and only if
\begin{equation} \label{F: [1+A(s)]twovec}
  [1+(r_\lambda-s)A_{\lambda+i0}(s)]\twovector{\hatlmbu{1}}{0}
\end{equation}
is a vector of order 2, which by Theorem \ref{T: beau vectors} is co-linear to a vector of the form $\twovector{\hat u_{\lambda+i0}(r_\lambda)}{b}.$
We calculate the first component of the vector (\ref{F: [1+A(s)]twovec}):
\begin{equation*}
 \begin{split}
    \euF_{\lambda+i0}(s)\hatlmbu{1} & + (s-r_\lambda)^2 \euD_{\lambda+i0}(s) \scal{\euF^*_{\lambda+i0}(s)\hat \psi}{\hatlmbu{1}} \hat u_{\lambda+i0}(s)
    \\ & =  \hat u_{\lambda+i0}(r_\lambda) + (s-r_\lambda)^2 \euD_{\lambda+i0}(s) \scal{\hat \psi}{\hat u_{\lambda+i0}(r_\lambda)} \hat u_{\lambda+i0}(s)
    \\ & =  \hat u_{\lambda+i0}(r_\lambda) - \frac{\scal{\hat \psi}{\hat u_{\lambda+i0}(r_\lambda)}}{\scal{\hat \psi}{\hat u_{\lambda+i0}(s)}} \hat u_{\lambda+i0}(s),
 \end{split}
\end{equation*}
where the second equality follows from definition~(\ref{F: def of euA(s)}) of~$\euD_{\lambda+i0}(s)$ and $\alpha = 0.$
Hence, the vector $\twovector{\hatlmbu{1}}{0}$ is a resonance vector of order~$3$ if and only if the vector
\begin{equation*}
    \hat u_{\lambda+i0}(r_\lambda) - \frac{\scal{\hat \psi}{\hat u_{\lambda+i0}(r_\lambda)}}{\scal{\hat \psi}{\hat u_{\lambda+i0}(s)}} \hat u_{\lambda+i0}(s)
\end{equation*}
is non-zero and co-linear to the vector $\hat u_{\lambda+i0}(r_\lambda).$
On the other hand, by Theorem~\ref{T: beau vectors} the vector $\twovector{\hatlmbu{1}}{0}$ has order three if and only if the vectors
$\hat u_{\lambda+i0}(r_\lambda)$ and $\hatlmbu{1}$ are linearly independent. Since the operator $\euF_{\lambda+i0}(s)$
is invertible, this holds if and only if the vectors
$$
  \euF_{\lambda+i0}(s) \hat u_{\lambda+i0}(r_\lambda) = \hat u_{\lambda+i0}(s)
$$
and
$$
  \euF_{\lambda+i0}(s)\hatlmbu{1} = \hat u_{\lambda+i0}(r_\lambda)
$$
are linearly independent. We conclude that $\twovector{\hatlmbu{1}}{0}$ is a vector of order 3 if and only if
$$
  \scal{\hat \psi}{\hat u_{\lambda+i0}(r_\lambda)} = 0.
$$
If this is the case then the vector space~$\Upsilon^3_{\lambda+i0}(r_\lambda)$ is three-dimensional and is generated by vectors
$$
  \twovector{0}{1}, \ \ \twovector{\hat u_{\lambda+i0}(r_\lambda)}{0} \ \ \text{and} \ \ \twovector{\hatlmbu{1}}{0},
$$
which have orders $1,2$ and $3$ respectively.
We have also proved that $d=2$ if and only if $\scal{\hat \psi}{\hat u_{\lambda+i0}(r_\lambda)} \neq 0.$
This gives the induction base.

Now assuming that the assertion holds for order of~$r_\lambda$ less than $d$ it will be proved for order of~$r_\lambda$ equal to $d.$
According to Theorem~\ref{T: beau vectors}, the resonance point~$r_\lambda$ has order $\geq d$ iff
the vector
$$
  \twovector{\hatlmbu{d-2}}{0}
$$
is a resonance vector of order~$d.$
In its turn, this vector
is a resonance vector of order~$d$ iff
\begin{equation} \label{F: u(k)=twovector}
  [1+(r_\lambda-s)A_{\lambda+i0}(s)]\twovector{\hatlmbu{d-2}}{0}
\end{equation}
is a vector of order $d-1,$
which by Lemma \ref{L: beau vectors} is a linear combination of vectors
$$
  \twovector{0}{1}, \ \twovector{\hat u_{\lambda+i0}(r_\lambda)}{0}, \ \twovector{\hatlmbu{1}}{0}, \ldots, \twovector{\hatlmbu{d-3}}{0}.
$$
The first component of the vector (\ref{F: u(k)=twovector}) is
\begin{equation} \label{F: bloody lengthy expression}
 \begin{split}
    \euF_{\lambda+i0}(s)\hatlmbu{d-2} & + (s-r_\lambda)^2 \euD_{\lambda+i0}(s) \scal{\euF^*_{\lambda+i0}(s)\hat \psi}{\hatlmbu{d-2}} \hat u_{\lambda+i0}(s)
    \\ & =  \hatlmbu{d-3} + (s-r_\lambda)^2 \euD_{\lambda+i0}(s) \scal{\hat \psi}{\hatlmbu{d-3}} \hat u_{\lambda+i0}(s)
    \\ & =  \hatlmbu{d-3} - \frac{\scal{\hat \psi}{\hatlmbu{d-3}}}{\scal{\hat \psi}{\hat u_{\lambda+i0}(s)}} \hat u_{\lambda+i0}(s).
 \end{split}
\end{equation}
Thus, order of~$r_\lambda \geq d$ iff this vector is a linear combination of the vectors
$$
  \hat u_{\lambda+i0}(r_\lambda), \hatlmbu{1}, \ldots, \hatlmbu{d-3}.
$$
By Theorem~\ref{T: beau vectors}, order~$r_\lambda \geq d$ iff the vectors
$$
  \hat u_{\lambda+i0}(r_\lambda), \hatlmbu{1}, \hatlmbu{2}, \ldots, \hatlmbu{d-2}
$$
are linearly independent.
Since the operator $\euF_{\lambda+i0}(s)$
is invertible, this holds iff  the vectors
$$
  \hat u_{\lambda+i0}(s), \hat u_{\lambda+i0}(r_\lambda), \hatlmbu{1}, \ldots, \hatlmbu{d-3}
$$
are linearly independent. It can now be concluded that order of~$r_\lambda$ is $\geq d$ iff the coefficient of $\hat u_{\lambda+i0}(s)$ in~(\ref{F: bloody lengthy expression}) is zero,
that is, iff $\scal{\hat \psi}{\hatlmbu{d-3}}=0.$ Combined with the induction assumption, this completes the proof.
\end{proof}
\begin{thm} \label{T: order k case} 
Let~$d$ be an integer not less than two.
The order of the real resonance point~$r_\lambda$ is equal to~$d$ if and only if the vectors
\begin{equation} \label{F: (uu,T(0)uu) = 0}
  \hat u_{\lambda+i0}(r_\lambda), \ \hat A_{\lambda+i0}(r_\lambda)\hat u_{\lambda+i0}(r_\lambda), \ldots, \hat A_{\lambda+i0}^{d-3}(r_\lambda)\hat u_{\lambda+i0}(r_\lambda)
\end{equation}
are orthogonal to the vector $\hat \psi$ but the vector $\hat A_{\lambda+i0}^{d-2}(r_\lambda)\hat u_{\lambda+i0}(r_\lambda)$ is not.
If this is the case, then for all $j=1,2,\ldots,d$ the vector space~$\Upsilon^j_{\lambda+i0}(r_\lambda)$ is $j$-dimensional and is generated by vectors
\begin{equation} \label{F: generators of Ups(j)}
  \twovector{0}{1}, \ \ \twovector{\hat u_{\lambda+i0}(r_\lambda)}{0}, \ \ \twovector{\hat A_{\lambda+i0}(r_\lambda)\hat u_{\lambda+i0}(r_\lambda)}{0}, \ \ \ldots, \ \
   \twovector{\hat A^{j-2}_{\lambda+i0}(r_\lambda)\hat u_{\lambda+i0}(r_\lambda)}{0},
\end{equation}
which have orders $1,2,\ldots,j$ respectively.
\end{thm}
\begin{proof} By Lemma~\ref{L: about euF(-1)} and by definition (\ref{F: vj(l+i0)(s)}) of the vectors $\hatlmbu{j}$ we have the equality
$$
  \hatlmbu{j} =  \SqBrs{1 + (s-r_\lambda)\hat A_{\lambda+i0}(r_\lambda)}^j \hat u_{\lambda+i0}(r_\lambda).
$$
Hence, the assertion to be proved is a direct consequence of Theorem~\ref{T: order k case (2)}.
\end{proof}
\begin{cor} \label{C: basis of Psi(j)} Under the conditions of Theorem~\ref{T: order k case}, if~$r_\lambda$ has order~$d$ then
the vector space~$\Psi_{\lambda+i0}(r_\lambda)$ is~$d$-dimensional and is generated by vectors
$$
  \twovector{\hat \psi}{0}, \ \ \twovector{\hat J\hat u_{\lambda+i0}(r_\lambda)}{0},
    \ \ \ldots, \ \
   \twovector{\hat J\hat A^{d-3}_{\lambda+i0}(r_\lambda)\hat u_{\lambda+i0}(r_\lambda)}{0},
   \twovector{\hat J\hat A^{d-2}_{\lambda+i0}(r_\lambda)\hat u_{\lambda+i0}(r_\lambda)}{\scal{\hat \psi}{\hat A^{d-2}_{\lambda+i0}(r_\lambda)\hat u_{\lambda+i0}(r_\lambda)}},
$$
which have orders $1,2,\ldots,d$ respectively. Further, the second component of the last vector is non-zero.
\end{cor}
\begin{proof} By Lemma~\ref{L: j-dimensions coincide}, the vector space $\Psi_{\lambda+i0}(r_\lambda)$ is the image of
$\Upsilon_{\lambda+i0}(r_\lambda)$ under the mapping~$J.$ Applying the operator~$J$ given by~(\ref{F: J=(J,uu;uu,alpha)})
to the vectors~(\ref{F: generators of Ups(j)}), which by Theorem~\ref{T: order k case} generate the vectors space $\Upsilon^j_{\lambda+i0}(r_\lambda)$,
one infers that~$d$ vectors
$$
  \twovector{\hat \psi}{0}, \ \ \twovector{\hat J\hat u_{\lambda+i0}(r_\lambda)}{\scal{\hat \psi}{\hat u_{\lambda+i0}(r_\lambda)}},
     \ \ \twovector{\hat J\hat A_{\lambda+i0}(r_\lambda)\hat u_{\lambda+i0}(r_\lambda)}{\scal{\hat \psi}{\hat A_{\lambda+i0}(r_\lambda)\hat u_{\lambda+i0}(r_\lambda)}}, \ \ \ldots, \ \
   \twovector{\hat J\hat A^{d-2}_{\lambda+i0}(r_\lambda)\hat u_{\lambda+i0}(r_\lambda)}{\scal{\hat \psi}{\hat A^{d-2}_{\lambda+i0}(r_\lambda)\hat u_{\lambda+i0}(r_\lambda)}}
$$
form a basis of $\Psi_{\lambda+i0}(r_\lambda).$ It is left to note that by Theorem~\ref{T: order k case} the second components of all these vectors except the last one
are zero.
\end{proof}

\subsection{Type I vectors for an embedded eigenvalue}

\label{Page: u pm and A pm}
In order to simplify formulas, we write $\hat u_+$ instead of $\hat u_{\lambda+i0}(r_\lambda)$ and $\hat A_+$ instead of $\hat A_{\lambda+i0}(r_\lambda):$
\begin{equation} \label{F: u pm and A pm}
  \hat u_\pm = \hat u_{\lambda \pm i0}(r_\lambda), \quad \hat A_\pm = \hat A_{\lambda \pm i0}(r_\lambda), \quad \hat B_\pm = \hat B_{\lambda \pm i0}(r_\lambda).
\end{equation}

For convenience we introduce the notation \label{Page: aj(pm)}
\begin{equation} \label{F: aj(pm)}
  a_{j,\pm} := \scal{\hat \psi}{\hat A^j_{\lambda\pm i0}(r_\lambda)\hat u_{\lambda\pm i0}(r_\lambda)}.
\end{equation}
In what follows, a vector $f \in \hat \clK$ will often be identified with the vector $\twovector{f}{0} \in \clK.$
Also, the vector $\twovector{0}{1} \in \clK$ will be written as 1.
By Theorem~\ref{T: order k case}, vectors
$$
  \hat A_+^{d-2}\hat u_+, \hat A_+^{d-3}\hat u_+, \ldots, \hat A_+\hat u_+, \hat u_+, 1
$$
form a basis of $\Upsilon_{\lambda+i0}(r_\lambda).$
By Corollary~\ref{C: basis of Psi(j)}, vectors
$$
  \hat B_-^{d-1}\hat \psi + a_{d-2,-}, \hat B_-^{d-2}\hat \psi, \ldots, \hat B_-\hat \psi, \hat \psi
$$
form a basis of $\Psi_{\lambda-i0}(r_\lambda).$

The following lemma directly follows from Theorem \ref{T: on vectors with property L}, but we still give another proof.
\begin{lemma} \label{L: hat u+=hat u-} Let~$k$ be a positive integer. If $d \geq 2k+1,$ then $\hat u_+ = \hat u_-,$
 $\hat A_+ \hat u_+ = \hat A_- \hat u_-,$ \ldots, $\hat A_+^{k-1} \hat u_+ = \hat A_-^{k-1} \hat u_-.$
\end{lemma}
\begin{proof} If $k=1,$ then $d \geq 3,$ and therefore, by Theorem~\ref{T: order k case}, $a_{0,+}$ and $a_{0,-}$ are zero,
that is,
$$
  0 = \scal{\hat \psi}{\hat u_\pm} = \scal{\hat \psi}{\hat T_\pm\hat \psi}.
$$
It follows that $\sqrt{\Im \hat T_+}\hat \psi=0,$ and therefore $\hat T_+\hat \psi = \hat T_-\hat \psi,$ that is, $\hat u_+ = \hat u_-.$

Assume that the assertion is true for $k=n$ and let $k=n+1.$ Then $d \geq 2n+3$ and therefore, by Theorem~\ref{T: order k case},
$a_{2n,\pm} = 0,$ that is,
$$
  \scal{\hat \psi}{(\hat T_+J)^{2n}\hat T_+\hat \psi} = 0.
$$
This implies that
\begin{equation*}
  \begin{split}
    0  & = \scal{(J\hat T_-)^{n}\hat \psi}{(\hat T_+J)^n\hat T_+\hat \psi}
    \\ & = \scal{J\hat A_-^{n-1}\hat u_+}{\hat T_+J\hat A_+^{n-1}\hat u_+}
    \\ & = \scal{J\hat A_+^{n-1}\hat u_+}{\hat T_+J\hat A_+^{n-1}\hat u_+},
  \end{split}
\end{equation*}
where the last equality follows from the induction assumption.
From this it follows that $(\Im \hat T_+)J\hat A_+^{n-1}\hat u_+= 0,$ that is,
$$
  \hat A_+ \hat A_+^{n-1}\hat u_+ = \hat A_- \hat A_+^{n-1}\hat u_+ = \hat A_- \hat A_-^{n-1}\hat u_-,
$$
where the last equality follows from the induction assumption.
\end{proof}


\subsection{Idempotents $P_{\lambda\pm i0}(r_\lambda)$ and $Q_{\lambda\pm i0}(r_\lambda)$}
\label{SS: about idempotent}

In this subsection we calculate the idempotents $P_{\lambda\pm i0}(r_\lambda).$
Since by~(\ref{F: lambda in hat (H,F)}) the operator-function $T_{\lambda+i0}(\hat H_s)$ is holomorphic at $s = r_\lambda,$ the functions
$T_{\lambda+i0}(\hat H_s)$ and $\scal{\hat \psi}{\hat u_{\lambda+i0}(s)}$
can be expanded into a Taylor series convergent in some neighbourhood of $s = r_\lambda$ as follows:
\begin{equation*}
    T_{\lambda+i0}(\hat H_s) = \sum_{k=0}^\infty (-1)^k (s-r_\lambda)^k \hat A^k_{\lambda+i0}(r_\lambda)T_{\lambda+i0}(\hat H_{r_\lambda}).
\end{equation*}
From this we obtain the equality
\begin{equation} \label{F: Neumann for uu+(s)}
    \scal{\hat \psi}{\hat u_{\lambda+i0}(s)} = \scal{\hat \psi}{T_{\lambda+i0}(\hat H_s)\hat \psi}
    = a_{0,+} - a_{1,+}(s-r_\lambda) + a_{2,+}(s-r_\lambda)^2 - a_{3,+}(s-r_\lambda)^3 + \ldots
\end{equation}
If $d$ is the order of~$r_\lambda,$ then by Theorem~\ref{T: order k case}, we have
$$
  a_{0,\pm} = \ldots = a_{d-3,\pm} = 0
$$
and the number $a_{d-2,\pm}$ is non-zero.

We shall need a Taylor series expansion for the function
$$
  (r_\lambda-s)\euD_{\lambda+i0}(s) = -\frac 1{\alpha +(r_\lambda-s)\scal{\hat\psi}{\hat u_{\lambda+i0}(s)}}.
$$
For this, we write the first few terms of the Taylor expansion of the inverse function:
\begin{equation} \label{F: (c0+c1s+...)(-1)}
  \begin{split}
  (c_0 - c_1(s-r_\lambda) & + c_2 (s-r_\lambda)^2 - c_3 (s-r_\lambda)^3 + \ldots)^{-1} \\ & = \frac 1{c_0} + \frac {c_1}{c_0^2}(s-r_\lambda) + \frac {c_1^2 - c_0c_2}{c_0^3}(s-r_\lambda)^2
     + \frac {c_1^3 -2c_0c_1c_2 + c_3c_0^2}{c_0^4}(s-r_\lambda)^3
   \\ & \qquad\qquad\qquad  + \frac {c_1^4 -3c_0c_1^2c_2 + 2c_0^2c_1c_3+c_0^2c_2^2-c_0^3c_4}{c_0^5}(s-r_\lambda)^4 + \ldots
  \end{split}
\end{equation}
Using the equality~(\ref{F: A+(s)=2 by 2 matrix}) for $A_{\lambda+i0}(s)$ and Proposition~\ref{P: Pz(rz)=res Az(s)}, one can calculate
the idempotent $P_{\lambda+i0}(r_\lambda)$ for points of not too high order.

\subsubsection{Order $d = 1$}
By Lemma~\ref{L: beau vectors}, the order~$d$ of the resonance point~$r_\lambda$ is equal to 1 if and only if
$\alpha \neq 0.$ If $\alpha \neq 0,$ then one can see that the $(1,1)$-entry of the matrix~(\ref{F: A+(s)=2 by 2 matrix}) is holomorphic at $s = r_\lambda,$
and therefore its residue vanishes. Hence, in this case $(1,1)$-entries of the idempotents $P_{\lambda\pm i0}(r_\lambda)$ are also zero,
and as a result these idempotents have rank one:
$$
  P_{\lambda\pm i0}(r_\lambda) = \left(\begin{matrix}
    0 & 0 \\
    \alpha^{-1} \la \hat \psi, \cdot \ra & 1
  \end{matrix}\right).
$$
It follows that
$$  Q_{\lambda\pm i0}(r_\lambda) = \left(\begin{matrix}
    0 & \alpha^{-1} \hat \psi \\
    0 & 1
  \end{matrix}\right)
\quad \text{and} \quad
  Q_{\lambda-i0}(r_\lambda)JP_{\lambda+i0}(r_\lambda) = \left(\begin{matrix}
    \alpha^{-1} \la \hat \psi, \cdot \ra\hat \psi & \hat \psi \\
    \la \hat \psi, \cdot \ra & \alpha
  \end{matrix}\right).
$$
Hence, $Q_{\lambda-i0}(r_\lambda)JP_{\lambda+i0}(r_\lambda)$ is a rank one operator with the range generated by vector $\twovector{\hat \psi}{\alpha}.$
Also, in this case the operators $\bfA_{\lambda\pm i0}(r_\lambda)$ are zero.

\subsubsection{Order $d = 2$}
\label{SSS: order 2}
\ By Theorem~\ref{T: order k case},
in this case $\alpha = 0$ and $\scal{\hat \psi}{\hat u_+} \neq 0.$
Since $\alpha = 0,$ it follows from definition~(\ref{F: def of euA(s)}) of $\euD_{\lambda+i0}(s),$
(\ref{F: Neumann for uu+(s)}) and (\ref{F: (c0+c1s+...)(-1)}) that
\begin{equation*}
  \begin{split}
     -\euD_{\lambda+i0}(s) & = \frac 1{(s-r_\lambda)^2\scal{\hat \psi}{\hat u_{\lambda+i0}(s)}}
     \\ & = \frac 1{(s-r_\lambda)^2} \brs{a_{0,+} - a_{1,+}(s-r_\lambda) + a_{2,+}(s-r_\lambda)^2 - \ldots}^{-1}
     \\ & = \frac 1{a_{0,+}} (s-r_\lambda)^{-2} + \frac {a_{1,+}}{a_{0,+}^2}(s-r_\lambda)^{-1} \ + \frac {a_{1,+}^2 - a_{0,+}a_{2,+}}{a_{0,+}^3} + \ldots
  \end{split}
\end{equation*}
Therefore the coefficient of $(s-r_\lambda)^{-1}$ in the $(1,1)$-entry of $A_{\lambda+i0}(s)$
is equal to
$
  a_{0,+}^{-1}\la \hat \psi, \cdot \ra\hat u_+
$
and the coefficient of $(s-r_\lambda)^{-1}$ in the $(2,1)$-entry of $A_{\lambda+i0}(s)$ is equal to
$$
  - \frac {a_{1,+}}{a_{0,+}^2}\la \hat \psi, \cdot \ra + \frac 1{a_{0,+}}\scal{\hat B_-\hat \psi}{\cdot}.
$$
Therefore, the idempotent operator $P_{\lambda+i0}(r_\lambda)$ is given by equality
\begin{equation} \label{F: P+ 2x2}
 P_{\lambda+i0}(r_\lambda) =
  \left(
    \begin{matrix}
       \frac 1{a_{0,+}}\la \hat \psi, \cdot \ra\hat u_+ & 0 \\
       - \frac {a_{1,+}}{a_{0,+}^2} \la \hat \psi, \cdot \ra + \frac 1{a_{0,+}}\scal{\hat B_-\hat \psi}{\cdot} & 1
    \end{matrix}
  \right).
\end{equation}
Similarly,
\begin{equation} \label{F: P- 2x2}
 P_{\lambda-i0}(r_\lambda) =
  \left(
    \begin{matrix}
       \frac 1{a_{0,-}}\la \hat \psi, \cdot \ra\hat u_- & 0 \\
       - \frac {a_{1,-}}{a_{0,-}^2} \la \hat \psi, \cdot \ra + \frac 1{a_{0,-}}\scal{\hat B_+\hat \psi}{\cdot} & 1
    \end{matrix}
  \right).
\end{equation}
Using these equalities one can check that in general $P_{\lambda+i0}(r_\lambda)P_{\lambda-i0}(r_\lambda) \neq P_{\lambda+i0}(r_\lambda).$
That is, the resonance point~$r_\lambda$ in general does not have property~$S.$
Further, since by~(\ref{F: Laurent for A+(s)}) the operator $\bfA_{\lambda \pm i0}(r_\lambda)$ is the coefficient of $(s-r_\lambda)^{-2}$ in the Laurent expansion of $A_{\lambda+i0}(s)$
at $r = r_\lambda,$ it follows from~(\ref{F: A+(s)=2 by 2 matrix}) that
$$
  \bfA_{\lambda \pm i0}(r_\lambda) = \left(
    \begin{matrix}
      0 & 0 \\
      - \frac {1}{\scal{\hat \psi}{\hat u_\pm}}\la \hat \psi, \cdot \ra & 0
    \end{matrix}
  \right).
$$
One can calculate that
\begin{equation*}
 \begin{split}
    & Q_{\lambda-i0}(r_\lambda)JP_{\lambda+i0}(r_\lambda)
  \\ & =    \left(
       \begin{matrix}
         \frac 1{\abs{a_{0,+}}^2}\scal{\hat u_+}{\hat B_+ \hat \psi}\la \hat \psi, \cdot \ra\hat \psi
             - 2\Re \frac {a_{1,+}}{a_{0,+}^2}\la \hat \psi, \cdot \ra\hat \psi + \frac 1{a_{0,+}}\scal{\hat B_-\hat \psi}{\cdot}\hat \psi
                     + \frac 1{\bar a_{0,+}}\la \hat \psi, \cdot \ra\hat B_-\hat \psi & \hat \psi \\
         \la \hat \psi, \cdot \ra  & 0
       \end{matrix}
     \right).
 \end{split}
\end{equation*}
A similar equality holds also for $Q_{\lambda+i0}(r_\lambda)JP_{\lambda-i0}(r_\lambda),$ which shows
that in general $$Q_{\lambda-i0}(r_\lambda)JP_{\lambda+i0}(r_\lambda) \neq Q_{\lambda+i0}(r_\lambda)JP_{\lambda-i0}(r_\lambda).$$
It follows from Proposition~\ref{P: property S}, that in general the point~$r_\lambda$ does not have property $S.$

From~(\ref{F: P+ 2x2}) and~(\ref{F: P- 2x2}) one can see that
the kernel of the idempotent $P_{\lambda \pm i0}(r_\lambda)$ is given by
$$
  \ker P_{\lambda \pm i0}(r_\lambda) = \set{\hat \phi - \scal{\hat \psi}{\hat u_\pm}^{-1}\scal{\hat \psi}{\hat A_\pm \hat \phi}\cdot 1 \ \ \colon \hat \phi \perp \hat \psi}.
$$

If $\hat J = 0,$ then it follows that
$$
  \ker P_{\lambda \pm i0}(r_\lambda) = {\linspan {\set{1, \hat \psi}}}^\perp.
$$
Thus, in this case the kernels $\ker P_{\lambda - i0}(r_\lambda)$ and $\ker P_{\lambda - i0}(r_\lambda)$ are equal, which
is the definition of resonance points with property $S.$  
Since the vector spaces $\Upsilon_{\lambda \pm i0}(r_\lambda) = \linspan {\set{1, \hat u_\pm}}$ are in general different, we conclude that
there exist real resonance points with property $S$ for which $\Upsilon_{\lambda+i0}(r_\lambda) \neq \Upsilon_{\lambda-i0}(r_\lambda).$
By Theorem~\ref{T: P(+)=P(-)}, it follows that in this case~$r_\lambda$ is not a resonance point of type~I.
Hence, this gives an example of a resonance point of order two with property~$S$ which is not of type~I.

These examples give a proof of the second part of Proposition~\ref{P: points without property S exist}.

\subsubsection{Order $d = 3$}

By Theorem~\ref{T: order k case}, in this case the vectors $\hat u_+$ and $\hat\psi$ are orthogonal while $\hat u_+$ and $\hat A_+ \hat\psi$ are not, so that
\begin{equation} \label{F: (uu,uu+)=0,(uu,uu-)=0}
  a_{0,+} = \scal{\hat \psi}{\hat u_+} = 0 \quad \text{and} \quad a_{1,+} = \scal{\hat \psi}{\hat A_+\hat u_+} \neq 0.
\end{equation}
The first of these two equalities implies that $\scal{\hat \psi}{\Im T_+ \hat \psi} = 0,$ and therefore $\Im T_+\hat \psi=0.$ It follows that
\begin{equation} \label{F: uu(+)=uu(-)}
  \hat u_+=\hat u_-, \quad \hat B_+\hat \psi=\hat B_-\hat \psi \quad \text{and} \quad
  a_{1,+} = \scal{\hat \psi}{\hat A_-\hat u_-} = a_{1,-}.
\end{equation}
Further, it follows from~(\ref{F: (uu,uu+)=0,(uu,uu-)=0}) that
the first term of the Neumann series~(\ref{F: Neumann for uu+(s)}) for $\scal{\hat \psi}{\hat u_{\lambda+i0}(s)}$ vanishes:
$$
  \scal{\hat \psi}{\hat u_{\lambda+i0}(s)} = -a_{1,+}(s-r_\lambda) + a_{2,+}(s-r_\lambda)^2 + \ldots
$$
and we get
\begin{equation*}
 \begin{split}
  \euD_{\lambda+i0}(s) & = - \frac 1{(s-r_\lambda)^2\scal{\hat \psi}{\hat u_{\lambda+i0}(s)}} \\
     & = \frac 1{(s-r_\lambda)^3} \brs{\frac 1{a_{1,+}} + \frac {a_{2,+}}{a_{1,+}^2}(s-r_\lambda)  + \frac {a_{2,+}^2 - a_{1,+}a_{3,+}}{a_{1,+}^3}(s-r_\lambda)^2 + \ldots}.
 \end{split}
\end{equation*}
Also,
$$
  1 +(r_\lambda-s)\hat B_{\lambda-i0}(s) = 1 +(r_\lambda-s) \hat B_- + (s-r_\lambda)^2\hat B_-^2 +(r_\lambda-s)^3\hat B_-^3 + \ldots
$$
and
$$
  \hat u_{\lambda+i0}(s) = \hat u_+ +(r_\lambda-s) \hat A_+ \hat u_+  + (s-r_\lambda)^2 \hat A_+^2 \hat u_+ - \ldots
$$
From this we find the coefficient of $\frac 1{s-r_\lambda}$ in the $(1,1)$-entry of $A_{\lambda+i0}(s),$ that is, the $(1,1)$-entry of the idempotent $P_{\lambda+i0}(r_\lambda):$
$$
  \hat P_+ := - \frac {a_{2,+}}{a_{1,+}^2} \la \hat \psi, \cdot \ra \hat u_+ + \frac 1{a_{1,+}}\brs{\la \hat \psi, \cdot \ra \hat A_+\hat u_+ + \scal{\hat B_-\hat \psi}{\cdot} \hat u_+},
$$
and we also find the $(2,1)$-entry of $A_{\lambda+i0}(s):$
$$
  \frac {a_{2,+}^2 - a_{1,+}a_{3,+}}{a_{1,+}^3}\la \hat \psi, \cdot \ra - \frac {a_{2,+}}{a_{1,+}^2}\scal{\hat B_-\hat \psi}{\cdot} + \frac 1{a_{1,+}}\scal{\hat B_-^2\hat \psi}{\cdot}.
$$
Hence,
$$
  P_{\lambda + i0}(r_\lambda) = \left(
       \begin{matrix}
         - \frac {a_{2,+}}{a_{1,+}^2} \la \hat \psi, \cdot \ra \hat u_+ + \frac 1{a_{1,+}}\brs{\la \hat \psi, \cdot \ra \hat A_+\hat u_+ + \scal{\hat B_-\hat \psi}{\cdot} \hat u_+} & 0 \\
         \frac {a_{2,+}^2 - a_{1,+}a_{3,+}}{a_{1,+}^3}\la \hat \psi, \cdot \ra - \frac {a_{2,+}}{a_{1,+}^2}\scal{\hat B_-\hat \psi}{\cdot} + \frac 1{a_{1,+}}\scal{\hat B_-^2\hat \psi}{\cdot} & 1
       \end{matrix}
     \right).
$$
The structure of this operator becomes a bit more transparent, if it is written as a matrix in the basis $(\hat B_-^2\hat \psi + a_{1,-}, \hat B_-\hat \psi, \hat \psi)$
of the range of $Q_{\lambda - i0}(r_\lambda)$ and in the basis $(\hat A_+\hat u_+, \hat u_+, 1)$ of the range of $P_{\lambda + i0}(r_\lambda)$ as follows:
$$
  P_{\lambda + i0}(r_\lambda) = \left(
       \begin{matrix}
         0  & 0 & \frac 1{a_{1,+}} \\
         0  & \frac 1{a_{1,+}} & - \frac {a_{2,+}}{a_{1,+}^2} \\
         \frac 1{a_{1,+}}  &  - \frac {a_{2,+}}{a_{1,+}^2}  & \frac {a_{2,+}^2 - a_{1,+}a_{3,+}}{a_{1,+}^3}
       \end{matrix}
     \right).
$$
Similarly, one can find $P_{\lambda - i0}(r_\lambda).$
Further, one can calculate that
$$
  \bfA_{\lambda + i0}(r_\lambda) = \left(
       \begin{matrix}
         - \frac 1{a_{1,+}}\la \hat \psi, \cdot \ra\hat u_+ & 0 \\
         \frac{a_{2,+}}{a_{1,+}^2} \la \hat \psi, \cdot \ra - \frac{1}{a_{1,+}} \scal{\hat B_-\hat \psi}{\cdot} & 0
       \end{matrix}
     \right).
$$
In the pair of bases $(\hat B_-^2\hat \psi + a_{1,-}, \hat B_-\hat \psi, \hat \psi)$
and $(\hat A_+\hat u_+, \hat u_+, 1)$
this operator takes the form
$$
  \bfA_{\lambda + i0}(r_\lambda) = \left(
       \begin{matrix}
         0 & 0 & 0 \\
         0  & 0 & -\frac 1{a_{1,+}} \\
         0  & -\frac 1{a_{1,+}} & \frac {a_{2,+}}{a_{1,+}^2}
       \end{matrix}
     \right).
$$
One can check that the (1,1)-entries $\hat P_+$ and $\hat P_-$ of the idempotents $P_{\lambda + i0}(r_\lambda)$ and $P_{\lambda - i0}(r_\lambda)$
satisfy the equality $\hat P_+\hat P_- = \hat P_+.$ This implies that the image of the operator $P_{\lambda + i0}(r_\lambda)P_{\lambda - i0}(r_\lambda) - P_{\lambda + i0}(r_\lambda)$ consists of
vectors of order 1.

These examples also show how to calculate the idempotents $P_{\lambda \pm i0}(r_\lambda)$
and nilpotent operators $\bfA_{\lambda \pm i0}(r_\lambda)$ in the case of arbitrary order.

\subsection{Example of calculation of resonance index}
\label{SS: resonance index: example}

The function $A_{\lambda+i0}(s)$ of the coupling constant $s$ has an eigenvalue $\sigma_\lambda(s)=(s-r_\lambda)^{-1}.$
When $\lambda+i0$ is shifted to $\lambda+iy$ with small positive $y,$ the eigenvalue $\sigma_\lambda(s)$ in general splits
into $N_\pm$ non-real eigenvalues in~$\mbC_\pm$ respectively. The difference $N_+-N_-$ is the resonance index.
To calculate the resonance index we need to find eigenvalues of $A_{\lambda+iy}(s)$ which belong to the group of
the eigenvalue $\sigma_\lambda(s),$ that is, which converge to $\sigma_\lambda(s)$ as $y \to 0^+.$
The eigenvalue equation
$$
  A_z u = \sigma u,
$$
where $u = \twovector{\hat u}{1}$ and where $A_z(s)$ is given in~(\ref{F: Az(s)=2x2 matrix}),
leads to the following two equations:
\begin{gather*}
    \hat A_z(s)\hat u + (r_\lambda-s)\euD_z(s)\scal{\euF_{z}^*(s)\hat \psi}{\hat u} \hat u_z(s) + \SqBrs{1 + (s-r_\lambda)\euD_z(s)\big((s-r_\lambda) \scal{\hat u_{\bar z}(s)}{\hat \psi} - \alpha\big)} \hat u_z(s) = \sigma \hat u,
\\
     \euD_z(s) \scal{\euF_{z}^*(s)\hat \psi}{\hat u} - \euD_z(s)\big((s-r_\lambda)\scal{\hat u_{\bar z}(s)}{\hat \psi} - \alpha\big) = \sigma.
\end{gather*}
From the second equality it follows that the first equality is equivalent to
$$
  \hat A_z(s) \hat u + \sigma (r_\lambda-s)\hat u_z(s) + \hat u_z(s) = \sigma \hat u.
$$
We consider the case $\hat J = 0.$ In this case $\hat A_z(s) = 0,$ $\euF_z(s) = 1,$
and $\scal{\hat \psi}{\hat u_z(s)}$ does not depend on~$s$ and is equal to $\scal{\hat \psi}{\hat u_z(r_\lambda)}.$ 
The first equation becomes
\begin{gather*}
  \sigma (r_\lambda-s)\hat u_z(s) + \hat u_z(s) = \sigma \hat u.
\end{gather*}
while the second equation turns into (using $\scal{\hat u_{\bar z}(s)}{\hat \psi} = \scal{\hat \psi}{\hat u_{z}(s)}$)
\begin{gather*}
  \euD_z(s) \brs{\scal{\hat \psi}{\hat u} + (r_\lambda-s)\scal{\hat \psi}{\hat u_{z}(s)} + \alpha} = \sigma.
\end{gather*}
If we exclude the vector $\hat u$ from these two equations we obtain the following quadratic equation for $\sigma:$
\begin{equation} \label{F: quad eq-n}
  \sigma ^2 - \sigma\euD_z(s)\brs{2 (r_\lambda - s)\scal{\hat \psi}{\hat u_z(s)} + \alpha} - \euD_z(s) \scal{\hat \psi}{\hat u_z(s)} = 0.
\end{equation}
We consider first the case of $\alpha = 0.$ In this case by definition~(\ref{F: def of euA(s)}) of~$\euD_z(s)$
$$
  \euD_z(s) = -\brs{iy + (s-r_\lambda)^2\scal{\hat \psi}{\hat u_z(s)}}^{-1},
$$
where as usual $z = \lambda+iy.$
Let
\begin{equation} \label{F: w(y)}
  w(y) = - \euD_z(s) \scal{\hat \psi}{\hat u_z(s)} = \scal{\hat \psi}{\hat u_z(s)}\brs{iy + (s-r_\lambda)^2\scal{\hat \psi}{\hat u_z(s)}}^{-1}.
\end{equation}
The equation~(\ref{F: quad eq-n}) for~$\sigma$ then becomes
$$
  \sigma^2 - 2(s-r_\lambda)w\sigma + w = 0.
$$
Its roots are
$$
  \sigma_{1,2}(y) = (s-r_\lambda) w \pm \sqrt{(s-r_\lambda)^2 w^2-w},
$$
where we agree that the complex square root belongs to either the upper half-plane or the positive semi-axis.
From~(\ref{F: w(y)}) one can find that as $y \to 0^+$
$$
  w(y) = (s-r_\lambda)^{-2} - \frac{iy}{\scal{\hat \psi}{\hat u_{\lambda+i0}(s)}}(s-r_\lambda)^{-4} + O(y^2).
$$
It follows that
$$
  (s-r_\lambda)^2 w^2-w = - \frac{iy}{\scal{\hat \psi}{\hat u_{\lambda+i0}(s)}}(s-r_\lambda)^{-4} + O(y^2).
$$

Let $\rho e^{i2\theta}$ be the polar form of $\frac{i}{\scal{\hat \psi}{\hat u_{\lambda+i0}(s)}}.$ Then one can see that
$$
  \sigma_{1,2}(y) = (s-r_\lambda)^{-1} \pm \sqrt {\rho y} e^{i\theta}(s-r_\lambda)^{-2} + O(y).
$$
Since $\scal{\hat \psi}{\hat u_{\lambda+i0}(s)} \neq 0,$ it follows that the roots approach $(s-r_\lambda)^{-1}$ from different half-planes~$\mbC_\pm.$ Therefore, in the case of
$\alpha = 0$ the resonance index is equal to $1-1=0.$

\bigskip

Now let $\alpha \neq 0.$ Since the resonance index does not depend on $s,$ to simplify calculations we choose $s = 1 + r_\lambda.$
In this case $\euD_z^{-1}(s) = \alpha - iy - \scal{\hat \psi}{\hat u_z}$ and the eigenvalue $\sigma_\lambda(s) = (s-r_\lambda)^{-1}$ is equal to~1.
One can calculate that the roots of~(\ref{F: quad eq-n}) are given by
$$
  \sigma_{1,2} = \frac{\alpha - 2\scal{\hat \psi}{\hat u_z} \pm \sqrt{\alpha^2-4iy\scal{\hat \psi}{\hat u_z}}}{2(\alpha -iy - \scal{\hat \psi}{\hat u_z})}
    = \frac{\alpha - 2\scal{\hat \psi}{\hat u_z} \pm \abs{\alpha} \brs{1-2iy\scal{\hat \psi}{\hat u_z}/\alpha^2 + O(y^2)}}{2(\alpha -iy - \scal{\hat \psi}{\hat u_z})}.
$$
Hence the root, which approaches the eigenvalue $\sigma_\lambda(s)=1$ as $y \to 0^+$ is 
$$
  \sigma(y) = \frac{\alpha - \scal{\hat \psi}{\hat u_z} - iy\scal{\hat \psi}{\hat u_z}/\alpha + O(y^2)}{\alpha -iy - \scal{\hat \psi}{\hat u_z}}
  = 1 + \frac{iy - iy\scal{\hat \psi}{\hat u_z}/\alpha + O(y^2)}{\alpha -iy - \scal{\hat \psi}{\hat u_z}}.
$$
Since
$$
  \sigma'(0) = \frac{i - i\scal{\hat \psi}{\hat u_{\lambda+i0}(r_\lambda)}/\alpha}{\alpha - \scal{\hat \psi}{\hat u_{\lambda+i0}(r_\lambda)}} = \frac i{\alpha},
$$
the root $\sigma(y)$ approaches 1 from above (and moreover, under a right angle), if $\alpha >0,$ and from below, if $\alpha<0.$
It follows that, in case of $\hat J = 0,$
$$
  \ind_{res}(\lambda;H_{r_\lambda},V) = \sign \alpha.
$$

\subsection{Examples of resonance points of orders three and four (in finite dimensions)}
One feature of resonance index theory is that it makes sense and gives non-trivial results for spectral points~$\lambda$
outside of essential spectrum (that is, for classical spectral flow) and even in finite dimensions.
For example, assume that there is a straight path of self-adjoint matrices $H_r = H_0 +rV;$
then the eigenvalues of~$H_r$ are analytic functions of~$r$ which may have extrema, or
critical points. Critical points of eigenvalues of~$H_r$ may have different orders. A natural question is how to construct
a path of self-adjoint matrices, such that an eigenvalue of the path has a critical point of a given order?
Theorem ~\ref{T: order k case} indicates how to construct such examples.

\subsubsection{Example 1}
Let
$$
  H_0 = \left(
 \begin{matrix}
   \lambda + \eps &  0 &  0 \\
   0 &  \lambda - \eps &  0 \\
   0 &  0 &  \lambda \\
 \end{matrix}
\right).
$$
Since~$\lambda$ is an eigenvalue of~$H_0,$
the point $r = 0$ is a resonant at~$\lambda$ point of the path~$H_0+rV$
for any perturbation~$V.$
The direction
$$
  V_1 = \left(
 \begin{matrix}
   0 &  0 &  1 \\
   0 &  0 &  1 \\
   1 &  1 &  0 \\
 \end{matrix}
\right)
$$
is not regularizing for the matrix~$H_0:$~$\lambda$ is a common eigenvalue of all operators~$H_r.$
That the perturbation $V_1$ is not regularizing can also be seen from the fact that the condition~(\ref{F: psi u(s)neq 0}) fails.

The following direction is regularizing:
$$
  V_2 = \left(
 \begin{matrix}
   1 &  0 &  1 \\
   0 &  1 &  1 \\
   1 &  1 &  0 \\
 \end{matrix}
\right).
$$
Since $\alpha = 0,$ the order of resonance point $r = 0$ is at least two.
Resonance index of the triple $(\lambda,H_0,V_2)$ is equal to $2-1=1.$

\subsubsection{Example 2}
For the matrix
$$
H_0=\left(
 \begin{matrix}
   \lambda + 1 &  0 &  0 &  0 \\
   0 &  \lambda + 1 &  0 &  0 \\
   0 &  0 &  \lambda - 1/2 &  0 \\
   0 &  0 &  0 &  \lambda \\
 \end{matrix}
\right)
$$
the direction
$$
V_1=\left(
 \begin{matrix}
   -2 &  0 &  0 &  1 \\
   0 &  -2 &  0 &  1 \\
   0 &  0 &  1 &  1 \\
   1 &  1 &  1 &  0 \\
 \end{matrix}
\right)
$$
is also not regularizing at~$\lambda$ for the same reason as above: the condition~(\ref{F: psi u(s)neq 0}) fails.

If $V_2$ is chosen as
$$V_2=\left(
 \begin{matrix}
   -4 &  0 &  0 &  1 \\
   0 &  -1 &  0 &  1 \\
   0 &  0 &  1 &  1 \\
   1 &  1 &  1 &  0 \\
 \end{matrix}
\right),
$$
then the condition~(\ref{F: (uu,T(0)uu) = 0}) holds with $d=3.$ As a result, the point $r = 0$ has order at least three.
According to Theorem \ref{T: order k case},
the order of the resonance point~$r_\lambda = 0$ is in fact three, since for the perturbation~$V_2$ the following condition fails:
\begin{equation*}
  \scal{\hat \psi}{T_{\lambda+i0}(\hat H_0)\hat JT_{\lambda+i0}(\hat H_0)\hat \psi} = 0.
\end{equation*}
But the regularizing direction
$$
V_3=\left(
 \begin{matrix}
   -3 &  0 &  0 &  1 \\
   0 &  -1 &  0 &  1 \\
   0 &  0 &  1 &  1 \\
   1 &  1 &  1 &  0 \\
 \end{matrix}
\right)
$$
satisfies the condition~(\ref{F: (uu,T(0)uu) = 0}) for $d=4,$
and, therefore, the corresponding resonance point~$r_\lambda = 0$ has order $4.$
Computer shows that $\ind_{res}(\lambda,H_0,V_2) = 2 - 1 = 1$ and $\ind_{res}(\lambda,H_0,V_3) = 2 - 2 = 0.$

%

\section{Open problems}
\label{S: open problems}

\subsection{On points $\lambda$ which are not essentially regular}
According to Theorem \ref{T: infty mult-ty then not essentially regular}, if a real number $\lambda$ is an eigenvalue of infinite multiplicity of any of operators from the affine space
$\clA,$ then $\lambda$ is not essentially regular. Is there a real number $\lambda$ which is not essentially regular and such that some operator $H \in \clA$
has finite multiplicity in a neighbourhood of~$\lambda?$

\subsection{Some questions about resonance matrix}
In section \ref{S: res.index and sign res matrix} it was shown (Theorem \ref{T: res.ind=sign res.matrix}) that the finite-rank self-adjoint operators
$Q_{\lambda+i0}(r_\lambda)JP_{\lambda-i0}(r_\lambda)$ and $Q_{\lambda-i0}(r_\lambda)JP_{\lambda+i0}(r_\lambda)$
have equal signatures. In subsection \ref{SS: property S} it was shown that
if a real resonance point~$r_\lambda$ has the generic property~$S$ then these operators
are in fact equal and vice versa, but points without property~$S$ also exist.

Do the spectral measures of operators $Q_{\lambda+i0}(r_\lambda)JP_{\lambda-i0}(r_\lambda)$ and $Q_{\lambda-i0}(r_\lambda)JP_{\lambda+i0}(r_\lambda)$ coincide?
What meaning do eigenvalues of self-adjoint operators $Q_{\lambda-i0}(r_\lambda)JP_{\lambda+i0}(r_\lambda)$ and $Q_{\lambda+i0}(r_\lambda)JP_{\lambda-i0}(r_\lambda)$ have?

\subsection{Some questions about type I points}

In section \ref{S: res point of type I} it was shown that if~$r_\lambda$ is a resonance point of type~I,
then $P_{\lambda+i0}(r_\lambda)=P_{\lambda-i0}(r_\lambda)$ and in particular $\Upsilon_{\lambda+i0}(r_\lambda)=\Upsilon_{\lambda-i0}(r_\lambda).$

\begin{ques}
Does $P_{\lambda+i0}(r_\lambda)=P_{\lambda-i0}(r_\lambda)$ imply that~$r_\lambda$ is of type~I?
\end{ques}

This question is a special case of the following question:
if a vector belongs to both vector spaces $\Upsilon_{\lambda+i0}(r_\lambda)$ and $\Upsilon_{\lambda-i0}(r_\lambda),$ then is it true that
(1) $u$ is a vector of type I, or at least (2)
order of $u$ as an element of $\Upsilon_{\lambda+i0}(r_\lambda)$ and $\Upsilon_{\lambda-i0}(r_\lambda)$ is the same.



%

\subsection{On multiplicity of~$H_0$}
Recall that a self-adjoint operator~$H_0$ on a Hilbert space $\hilb$ has multiplicity $m,$ if~$m$ is the smallest
of positive integers~$k$ such that for some~$k$ vectors $f_1, \ldots, f_k,$ the linear span
of vectors~$H_0^i f_j,$ $i=1,2,\ldots$ and $j=1,\ldots,k,$ is dense in the Hilbert space~$\hilb.$
\begin{conj} \label{Conj: Geom mult of H0}
If a self-adjoint operator~$H_0 \in \clA$ has multiplicity~$m,$ then for every essentially regular number~$\lambda$ at which~$H_0$ is resonant
the dimension of the vector space~$\Upsilon^1_{\lambda+i0}(r_\lambda)$ is not larger than $m.$
\end{conj}
Combined with the U-turn Theorem \ref{T: U-turn for res index}, this conjecture would imply that the resonance index cannot be larger than the multiplicity of~$H_0$
for any regularizing perturbation~$V.$
This is a reasonable conjecture, since one would not expect the multiplicity of the singular spectrum to be larger than the multiplicity of~$H_0.$

\subsection{Resonance index as a function of perturbation}
In this paper a fixed perturbation~$V$ has been considered.
An open matter of study is the dependence of the resonance index $\ind_{res}(\lambda; H_0, V)$ on the perturbation~$V.$

Let~$H_0$ be resonant at an essentially regular point~$\lambda.$
A regularizing direction~$V$ will be called \emph{simple} if the resonance point~$r_\lambda = 0$ has order 1.
In this case $\Upsilon^1_{\lambda+i0}(r_\lambda) = \Upsilon_{\lambda+i0}(r_\lambda)$ and therefore
by Theorem \ref{T: mult of s.c. spectrum: drastic version} for simple directions $V$
the vector space $\Upsilon_{\lambda+i0}(r_\lambda)$ does not depend on~$V.$

\begin{conj} \label{Conj: simple reg-ing directions}
If~$H_0$ is resonant at an essentially regular point~$\lambda,$
then the set of simple directions is open in the norm of the vector space $\clA_0(F),$
given by $\norm{F^*JF}_{\clA_0} = \norm{J}.$ Moreover, the set of non-simple directions is a meager subset of $\clA_0(F).$
Moreover, the resonance index $\ind_{res}(\lambda; H_0, V)$ is stable under small perturbations of a simple direction~$V.$
\end{conj}

%
%

\subsection{Resonance lines and eigenvalues}
Recall that a pair of self-adjoint operators~$H$ and~$V$ is called \emph{reducible},
if there exists a non-zero proper (closed) subspace $\clL$ of the Hilbert space $\hilb,$
such that $H \clL \subset \clL$ and $V \clL \subset \clL.$

By Proposition~\ref{P: Az 4.1.10}, for every essentially regular point $\lambda \in \Lambda(\clA,F),$
the resonance set $R(\lambda; \clA, F)$ is an analytic set, in the sense that every analytic curve either intersects
the set $R(\lambda; \clA, F)$ at a discrete set of points or it entirely belongs to the set $R(\lambda; \clA, F).$
There is a distinguished class of analytic curves --- the straight lines.
We suggest that the straight lines $\set{H_0+rV \colon r \in \mbR}$ in the resonance set $R(\lambda; \clA,F)$ have a special meaning.
\begin{conj} \ If $\set{H_0+rV \colon r \in \mbR}$ is a resonant at~$\lambda$ line,
then~$\lambda$ is a common eigenvalue of all operators~$H_0+rV.$
\end{conj}
This is motivated by the fact that embedded eigenvalues are highly
unstable, and there has to be a reason for them not to get
dissolved under perturbations $rV$ for all $r \in \mbR.$ 

If $\set{H_0+rV \colon r \in \mbR}$ is a resonant at~$\lambda$ line, then as
simple finite-dimensional examples show the eigenvectors corresponding to~$\lambda$ may not in general be common for all
operators~$H_0+rV,$ $r \in \mbR.$

\subsection{On resonance points~$r_z$ as functions of $z$}
\subsubsection{On the analytic continuation of resonance points~$r_z$}
A resonance point $r_z$ corresponding to $z$ is a holomorphic function of $z.$
Here we write $r(z)$ instead of $r_z$ and call $r(z)$ a resonance function.
This function in general is multi-valued and it can have continuous branching points of a finite period;
examples can easily be constructed even in a finite-dimensional Hilbert space~$\hilb.$
A point $z_0$ of the complement of the essential spectrum will be called an \emph{absorbing point}
if $r(z) \to \infty$ as $z$ approaches $z_0$ along some half-interval $\gamma_1$ from the domain of holomorphy of~$r(z).$
It can be shown that if
$z_0$ is an absorbing point, then $r(z) \to \infty$ as $z$ approaches $z_0$ along any half-interval $\gamma_2$ from the domain of holomorphy of $r(z)$
which is homotopic to $\gamma_1$ in the domain of holomorphy. Recall that the domain of holomorphy of $r(z)$ is in general a multi-sheet Riemannian surface.

In the following conjecture we collect some questions regarding the function~$r(z).$

\begin{ques}
Are the following assertions true?
\begin{enumerate}
  \item If the limit $r_{\lambda+i0} := \lim_{y\to 0^+} r_{\lambda+iy}$ exists and is a real number, then as $y \to 0^+$ the number $r_{\lambda+iy}$ approaches
$r_{\lambda+i0}$ under a non-zero angle.
  \item Derivative of a resonance function $r(z)$ at a continuous branching point~$z_0$ is equal to~$\infty.$
  \item Let $r(z)$ be a resonance function. If $r(z)$ is holomorphic at a point $z_0$ (and does not branch at $z_0$) then
the derivative $r'(z_0)$ is not zero.
  \item If $z_0$ is a continuous branching point of a resonance function $r(z),$
then the inverse $z(r)$ of $r(z)$ is a single-valued function in a neighbourhood of $r_0 = r(z_0).$
  \item On any compact subset of $\mbC \setminus \sigma_{ess}$ a resonance function~$r_z$ can have only a finite number of isolated continuous branching points.
  In general, what can be said about the distribution of branching points of a resonance function $r_z?$
  \item A resonance function~$r(z)$ has a cycle of largest period~$d$ at a continuous branching point $z = z_0$ if and only if $r_{z_0}$ has order $d.$
  \item \label{Q: I7} Resonance functions do not have (a) non-real (b) real absorbing points, including isolated absorbing points.
  \item \label{Q: I8} Any resonance function $r(z)$ admits analytic continuation, possibly multi-valued, to the complement of the essential spectrum
with only one possible type of isolated singularity: continuous branching points of finite period.
\end{enumerate}
\end{ques}

Clearly, (\ref{Q: I8}) implies (\ref{Q: I7}).
In can be shown that these two statements are equivalent.

\subsubsection{On the splitting property of resonance points~$r_z$}
Let~$\lambda$ be an essentially regular point, let~$H_0$ be a self-adjoint operator from $\clA$
and let~$V$ be a regularizing direction at~$\lambda.$ Let~$r_\lambda$ be a real resonance point of the line $H_r = H_0+rV$
and let~$r_z^1, \ldots, r_z^N$ be the resonance points of the group of~$r_\lambda.$
\begin{conj} \label{Conj: rz are non-degenerate}
(1) If the pair $(H_0,V)$ is irreducible, then all resonance points~$r_z^1, \ldots, r_z^N$
of the group of~$r_\lambda$ considered as functions of~$z$ are non-degenerate.
More generally, for an irreducible pair every resonance point~$r_z$ as a function of~$z$ is non-degenerate.

(2) 
All resonance points~$r_z^1, \ldots, r_z^N$ of the group of~$r_\lambda$ considered as functions of~$z$ have order~1.
More generally, every resonance point~$r_z$ as a function of~$z$ has order 1.
\end{conj}

\subsubsection{Analytic continuation through gaps in the essential spectrum}
Assume that there is an island~$I$ in the essential spectrum, that is,~$I$ is a closed interval such that
for some $\eps>0$ the intersection of $\sigma_{ess}$ and $(a-\eps,b+\eps)$ is equal to $[a,b].$
Assume that a resonance function $r(z)$ can be continued analytically over the island.
The analytic continuation back to the initial point may differ from the original function, of course. What can be said about the period of this analytic continuation?

What can be said about an integral of~$r_z$ over a contour which encloses an island of essential spectrum?


\subsection{Mittag-Leffler representation of~$A_z(s)$}
Is it true that the function~$A_z(s)$ satisfies the equality
$$
  A_z(s) = \sum_{r_z} A_z(s) P_z(r_z),
$$
where the sum is taken over all resonance points~$r_z,$ and where the product
$A_z(s) P_z(r_z)$ stands for the Laurent series~(\ref{F: Az(s)Pz(rz)=sum...})?
Note that this assertion holds for finite-rank perturbations~$V,$ in which case the sum above is finite.
In general, though, this seems to be unlikely.

\subsection{On regular resonance vectors}
Theorem~\ref{T: res eq-n and eigenvectors} asserts that if~$\chi \in \hilb$ is an eigenvector of a resonant at~$\lambda$ operator~$H_0,$
then the vector $F\chi$ is a resonance vector of order~1. The resonance vector $F \chi$ is regular by definition.
\begin{conj} \label{Conj: res eq-n and eigenvectors}
If~$F\chi$ is a resonance vector of order~1, then~$\chi$ is an eigenvector of~$H_0.$
\end{conj}
This assertion is proved in Theorem~\ref{T: if lambda notin ess sp, ...} under an additional condition that~$\lambda$ does not belong to the essential spectrum.

\subsection{On singular ssf for trace-class perturbations}

Similarly to the definition of the singular spectral shift function one can define pure point and singular continuous spectral shift functions
as distributions by formulas
$$
  \xi^{(pp)}(\phi) = \int_0^1 \Tr (V\phi(H_r^{(pp)})\,dr \ \ \text{and} \  \xi^{(sc)}(\phi) = \int_0^1 \Tr (V\phi(H_r^{(sc)})\,dr, \ \ \phi \i C_c(\mbR),
$$
where $H_r^{(pp)}$ and $H_r^{(sc)}$ are pure point and singular continuous parts of $H_r$ respectively.
Clearly, $\xis = \xi^{(pp)} + \xi^{(sc)}.$

The density of distributions $\xi^{(pp)}$ and $\xi^{(sc)}$ we shall denote by the same symbols.
\begin{conj} Let $H_0$ be an arbitrary self-adjoint operator.
If $V$ is trace class then restriction of pure point spectral shift function for the pair $H_0$ and $H_0+V$ to the essential spectrum of $H_0$ is zero.
\end{conj}
That is, for trace class perturbations restriction of $\xis$ to $\sigma_{ess}$ coincides with $\xi^{(sc)}.$

\subsection{On pure point and singular continuous parts of resonance index}
Material of this subsection and motivation for it are based on Section \ref{S: Geom meaning of Ups(1)}.

In addition to our usual assumptions about operators $H_0,F$ and $V$ we assume that $V$ is a positive operator.
Let $r_\lambda$ be a resonance point of the triple $(\lambda, H_0, V).$
Since $V$ is positive, we have the equality
$$
  \ind_{res}(\lambda; H_{r_\lambda},V) = \dim \Upsilon_{\lambda+i0}(r_\lambda) = \dim \Upsilon^1_{\lambda+i0}(r_\lambda).
$$
We define the pure point and singular continuous parts of the resonance index by formulas
$$
  \ind^{(pp)}_{res}(\lambda; H_{r_\lambda},V) = \dim \clV_\lambda
$$
and
$$
  \ind^{(sc)}_{res}(\lambda; H_{r_\lambda},V) = \ind_{res}(\lambda; H_{r_\lambda},V) - \dim \clV_\lambda,
$$
where $\clV_\lambda$ is the vector space of eigenvectors of $H_{r_\lambda}$ corresponding to eigenvalue $\lambda.$
\begin{conj} For a.e. $\lambda$
$$
  \xi^{(pp)}(\lambda; H_1,H_0) = \sum_{r \in [0,1]} \ind^{(pp)}_{res}(\lambda; H_r,V).
$$
\end{conj}
This equality is equivalent to
$$
  \xi^{(sc)}(\lambda; H_1,H_0) = \sum_{r \in [0,1]} \ind^{(sc)}_{res}(\lambda; H_r,V).
$$

For non-sign-definite operators $V$ it is not clear how one can define the pure point and singular parts of the resonance index.

\subsection{On singular spectral shift function for relatively compact perturbations}
Singular spectral shift function is well-defined for relatively trace-class perturbations.
For such perturbations it admits three other descriptions as singular $\mu$-invariant, total resonance index
and total signature of resonance matrix. These three descriptions are well-defined for relatively compact perturbations too
provided that the limiting absorption principle holds, and in this case they are all equal, see Section \ref{S: res.index and sign res matrix}, Theorem \ref{T: res.ind=sign res.matrix}
and \cite{Az11}. While for relatively trace-class perturbations the spectral shift function in general is not defined
it is quite possible that in this case the singular spectral shift function still makes sense and is equal to the other three integer-valued functions.
Indeed, while the operator $VE_\Delta^{H}$ for a bounded Borel set $\Delta$ may fail to be trace-class it is still possible that
the operator $VE_\Delta^{H^{(s)}}$ is trace-class for a sufficiently large class of Borel sets $\Delta$ (for example for compact subsets $\Delta$ of $\Lambda(H_0,F)$).
This would allow to use modification of Birman-Solomyak formula to define the singular spectral shift function.
The second step would be to show that this function is integer-valued and is equal to the other three functions.

\begin{conj} (a) Assume that for an affine space $\clA$ of self-adjoint operators with rigging $F$
assumptions of section \ref{S: Preliminaries} are satisfied including the limiting absorption principle.
Let $H_0$ be a self-adjoint operator from $\clA$ and $V$ be a self-adjoint operator from the corresponding vector space~$\clA_0.$
For any compact subset $K$ of the set $\Lambda(H_0,F)$ the operator
$F E_K^{H^{(s)}}$ is trace class, where $H^{(s)}$ is the singular part of~$H=H_0+V.$

(b) The measure
$$
  K \mapsto \int_0^1 \Tr(V E_K^{H_r^{(s)}})\,dr
$$
is well-defined on compact subsets of $\Lambda(H_0,F)$ and is absolutely continuous.

(c) Density of this measure is a.e. integer-valued and coincides with total resonance index
of the pair $H_0$ and $H_1.$
\end{conj}

\section*{Index}
\noindent
$a_{j,\pm},$~(\ref{F: aj(pm)}), p.\,\pageref{Page: aj(pm)}.\\
$\clA$ affine space of self-adjoint operators,~(\ref{F: clA}), p.\,\pageref{Page: clA}. \\
$\clA_0=\clA_0(F)$ a vector space of self-adjoint perturbations, p.\,\pageref{Page: clA0}. \\
$A_z(s)$ non-self-adjoint compact operator,~(\ref{F: Az(s)}), p.\,\pageref{Page: Az(s)}. \\
$\ulA_z(s)$ non-self-adjoint compact operator,~(\ref{F: ulA and ulB}), p.\,\pageref{Page: ulAz(s)}. \\
$\tilde A_{z,r_z}(s)$ holomorphic part of Laurent expansion of $A_z(s),$~(\ref{F: Laurent for A+(s)}), p.\,\pageref{Page: tilde A}. \\
$\hat A_\pm,$~(\ref{F: u pm and A pm}), p.\,\pageref{Page: u pm and A pm}. \\
$\bfA_z(r_z)$ a finite rank nilpotent operator,~(\ref{F: def of bfA}), p.\,\pageref{Page: bfA(rz)}.\\
$\ubfA_z(r_z)$ a finite rank nilpotent operator,~(\ref{F: def of ubfA and ubfB}), p.\,\pageref{Page: ubfA(rz)}.\\
$\bfA_z(r_\lambda)$ a finite rank nilpotent operator,~(\ref{F: bfAz(rl)}), p.\,\pageref{Page: bfA(rl)}. \\
$\hat A_z(s),$~(\ref{F: hat Az(s)}), p.\,\pageref{Page: hat A}.\\
$B_z(s)$ non-self-adjoint compact operator,~(\ref{F: Bz(s)}), p.\,\pageref{Page: Bz(s)}. \\
$\ulB_z(s)$ non-self-adjoint compact operator,~(\ref{F: ulA and ulB}), p.\,\pageref{Page: ulAz(s)}. \\
$\hat B_\pm,$~(\ref{F: u pm and A pm}), p.\,\pageref{Page: u pm and A pm}. \\
$\bfB_z(r_z)$ a finite rank nilpotent operator,~(\ref{F: def of bfB}), p.\,\pageref{Page: bfB(rz)}.\\
$\ubfB_z(r_z)$ a finite rank nilpotent operator,~(\ref{F: def of ubfA and ubfB}), p.\,\pageref{Page: ubfA(rz)}.\\
$\bfB_z(r_\lambda)$ a finite rank nilpotent operator,~(\ref{F: bfBz(rl)}), p.\,\pageref{Page: bfB(rl)}. \\
$\hat B_z(s),$~(\ref{F: hat Az(s)}), p.\,\pageref{Page: hat A}.\\
$\mbC_+$ open upper complex half-plane. \\
$\mbC_-$ open lower complex half-plane. \\
$d$ order of a resonance point~$r_z,$~(\ref{F: d}), p.\,\pageref{Page: point of order d}. \\
$\order(u),$ order of a resonance vector $u,$ p.\,\pageref{Page: order of vector}.\\
$\euD_z(s),$ formula~(\ref{F: def of euA(s)}), p.\,\pageref{Page: euDz(s)}.\\
$\euE_\lambda$ evaluation operator,~(\ref{F: euE(F*psi)=sqrt Im T}), p.\,\pageref{Page: euE(lambda)}. \\
$e_+(H)$ the set of positive eigenvalues of \Schrodinger operator, p.\,\pageref{Page e+(H)}. \\
$F$ rigging operator,~(\ref{F: F}), p.\,\pageref{Page: F}. \\
$\euF_z(s)$~(\ref{F: euF}), p.\,\pageref{Page: euF}. \\
$\hilb$ the ``main'' Hilbert space. \\
$\hilb_\pm$ Hilbert spaces of the triple $(\hilb_+,\hilb,\hilb_-),$~(\ref{F: hilb+}), p.\,\pageref{Page: hilb+}. \\
$H, \ H_0, \ H_r, \ H_{r_\lambda}, \ H_s$ self-adjoint operators, elements of the affine space~$\clA.$ p.\,\pageref{Page: Hr}. \\
$\hat H_0$ p.\,\pageref{Page: hat H0}.\\
$\hat H_s,$~(\ref{F: hat Hs}), p.\,\pageref{Page: hat Hs}.\\
$\im(A)$ image of an operator~$A,$ p.\,\pageref{Page: im(A)}.\\
$\Im A$ imaginary part of an operator~$A,$ p.\,\pageref{Page: Re(A)}.\\
$\ind_{res}(\lambda; H, V)$ resonance index,~(\ref{F: res(ind)}), p.\,\pageref{Page: res(ind)}. \\
$J$ self-adjoint bounded operator on $\clK,$ p.\,\pageref{Page: J}. \\
$\clK$ the ``auxiliary'' Hilbert space. \\
$\clK_\pm$ Hilbert spaces of the triple $(\clK_+,\clK,\clK_-),$ p.\,\pageref{Page: hilb+}. \\
$\ker(A)$ kernel of an operator~$A,$ p.\,\pageref{Page: im(A)}.\\
$\clLs_z(r_z)$ a subspace of $\Upsilon_z(r_z),$ \pageref{Page: clLs}. \\
$\clLw_z(r_z)$ a subspace of $\clLs_z(r_z),$ \pageref{Page: clLw}. \\
$m$ the geometric multiplicity of a resonance point,~(\ref{F: N}), p.\,\pageref{Page: m}. \\
$N$ dimension of~$\Upsilon_{\lambda\pm i0}(r_\lambda),$ algebraic multiplicity of a resonance point,~(\ref{F: N}), pp.\,\pageref{Page: N}, \pageref{Page: N pm}. \\
$N_+$ the number of up-points in the group of a real resonance point~$r_\lambda,$ p.\,\pageref{Page: N pm}. \\
$N_-$ the number of down-points in the group of a real resonance point~$r_\lambda,$ p.\,\pageref{Page: N pm}. \\
$P_z(r_z)$ finite-rank idempotent operator, p.\,\pageref{Page: Pz(rz)}. \\
$\ulP_z(r_z)$ finite-rank idempotent operator, p.\,\pageref{Page: ulPz(rz)}. \\
$P_z(r_\lambda)$ finite-rank idempotent operator, p.\,\pageref{Page: Pz(rl)}. \\
$P_z^\uparrow(r_\lambda),$ $P_z^\downarrow(r_\lambda),$ p.\,\pageref{Page: P(up), P(down)}. \\
$Q_z(r_z)$ finite-rank idempotent operator, p.\,\pageref{Page: Qz(rz)}. \\
$\ulQ_z(r_z)$ finite-rank idempotent operator, p.\,\pageref{Page: ulQz(rz)}. \\
$Q_z(r_\lambda)$ finite-rank idempotent operator, p.\,\pageref{Page: Qz(rl)}. \\
$r$ a real number, a coupling constant, p.\,\pageref{Page: r}. \\
$r_\lambda$ a real resonance point, p.\,\pageref{Page: r(lambda)}. \\
$r_z$ resonance point, p.\,\pageref{Page: rz}. \\
$r^1_z, \ldots r_z^N,$ resonance points of the group of~$r_\lambda,$ p.\,\pageref{Page: rz1...rzN}. \\
$R_z(H_s)$ resolvent of $H_s,$ p.\,\pageref{Page: Rz(s)}. \\
$R(\lambda; \clA,F)$ the resonance set,~(\ref{F: R(lambda)}), p.\,\pageref{Page: R(l,A,F)}. \\
$R(\lambda; H_0,V)$ the resonance set, p.\,\pageref{Page: resonance set}. \\
$\Rindex$ $R$-index of a finite-rank operator without non-zero real eigenvalues, p.\,\pageref{Page: Rindex}. \\
$R$-index, the same as $\Rindex,$ p.\,\pageref{Page: Rindex}. \\
$\clR$ the class of all finite-rank operators without non-zero real eigenvalues, p.\,\pageref{Page: clR}. \\
$\clR_N, \ \clR_{\leq N},$ p.\,\pageref{Page: clR(N)}. \\
$\Re A$ real part of an operator~$A,$ p.\,\pageref{Page: Re(A)}.\\
$s$ a real or complex number, a coupling constant. \\
$\sign(A)$ signature of a finite-rank self-adjoint operator~$A,$~(\ref{F: sign(A)}), p.\,\pageref{Page: sign(A)}. \\
$T_z(H_s)$ non-self-adjoint compact operator,~(\ref{F: Tz(H)}), p.\,\pageref{Page: Tz(s)}. \\
$u$ a resonance vector, p.\,\pageref{Page: uu}. \\
$\hat u_z(s),$ formula~(\ref{F: uz(s)}), p.\,\pageref{Page: uz(s)}. \\
$\hatlmbu{j},$ formula~(\ref{F: vj(l+i0)(s)}), p.\,\pageref{Page: vj(l+i0)(s)}. \\
$\hat u_\pm,$~(\ref{F: u pm and A pm}), p.\,\pageref{Page: u pm and A pm}. \\
$V$ a self-adjoint operator from $\clA_0(F),$~(\ref{F: V}), p.\,\pageref{Page: V}.\\
$z = \lambda+iy$ complex number, an element of $\Pi.$ \\

\noindent
$\gamma(u), \gamma_z(u),$ depth of a resonance vector, p.\,\pageref{Page: depth of vector}.\\
$\lambda$ a real number, an element of the spectral line. \\
$\lambda\pm i0$ an element of $\partial \Pi_\pm.$ \\
$\Lambda(H,F)$~(\ref{F: Lambda(H,F)}), p.\,\pageref{Page: Lambda(H,F)}. \\
$\Lambda(\clA,F)$ the set of essentially regular points~$\lambda,$~(\ref{F: Lambda(clA,F)}), p.\,\pageref{Page: Lambda(clA,F)}. \\
$\Pi = \bar \mbC_+ \sqcup \bar \mbC_-, \ \Pi_\pm = \bar \mbC_\pm , \ \partial \Pi, \ \partial \Pi_\pm,$  p.\,\pageref{Page: Pi}. \\
$\sigma_z(s),$ an eigenvalue of~$A_z(s),$ p.\,\pageref{Page: sigma z(s)}. \\
$\sigma_A$ spectrum of an operator $A.$ \\
$\mu_A$ spectral measure of a compact operator~$A,$ \pageref{Page: mu(A)}. \\
$\Upsilon_z(r_z)$ vector space of resonance vectors, p.\,\pageref{Page: Upsilon}. \\
$\Upsilon_z(r_\lambda)$ vector space of resonance vectors, p.\,\pageref{Page: Upsilon(rl)}. \\
$\Upsilon^j_z(r_z)$ vector space of resonance vectors of order $j,$ p.\,\pageref{Page: Upsilon(j)}. \\
$\psi$ a co-resonance vector, p.\,\pageref{Page: psi}. \\
$\Psi_z(r_z)$ vector space of co-resonance vectors, p.\,\pageref{Page: Psi}. \\
$\Psi^j_z(r_z)$ vector space of co-resonance vectors of order $j,$ p.\,\pageref{Page: Psi(j)}. \\
$\Psi_z(r_\lambda)$ vector space of co-resonance vectors, p.\,\pageref{Page: Psi(rl)}. \\

\noindent
Co-resonance vector, p.\,\pageref{Page: cores vector}. \\
\mbox{ }\quad --- of order $j,$ p.\,\pageref{Page: cores vector}. \\
Essentially regular point, p.\,\pageref{Page: ess reg point}.\\
\mbox{ }\quad ---  line, p.\,\pageref{Page: ess regular line}.\\
Jordan decomposition, p.\,\pageref{Page: Jordan decomp-n}.\\
Operator resonant at~$\lambda,$~(\ref{F: op-r resonant at lambda}), p.\,\pageref{Page: op-r res at lambda} \\
Resonance index, p.\,\pageref{Page: res index}. \\
\mbox{ }\quad --- anti-down-point, p.\,\pageref{Page: up-point}\\
\mbox{ }\quad --- anti-up-point, p.\,\pageref{Page: up-point}\\
\mbox{ }\quad --- down-point, p.\,\pageref{Page: up-point}\\
\mbox{ }\quad --- matrix (of a finite set of resonance points), p.\,\pageref{Page: res matrix(1)} \\
\mbox{ }\quad --- point, p.\,\pageref{Page: res point}. \\
\mbox{ }\quad --- --- of order~$d,$ p.\,\pageref{Page: point of order d}\\
\mbox{ }\quad --- --- of type~I, p.\,\pageref{Page: point of type I}\\
\mbox{ }\quad --- --- with property~$C$, p.\,\pageref{Page: point with property C(2)}\\
\mbox{ }\quad --- --- with property~$P$, p.\,\pageref{Page: point with property P}\\
\mbox{ }\quad --- --- with property~$S$, p.\,\pageref{Page: point with property S}\\
\mbox{ }\quad --- --- with property~$U$, p.\,\pageref{Page: point with property U}\\
\mbox{ }\quad --- ---, Young diagram of, p.\,\pageref{Page: Young diagram}\\
\mbox{ }\quad --- up-point, p.\,\pageref{Page: up-point}\\
\mbox{ }\quad --- vector, p.\,\pageref{Page: res vector} \\
\mbox{ }\quad --- --- of depth $j,$ p.\,\pageref{Page: depth of vector} \\
\mbox{ }\quad --- --- of order $j,$ p.\,\pageref{Page: res vector} \\
\mbox{ }\quad --- --- of type~I, p.\,\pageref{Page: type I vector} \\
\mbox{ }\quad --- --- regular, p.\,\pageref{Page: regular vector} \\


\rndef{\emph}[1]{{\it #1}}

\mathsurround 0pt
\ndef{\AndSoOn}{$\dots$}



\end{document}